\documentclass{amsbook}%
\usepackage{graphicx}
\usepackage{amscd}
\usepackage{amsmath}
\usepackage{amsfonts}
\usepackage{amssymb}%
\providecommand{\U}[1]{\protect\rule{.1in}{.1in}}
\providecommand{\U}[1]{\protect\rule{.1in}{.1in}}
\providecommand{\U}[1]{\protect\rule{.1in}{.1in}}
\providecommand{\U}[1]{\protect\rule{.1in}{.1in}}
\theoremstyle{plain}

\newtheorem{condition}{Condition}

\newtheorem{corollary}{Corollary}

\newtheorem{definition}{Definition}
\newtheorem{example}{Example}

\newtheorem{lemma}{Lemma}

\newtheorem{proposition}{Proposition}
\newtheorem{remark}{Remark}

\newtheorem{theorem}{Theorem}
\numberwithin{equation}{section}

\begin{document}
\frontmatter
\title[Fractional Sobolev inequalities]{FRACTIONAL SOBOLEV INEQUALITIES: SYMMETRIZATION, ISOPERIMETRY AND INTERPOLATION}
\dedicatory{To Iolanda, Quim and Vanda}\author{Joaquim Mart\'{\i}n$^{\ast}$}
\address{Department of Mathematics\\
Universitat Aut\`{o}noma de Barcelona}
\email{jmartin@mat.uab.cat}
\author{Mario Milman**}
\address{Department of Mathematics\\
Florida Atlantic University\\
Boca Raton, Fl. 33431}
\email{mario.milman@gmail.com}
\urladdr{https://sites.google.com/site/mariomilman}
\thanks{$^{\ast}$Partially supported in part by Grants MTM2010-14946, MTM-2010-16232.}
\thanks{**This work was partially supported by a grant from the Simons Foundation
(\#207929 to Mario Milman).}
\thanks{This paper is in final form and no version of it will be submitted for
publication elsewhere.}

\begin{abstract}
We obtain new oscillation inequalities in metric spaces in terms of the Peetre
$K-$functional and the isoperimetric profile. Applications provided include a
detailed study of Fractional Sobolev inequalities and the Morrey-Sobolev
embedding theorems in different contexts. In particular we include a detailed
study of Gaussian measures as well as probability measures between Gaussian
and exponential. We show a kind of reverse P\'{o}lya-Szeg\"{o} principle that
allows us to obtain continuity as a self improvement from boundedness, using
symetrization inequalities. Our methods also allow for precise estimates of
growth envelopes of generalized Sobolev and Besov spaces on metric spaces. We
also consider embeddings into $BMO$ and their connection to Sobolev embeddings.

\end{abstract}
\maketitle
\tableofcontents

\section{Preface}

This paper is devoted to the study of fractional Sobolev inequalities and
Morrey-Sobolev type embedding theorems in metric spaces, using symmetrization.
The connection with isoperimetry plays a crucial role. The aim was to provide
a unified account and develop the theory in the general setting of metric
measure spaces whose isoperimetric profiles satisfy suitable assumptions. In
particular, the use of new pointwise inequalities for suitable defined moduli
of continuity allow us to treat in a unified way Euclidean and Gaussian
measures as well as a large class of different geometries. We also study the
role of isoperimetry in the estimation of $BMO$ oscillations. The connection
with Interpolation/Approximation theory also plays a crucial role in our
development and suggests further applications to optimization...

The authors are grateful to the institutions and agencies that supported their
research over the long course of time required to complete this work. In
particular, the second named author acknowledges a one semester sabbatical
provided by FAU.

\mainmatter

\chapter{Introduction\label{intro}}

\setcounter{section}{1} In this paper we establish general versions of
fractional Sobolev embeddings, including Morrey-Sobolev type embedding
theorems, in the context of metric spaces, using symmetrization methods. The
connection of the underlying inequalities with interpolation and isoperimetry
plays a crucial role.

We shall consider connected, measure metric spaces $\left(  \Omega
,d,\mu\right)  $ equipped with a finite Borel measure $\mu$. For measurable
functions $u:\Omega\rightarrow\mathbb{R},$ the distribution function is
defined by
\[
\mu_{u}(t)=\mu\{x\in{\Omega}:u(x)>t\}\text{ \ \ \ \ }(t\in\mathbb{R}).
\]
The signed \textbf{decreasing rearrangement} of $u,$ which we denote by
$u_{\mu}^{\ast},$ is the right-continuous non-increasing function from
$[0,\mu(\Omega))$ into $\mathbb{R}$ that is equimeasurable with $u;$ i.e.
$u_{\mu}^{\ast}$ satisfies
\[
\mu_{u}(t)=\mu\{x\in{\Omega}:u(x)>t\}=m(\left\{  s\in\lbrack0,\mu
(\Omega)):u_{\mu}^{\ast}(s)>t\right\}  )\text{ , \ }t\in\mathbb{R}%
\]
(where $m$ denotes the Lebesgue measure on $[0,\mu(\Omega)).$ The maximal
average$\ $of $u_{\mu}^{\ast}$ is defined by
\[
u_{\mu}^{\ast\ast}(t)=\frac{1}{t}\int_{0}^{t}u_{\mu}^{\ast}(s)ds,\text{
}(t>0).
\]

For a Borel set $A\subset\Omega,$ the \textbf{perimeter} or \textbf{Minkowski
content} of $A$ is defined by
\[
P(A;\Omega)=\lim\inf_{h\rightarrow0}\frac{\mu\left(  \left\{  x\in
\Omega:d(x,A)<h\right\}  \right)  -\mu\left(  A\right)  }{h}.
\]
The \textbf{isoperimetric profile} $I_{\Omega}(t),t\in(0,\mu(\Omega)),$ is
maximal with respect to the inequality
\begin{equation}
I_{\Omega}(\mu(A))\leq P(A;\Omega). \label{alaisi}%
\end{equation}

From now on we only consider connected metric measure spaces whose
isoperimetric profile $I_{\Omega}$ is zero at zero, continuous, concave and
symmetric around $\mu(\Omega)/2.$

The starting point of the discussion are the rearrangement
inequalities\footnote{For more detailed information we refer to Chapter
\ref{preliminar} below.} of \cite{mamiadv} and \cite{mamicon}, where we showed
that\footnote{See also the extensive list of references provided in
\cite{mamiadv}.}, under our current assumptions on the profile $I_{\Omega}$,
for all Lipschitz functions $f$ on $\Omega$ (briefly $f\in Lip(\Omega))$,
\begin{equation}
\left\vert f\right\vert _{\mu}^{\ast\ast}(t)-\left\vert f\right\vert _{\mu
}^{\ast}(t)\leq\frac{t}{I(t)}\left\vert \nabla f\right\vert _{\mu}^{\ast\ast
}(t),\text{ }0<t<\mu(\Omega),\text{ } \label{cuatro}%
\end{equation}
where%
\[
|\nabla f(x)|=\limsup_{d(x,y)\rightarrow0}\frac{|f(x)-f(y)|}{d(x,y)}.
\]
In fact, in \cite{mamiadv} we showed that (\ref{cuatro}) is equivalent to
(\ref{alaisi}).

Since the integrability properties do not change by rearrangements (i.e.
integrability properties are *rearrangement invariant*), rearrangement
inequalities are particularly useful to prove embeddings of Sobolev spaces
into rearrangement invariant spaces\footnote{Roughly speaking, a rearrangement
invariant space is a Banach function space where the norm of a function
depends only on the $\mu$-measure of its level sets.}. On the other hand, the
use of rearrangement inequalities to study smoothness of functions is harder
to implement. The main difficulty here is that while the classical
P\'{o}lya-Szeg\"{o} principle (cf. \cite{leoni}, \cite{bes}) roughly states
that symmetrizations are smoothing, i.e. they preserve the (up to first order)
smoothness of Sobolev/Besov functions, the converse does not hold in general.
In other words, it is not immediate how to deduce smoothness properties of $f$
from inequalities on $f_{\mu}^{\ast}.$ From this point of view, one could
describe some of the methods we develop in this paper as \textquotedblleft
suitable converses to the P\'{o}lya-Szeg\"{o} principle\textquotedblright.

As it turns out, related issues have been studied long ago, albeit in a less
general context, by A. Garsia and his collaborators. The original impetus of
Garsia's group was to study the path continuity of certain stochastic
processes (cf. \cite{garsia0}, \cite{garr}); a classical topic in Probability
theory. This task led Garsia et al. to obtain rearrangement inequalities for
general moduli of continuity, including $L^{p}$ or even Orlicz moduli of
continuity. Moreover, in \cite{garsiagrenoble}, \cite{garsiaind}, and
elsewhere (cf. \cite{garsia}), these symmetrization inequalities were also
applied to problems in Harmonic Analysis and, in particular, to study the
absolute convergence of Fourier series. From our point of view, a remarkable
aspect of the approach of Garsia et al. (cf. \cite{garsiaind}) is precisely
that the sought continuity can be recovered using rearrangement inequalities.
In other words, one can reinterpret this part of the Garsia-Rodemich analysis
as an approach to the Morrey-Sobolev embedding theorem using rearrangement inequalities.

It will be instructive to show how Garsia's analysis can be combined with
(\ref{cuatro}). To fix ideas we consider the setting of Garsia-Rodemich: The
metric measure space is ($(0,1)^{n},\left\vert \cdot\right\vert ,dx)$ (that is
$(0,1)^{n}$ provided with the Euclidean distance and Lebesgue measure). For
functions $f\in Lip(0,1)^{n}$ the inequality (\ref{cuatro}) takes the
following form\footnote{The rearrangement of $f$ with respect to the Lebesgue
measure is simply denoted by $f^{\ast}.$}%
\[
\left\vert f\right\vert ^{\ast\ast}(t)-\left\vert f\right\vert ^{\ast}(t)\leq
c_{n}\frac{t}{\min(t,1-t)^{1-1/n}}\left\vert \nabla f\right\vert ^{\ast\ast
}(t),\text{ \ }0<t<1.
\]
In fact (cf. Chapter \ref{capitap}), the previous inequality remains true for
all functions $f\in W_{L^{p}}^{1}(0,1)^{n}$ (where $1\leq p<\infty$)$,$ and
$W_{L^{p}}^{1}(0,1)^{n}$ is the Sobolev space of real-valued weakly
differentiable functions on $(0,1)^{n}$ whose first-order derivatives belong
to $L^{p}$). Moreover, as we shall see (cf. Chapter \ref{contchap}), the
inequality also holds for (signed) rearrangements; i.e.\ for all $f\in
W_{L^{p}}^{1}(0,1)^{n},$ we have that,%
\begin{equation}
f^{\ast\ast}(t)-f^{\ast}(t)\leq c_{n}\frac{t}{\min(t,1-t)^{1-1/n}}\left\vert
\nabla f\right\vert ^{\ast\ast}(t),\text{ \ }0<t<1.\label{cuatrodos}%
\end{equation}
Suppose that $p>n.$ Integrating, and using the fundamental theorem of
calculus\footnote{Recall that $\frac{d}{dt}\left(  f^{\ast\ast}(t)\right)
=-\frac{\left(  f^{\ast\ast}(t)-f^{\ast}(t)\right)  }{t}.$}$,$ we get
\begin{align}
f^{\ast\ast}(0)-f^{\ast\ast}(1) &  =\int_{0}^{1}\left(  f^{\ast\ast
}(t)-f^{\ast}(t)\right)  \frac{dt}{t}\nonumber\\
&  \leq c_{n}\int_{0}^{1}\left\vert \nabla f\right\vert ^{\ast\ast}%
(t)\frac{dt}{\min(t,1-t)^{1-1/n}}\nonumber\\
&  \leq c_{n,p}\left\Vert \left\vert \nabla f\right\vert \right\Vert _{L^{p}%
}\left\Vert \frac{1}{\min(t,1-t)^{1-1/n}}\right\Vert _{L^{p^{\prime}}}\text{
(by H\"{o}lder's inequality)}\nonumber\\
&  =C_{n.p}\left\Vert \left\vert \nabla f\right\vert \right\Vert _{L^{p}%
},\nonumber
\end{align}
where the last inequality follows from the fact that for $p>n,\left\Vert
\frac{1}{\min(t,1-t)^{1-1/n}}\right\Vert _{L^{p^{\prime}}}<\infty.$
Summarizing our findings, we have
\begin{equation}
ess\sup_{x\in(0,1)^{n}}f-\int_{0}^{1}f=f^{\ast\ast}(0)-f^{\ast\ast}(1)\leq
C_{n.p}\left\Vert \left\vert \nabla f\right\vert \right\Vert _{L^{p}%
}.\label{vale}%
\end{equation}
Applying (\ref{vale}) to $-f$ yields
\begin{equation}
\int_{0}^{1}f-ess\inf_{x\in(0,1)^{n}}f\leq C_{n.p}\left\Vert \left\vert \nabla
f\right\vert \right\Vert _{L^{p}}.\label{vale1}%
\end{equation}
Therefore, adding (\ref{vale}) and (\ref{vale1}) we obtain
\begin{equation}
Osc(f;(0,1)^{n}):=ess\sup_{x\in(0,1)^{n}}f-ess\inf_{x\in(0,1)^{n}}%
f\leq2C_{n.p}\left\Vert \left\vert \nabla f\right\vert \right\Vert _{L^{p}%
}.\label{persiste1}%
\end{equation}

We have shown that (\ref{cuatro}) gives us good control of the oscillation of
the original function on the whole cube $(0,1)^{n}.$ To control the
oscillation on any cube $Q\subset(0,1)^{n}$, we use a modification of an
argument that originates\footnote{In the original one dimensional argument
(cf. \cite{garsiagrenoble}, \cite{garsiaind}), one controls the oscillation of
$f$ in terms of an expression that involves the modulus of continuity, rather
than the gradient.} in the work of Garsia et al. (cf. \cite{garsiagrenoble}).
The idea is that if an inequality scales appropriately, one can re-scale.
Namely, given two fixed points $x<y\in(0,1)$, one can apply the inequality at
hand to the re-scaled function\footnote{For example, consider the case $n=1.$
Given $0<x<y<1,$ the inequality (\ref{persiste1}) applied to $\tilde{f}$
yields%
\begin{align*}
ess\sup_{[x,y]}f-ess\text{ }\inf_{[x,y]}f  &  \leq c_{p}\left\vert
y-x\right\vert \left(  \int_{0}^{1}\left\vert f^{\prime}(x+t(y-x))\right\vert
^{p}dt\right)  ^{1/p}\\
&  =c_{p}\left\vert y-x\right\vert ^{1-1/p}\left(  \int_{0}^{1}\left\vert
f^{\prime}\right\vert ^{p}\right)  ^{1/p}.
\end{align*}
} $\tilde{f}(t)=f(x+t(y-x)),t\in\lbrack0,1].$ This type of \textquotedblleft
change of scale argument\textquotedblright\ can be extended to the cube
$(0,1)^{n},$ but for general domains becomes unmanageable. Therefore, we
needed to reformulate the idea somewhat. From our point of view the idea is
that if we control $\left\vert \nabla f\right\vert $ on $(0,1)^{n}$ then we
ought to be able to control its restrictions$.$ The issue then becomes: How do
our inequalities scale under restrictions? Again for $(0,1)^{n}$ all goes
well. In fact, if $f\in W_{L^{p}}^{1}((0,1)^{n})$ then, for any open cube
$Q\subset(0,1)^{n},$ we have $f\chi_{Q}\in W_{L^{p}}^{1}(Q).$ Moreover, the
fundamental inequality (\ref{cuatrodos}) has the following scaling
\[
\left(  f\chi_{Q}\right)  ^{\ast\ast}(t)-\left(  f\chi_{Q}\right)  ^{\ast
}(t)\leq c_{n}\frac{t}{\min(t,\left\vert Q\right\vert -t)^{1-1/n}}\left\vert
\nabla\left(  f\chi_{Q}\right)  \right\vert ^{\ast\ast}(t),\text{ }0<t<|Q|.
\]
Using the previous argument applied to $f\chi_{Q}$ we thus obtain%
\[
Osc(f;Q)\leq c_{n,p}\left\Vert \frac{t}{\min(t,\left\vert Q\right\vert
-t)^{1-1/n}}\right\Vert _{L^{p^{\prime}}(0,|Q|)}\left\Vert \left\vert \nabla
f\right\vert \right\Vert _{L^{p}(Q)}.
\]
By computation, and a classical argument, it is easy to see from here that
(cf. Remark \ref{rema} in Chapter \ref{capitap})
\[
\left\vert f(y)-f(z)\right\vert \leq c_{n,p}\left\vert y-z\right\vert
^{(1-\frac{n}{p})}\left\Vert \left\vert \nabla f\right\vert \right\Vert
_{p},\text{a.e. }y,z.
\]
When $p=n$ this argument fails but, nevertheless, by a simple modification, it
yields a result due independently to Stein \cite{ST1} and C. P. Calder\'{o}n
\cite{calx}: Namely\footnote{We refer to [(\ref{montana}), (\ref{vermont}),
Chapter \ref{chapbmo}] for the definition of Lorentz spaces.}, if $\left\Vert
\left\vert \nabla f\right\vert \right\Vert _{L^{n,1}}<\infty,$ then $f$ is
essentially continuous (cf. [Remark \ref{rema}, Chapter \ref{capitap}]).

We will show that, with suitable technical adjustments, this method can be
extended to the metric setting\footnote{For a different approach to the
Morrey-Sobolev theorem on metric spaces we refer to the work of Coulhon
\cite{cou1}, \cite{cou}.}. To understand the issues involved let us note that,
since our inequalities are formulated in terms of isoperimetric profiles, to
achieve the local control or \textquotedblleft the change of scales (in our
situation through the restrictions)"\ we need suitable control of the
(relative) isoperimetric profiles on the (new) metric spaces obtained by
restriction. More precisely, if $Q\subset\Omega$ is an open set, we shall
consider the metric measure space $\left(  Q,d_{\mid Q},\mu_{\mid Q}\right)
.$ Then the problem we face is that, in general, the isoperimetric profile
$I_{Q}$ of $\left(  Q,d_{\mid Q},\mu_{\mid Q}\right)  $ is different from the
isoperimetric profile $I_{\Omega}$ of $\left(  \Omega,d,\mu\right)  .$ What we
needed is to control the \textquotedblleft relative isoperimetric
inequality\textquotedblright\footnote{Recall that given an open set
$G\subset\Omega$ , and a set $A\subset G,$ the \textbf{perimeter} of $A$
\textbf{relative} to $G$ (cf. Chapter \ref{preliminar}) is defined by
\[
P(A;G)=\lim\inf_{h\rightarrow0}\frac{\mu\left(  \left\{  x\in
G:d(x,A)<h\right\}  \right)  -\mu\left(  A\right)  }{h}.
\]
The corresponding \textbf{relative isoperimetric profile} of $G\subset\Omega$
is given by
\[
I_{G}(s)=I_{(G,d,\mu)}(s)=\inf\left\{  P(A;G):\text{ }A\subset G,\text{ }%
\mu(A)=s\right\}  .
\]
}, and make sure the corresponding inequalities scale appropriately. We say
that an \textbf{isoperimetric inequality relative to} $G$ holds, if there
exists a positive constant $C_{G}$ such that
\[
I_{G}(s)\geq C_{G}\min(I_{\Omega}(s),I_{\Omega}(\mu(G)-s)).
\]
We will say that a metric measure space $\left(  \Omega,d,\mu\right)  $ has
the \textbf{relative uniform isoperimetric property }if\ there is a constant
$C$ such that for any ball $B\ $in $\Omega,$ its \textbf{relative
isoperimetric profile }$I_{B}$ satisfies
\[
I_{B}(s)\geq C\min(I_{\Omega}(s),I_{\Omega}(\mu(B)-s)),\text{ \ \ }%
0<s<\mu(B).
\]

For metric spaces $(\Omega,d,\mu)$ satisfying the relative uniform
isoperimetric property we have the scaling that we need to apply the previous
analysis. This theme is developed in detail in Chapter \ref{contchap}. The
previous discussion implicitly shows that $(0,1)^{n}$ has the relative uniform
isoperimetric property. We shall show in Chapter \ref{capitap} that many
familiar metric measure spaces also have the relative uniform isoperimetric property.

We now turn to the main objective of this paper which is to develop the
corresponding theory for fractional order Besov-Sobolev spaces. This is,
indeed, the original setting of Garsia's work, and our aim in this paper is to
extend it to the metric setting. The first part of our program for Besov
spaces was to formulate a suitable replacement of (\ref{cuatro}) for the
fractional setting. To explain the peculiar form of the underlying
inequalities that we need requires some preliminary background information.

Let $X=X(\mathbb{R}^{n})$ be a rearrangement invariant space\footnote{See
[(\ref{secc:ri}), Chapter \ref{preliminar}].} on $\mathbb{R}^{n},$ and let
$\omega_{X}$ be the modulus of continuity associated with $X$ defined for
$g\in X$ by
\[
\omega_{X}\left(  t,g\right)  =\sup_{\left\vert h\right\vert \leq t}\left\Vert
g(\cdot+h)-g(\cdot)\right\Vert _{X}.
\]
It is known (for increasing levels of generality see \cite{garsiagrenoble},
\cite{ko}, \cite{mamiproc}, \cite{Mar} and the references therein), that there
exists $c=c_{n}>0$ such that, for all functions $f\in X(\mathbb{R}^{n}%
)+\dot{W}_{X}^{1}(\mathbb{R}^{n}),$%
\begin{equation}
\left\vert f\right\vert ^{\ast\ast}(t)-\left\vert f\right\vert ^{\ast}(t)\leq
c_{n}\frac{\omega_{X}\left(  t^{1/n},f\right)  }{\phi_{X}(t)},\text{ }t>0,
\label{tres}%
\end{equation}
where $\dot{W}_{X}^{1}(\mathbb{R}^{n})$ is the homogeneous Sobolev space
defined by means of the seminorm $\left\Vert u\right\Vert _{\dot{W}_{X}%
^{1}(\mathbb{R}^{n})}:=\left\Vert \left\vert \nabla u\right\vert \right\Vert
_{X(\mathbb{R}^{n})},$ $\phi_{X}(t)$ is the fundamental function\footnote{For
the definition see [(\ref{dasfundamental}), Chapter \ref{preliminar}] below.}
of $X,$ and $\left\vert f\right\vert ^{\ast}$ is the rearrangement of
$\left\vert f\right\vert $ with respect to the Lebesgue measure\footnote{In
the background of inequalities of this type lies a form of the
P\'{o}lya-Szeg\"{o} principle that states that symmetric rearrangements do not
increase Besov norms (cf. \cite{almlie}, \cite{mamiproc} and the references
therein).}.

The inequalities we seek are extensions of (\ref{tres}) to the metric setting.
Note that, in some sense, one can consider (\ref{tres}) as an extension, by
interpolation, of (\ref{cuatro}). Therefore, it is natural to ask: How should
(\ref{tres}) be reformulated in order to make sense for metric spaces? Not
only we need a suitable substitute for the modulus of continuity $\omega_{X},$
but a suitable re-interpretation of the factor $``t^{1/n}"$ is required as
well. We now discuss these issues in detail.

There are several known alternative, although possibly non equivalent,
definitions of modulus of continuity in the general setting of metric measure
spaces $(\Omega,d,\mu)$ (cf. \cite{Akos} for the interpolation properties of
Besov spaces on metric spaces). Given our background on approximation theory,
it was natural for us to choose the universal object that is provided by
interpolation/approximation theory, namely the Peetre $K-$functional. Indeed
on $\mathbb{R}^{n},$ the Peetre $K-$functional is defined by:
\[
K(t,f;X(\mathbb{R}^{n}),\dot{W}_{X}^{1}(\mathbb{R}^{n})):=\inf\{\left\Vert
f-g\right\Vert _{X}+t\left\Vert \left\vert \nabla g\right\vert \right\Vert
_{X}:g\in\dot{W}_{X}^{1}(\mathbb{R}^{n})\}.
\]
Considering the $K-$functional is justified since it is well known
that\footnote{Here the symbol $f\simeq g$ indicates the existence of a
universal constant $c>0$ (independent of all parameters involved) such that
$(1/c)f\leq g\leq c\,f$. Likewise the symbol $f\preceq g$ will mean that there
exists a universal constant $c>0$ (independent of all parameters involved)
such that $f\leq c\,g$.} (cf. \cite[Chapter 5, formula (4.41)]{bs})
\[
K(t,f;X(\mathbb{R}^{n}),\dot{W}_{X}^{1}(\mathbb{R}^{n}))\simeq\omega
_{X}(t,f).
\]
In the general case of metric measure spaces $(\Omega,d,\mu)$ we shall
consider:
\begin{align*}
K(t,f;X(\Omega),S_{X}(\Omega))  &  :=\\
\inf\{\left\Vert f-g\right\Vert _{X(\Omega)}+t\left\Vert \left\vert \nabla
g\right\vert \right\Vert _{X(\Omega)}  &  :g\in S_{X}(\Omega)\},
\end{align*}
where $X(\Omega)$ is a r.i. space on $\Omega,$ and $S_{X}(\Omega)=\{f\in
Lip(\Omega):$ $\left\Vert \left\vert \nabla f\right\vert \right\Vert
_{X(\Omega)}<\infty\}.$ We shall thus think of $K(t,f;X(\Omega),S_{X}%
(\Omega))$ as \textquotedblleft a modulus of continuity\textquotedblright.

Now, given our experience with the inequality (\ref{cuatro}), we were led to
conjecture the following reformulation\footnote{At least for the metric
measure spaces $(\Omega,d,\mu)$ considered in \cite{mamiadv} (for which, in
particular, (\ref{cuatro}) holds).} of (\ref{tres}): There exists a universal
constant $c>0,$ such that for every r.i. space $X(\Omega),$ and for all $f\in
X(\Omega)+S_{X}(\Omega),$ we have
\begin{equation}
\left\vert f\right\vert _{\mu}^{\ast\ast}(t)-\left\vert f\right\vert _{\mu
}^{\ast}(t)\leq c\frac{K\left(  \frac{t}{I_{\Omega}(t)},f;X(\Omega
),S_{X}(\Omega)\right)  }{\phi_{X}(t)},\text{ }0<t<\mu(\Omega). \label{seis}%
\end{equation}
We presented this conjectural inequality when lecturing on the topic. In
particular, we communicated the conjecture to M. Mastylo, who recently proved
in \cite{mastylo} that indeed (\ref{seis}) holds for $t\in(0,\mu(\Omega)/4),$
and for all rearrangement invariant spaces $X$ that are \textquotedblleft far
away from $L^{1}$ and from $L^{\infty}"$.

The result of \cite{mastylo}, while in many respects satisfying, leaves some
important questions open. Indeed, the restrictions placed on the range of $t$
(i.e. the measure of the sets), as well as those placed on the spaces,
precludes the investigation of the isoperimetric nature of (\ref{seis}). In
particular, while the equivalence of (\ref{cuatro}) and the isoperimetric
inequality (\ref{alaisi}) is known to hold (cf. \cite{mamiadv}), the possible
equivalence of (\ref{seis}) with the isoperimetric inequality (\ref{alaisi})
apparently cannot be answered without involving the space $L^{1}$.

One of our main results in Chapter \ref{main} shows that (\ref{seis})
crucially holds for all $t\in(0,\mu(\Omega)/2)$ and without restrictions on
the function spaces $X$ $\ $(cf. [Theorem \ref{main1}, Chapter \ref{main}]
below). The possibility of including $X=L^{1}$ allows us to prove the
following fractional Sobolev version of the celebrated Maz'ya
equivalence\footnote{Which claims the equivalence between the
Gagliardo-Nirenberg inequality and the isoperimetric inequality.} (cf.
\cite{ma})).

\begin{theorem}
\label{ref teo 1} Let $(\Omega,d,\mu)$ be a metric measure space that
satisfies our standard assumptions. Then,

(i) (cf. Theorem \ref{main1}) For all rearrangement invariant spaces
$X(\Omega)$, and for all $f\in X(\Omega)+S_{X}(\Omega),$
\begin{equation}
\left\vert f\right\vert _{\mu}^{\ast\ast}(t)-\left\vert f\right\vert _{\mu
}^{\ast}(t)\leq16\frac{K\left(  \frac{t}{I_{\Omega}(t)},f;X(\Omega
),S_{X}(\Omega)\right)  }{\phi_{X}(t)},\text{ }t\in(0,\mu(\Omega)/2),
\label{desiK}%
\end{equation}

\begin{equation}
\left\vert f-f_{\Omega}\right\vert _{\mu}^{\ast\ast}(t)-\left\vert
f-f_{\Omega}\right\vert _{\mu}^{\ast}(t)\leq16\frac{K\left(  \frac
{t}{I_{\Omega}(t)},f;X(\Omega),S_{X}(\Omega)\right)  }{\phi_{X}(t)},\text{
}t\in(0,\mu(\Omega)), \label{desiK9}%
\end{equation}
where
\begin{equation}
f_{\Omega}=\frac{1}{\mu(\Omega)}\int_{\Omega}fd\mu. \label{average}%
\end{equation}

(ii) (cf. Theorem \ref{ref teo 2}) Conversely, suppose that $G:(0,\mu
(\Omega))\rightarrow\mathbb{R}_{+},$ is a continuous function, which is
concave and symmetric around $\mu(\Omega)/2,$ and there exists a constant
$c>0$ such that\footnote{In other words we assume that (\ref{desiK}) holds for
$X=L^{1}(\Omega),$ and with $\frac{t}{G(t)}$ replacing $\frac{t}{I_{\Omega
}(t)}.$} for all $f\in X(\Omega)+S_{X}(\Omega),$%
\[
\left\vert f\right\vert _{\mu}^{\ast\ast}(t)-\left\vert f\right\vert _{\mu
}^{\ast}(t)\leq c\frac{K\left(  \frac{t}{G(t)},f;X(\Omega),S_{X}%
(\Omega)\right)  }{t},\text{ }t\in(0,\mu(\Omega)/2).
\]
Then, for all Borel sets with $\mu(A)\leq\mu(\Omega)/2,$ we have the
isoperimetric inequality
\[
G(\mu(A))\leq cP(A,\Omega).
\]
As a consequence, there exists a constant $c>0$ such that for all $t\in
(0,\mu(\Omega)),$%
\[
G(t)\leq cI_{\Omega}(t).
\]

\end{theorem}

Using (\ref{desiK}) and (\ref{desiK9}) as a starting point we can study
embeddings of Besov spaces in metric spaces and, in particular, the
corresponding Morrey-Sobolev-Besov embedding.

We now focus the discussion on the fractional Morrey-Sobolev theorem. We start
by describing the inequalities of \cite{garsiaind}, \cite{garsia}, which for
$L^{p}$ spaces\footnote{Importantly, Garsia-Rodemich also can deal with
$X=L^{p},$ or $X=L_{A}$ (Orlicz space), our approach covers all r.i. spaces
and works for a large class of metric spaces.} on $[0,1]$ take the following
form%
\begin{equation}
\left.
\begin{array}
[c]{c}%
f^{\ast}(x)-f^{\ast}(1/2)\\
f^{\ast}(1/2)-f^{\ast}(1-x)
\end{array}
\right\}  \leq c\int_{x}^{1}\frac{\omega_{L^{p}}(t,f)}{t^{1/p}}\frac{dt}%
{t},\text{ }x\in\left(  0,\frac{1}{2}\right]  . \label{garintro1}%
\end{equation}
Letting $x\rightarrow0$ in (\ref{garintro1}) and adding the two inequalities
then yields%
\[
ess\sup_{[0,1]}f-ess\text{ }\inf_{[0,1]}f\leq c\int_{0}^{1}\frac{\omega
_{L^{p}}(t,f)}{t^{1/p}}\frac{dt}{t}.
\]
Using the *change of scale argument* leads to%
\[
\left|  f(x)-f(y)\right|  \leq2c\int_{0}^{\left|  x-y\right|  }\frac
{\omega_{L^{p}}(t,f)}{t^{1/p}}\frac{dt}{t};\text{ }x,y\in\lbrack0,1],
\]
from which the essential continuity\footnote{An application of H\"{o}lder's
inequality also yields Lip conditions.} of $f$ is apparent if we know that
$\int_{0}^{1}\frac{\omega_{L^{p}}(t,f)}{t^{1/p}}\frac{dt}{t}<\infty.$ To
obtain the $n-$dimensional version of (\ref{garintro1}) for $[0,1]^{n}$,
Garsia et al. had to develop deep combinatorial techniques. The corresponding
$n-$dimensional inequality is given by (cf. \cite[(3.6)]{garsia} and the
references therein)%
\[
\left.
\begin{array}
[c]{c}%
f^{\ast}(x)-f^{\ast}(1/2)\\
f^{\ast}(1/2)-f^{\ast}(1-x)
\end{array}
\right\}  \leq c\int_{x}^{1}\frac{\omega_{L^{p}}(t^{1/n},f)}{t^{1/p}}\frac
{dt}{t},\text{ }x\in x\in\left(  0,\frac{1}{2}\right]  ,
\]
which by the now familiar argument yields%
\begin{equation}
\left|  f(x)-f(y)\right|  \leq C_{p,n}\int_{0}^{\left|  x-y\right|  }%
\frac{\omega_{L^{p}}(t,f)}{t^{n/p}}\frac{dt}{t};\text{ }x,y\in\left[
0,\frac{1}{2}\right]  ^{n}. \label{garintro4}%
\end{equation}
However, as pointed out above, the change of scale technique is apparently not
available for more general domains. Moreover, as witnessed by the difficulties
already encountered by Garsia et al. when proving inequalities on
$n-$dimensional cubes, it was not even clear at that time what form the
rearrangement inequalities would take in general. In particular, Garsia et al.
do not use isoperimetry.

For more general function spaces we need to reformulate Theorem
\ref{ref teo 1} above as follows (cf. Chapter \ref{main})

\begin{theorem}
\label{teo2} (cf. [Theorem \ref{main2}, Chapter \ref{main}]) Let $\left(
\Omega,d,\mu\right)  $ be a metric measure space that satisfies our standard
assumptions, and let $X$ be a r.i. space on $\Omega.$ Then, there exists a
constant $c>0$ such that for all $f\in X+S_{X}(\Omega),$%
\begin{equation}
\left\vert f\right\vert _{\mu}^{\ast\ast}(t/2)-\left\vert f\right\vert _{\mu
}^{\ast}(t/2)\leq c\frac{K\left(  \psi(t),f;X,S_{X}(\Omega)\right)  }{\phi
_{X}(t)},\text{ }0<t<\mu(\Omega), \label{desiKK}%
\end{equation}
where%
\[
\psi(t)=\frac{\phi_{X}(t)}{t}\left\Vert \frac{s}{I_{\Omega}(s)}\chi
_{(0,t)}(s)\right\Vert _{\bar{X}^{^{\prime}}},
\]
and $\bar{X}^{^{\prime}}$ denotes the associated\footnote{[cf. Chapter
\ref{preliminar} for the definition]} space of $\bar{X}$.
\end{theorem}

\begin{remark}
Note that when $X=L^{1}$ the inequalities (\ref{desiKK}) and (\ref{desiK}) are
equivalent, modulo constants.
\end{remark}

A second step of our program is the routine, but crucially important,
reformulation of rearrangement inequalities using signed rearrangements. Once
this is done, our generalized Morrey-Sobolev-Garsia-Rodemich theorem can be
stated as follows

\begin{theorem}
\label{teo3} (cf. [Theorem \ref{continuo}, Chapter \ref{contchap}]) Let
$\left(  \Omega,d,\mu\right)  $ be a metric measure space that satisfies our
standard assumptions and has the relative uniform isoperimetric
property\textbf{. }Let $X$ be a r.i. space in $\Omega$ such that
\[
\left\Vert \frac{1}{I_{\Omega}(s)}\right\Vert _{\bar{X}^{^{\prime}}}<\infty.
\]
If $f\in X+S_{X}(\Omega)$ satisfies
\[
\int_{0}^{\mu(\Omega)}\frac{K\left(  \phi_{X}(t)\left\Vert \frac{1}{I_{\Omega
}(s)}\chi_{(0,t)}(s)\right\Vert _{\bar{X}^{^{\prime}}},f;X,S_{X}%
(\Omega)\right)  }{\phi_{_{X}}(t)}\frac{dt}{t}<\infty,
\]
then $f$ is essentially bounded and essentially continuous.
\end{theorem}

To see the connection\footnote{The full metric version of Garsia's inequality
is given in Chapter \ref{Garetal}.} between our inequalities and those of
Garsia et al. let us fix ideas and set $\mu(\Omega)=1.$ Observe that on
account of (\ref{desiK}), and the fact that $K\left(  \frac{t}{I(t)}%
,f;X,S_{X}\right)  $ is increasing and $\phi_{X}(t)$ is concave, we have%
\begin{align*}
f_{\mu}^{\ast\ast}(t)-f_{\mu}^{\ast}(t)  &  \leq16\frac{K\left(  \frac
{t}{I_{\Omega}(t)},f;X,S_{X}\right)  }{\phi_{X}(t)}\\
&  \leq\frac{16}{\ln2}2\int_{t}^{2t}\frac{K\left(  \frac{s}{I_{\Omega}%
(s)},f;X,S_{X}\right)  }{\phi_{X}(s)}\frac{ds}{s},\text{ }t\in(0,1/2].
\end{align*}
Combining the last inequality with (cf. \cite[(4.1) pag 1222]{bmr})
\begin{equation}
f_{\mu}^{\ast}\left(  \frac{t}{2}\right)  -f_{\mu}^{\ast}(t)\leq2\left(
f_{\mu}^{\ast\ast}(t)-f_{\mu}^{\ast}(t)\right)  , \label{davobonee}%
\end{equation}
yields
\[
f_{\mu}^{\ast}\left(  \frac{t}{2}\right)  -f_{\mu}^{\ast}(t)\leq\frac{16}%
{\ln2}2\int_{t}^{2t}\frac{K\left(  \frac{s}{I_{\Omega}(s)},f;X,S_{X}\right)
}{\phi_{X}(s)}\frac{ds}{s},\text{ }t\in(0,1/2].
\]
Therefore, for $n=2,....$%
\[
f_{\mu}^{\ast}\left(  \frac{1}{2^{n+1}}\right)  -f_{\mu}^{\ast}\left(
\frac{1}{2}\right)  =%
{\displaystyle\sum\limits_{j=1}^{n}}
f_{\mu}^{\ast}\left(  \frac{1}{2^{j+1}}\right)  -f_{\mu}^{\ast}\left(
\frac{1}{2^{j}}\right)  \leq\frac{16}{\ln2}2\int_{0}^{1}\frac{K\left(
\frac{s}{I_{\Omega}(s)},f;X,S_{X}\right)  }{\phi_{X}(s)}\frac{ds}{s},
\]
and letting $n\rightarrow\infty,$ we find%
\[
ess\sup_{\Omega}f-f_{\mu}^{\ast}\left(  \frac{1}{2}\right)  \leq c\int_{0}%
^{1}\frac{K\left(  \frac{s}{I_{\Omega}(s)},f;X,S_{X}\right)  }{\phi_{X}%
(s)}\frac{ds}{s}.
\]
Likewise, applying the previous inequality to $-f$ $\ $and adding the two
resulting inequalities, and then observing that $(-f)_{\mu}^{\ast}(\frac{1}%
{2})=-f_{\mu}^{\ast}(\frac{1}{2}),$ yields%
\[
ess\sup_{\Omega}f-ess\text{ }\inf_{\Omega}f\leq c\int_{0}^{1}\frac{K\left(
\frac{s}{I_{\Omega}(s)},f;X,S_{X}\right)  }{\phi_{X}(s)}\frac{ds}{s}.
\]
From this point we can proceed to study the continuity or Lip properties of
$f$ using the arguments\footnote{Interestingly the one dimensional case
studied by Garsia et al somehow does not follow directly since the
isoperimetric profile for the unit interval $(0,1)$ is $1.$ Therefore in this
case the isoperimetric profile does not satisfy the assumptions of
\cite{mamiadv}. Nevertheless, our inequalities remain true and provide an
alternate approach to the Garsia inequalities. See Chapter \ref{Garetal}%
\ below for complete details.} outlined above.

Next let us consider limiting cases of the Sobolev-Besov embeddings connected
with these inequalities and the role of $BMO.$ Note that the inequality
(\ref{garintro4}) can be reformulated as the embedding of $B_{p}%
^{n/p,1}([0,1]^{n})$ into the space of continuous functions $C([0,1]^{n})$%

\begin{equation}
B_{p}^{n/p,1}([0,1]^{n})\subset C([0,1]^{n}),\text{where }%
n/p<1.\label{extienda}%
\end{equation}
Moreover, since%
\[
\omega_{L^{p}}\left(  t,f\right)  \leq c_{p,n}t\left\Vert f\right\Vert
_{W_{L^{p}}^{1}},
\]
we also have%
\[
W_{L^{p}}^{1}\left(  [0,1]^{n}\right)  \subset B_{p}^{n/p,1}([0,1]^{n}).
\]
Therefore, (\ref{extienda}) implies the (Morrey-Sobolev) continuity of Sobolev
functions in $W_{p}^{1}$ when $p>n.$ On the other hand, if we consider the
Besov condition defined by the right hand side of (\ref{tres}), when $X=L^{p}%
$, and $n/p<1,$ we find\footnote{Note that we have%
\[
\left\Vert f\right\Vert _{B_{p}^{n/p,\infty}([0,1]^{n})}\leq c_{p,n}\left\Vert
f\right\Vert _{B_{p}^{n/p,1}([0,1]^{n})}.
\]
}%
\[
\left\Vert f\right\Vert _{B_{p}^{n/p,\infty}([0,1]^{n})}=\sup_{t\in
\lbrack0,1]}\frac{\omega_{L^{p}}\left(  t,f\right)  }{t^{n/p}}.
\]
Now, for functions in $B_{p}^{n/p,\infty}([0,1]^{n})$ we don't expect
boundedness, and in fact, apparently the best we can say directly from our
rearrangement inequalities, follows from (\ref{tres}):%
\begin{equation}
\sup_{\lbrack0,1]}\left(  f^{\ast\ast}(t)-f^{\ast}(t)\right)  \leq c\left\Vert
f\right\Vert _{B_{p}^{n/p,\infty}([0,1]^{n})}.\label{ufa}%
\end{equation}
In view of the celebrated result of Bennett-DeVore-Sharpley \cite{bds} (cf.
also \cite{bs}) that characterizes the rearrangement invariant hull of $BMO$
via the left hand side of (\ref{ufa}), we see that (\ref{ufa}) gives $f\in
B_{p}^{n/p,\infty}([0,1]^{n})\Rightarrow f^{\ast}\in BMO[0,1].$ In fact, a
stronger result is known and readily available%
\begin{equation}
B_{p}^{n/p,\infty}([0,1]^{n})\subset BMO.\label{geometria}%
\end{equation}
It turns out that our approach to the estimation of oscillations allows us to
extend (\ref{geometria}) to other geometries. Our method reflects the
remarkable connections between oscillation, isoperimetry, interpolation, and
rescalings. We briefly explain the ideas behind our approach to
(\ref{geometria}). Given a metric measure space $\left(  \Omega,d,\mu\right)
$ satisfying our standard assumptions, we have the well known Poincar\'{e}
inequality (cf. \cite[page 150]{mamiadv} and the references therein) given by
\begin{equation}
\int_{{\Omega}}\left\vert f(x)-m(f)\right\vert d\mu\leq\frac{\mu(\Omega
)}{2I_{\Omega}(\mu(\Omega)/2)}\int_{{\Omega}}\left\vert \nabla f(x)\right\vert
d\mu,\text{ for all }f\in S_{L^{1}}(\Omega),\label{poitier}%
\end{equation}
where $m(f)$ is a median\footnote{For the definition of median see
[(\ref{mediandef}), Chapter \ref{main}].} of $f.$ Then given a r.i. space $X$
on $\Omega$ we can extend (\ref{poitier}) by \textquotedblleft by
interpolation\textquotedblright\ and obtain a $K-$Poincar\'{e}
inequality\footnote{See (\ref{average}) above.} (cf. [Theorem \ref{teobmo},
Chapter \ref{main}])
\[
\frac{1}{\mu(\Omega)}\int_{\Omega}\left\vert f-f_{\Omega}\right\vert d\mu\leq
c\frac{K\left(  \frac{\mu(\Omega)/2}{I_{\Omega}(\mu(\Omega)/2)},f;X,S_{X}%
\right)  }{\phi_{X}(\mu(\Omega))}.
\]
Now, if $\Omega$ supports the isoperimetric rescalings described above, the
previous inequality self improves to (cf. [Theorem \ref{bmomarkao}, chapter
\ref{chapbmo}])%
\[
\left\Vert f\right\Vert _{BMO(\Omega)}=\sup_{B\text{ balls}}\frac{1}{\mu
(B)}\int_{B}\left\vert f-f_{B}\right\vert d\mu\leq c\sup_{t<\mu(\Omega)}%
\frac{K\left(  \frac{t}{I_{\Omega}(t)},f;X,S_{X}\right)  }{\phi_{X}(t)}.
\]
It is easy to see that for metric spaces with Euclidean type isoperimetric
profiles, i.e. $I_{\Omega}(t)\succeq t^{1-1/n},$ on $(0,1/2),$ we recover
(\ref{geometria}) (cf. [Corollary \ref{corolariomarkao}, Chapter
\ref{chapbmo}]). Indeed, the result exhibits a new connection between the
geometry of the ambient space and the embedding of Besov and BMO spaces. For
further examples we refer to Chapter \ref{chapbmo} (cf. [(\ref{dependitalpha},
Chapter \ref{chapbmo}]).

In the opposite direction, we can use the insights we gained ``interpolating
between a r.i. space $X$ and the corresponding space of Lip functions $S_{X}"$
to obtain analogous results interpolating with $BMO.$ This is not a
coincidence; for recall the well known Euclidean interpretation of $BMO$ as a
limiting $Lip$ condition. This can be seen by means of writing Lip$_{\alpha}$
conditions on a fixed Euclidean cube $Q$ as%
\[
\left\|  f\right\|  _{Lip_{\alpha}}\simeq\sup_{\substack{Q^{\prime}\subset
Q\\Q^{\prime}\text{ cub}}}\frac{1}{\left|  Q^{\prime}\right|  ^{1-\alpha/n}%
}\int_{Q^{\prime}}\left|  f-f_{Q^{\prime}}\right|  dx<\infty.
\]
In this fashion $BMO$ appears as the limiting case of $Lip_{\alpha}$
conditions when $\alpha\rightarrow0.$ With this intuition at hand we were led
to formulate the corresponding version of Theorem \ref{ref teo 1} for $BMO.$
It reads as follows\footnote{In particular, we arrive, albeit through a very
different route than the original, to an extension of an inequality of
Bennett-DeVore-Sharpley (cf. \cite[combine Theorem 7.3 and Theorem 8.8]{bs})
for the space $L^{1}.$}

\begin{theorem}
(cf. [Theorem \ref{teoremarkao}, Chapter \ref{chapbmo}]) Suppose that
$(\Omega,d,\mu)$ is a metric measure space that satisfies our usual
assumptions and, moreover, is such that the Bennett-DeVore-Sharpley inequality%
\begin{equation}
\sup_{t}\left(  f_{\mu}^{\ast\ast}(t)-f_{\mu}^{\ast}(t)\right)  \leq
c\left\Vert f\right\Vert _{BMO(\Omega)}, \label{otravez}%
\end{equation}
holds. Then,
\[
f_{\mu}^{\ast\ast}(t)-f_{\mu}^{\ast}(t)\leq c\frac{K(\phi_{X}(t),f;X(\Omega
),BMO(\Omega))}{\phi_{X}(t)},\,0<t<\mu(\Omega).
\]

\end{theorem}

From this point of view the Bennett-DeVore-Sharpley inequality (\ref{otravez})
takes the role of our basic inequality (\ref{cuatro}). In this respect it is
important to note that (\ref{otravez}) has been shown to hold in great
generality, for example it holds for doubling measures (cf. \cite{saghschv}).
Finally, in connection with $BMO,$ we considered the role of signed
rearrangements$.$ Here the import of this notion is that signed rearrangements
provide a theoretical method to compute medians (cf. Theorem \ref{median}) and
thus quickly lead to a version of the limiting case of the
John-Stromberg-Jawerth-Torchinsky inequality (cf. Chapter \ref{chapbmo},
(\ref{mediana})).

Our approach to the (Morrey-)Sobolev embedding theorem also leads to the
consideration of \textquotedblleft Lorentz spaces with negative
indices\textquotedblright\ (cf. Chapter \ref{negativo}), providing still a
very suggestive approach\footnote{Although in this paper we only concentrate
on the role that these spaces play on the theory of embeddings, one cannot but
feel that a detailed study of these spaces could be useful for other questions
connected with interpolation/approximation.} to these results, at least in the
Euclidean case.

In Chapter \ref{capitap} and Chapter \ref{chapgauss} we have considered
explicit versions of our results in different classical contexts. In
particular, in Chapter \ref{chapgauss}, we obtain new fractional Sobolev
inequalities for Gaussian measures, as well as for probability measures that
are in between Gaussian and exponential. For example, for Gaussian measure on
$\mathbb{R}^{n},$ we have for $1\leq q<\infty,\theta\in(0,1)$%
\[
\left\{  \int_{0}^{1/2}\left|  f\right|  _{\gamma_{n}}^{\ast}(t)^{q}\left(
\log\frac{1}{t}\right)  ^{\frac{q\theta}{2}}dt\right\}  ^{1/q}\leq c\left\|
f\right\|  _{B_{L^{q}}^{\theta,q}(\gamma_{n})},
\]
where $c$ is independent of the dimension. Likewise, the same proof yields
that for probability measures on the real line of the form\footnote{where
$Z_{r}^{-1}$ is a normalizing constant.}%
\[
d\mu_{r}(x)=Z_{r}^{-1}\exp\left(  -\left|  x\right|  ^{r}\right)  dx,\text{
}r\in(1,2]
\]
and their tensor products%
\[
\mu_{r,n}=\mu_{p}^{\otimes n},
\]
we have%
\[
\left\{  \int_{0}^{1/2}\left|  f\right|  _{\mu_{r,n}}^{\ast}(t)^{q}\left(
\log\frac{1}{t}\right)  ^{q\theta\left(  1-1/r\right)  }dt\right\}  ^{1/q}\leq
c\left\|  f\right\|  _{B_{L^{q}}^{\theta,q}(\mu_{r,n})}.
\]
We refer to Chapter \ref{chapgauss} for the details, where the reader will
also find a treatment of the case $q=\infty,$ which yields the corresponding
improvements on the exponential integrability:%
\[
\sup_{t\in(0,\frac{1}{2})}\left(  \left|  f\right|  _{\mu_{r,n}}^{\ast\ast
}(t)-\left|  f\right|  _{\mu_{r,n}}^{\ast}(t)\right)  \left(  \log\frac{1}%
{t}\right)  ^{(1-\frac{1}{r})\theta}\leq c\left\|  f\right\|  _{\dot
{B}_{L^{\infty}}^{\theta,\infty}(\mu_{r,n})}.
\]

Applications to the computation of envelopes of function spaces in the sense
of Triebel-Haroske and their school are provided in Chapter \ref{triebel}.

The table of contents will serve to show the organization of the paper. We
have tried to make the reading of the chapters in the second part of the paper
as independent of each other as possible.

\begin{flushleft}

\textbf{Acknowledgement}.\ \ \textit{We are extremely grateful to the referees
for their painstaking review that gave us the opportunity to correct, and
fill-in gaps in some arguments, and thus contributed to substantially improve
the quality of the paper.}
\end{flushleft}

\chapter{Preliminaries\label{preliminar}}

\section{Background}

Our notation in the paper will be for the most part standard. In this paper we
shall only consider\footnote{See also Condition \ref{condition} below.}
connected measure metric spaces $\left(  \Omega,d,\mu\right)  $ equipped with
a finite Borel measure $\mu$, which we shall simply refer to, as
\textquotedblleft measure metric spaces\textquotedblright. For measurable
functions $u:\Omega\rightarrow\mathbb{R},$ the distribution function of $u$ is
given by
\[
\mu_{u}(t)=\mu\{x\in{\Omega}:u(x)>t\}\text{ \ \ \ \ }(t\in\mathbb{R}).
\]
The signed \textbf{decreasing rearrangement}\footnote{Note that this notation
is somewhat unconventional. In the literature it is common to denote the
decreasing rearrangement of $\left\vert u\right\vert $ by $u_{\mu}^{\ast},$
while here it is denoted by $\left\vert u_{\mu}\right\vert ^{\ast}$ since we
need to distinguish between the rearrangements of $u$ and $\left\vert
u\right\vert .$ In particular, the rearrangement of $u$ can be negative. We
refer the reader to \cite{rako} and the references quoted therein for a
complete treatment.} of a function $u$ is the right-continuous non-increasing
function from $[0,\mu(\Omega))$ into $\mathbb{R}$ which is equimeasurable with
$u.$ It can be defined by the formula%
\[
u_{\mu}^{\ast}(s)=\inf\{t\geq0:\mu_{u}(t)\leq s\},\text{ \ }s\in\lbrack
0,\mu(\Omega)),
\]
and satisfies
\[
\mu_{u}(t)=\mu\{x\in{\Omega}:u(x)>t\}=m\left\{  s\in\lbrack0,\mu
(\Omega)):u_{\mu}^{\ast}(s)>t\right\}  \text{ , \ }t\in\mathbb{R}%
\]
(where $m$ denotes the Lebesgue measure on $[0,\mu(\Omega)).$ It follows from
the definition that
\begin{equation}
\left(  u+v\right)  _{\mu}^{\ast}(s)\leq u_{\mu}^{\ast}(s/2)+v_{\mu}^{\ast
}(s/2). \label{aa1}%
\end{equation}
Moreover,
\[
u_{\mu}^{\ast}(0^{+})=ess\sup_{\Omega}u\text{ \ \ and \ \ \ }u_{\mu}^{\ast
}(\mu(\Omega)^{-})=ess\inf_{\Omega}u.
\]

The maximal average $u_{\mu}^{\ast\ast}(t)$ is defined by
\[
u_{\mu}^{\ast\ast}(t)=\frac{1}{t}\int_{0}^{t}u_{\mu}^{\ast}(s)ds=\frac{1}%
{t}\sup\left\{  \int_{E}u(s)d\mu:\mu(E)=t\right\}  ,t>0.
\]
The operation $u\rightarrow u_{\mu}^{\ast\ast}$ is sub-additive, i.e.
\begin{equation}
\left(  u+v\right)  _{\mu}^{\ast\ast}(s)\leq u_{\mu}^{\ast\ast}(s)+v_{\mu
}^{\ast\ast}(s). \label{a2}%
\end{equation}
Moreover, since $u_{\mu}^{\ast}$ is decreasing, $u_{\mu}^{\ast\ast}$ is also
decreasing and $u_{\mu}^{\ast}\leq u_{\mu}^{\ast\ast}$.

The following lemma proved in \cite[Lemma 2.1]{garsiagrenoble} will be useful
in what follows.

\begin{lemma}
\label{garlem} Let $f$ and $f_{n}$, $n=1,..,$ be integrable on $\Omega.$
Suppose that
\[
\lim_{n}\int_{\Omega}\left|  f_{n}(x)-f(x)\right|  d\mu=0.
\]
Then
\begin{align*}
\left(  f_{n}\right)  _{\mu}^{\ast\ast}(t)\rightarrow f_{\mu}^{\ast\ast}(t)
&  ,\text{ uniformly for }t\in\lbrack0,\mu(\Omega)]\text{, and }\\
\left(  f_{n}\right)  _{\mu}^{\ast}(t)\rightarrow f_{\mu}^{\ast}(t)\text{ }
&  \text{at all points of continuity of }f_{\mu}^{\ast}.
\end{align*}

\end{lemma}

When the measure is clear from the context, or when we are dealing with
Lebesgue measure, we may simply write $u^{\ast}$ and $u^{\ast\ast}$, etc.

For a Borel set $A\subset\Omega,$ the \textbf{perimeter} or \textbf{Minkowski
content} of $A$ is defined by
\[
P(A;\Omega)=\lim\inf_{h\rightarrow0}\frac{\mu\left(  A_{h}\right)  -\mu\left(
A\right)  }{h},
\]
where $A_{h}=\left\{  x\in\Omega:d(x,A)<h\right\}  $ is the open
$h-$neighborhood of $A.$

The \textbf{isoperimetric profile} is defined by
\[
I_{\Omega}(s)=I_{(\Omega,d,\mu)}(s)=\inf\left\{  P(A;\Omega):\text{ }%
\mu(A)=s\right\}  ,
\]
i.e. $I_{(\Omega,d,\mu)}:[0,\mu(\Omega)]\rightarrow\left[  0,\infty\right)  $
is the pointwise maximal function such that%
\begin{equation}
P(A;\Omega)\geq I_{\Omega}(\mu(A)), \label{ena}%
\end{equation}
holds for all Borel sets $A$. A set $A$ for which equality in (\ref{ena}) is
attained will be called an \textbf{isoperimetric domain. }Again when no
confusion arises we shall drop the subindex $\Omega$ and simply write $I.$

We will always assume that the metric measure spaces $(\Omega,d,\mu)$
considered satisfy the following condition

\begin{condition}
\label{condition}We will assume throughout the paper that our metric measure
spaces $(\Omega,d,\mu)$ are such that the isoperimetric profile $I_{(\Omega
,d,\mu)}$ is a concave continuous function, increasing on $(0,\mu(\Omega)/2),$
symmetric around the point $\mu(\Omega)/2$ that, moreover, vanishes at zero.
We remark that these assumptions are fulfilled for a large class of metric
measure spaces\footnote{These assumptions are satisfied for the classical
examples (cf. \cite{bobk}, \cite{MiE}, \cite{Bayle} and the references
therein)}.
\end{condition}

A continuous, concave function, $J:[0,\mu(\Omega)]\rightarrow\left[
0,\infty\right)  $, increasing on $(0,\mu(\Omega)/2)$, symmetric around the
point $\mu(\Omega)/2,$ and such that
\begin{equation}
I_{\Omega}\geq J, \label{enaena}%
\end{equation}
will be called an \textbf{isoperimetric estimator} for $(\Omega,d,\mu).$ Note
that (\ref{enaena}) and the fact that $I_{\Omega}(0)=0$ implies that $J(0)=0.$

For a Lipschitz function $f$ on $\Omega$ (briefly $f\in Lip(\Omega))$ we
define the \textbf{modulus of the gradient} by\footnote{In fact one can define
$\left|  \nabla f\right|  $ for functions $f$ that are Lipschitz on every ball
in $(\Omega,d)$ (cf. \cite[pp. 2, 3]{bobk} for more details).}
\[
|\nabla f(x)|=\limsup_{d(x,y)\rightarrow0}\frac{|f(x)-f(y)|}{d(x,y)}.
\]

Let us recall some results that relate isoperimetry and rearrangements (see
\cite{mamiadv},\cite{mamiproc}).

\begin{theorem}
The following statements hold

\begin{enumerate}
\item Isoperimetric inequality: $\forall A\subset\Omega,$ Borel set$,$%
\[
P(A;\Omega)\geq I_{\Omega}(\mu(A)).
\]

\item Oscillation inequality: $\forall f\in Lip(\Omega),$%
\begin{equation}
(\left\vert f\right\vert _{\mu}^{\ast\ast}(t)-\left\vert f\right\vert _{\mu
}^{\ast}(t))\frac{I_{\Omega}(t)}{t}\leq\frac{1}{t}\int_{0}^{t}\left\vert
\nabla f\right\vert _{\mu}^{\ast}(s)ds,\text{ \ \ }0<t<\mu(\Omega).
\label{reod00}%
\end{equation}

\end{enumerate}
\end{theorem}

\begin{lemma}
\label{ll1}Let $h$ be a bounded Lip function on $\Omega$. Then there exists a
sequence of bounded functions $\left(  h_{n}\right)  _{n}$ $\subset
Lip(\Omega)$ , such that

\begin{enumerate}
\item
\begin{equation}
\left|  \nabla h_{n}(x)\right|  \leq(1+\frac{1}{n})\left|  \nabla h(x)\right|
,\text{ \ }x\in\Omega. \label{cota01}%
\end{equation}

\item
\begin{equation}
h_{n}\underset{n\rightarrow\infty}{\rightarrow}h\text{ in }L^{1}.
\label{converge}%
\end{equation}

\item
\begin{equation}
\int_{0}^{t}\left|  \left(  -\left|  h_{n}\right|  _{\mu}^{\ast}\right)
^{\prime}(\cdot)I_{\Omega}(\cdot)\right|  ^{\ast}(s)ds\leq\int_{0}^{t}\left|
\nabla h_{n}\right|  _{\mu}^{\ast}(s)ds\text{, \ \ \ \ \ \ }0<t<\mu(\Omega).
\label{aa}%
\end{equation}
(The second rearrangement on the left hand side is with respect to the
Lebesgue measure on $\left[  0,\mu(\Omega)\right)  $).
\end{enumerate}
\end{lemma}

\section{Rearrangement invariant spaces\label{secc:ri}}

We recall briefly the basic definitions and conventions we use from the theory
of rearrangement-invariant (r.i.) spaces, and refer the reader to \cite{bs}
and \cite{KPS} for a complete treatment.

Let $X=X({\Omega})$ be a Banach function space on $({\Omega},d,\mu)$, with the
Fatou property\footnote{This means that if $f_{n}\geq0,$ and $f_{n}\uparrow
f,$ then $\left\Vert f_{n}\right\Vert _{X}\uparrow\left\Vert f\right\Vert
_{X}$ (i.e. the monotone convergence theorem holds in the $X$ norm). The
nomenclature is somewhat justified by the fact that this property is
equivalent to the validity of the usual Fatou Lemma in the $X$ norm (cf.
\cite{bs}).}. We shall say that $X$ is a \textbf{rearrangement-invariant}
(r.i.) space, if $g\in X$ implies that all $\mu-$measurable functions $f$ with
$\left\vert f\right\vert _{\mu}^{\ast}=\left\vert g\right\vert _{\mu}^{\ast},$
also belong to $X$ and moreover, $\Vert f\Vert_{X}=\Vert g\Vert_{X}$. The
functional $\Vert\cdot\Vert_{X}$ \ will be called a rearrangement invariant
norm. Typical examples of r.i. spaces are the $L^{p}$-spaces, Orlicz spaces,
Lorentz\footnote{See [(\ref{montana}), (\ref{vermont}), Chapter \ref{chapbmo}]
for the definition of Lorentz spaces.} spaces, Marcinkiewicz spaces, etc.

On account of the fact that $\mu(\Omega)<\infty,$ for any r.i. space
$X({\Omega})$ we have%

\[
L^{\infty}(\Omega)\subset X(\Omega)\subset L^{1}(\Omega),
\]
with continuous embeddings.

For rearrangement invariant norms (or seminorms) $\left\Vert .\right\Vert
_{X},$ we can compare the size of elements by comparing their averages, as
expressed by a majorization principle, sometimes referred to as the
Calder\'{o}n-Hardy Lemma:
\begin{align}
\text{Suppose that }\int_{0}^{t}\left\vert f\right\vert _{\mu}^{\ast}(s)ds  &
\leq\int_{0}^{t}\left\vert g\right\vert _{\mu}^{\ast}(s)ds,\text{ holds for
all }0<t<\mu(\Omega)\text{ }\label{mayorante}\\
&  \Rightarrow\left\Vert f\right\Vert _{X}\leq\left\Vert g\right\Vert _{X}
.\nonumber
\end{align}
\ 

The associated space $X^{\prime}(\Omega)$ is defined using the r.i. norm given
by%
\[
\left\Vert h\right\Vert _{X^{\prime}(\Omega)}=\sup_{g\neq0}\frac{\int_{\Omega
}\left\vert g(x)h(x)\right\vert d\mu}{\left\Vert g\right\Vert _{X(\Omega)}%
}=\sup_{g\neq0}\frac{\int_{0}^{\mu(\Omega)}\left\vert h\right\vert _{\mu
}^{\ast}(s)\left\vert g\right\vert _{\mu}^{\ast}(s)ds}{\left\Vert g\right\Vert
_{X(\Omega)}}.
\]
In particular, the following \textbf{generalized H\"{o}lder's inequality }holds%

\begin{equation}
\int_{\Omega}\left|  g(x)h(x)\right|  d\mu\leq\left\|  g\right\|  _{X(\Omega
)}\left\|  h\right\|  _{X^{\prime}(\Omega)}. \label{hol}%
\end{equation}

The \textbf{fundamental function\ }of $X({\Omega})$ is defined by
\begin{equation}
\phi_{X}(s)=\left\Vert \chi_{E}\right\Vert _{X},\text{ \ }0\leq s\leq
\mu(\Omega), \label{dasfundamental}%
\end{equation}
where $E$ is any measurable subset of $\Omega$ with $\mu(E)=s.$ We can assume
without loss of generality that $\phi_{X}$ is concave (cf. \cite{bs}).
Moreover,%
\begin{equation}
\phi_{X^{\prime}}(s)\phi_{X}(s)=s. \label{dual}%
\end{equation}
For example, if $X=L^{p}$ or $X=L^{p,q}$ (a Lorentz space), then $\phi_{L^{p}%
}(t)=\phi_{L^{p,q}}(t)=t^{1/p},$ if $1\leq p<\infty,$ while for $p=\infty,$
$\phi_{L^{\infty}}(t)\equiv1.$ If $N$ is a Young's function, then the
fundamental function of the Orlicz space $X=L_{N}$ is given by $\phi_{L_{N}%
}(t)=1/N^{-1}(1/t).$

The Lorentz $\Lambda(X)$ space and the Marcinkiewicz space $M(X)$ associated
with $X$ are defined by the quasi-norms
\begin{equation}
\left\|  f\right\|  _{M(X)}=\sup_{t}f_{\mu}^{\ast}(t)\phi_{X}(t),\text{
\ \ }\left\|  f\right\|  _{\Lambda(X)}=\int_{0}^{\mu(\Omega)}f_{\mu}^{\ast
}(t)d\phi_{X}(t). \label{fier}%
\end{equation}
Notice that
\[
\phi_{M(X)}(t)=\phi_{\Lambda(X)}(t)=\phi_{X}(t),
\]
and, moreover,%
\begin{equation}
\Lambda(X)\subset X\subset M(X). \label{tango}%
\end{equation}

Let $X({\Omega})$ be a r.i. space, then there exists a \textbf{unique} r.i.
space $\bar{X}=\bar{X}(0,\mu(\Omega))$ on $\left(  (0,\mu(\Omega)),m\right)
$, ($m$ denotes the Lebesgue measure on the interval $(0,\mu(\Omega))$) such
that%
\begin{equation}
\Vert f\Vert_{X({\Omega})}=\Vert\left\vert f\right\vert _{\mu}^{\ast}%
\Vert_{\bar{X}(0,\mu(\Omega))}. \label{repppp}%
\end{equation}
$\bar{X}$ is called the \textbf{representation space} of $X({\Omega})$. The
explicit norm of $\bar{X}(0,\mu(\Omega))$ is given by (see \cite[Theorem 4.10
and subsequent remarks]{bs})
\begin{equation}
\Vert h\Vert_{\bar{X}(0,\mu(\Omega))}=\sup\left\{  \int_{0}^{\mu(\Omega
)}\left\vert h\right\vert ^{\ast}(s)\left\vert g\right\vert _{\mu}^{\ast
}(s)ds:\left\Vert g\right\Vert _{X^{\prime}(\Omega)}\leq1\right\}
\label{representada}%
\end{equation}
(the first rearrangement is with respect to the Lebesgue measure on $\left[
0,\mu(\Omega)\right)  $).

Classically conditions on r.i. spaces can be formulated in terms of the
boundedness of the Hardy operators defined by
\[
Pf(t)=\frac{1}{t}\int_{0}^{t}f(s)ds;\text{ \ \ \ }Qf(t)=\int_{t}^{\mu(\Omega
)}f(s)\frac{ds}{s}.
\]
The boundedness of these operators on r.i. spaces can be best described in
terms of the so called \textbf{Boyd indices}\footnote{Introduced by D.W. Boyd
in \cite{boyd}.} defined by
\[
\bar{\alpha}_{X}=\inf\limits_{s>1}\dfrac{\ln h_{X}(s)}{\ln s}\text{ \ \ and
\ \ }\underline{\alpha}_{X}=\sup\limits_{s<1}\dfrac{\ln h_{X}(s)}{\ln s},
\]
where $h_{X}(s)=\sup_{\left\|  f\right\|  _{\bar{X}}\leq1}$ $\left\|
E_{s}f\right\|  _{\bar{X}}$ denotes the norm of the compression/dilation
operator $E_{s}$ on $\bar{X}$, defined for $s>0,$ by
\[
E_{s}f(t)=\left\{
\begin{array}
[c]{ll}%
f^{\ast}(\frac{t}{s}) & 0<t<s,\\
0 & s<t<\mu(\Omega).
\end{array}
\right.
\]
The operator $E_{s}$ is bounded on $\bar{X}$ on every r.i. space $X(\Omega),$
and moreover,
\begin{equation}
h_{X}(s)\leq\max(1,s),\text{ for all }s>0. \label{ccdd}%
\end{equation}
For example, if $X=L^{p}$, then $\overline{\alpha}_{L^{p}}=\underline{\alpha
}_{L^{p}}=\frac{1}{p}.$ It is well known that
\begin{equation}%
\begin{array}
[c]{c}%
P\text{ is bounded on }\bar{X}\text{ }\Leftrightarrow\overline{\alpha}%
_{X}<1,\\
Q\text{ is bounded on }\bar{X}\text{ }\Leftrightarrow\underline{\alpha}_{X}>0.
\end{array}
\label{alcance}%
\end{equation}

We shall also need to consider the restriction of functions of the r.i. space
$X(\Omega)$ to measurable subsets $G\subset\Omega$ with $\mu(G)\neq0.$ We can
then consider $G$ as a metric measure space $(G,d_{\mid G},\mu_{\mid G})$
where the corresponding distance and the measure are obtained by the
restrictions of the distance $d$ and the measure $\mu$ to $G.$ We shall denote
the r.i. space $X(G,d_{\mid G},\mu_{\mid G})$ by $X_{r}(G).$ Given
$u:G\rightarrow\mathbb{R}$, $u\in X_{r}(G),$ we let $\tilde{u}:\Omega
\rightarrow\mathbb{R}$, be its extension to $\Omega$ defined by
\begin{equation}
\tilde{u}(x)=\left\{
\begin{array}
[c]{ll}%
u(x) & x\in G,\\
0 & x\in\text{ }\Omega\setminus G.
\end{array}
\right.  \label{chavez}%
\end{equation}
Then,%
\[
\left\Vert u\right\Vert _{X_{r}(G)}=\left\Vert \tilde{u}\right\Vert
_{X(\Omega)}.
\]

\begin{proposition}
Let $X(\Omega)$ be a r.i. space on $\Omega,$ and let $G$ be a measurable
subset of $\Omega$ with $\mu(G)\neq0$. Then,

\begin{enumerate}
\item If $u\in X(\Omega),$ then $u\chi_{G}\in X_{r}(G)$ and
\[
\left\|  u\chi_{G}\right\|  _{X_{r}\left(  G\right)  }\leq\left\|  u\right\|
_{X\left(  \Omega\right)  }.
\]

\item Let $\bar{X}_{r}$ be the representation space of $X_{r}\left(  G\right)
$ and let $\bar{X}$ be the representation space of $X\left(  \Omega\right)  $.
Let $u\in X_{r}\left(  G\right)  .$ Then%
\[
\left\|  u\right\|  _{X_{r}(G)}=\left\|  \widetilde{{\left(  \left|  u\right|
_{\mu_{\mid G}}^{\ast}\right)  }}\right\|  _{\bar{X}},
\]
where given $h:(0,\mu(G))\rightarrow(0,\infty),$ $\tilde{h}$ denotes its
continuation by $0$ outside $(0,\mu(G)).$ Thus by the uniqueness of the
representation space, if $h\in\bar{X}_{r},$ then
\[
\left\|  h\right\|  _{\bar{X}_{r}}=\left\|  \tilde{h}\right\|  _{\bar{X}}.
\]

\item The fundamental function of $X_{r}\left(  G\right)  $ is given by
\begin{equation}
\phi_{X_{r}\left(  G\right)  }(s)=\phi_{X\left(  \Omega\right)  }(s)\text{
\ \ }(0\leq s\leq\mu(G)).\text{\ } \label{galio}%
\end{equation}

\item Let $\left(  X_{r}\left(  G\right)  \right)  ^{\prime}$ be the
associated space of $X_{r}\left(  G\right)  .$ Then%
\begin{equation}
\left(  X_{r}\left(  G\right)  \right)  ^{\prime}=\left(  X(\Omega)^{\prime
}\right)  _{r}(G). \label{dualres}%
\end{equation}

\end{enumerate}
\end{proposition}

\begin{proof}
Part $1$ and $4$ are elementary. For Part $2,$ note that if $u\in X_{r}(G),$
then%
\begin{align*}
\left\Vert u\right\Vert _{X_{r}(G)}  &  =\left\Vert \tilde{u}\right\Vert
_{X(\Omega)}\\
&  =\left\Vert \left(  \left\vert \tilde{u}\right\vert \right)  _{\mu}^{\ast
}\right\Vert _{\bar{X}}\text{ \ (by (\ref{repppp})).}%
\end{align*}
Since $\mu_{\mid G}=\mu$ on $G,$ it follows from the definition of $\tilde{u}$
that
\[
\left(  \left\vert \tilde{u}\right\vert \right)  _{\mu}^{\ast}=\left\{
\begin{array}
[c]{ll}%
\left(  \left\vert u\right\vert \right)  _{\mu_{\mid G}}^{\ast}(s) &
s\in(0,\mu(G)),\\
0 & s\in(\mu(G),\mu(\Omega)).
\end{array}
\right.
\]
Thus%
\[
\left\Vert \left(  \left\vert \tilde{u}\right\vert \right)  _{\mu}^{\ast
}\right\Vert _{\bar{X}}=\left\Vert \left(  \left\vert u\right\vert \right)
_{\mu_{\mid G}}^{\ast}\chi_{(0,\mu(G))}\right\Vert _{\bar{X}}=\left\Vert
\widetilde{{\left(  \left\vert u\right\vert _{\mu_{\mid G}}^{\ast}\right)  }%
}\right\Vert _{\bar{X}}.
\]

Now Part $3$ follows from Part $2$ taking in account that
\[
\phi_{X\left(  \Omega\right)  }(s)=\phi_{\bar{X}}(s).
\]

\end{proof}

In what follows, when $G\subset\Omega$ is clear from the context, and $u$ is a
function defined on $G,$ we shall use the notation $\tilde{u}$ to denote its
extension (by zero) defined by (\ref{chavez}) above.

\section{Some remarks about Sobolev spaces}

Let $(\Omega,d,\mu)$ be a connected metric measure space with finite measure,
and let $X$ be a r.i. space on $\Omega.$ We let $S_{X}=S_{X}({\Omega})=\{f\in
Lip(\Omega):\left\vert \nabla f\right\vert \in X(\Omega)\},$ equipped with the
seminorm
\[
\left\Vert f\right\Vert _{S_{X}}=\left\Vert \left\vert \nabla f\right\vert
\right\Vert _{X}.
\]

At some point in our development we also need to consider restrictions of
Sobolev functions. Let $G\subset\Omega$ be an open subset, then if $f\in
S_{X}({\Omega})$ we have $f\chi_{{G}}\in S_{X_{r}}(G),$and%
\begin{align*}
\left\Vert f\chi_{{G}}\right\Vert _{S_{X_{r}}({G})}  &  \leq\left\Vert
\left\vert \nabla f\right\vert \chi_{{G}}\right\Vert _{X({\Omega})}\\
&  \leq\left\Vert \left\vert \nabla f\right\vert \right\Vert _{X({\Omega})}\\
&  =\left\Vert f\right\Vert _{S_{X}({\Omega})}.
\end{align*}

$K-$functionals play an important role in this paper. The $K-$functional for
the pair $(X(\Omega),S_{X}(\Omega))$ for is defined by
\begin{equation}
K(t,f;X(\Omega),S_{X}(\Omega))=\inf_{g\in S_{X}(\Omega)}\{\left\Vert
f-g\right\Vert _{X(\Omega)}+t\left\Vert \left\vert \nabla g\right\vert
\right\Vert _{X(\Omega)}\}. \label{defK1}%
\end{equation}
If ${G}$ is an open subset of ${\Omega},$ each competing decomposition for the
calculation of the $K-$functional of $f$, $K(t,f;X(\Omega),S_{X}(\Omega)),$
produces by restriction a decomposition for the calculation of the
$K-$functional of $f\chi_{{G}},$ and we have%
\[
K(t,f\chi_{{G}};X_{r}(G),S_{X_{r}}(G))\leq K(t,f;X(\Omega),S_{X}(\Omega)).
\]

Notice that from our definition of $S_{X}(\Omega)$ it does not follow that
$h\in S_{X}$ implies that $h\in X.$ However, under mild conditions on $X,$ one
can guarantee that $h\in X$. Indeed, using the isoperimetric profile
$I=I_{(\Omega,d,\mu)}$, let us define the associated \textbf{isoperimetric
Hardy operator} by
\[
Q_{I}f(t)=\int_{t}^{\mu(\Omega)}f(s)\frac{ds}{I(s)}\text{ \ \ }(f\geq0).
\]
Suppose that there exists an absolute constant $c>0$ such that, for all
$f\in\bar{X},$ such that $f\geq0,$ and with supp$(f)\subset(0,\mu(\Omega)/2),$
we have
\begin{equation}
\left\Vert Q_{I}f\right\Vert _{\bar{X}}\leq c\left\Vert f\right\Vert _{\bar
{X}}. \label{condita}%
\end{equation}
Then, it was shown in \cite{mamiadv} that for all $h\in S_{X},$%
\[
\left\Vert h-\frac{1}{\mu(\Omega)}\int_{\Omega}h\right\Vert _{X}%
\preceq\left\Vert \left\vert \nabla h\right\vert \right\Vert _{X}.
\]
Therefore, since constant functions belong to $X$ we can then conclude that
indeed $h\in X.$ It is easy to see that if $\underline{\alpha}_{X}>0,$
condition (\ref{condita}) is satisfied. Indeed, from the concavity of $I,$ it
follows that $\frac{I(s)}{s}$ is decreasing, therefore%
\[
\frac{I(\mu(\Omega)/2)}{\mu(\Omega)/2}\leq\frac{I(s)}{s},\text{ }s\in
(0,\mu(\Omega)/2).
\]
It follows that if $s\in(0,\mu(\Omega)/2),$ then%
\[
s\leq\frac{\mu(\Omega)/2}{I(\mu(\Omega)/2)}I(s)=cI(s).
\]
Consequently, for all $f\geq0,$ with supp$(f)\subset(0,\mu(\Omega)/2),$%
\[
Q_{I}f(t)=\int_{t}^{\mu(\Omega)/2}f(s)\frac{ds}{I(s)}\leq c\int_{t}%
^{\mu(\Omega)/2}f(s)\frac{ds}{s}=Qf(t).
\]
Therefore,%
\[
\left\Vert Q_{I}f\right\Vert _{\bar{X}}\leq c\left\Vert Qf\right\Vert
_{\bar{X}}\leq c_{X}\left\Vert f\right\Vert _{\bar{X}},
\]
where the last inequality follows from the fact that $\underline{\alpha}%
_{X}>0.$ We can avoid placing restrictions on $X$ if instead we impose more
conditions on the isoperimetric profile. For example, suppose that the
following condition \footnote{A typical example is $I(t)\simeq t^{1-1/n},$
near zero.} on $I$ holds$:$
\begin{equation}
\int_{0}^{\mu(\Omega)/2}\frac{ds}{I(s)}=c<\infty. \label{ciso}%
\end{equation}
Then, for $f\in L^{\infty}$ we have%
\begin{align*}
Q_{I}f(t)  &  \leq\left\Vert f\right\Vert _{L^{\infty}}\int_{t}^{\mu(\Omega
)}\frac{ds}{I(s)}\\
&  \leq c\left\Vert f\right\Vert _{L^{\infty}}.
\end{align*}
Consequently, $Q_{I}$ is bounded on $L^{\infty}.$ Since, as we have already
seen $Q_{I}$ $\leq Q,$ it follows that $Q_{I}$ is also bounded on $L^{1},$ and
therefore, by Calder\'{o}n's interpolation theorem, $Q_{I}$ is bounded on any
r.i. space $X.$ In particular, (\ref{condita}) is satisfied.

\chapter{Oscillations, K-functionals and Isoperimetry\label{main}}

\section{Summary}

Let $(\Omega,d,\mu)$ be a metric measure space satisfying the usual
assumptions, and let $X$ be a r.i. space on $\Omega$. In this chapter we show
(cf. Theorem \ref{main1} in Section \ref{secc:primera}) that for all $f\in
X+S_{X},$
\begin{equation}
\left\vert f\right\vert _{\mu}^{\ast\ast}(t)-\left\vert f\right\vert _{\mu
}^{\ast}(t)\leq16\frac{K\left(  \frac{t}{I_{\Omega}(t)},f;X,S_{X}\right)
}{\phi_{X}(t)},\text{ }0<t\leq\mu(\Omega)/2. \label{polaka1}%
\end{equation}
Moreover, if $f_{\Omega}=\frac{1}{\mu(\Omega)}\int_{\Omega}fd\mu,$ then%
\begin{equation}
\left\vert f-f_{\Omega}\right\vert _{\mu}^{\ast\ast}(t)-\left\vert
f-f_{\Omega}\right\vert _{\mu}^{\ast}(t)\leq16\frac{K\left(  \frac
{t}{I_{\Omega}(t)},f;X,S_{X}\right)  }{\phi_{X}(t)},\text{ }0<t\leq\mu
(\Omega). \label{polakita}%
\end{equation}

This extends one of the main results of \cite{mastylo}. In Section
\ref{sec:variante} we prove a variant of inequality (\ref{polaka1}) that will
play an important role in Chapter \ref{contchap}, where embeddings into the
space of continuous functions will be analyzed. In this variant we replace
$\frac{t}{I_{\Omega}(t)}$ by a smaller function that depends on the space $X,$
more specifically we show that%
\[
\left|  f\right|  _{\mu}^{\ast\ast}(t)-\left|  f\right|  _{\mu}^{\ast}%
(t)\leq4\frac{K\left(  \psi(t),f;X,S_{X}\right)  }{\phi_{X}(t)},\text{
}0<t\leq\mu(\Omega)/2,
\]
where
\[
\psi(t)=\frac{\phi_{X}(t)}{t}\left\|  \frac{s}{I_{\Omega}(s)}\chi
_{(0,2t)}(s)\right\|  _{\bar{X}^{^{\prime}}}.
\]
In Section \ref{sec:iso} we show that (\ref{polaka1}) for $X=L^{1}$ implies an
isoperimetric inequality.

Underlying these results is the following estimation of the oscillation
(without rearrangements): there exists a constant $c>0$ such that
\[
\frac{1}{\mu(\Omega)}\int_{\Omega}\left\vert f-f_{\Omega}\right\vert d\mu\leq
c\frac{K\left(  \frac{\mu(\Omega)/2}{I_{\Omega}(\mu(\Omega)/2)},f;X,S_{X}%
\right)  }{\phi_{X}(\mu(\Omega))}.
\]

\section{Estimation of the oscillation in terms of $K-$%
functionals\label{secc:primera}}

The leitmotif of this chapter are the remarkable connections between
oscillations, optimization and isoperimetry.

\begin{definition}
\label{mediandef} Let $f:\Omega\rightarrow\mathbb{R}$ be an integrable
function. We say that $m(f)$ is a median of $f$ if%
\[
\mu\{f>m(f)\}\geq\mu(\Omega)/2;\text{ and }\mu\{f<m(f)\}\geq\mu(\Omega)/2.
\]

\end{definition}

Recall the following basic property of medians (cf. \cite{tor}, \cite[page
134]{mamiadv}, etc.)%
\begin{align}
\frac{1}{2}\left(  \frac{1}{\mu(\Omega)}\int_{\Omega}\left\vert f-f_{\Omega
}\right\vert d\mu\right)   &  \leq\inf_{c}\frac{1}{\mu(\Omega)}\int_{\Omega
}\left\vert f-c\right\vert d\mu\nonumber\\
&  \leq\frac{1}{\mu(\Omega)}\int_{\Omega}\left\vert f-m(f)\right\vert
d\mu\label{basicamed}\\
&  \leq3\left(  \frac{1}{\mu(\Omega)}\int_{\Omega}\left\vert f-f_{\Omega
}\right\vert d\mu\right)  .\nonumber
\end{align}

The starting point of our analysis is the well known Poincar\'{e} inequality
(cf. \cite{bobk}, \cite[page 150]{mamiadv}):
\begin{equation}
\int_{{\Omega}}\left\vert f-m(f)\right\vert d\mu\leq\frac{\mu(\Omega
)}{2I_{\Omega}(\mu(\Omega)/2)}\int_{{\Omega}}\left\vert \nabla f\right\vert
d\mu,\text{ for all }f\in S_{L^{1}}(\Omega). \label{pr1}%
\end{equation}

Our next result is an extension of the Poincar\'{e} type inequality
(\ref{pr1}) by interpolation (i.e. using $K-$functionals).

\begin{theorem}
\label{teobmo} Let $X$ be a r.i space on $\Omega.$ Then for all $f\in X,$
\begin{equation}
\frac{1}{\mu(\Omega)}\int_{\Omega}\left\vert f-f_{\Omega}\right\vert d\mu
\leq2\frac{K\left(  \frac{\mu(\Omega)/2}{I_{\Omega}(\mu(\Omega)/2)}%
,f;X,S_{X}\right)  }{\phi_{X}(\mu(\Omega))}. \label{poinbesov}%
\end{equation}

\end{theorem}

\begin{proof}
For an arbitrary decomposition $f=(f-h)+h,$ with $h\in S_{X},$ we have%
\begin{align*}
\frac{1}{\mu(\Omega)}\int_{{\Omega}}\left\vert f-m(h)\right\vert d\mu &
=\frac{1}{\mu(\Omega)}\int_{{\Omega}}\left\vert f-h+\left(  h-m(h)\right)
\right\vert d\mu\\
&  \leq\frac{1}{\mu(\Omega)}\int_{{\Omega}}\left\vert f-h\right\vert
d\mu+\frac{1}{\mu(\Omega)}\int_{{\Omega}}\left\vert h-m(h)\right\vert d\mu\\
&  \leq\frac{1}{\mu(\Omega)}\left\Vert f-h\right\Vert _{L^{1}}+\frac
{1}{2I_{\Omega}(\mu(\Omega)/2)}\left\Vert \left\vert \nabla h\right\vert
\right\Vert _{L^{1}}\text{ (by (\ref{pr1}))}\\
&  \leq\frac{1}{\mu(\Omega)}\left\Vert f-h\right\Vert _{X}\phi_{X^{\prime}%
}(\mu(\Omega))+\\
&  \frac{1}{2I_{\Omega}(\mu(\Omega)/2)}\left\Vert \left\vert \nabla
h\right\vert \right\Vert _{X}\phi_{X^{\prime}}(\mu(\Omega))\text{ (by
H\"{o}lder's inequality)}\\
&  =\frac{1}{\mu(\Omega)}\phi_{X^{\prime}}(\mu(\Omega))\left(  \left\Vert
f-h\right\Vert _{X}+\frac{\mu(\Omega)/2}{I_{\Omega}(\mu(\Omega)/2)}\left\Vert
\left\vert \nabla h\right\vert \right\Vert _{X}\right) \\
&  =\frac{1}{\phi_{X}(\mu(\Omega))}\left(  \left\Vert f-h\right\Vert
_{X}+\frac{\mu(\Omega)/2}{I_{\Omega}(\mu(\Omega)/2)}\left\Vert \left\vert
\nabla h\right\vert \right\Vert _{X}\right)  .
\end{align*}
Combining the last inequality with (\ref{basicamed}), we obtain that for all
decompositions $f=(f-h)+h,$ with $h\in S_{X},$
\[
\frac{1}{\mu(\Omega)}\int_{\Omega}\left\vert f-f_{\Omega}\right\vert d\mu
\leq\frac{2}{\phi_{X}(\mu(\Omega))}\left(  \left\Vert f-h\right\Vert
_{X}+\frac{\mu(\Omega)/2}{I_{\Omega}(\mu(\Omega)/2)}\left\Vert \left\vert
\nabla h\right\vert \right\Vert _{X}\right)  .
\]
Consequently,%
\begin{align*}
\frac{1}{\mu(\Omega)}\int_{\Omega}\left\vert f-f_{\Omega}\right\vert d\mu &
\leq\frac{2}{\phi_{X}(\mu(\Omega))}\inf_{h\in S_{X}}\left(  \left\Vert
f-h\right\Vert _{X}+\frac{\mu(\Omega)/2}{I_{\Omega}(\mu(\Omega)/2)}\left\Vert
\left\vert \nabla h\right\vert \right\Vert _{X}\right) \\
&  =\frac{2}{\phi_{X}(\mu(\Omega))}K\left(  \frac{\mu(\Omega)/2}{I_{\Omega
}(\mu(\Omega)/2)},f;X,S_{X}\right)  ,
\end{align*}
as we wished to show.
\end{proof}

The main result of this section is the following

\begin{theorem}
\label{main1}Let $X$ be a r.i. space on $\Omega.$ Then (\ref{polaka1}) and
(\ref{polakita}) hold for all $f\in X+S_{X}.$
\end{theorem}

\begin{proof}
It will be useful to note for future use that if $\left\|  \cdot\right\|  $
denotes either $\left\|  \cdot\right\|  _{X}$ or $\left\|  \cdot\right\|
_{S_{X}},$ we have
\begin{equation}
\left\|  \left|  f\right|  \right\|  \leq\left\|  f\right\|  . \label{dolores}%
\end{equation}
Let $\varepsilon>0$, and consider any decomposition $f=f-h+h$ with $h\in
S_{X},$ such that
\begin{equation}
\left\|  f-h\right\|  _{X}+t\left\|  \left|  \nabla h\right|  \right\|
_{X}\leq K\left(  t,f;X,S_{X}\right)  +\varepsilon. \label{dolores1}%
\end{equation}
Since by (\ref{dolores}), $h\in S_{X}$ implies that $\left|  h\right|  \in
S_{X},$ this decomposition of $f$ produces the following decomposition of
$\left|  f\right|  :$
\[
\left|  f\right|  =\left|  f\right|  -\left|  h\right|  +\left|  h\right|  .
\]
Therefore, by (\ref{dolores}) and (\ref{dolores1}) we have
\begin{align*}
\left\|  \left|  f\right|  -\left|  h\right|  \right\|  _{X}+t\left\|  \left|
\nabla\left|  h\right|  \right|  \right\|  _{X}  &  \leq\left\|  f-h\right\|
_{X}+t\left\|  \left|  \nabla h\right|  \right\|  _{X}\\
&  \leq K\left(  t,f;X,S_{X}\right)  +\varepsilon.
\end{align*}
Consequently,
\begin{equation}
\inf_{0\leq h\in S_{X}}\left\{  \left\|  \left|  f\right|  -h\right\|
_{X}+t\left\|  \left|  \nabla h\right|  \right\|  _{X}\right\}  \leq K\left(
t,f;X,S_{X}\right)  . \label{porfin}%
\end{equation}

In what follows we shall use the following notation:%
\[
K\left(  t,f\right)  :=K\left(  t,f;X,S_{X}\right)  .
\]

We shall start by proving (\ref{polaka1}). The proof will be divided in two parts.

\textbf{Part 1}: $t\in(0,\mu(\Omega)/4].$

Given $0\leq h\in S_{X}$ consider the decomposition
\[
\left|  f\right|  =\left(  \left|  f\right|  -h\right)  +h.
\]
By (\ref{a2}) we have
\[
\left|  f\right|  _{\mu}^{\ast\ast}(t)\leq\left|  \left|  f\right|  -h\right|
_{\mu}^{\ast\ast}(t)+\left|  h\right|  _{\mu}^{\ast\ast}(t),
\]
and by (\ref{aa1}) we get
\[
\left|  h\right|  _{\mu}^{\ast}(2t)-\left|  \left|  f\right|  -h\right|
_{\mu}^{\ast}(t)\leq\left|  f\right|  _{\mu}^{\ast}(t).
\]
Combining the previous estimates we can write
\begin{align}
\left|  f\right|  _{\mu}^{\ast\ast}(t)-\left|  f\right|  _{\mu}^{\ast}(t)  &
\leq\left|  \left|  f\right|  -h\right|  _{\mu}^{\ast\ast}(t)+\left|
h\right|  _{\mu}^{\ast\ast}(t)-(\left|  h\right|  _{\mu}^{\ast}(2t)-\left|
\left|  f\right|  -h\right|  _{\mu}^{\ast}(t))\nonumber\\
&  =\left|  \left|  f\right|  -h\right|  _{\mu}^{\ast\ast}(t)+\left|  \left|
f\right|  -h\right|  _{\mu}^{\ast}(t)+\left|  h\right|  _{\mu}^{\ast\ast
}(t)-\left|  h\right|  _{\mu}^{\ast}(2t)\nonumber\\
&  \leq2\left|  \left|  f\right|  -h\right|  _{\mu}^{\ast\ast}(t)+\left(
\left|  h\right|  _{\mu}^{\ast\ast}(t)-\left|  h\right|  _{\mu}^{\ast
}(2t)\right) \nonumber\\
&  =2\left|  \left|  f\right|  -h\right|  _{\mu}^{\ast\ast}(t)+\left(  \left|
h\right|  _{\mu}^{\ast\ast}(t)-\left|  h\right|  _{\mu}^{\ast\ast}(2t)\right)
+(\left|  h\right|  _{\mu}^{\ast\ast}(2t)-\left|  h\right|  _{\mu}^{\ast
}(2t))\nonumber\\
&  =(I)+(II)+(III). \label{flaca}%
\end{align}
We first show that $\ (II)\leq(III).$ Recall that $\frac{d}{dt}(-\left|
g\right|  _{\mu}^{\ast\ast}(t))=\frac{\left|  g\right|  _{\mu}^{\ast\ast
}(t)-\left|  g\right|  _{\mu}^{\ast}(t)}{t},$ then using the fundamental
theorem of Calculus, and then the fact that $t\left(  \left|  g\right|  _{\mu
}^{\ast\ast}(t)-\left|  g\right|  _{\mu}^{\ast}(t)\right)  =\int_{\left|
g\right|  _{\mu}^{\ast}(t)}^{\infty}\mu_{\left|  g\right|  }(s)ds$ is
increasing, to estimate $(II)$ as follows:%
\begin{align*}
(II)  &  =\left|  h\right|  _{\mu}^{\ast\ast}(t)-\left|  h\right|  _{\mu
}^{\ast\ast}(2t)\\
&  =\int_{t}^{2t}\left(  \left|  h\right|  _{\mu}^{\ast\ast}(s)-\left|
h\right|  _{\mu}^{\ast}(s)\right)  \frac{ds}{s}\\
&  \leq2t\left(  \left|  h\right|  _{\mu}^{\ast\ast}(2t)-\left|  h\right|
_{\mu}^{\ast}(2t)\right)  \int_{t}^{2t}\frac{ds}{s^{2}}\\
&  =\left(  \left|  h\right|  _{\mu}^{\ast\ast}(2t)-\left|  h\right|  _{\mu
}^{\ast}(2t)\right) \\
&  =(III).
\end{align*}
Inserting this information in (\ref{flaca}), applying ([(\ref{reod00}),
Chapter \ref{preliminar}]) and using the fact that $I_{\Omega}(s)$ is
increasing $\ $on $(0,\mu(\Omega)/2),$ yields,
\begin{align}
\left|  f\right|  _{\mu}^{\ast\ast}(t)-\left|  f\right|  _{\mu}^{\ast}(t)  &
\leq2\left(  \left|  \left|  f\right|  -h\right|  _{\mu}^{\ast\ast}%
(t)+\frac{2t}{I_{\Omega}(2t)}\left|  \nabla h\right|  ^{\ast\ast}(2t)\right)
\nonumber\\
&  \leq4\left(  \left|  \left|  f\right|  -h\right|  _{\mu}^{\ast\ast
}(t)+\frac{t}{I_{\Omega}(t)}\left|  \nabla h\right|  ^{\ast\ast}(t)\right)
\nonumber\\
&  =4\left(  A(t)+B(t)\right)  \text{.} \label{det5}%
\end{align}
We now estimate the two terms on the right hand side of (\ref{det5}). For the
term $A(t):$ Note that for any $g\in X,$%
\[
\left|  g\right|  _{\mu}^{\ast\ast}(t)=\frac{1}{t}\int_{0}^{t}\left|
g\right|  _{\mu}^{\ast}(s)ds=\frac{1}{t}\int_{0}^{1}\left|  g\right|  _{\mu
}^{\ast}(s)\chi_{(0,t)}(s)ds.
\]
Therefore, by H\"{o}lder's inequality (cf. [(\ref{hol}) and (\ref{dual}) of
Chapter\ref{preliminar}]) we have
\begin{align}
\left|  \left|  f\right|  -h\right|  _{\mu}^{\ast\ast}(t)  &  =\frac{1}{t}%
\int_{0}^{1}\left|  \left|  f\right|  -h\right|  _{\mu}^{\ast}(s)\chi
_{(0,t)}(s)ds\nonumber\\
&  \leq\left\|  \left(  \left|  f\right|  -h\right)  \right\|  _{X}\frac
{\phi_{X^{\prime}}(t)}{t}\nonumber\\
&  =\left\|  \left(  \left|  f\right|  -h\right)  \right\|  _{X}\frac{1}%
{\phi_{X}(t)}. \label{det1}%
\end{align}
Similarly, for $B(t)$ we get
\begin{equation}
B(t)=\frac{t}{I(t)}\left|  \nabla h\right|  _{\mu}^{\ast\ast}(t)\leq\frac
{t}{I_{\Omega}(t)}\frac{\left\|  \left|  \nabla h\right|  \right\|  _{X}}%
{\phi_{X}(t)}. \label{det2}%
\end{equation}
Inserting (\ref{det1}) and (\ref{det2}) back in (\ref{det5}) we find that$,$%
\[
\left|  f\right|  _{\mu}^{\ast\ast}(t)-\left|  f\right|  _{\mu}^{\ast}%
(t)\leq\frac{4}{\phi_{X}(t)}\left(  \left\|  \left|  f\right|  -h\right\|
_{X}+\frac{t}{I_{\Omega}(t)}\left\|  \left|  \nabla h\right|  \right\|
_{X}\right)  .
\]
Therefore, by (\ref{porfin}),
\begin{align*}
\left|  f\right|  _{\mu}^{\ast\ast}(t)-\left|  f\right|  _{\mu}^{\ast}(t)  &
\leq\frac{4}{\phi_{X}(t)}\inf_{0\leq h\in S_{X}}\left(  \left\|  \left|
f\right|  -h\right\|  _{X}+\frac{t}{I_{\Omega}(t)}\left\|  \left|  \nabla
h\right|  \right\|  _{X}\right) \\
&  \leq4\frac{K\left(  \frac{t}{I_{\Omega}(t)},f\right)  }{\phi_{X}(t)}.
\end{align*}

\textbf{Part II:} $t\in(\mu(\Omega)/4,\mu(\Omega)/2].$

Using that the function $t(\left|  f\right|  _{\mu}^{\ast\ast}(t)-\left|
f\right|  _{\mu}^{\ast}(t))$ is increasing, we get,
\begin{align*}
t(\left|  f\right|  _{\mu}^{\ast\ast}(t)-\left|  f\right|  _{\mu}^{\ast}(t))
&  \leq\mu(\Omega)/2(\left|  f\right|  _{\mu}^{\ast\ast}(\mu(\Omega
)/2)-\left|  f\right|  _{\mu}^{\ast}(\mu(\Omega)/2))\\
&  =\int_{0}^{\frac{\mu(\Omega)}{2}}\left(  \left|  f\right|  _{\mu}^{\ast
}(s)-\left|  f\right|  _{\mu}^{\ast}(\mu(\Omega)/2)\right)  ds\\
&  =I.
\end{align*}
Now we use the following elementary inequality to estimate the difference
$\left|  f\right|  _{\mu}^{\ast}(s)-\left|  f\right|  _{\mu}^{\ast}(\mu
(\Omega)/2)$ (cf. \cite[pag 94]{garsiagrenoble}, \cite[(2.5)]{ko}): For any
$\sigma\in\mathbb{R},$ $0<r\leq\tau<\mu(\Omega),$ we have%
\begin{equation}
\left|  f\right|  _{\mu}^{\ast}(r)-\left|  f\right|  _{\mu}^{\ast}(\tau
)\leq\left|  f-\sigma\right|  _{\mu}^{\ast}(r)+\left|  f-\sigma\right|  _{\mu
}^{\ast}(\mu(\Omega)-\tau). \label{formula1}%
\end{equation}
Consequently, if $0<s<\mu(\Omega)/2$ and we let $r=s,\tau=\mu(\Omega)/2,$ then
(\ref{formula1}) yields%
\[
\left|  f\right|  _{\mu}^{\ast}(s)-\left|  f\right|  _{\mu}^{\ast}(\mu
(\Omega)/2))\leq\left|  f-\sigma\right|  _{\mu}^{\ast}(s)+\left|
f-\sigma\right|  _{\mu}^{\ast}(\mu(\Omega)/2).
\]
The term $I$ above can be now be estimated as follows. For any $\sigma
\in\mathbb{R},$ we have%
\begin{align*}
I  &  \leq\int_{0}^{\frac{\mu(\Omega)}{2}}\left|  f-\sigma\right|  _{\mu
}^{\ast}(s)+\left|  f-\sigma\right|  _{\mu}^{\ast}(\mu(\Omega)/2)ds\\
&  \leq2\int_{0}^{\frac{\mu(\Omega)}{2}}\left|  f-\sigma\right|  _{\mu}^{\ast
}(s)ds\\
&  \leq2\int_{0}^{\mu(\Omega)}\left|  f-\sigma\right|  _{\mu}^{\ast}(s)ds.
\end{align*}
Selecting $\sigma=f_{\Omega}$, yields%
\begin{align*}
I  &  \leq2\int_{0}^{\mu(\Omega)}\left|  f-f_{\Omega}\right|  _{\mu}^{\ast
}(s)ds\\
&  =2\int_{\Omega}\left|  f-f_{\Omega}\right|  d\mu.
\end{align*}
At this point we can apply (\ref{poinbesov}) to obtain%
\[
I\leq2\mu(\Omega)\frac{K\left(  \frac{\mu(\Omega)/2}{I_{\Omega}(\mu
(\Omega)/2)},f\right)  }{\phi_{X}(\mu(\Omega))}.
\]
A well known elementary fact about $K-$functionals is that they are concave
functions (cf. \cite{bs}), in particular $\frac{K(t,f)}{t}$ is a decreasing
function. Thus, since $I_{\Omega}(t)$ is increasing on $(0,\mu(\Omega)/2)),$
we see that
\begin{align*}
K\left(  \frac{\mu(\Omega)/2}{I_{\Omega}(\mu(\Omega)/2)},f\right)   &
\leq\frac{\mu(\Omega)/2}{I_{\Omega}(\mu(\Omega)/2)}\frac{I_{\Omega}(t)}%
{t}K\left(  \frac{t}{I_{\Omega}(t)},f\right) \\
&  \leq\frac{\mu(\Omega)/2}{I_{\Omega}(\mu(\Omega)/2)}\frac{I_{\Omega}%
(\mu(\Omega)/2)}{t}K\left(  \frac{t}{I_{\Omega}(t)},f\right)  \text{ }\\
&  \leq2K\left(  \frac{t}{I_{\Omega}(t)},f\right)  \text{ (since }t>\mu
(\Omega)/4).
\end{align*}
Therefore,%
\begin{align*}
\left|  f\right|  _{\mu}^{\ast\ast}(t)-\left|  f\right|  _{\mu}^{\ast}(t))  &
\leq4\frac{\mu(\Omega)}{t}\frac{K\left(  \frac{t}{I_{\Omega}(t)},f\right)
}{\phi_{X}(\mu(\Omega))}\\
&  \leq16\frac{K\left(  \frac{t}{I_{\Omega}(t)},f\right)  }{\phi_{X}%
(\mu(\Omega))}\text{ (since }t>\mu(\Omega)/4)\\
&  \leq16\frac{K\left(  \frac{t}{I_{\Omega}(t)},f\right)  }{\phi_{X}(t)}\text{
(since }\frac{1}{\phi_{X}}\text{ decreases).}%
\end{align*}

Let us now to prove (\ref{polakita}). Once again we divide the proof in two
cases. For $t$ $\in$ $(0,\mu(\Omega)/2)$ we claim that (\ref{polakita}) is a
consequence of (\ref{polaka1}). Indeed, if we apply (\ref{polaka1}) to the
function $f-f_{\Omega}$ and then use the fact that the $K-$functional is
subadditive and zero on constant functions, we obtain
\begin{align*}
\left|  f-f_{\Omega}\right|  _{\mu}^{\ast\ast}(t)-\left|  f-f_{\Omega}\right|
_{\mu}^{\ast}(t)  &  \leq16K\left(  \frac{t}{I_{\Omega}(t)},f-f_{\Omega
}\right) \\
&  \leq16K\left(  \frac{t}{I_{\Omega}(t)},f\right)  +K\left(  \frac
{t}{I_{\Omega}(t)},f_{\Omega}\right) \\
&  =16K\left(  \frac{t}{I_{\Omega}(t)},f\right)  .
\end{align*}

Suppose now that $t\in(\mu(\Omega)/2,\mu(\Omega)).$ We have%
\begin{align*}
\left\vert f-f_{\Omega}\right\vert _{\mu}^{\ast\ast}(t)-\left\vert
f-f_{\Omega}\right\vert _{\mu}^{\ast}(t)  &  \leq\left\vert f-f_{\Omega
}\right\vert _{\mu}^{\ast\ast}(t)\\
&  \leq\frac{1}{t}\int_{0}^{t}\left\vert f-f_{\Omega}\right\vert _{\mu}^{\ast
}(s)ds\\
&  \leq\frac{2}{\mu(\Omega)}\int_{0}^{\mu(\Omega)}\left\vert f-f_{\Omega
}\right\vert _{\mu}^{\ast}(s)ds\\
&  =\frac{2}{\mu(\Omega)}\int_{\Omega}\left\vert f-f_{\Omega}\right\vert
d\mu\\
&  \leq4\frac{K\left(  \frac{\mu(\Omega)/2}{I_{\Omega}(\mu(\Omega
)/2)},f\right)  }{\phi_{X}(\mu(\Omega))}\text{ (by (\ref{poinbesov})).}%
\end{align*}
Recalling that $\frac{t}{I_{\Omega}(t)}$ is increasing and $\frac{1}{\phi
_{X}(t)}$ is decreasing, we can continue with
\begin{align*}
4\frac{K\left(  \frac{\mu(\Omega)/2}{I_{\Omega}(\mu(\Omega)/2)},f\right)
}{\phi_{X}(\mu(\Omega))}  &  \leq4\frac{K\left(  \frac{t}{I_{\Omega}%
(t)},f\right)  }{\phi_{X}(\mu(\Omega))}\\
&  \leq4\frac{K\left(  \frac{t}{I_{\Omega}(t)},f\right)  }{\phi_{X}%
(t)}\text{,}%
\end{align*}
an the desired result follows.
\end{proof}

A useful variant of the previous result can be stated as follows (cf.
\cite{mastylo} for a somewhat weaker result).

\begin{theorem}
\label{teo:elotro}Let $X=X(\Omega)$ be a r.i. space with $\underline{\alpha
}_{X}>0.$ Then, there exists a constant $c=c(X)>0$ such that for all $f\in
X,$
\begin{equation}
\left\Vert \left(  \left\vert f\right\vert _{\mu}^{\ast}(\cdot)-\left\vert
f\right\vert _{\mu}^{\ast}(t/2)\right)  \chi_{(0,t/2)}(\cdot)\right\Vert
_{\bar{X}}\leq cK\left(  \frac{t}{I_{\Omega}(t)},f;X,S_{X}\right)  ,\text{
}0<t\leq\mu(\Omega). \label{fertilizada}%
\end{equation}

\end{theorem}

\begin{proof}
Fix $t\in(0,\mu(\Omega)].$ We will first suppose that $f$ is bounded. We claim
that in the computation of $K\left(  t,f;X,S_{X}\right)  $ we can restrict
ourselves to consider decompositions with bounded $h.$ In fact, we have
\begin{equation}
\inf_{0\leq h\leq\left\Vert f\right\Vert _{\infty},h\in S_{X}}\left\{
\left\Vert \left\vert f\right\vert -h\right\Vert _{X}+t\left\Vert \left\vert
\nabla h\right\vert \right\Vert _{X}\right\}  \leq2\inf_{0\leq h\in S_{X}%
}\left\{  \left\Vert \left\vert f\right\vert -h\right\Vert _{X}+t\left\Vert
\left\vert \nabla h\right\vert \right\Vert _{X}\right\}  . \label{porfin1}%
\end{equation}
To see this consider any competing decomposition with $0\leq h\in S_{X}.$ Let
\[
g=\min(h,\left\Vert f\right\Vert _{\infty})\in S_{X}.
\]
Then,
\[
\left\vert \nabla g\right\vert \leq\left\vert \nabla h\right\vert ,
\]
and, moreover, we have
\begin{align*}
\left\Vert \left\vert f\right\vert -g\right\Vert _{X}+t\left\Vert \left\vert
\nabla g\right\vert \right\Vert _{X}  &  \leq\left\Vert \left(  \left\vert
f\right\vert -h\right)  \chi_{\left\{  h\leq\left\Vert f\right\Vert _{\infty
}\right\}  }\right\Vert _{X}+\left\Vert \left(  \left\vert f\right\vert
-\left\Vert f\right\Vert _{\infty}\right)  \chi_{\left\{  h>\left\Vert
f\right\Vert _{\infty}\right\}  }\right\Vert _{X}\\
&  +t\left\Vert \left\vert \nabla h\right\vert \right\Vert _{X}\\
&  \leq\left\Vert \left\vert f\right\vert -h\right\Vert _{X}+\left\Vert
\left(  \left\vert f\right\vert -\left\Vert f\right\Vert _{\infty}\right)
\chi_{\left\{  h>\left\Vert f\right\Vert _{\infty}\right\}  }\right\Vert
_{X}+t\left\Vert \left\vert \nabla h\right\vert \right\Vert _{X}.
\end{align*}
Now, since
\[
\left\vert \left(  \left\vert f\right\vert -\left\Vert f\right\Vert _{\infty
}\right)  \chi_{\left\{  h>\left\Vert f\right\Vert _{\infty}\right\}
}\right\vert =\left(  \left\Vert f\right\Vert _{\infty}-\left\vert
f\right\vert \right)  \chi_{\left\{  h>\left\Vert f\right\Vert _{\infty
}\right\}  }\leq\left(  h-\left\vert f\right\vert \right)  \chi_{\left\{
h>\left\Vert f\right\Vert _{\infty}\right\}  },
\]
we see that
\[
\left\Vert \left(  \left\vert f\right\vert -\left\Vert f\right\Vert _{\infty
}\right)  \chi_{\left\{  h>\left\Vert f\right\Vert _{\infty}\right\}
}\right\Vert _{X}=\left\Vert \left(  h-\left\vert f\right\vert \right)
\chi_{\left\{  h>\left\Vert f\right\Vert _{\infty}\right\}  }\right\Vert
_{X}\leq\left\Vert \left\vert f\right\vert -h\right\Vert _{X}.
\]
Consequently%
\[
\left\Vert \left\vert f\right\vert -g\right\Vert _{X}+t\left\Vert \left\vert
\nabla g\right\vert \right\Vert _{X}\leq2\left\Vert \left\vert f\right\vert
-h\right\Vert _{X}+t\left\Vert \left\vert \nabla h\right\vert \right\Vert
_{X},
\]
and (\ref{porfin1}) follows.

Let $0\leq h$ be a bounded $Lip(\Omega)$ function$,$ and fix $g\in\bar
{X}^{\prime}$ with $\left\Vert g\right\Vert _{\bar{X}^{\prime}}=1.$ Recall
that $\bar{X}^{\prime}$ is a r.i. space on $\left(  [0,\mu(\Omega)],m\right)
$, where $m$ denotes Lebesgue measure. We let $|g|^{\ast}:=|g|_{m}^{\ast}.$
Consider the decomposition
\[
\left\vert f\right\vert =\left(  \left\vert f\right\vert -h\right)  +h.
\]
Writing $h=\left\vert f\right\vert +(-\left(  \left\vert f\right\vert
-h\right)  ),$ we see that $\left\vert h\right\vert _{\mu}^{\ast}%
(t)\leq\left\vert f\right\vert _{\mu}^{\ast}(t/2)+\left\vert \left\vert
f\right\vert -h\right\vert _{\mu}^{\ast}(t/2).$ We use this inequality to
provide a lower bound for $\left\vert f\right\vert _{\mu}^{\ast}(t/2),$ namely%
\[
\left\vert h\right\vert _{\mu}^{\ast}(t)-\left\vert \left\vert f\right\vert
-h\right\vert _{\mu}^{\ast}(t/2)\leq\left\vert f\right\vert _{\mu}^{\ast
}(t/2).
\]
Then we have%
\begin{align}
&  \int_{0}^{t/2}(\left\vert f\right\vert _{\mu}^{\ast}(s)-\left\vert
f\right\vert _{\mu}^{\ast}(t/2))\chi_{(0,t/2)}(s)\left\vert g\right\vert
^{\ast}(s)ds\nonumber\\
&  =\int_{0}^{t/2}\left\vert f\right\vert _{\mu}^{\ast}(s)\left\vert
g\right\vert ^{\ast}(s)ds-\left\vert f\right\vert _{\mu}^{\ast}(t/2)\int%
_{0}^{t/2}\left\vert g\right\vert ^{\ast}(s)ds\\
&  \leq\int_{0}^{t/2}\left\vert f\right\vert _{\mu}^{\ast}(s)\left\vert
g\right\vert ^{\ast}(s)ds-\left(  \left\vert h\right\vert _{\mu}^{\ast
}(t)-\left\vert \left\vert f\right\vert -h\right\vert _{\mu}^{\ast
}(t/2)\right)  \int_{0}^{t/2}\left\vert g\right\vert ^{\ast}(s)ds.
\label{este0}%
\end{align}
To estimate the first integral on the right hand side without sacrificing
precision on the range of the variable $^{\ast}t^{\ast}$ requires an argument.
We shall use the majorization principle [(\ref{mayorante}), Chapter
\ref{preliminar}] as follows. Since the operation $f\rightarrow f^{\ast\ast}$
is sub-additive, for any $r>0$ we have%
\begin{align*}
\int_{0}^{r}\left\vert f\right\vert _{\mu}^{\ast}(s)\chi_{(0,\frac{t}{2}%
)}(s)ds  &  =\int_{0}^{\min\{r,\frac{t}{2}\}}\left\vert f\right\vert _{\mu
}^{\ast}(s)ds\leq\int_{0}^{\min\{r,\frac{t}{2}\}}\left\vert \left\vert
f\right\vert -h\right\vert _{\mu}^{\ast}(s)ds+\int_{0}^{\min\{r,\frac{t}{2}%
\}}\left\vert h\right\vert _{\mu}^{\ast}(s)ds\\
&  =\int_{0}^{\min\{r,\frac{t}{2}\}}\left(  \left\vert \left\vert f\right\vert
-h\right\vert _{\mu}^{\ast}(s)+\left\vert h\right\vert _{\mu}^{\ast
}(s)\right)  ds\\
&  \leq\int_{0}^{r}\left(  \left\vert \left\vert f\right\vert -h\right\vert
_{\mu}^{\ast}(s)+\left\vert h\right\vert _{\mu}^{\ast}(s)\right)
\chi_{(0,\frac{t}{2})}(s)ds.
\end{align*}
Now, since for each fixed $t$ the functions $H_{1}=\left\vert f\right\vert
_{\mu}^{\ast}(\cdot)\chi_{(0,\frac{t}{2})}(\cdot)$ and $H_{2}=\left(
\left\vert \left\vert f\right\vert -h\right\vert _{\mu}^{\ast}(\cdot
)+\left\vert h\right\vert _{\mu}^{\ast}(\cdot)\right)  \chi_{(0,\frac{t}{2}%
)}(\cdot)$ are decreasing, we can apply the Calder\'{o}n-Hardy Lemma to the
(Lorentz) function seminorm defined by (cf. \cite[Theorem 1]{lor})%
\[
\left\Vert H\right\Vert _{\Lambda g}=\int_{0}^{\mu(\Omega)}\left\vert
H\right\vert ^{\ast}(s)\left\vert g\right\vert ^{\ast}(s)ds.
\]
We obtain%
\[
\left\Vert H_{1}\right\Vert _{\Lambda g}\leq\left\Vert H_{2}\right\Vert
_{\Lambda g}.
\]
It follows that%
\begin{align*}
\int_{0}^{t/2}\left\vert f\right\vert _{\mu}^{\ast}(s)\left\vert g\right\vert
^{\ast}(s)ds  &  =\int_{0}^{\mu(\Omega)}\left\vert f\right\vert _{\mu}^{\ast
}(s)\chi_{(0,\frac{t}{2})}(s)\left\vert g\right\vert ^{\ast}(s)ds\\
&  \leq\int_{0}^{\mu(\Omega)}\left(  \left\vert \left\vert f\right\vert
-h\right\vert _{\mu}^{\ast}(s)+\left\vert h\right\vert _{\mu}^{\ast
}(s)\right)  \chi_{(0,\frac{t}{2})}(s)\left\vert g\right\vert ^{\ast}(s)ds\\
&  =\int_{0}^{t/2}\left(  \left\vert \left\vert f\right\vert -h\right\vert
_{\mu}^{\ast}(s)+\left\vert h\right\vert _{\mu}^{\ast}(s)\right)  \left\vert
g\right\vert ^{\ast}(s)ds.
\end{align*}
Inserting this estimate back in (\ref{este0}) we have,%
\begin{align*}
&  \int_{0}^{t/2}(\left\vert f\right\vert _{\mu}^{\ast}(s)-\left\vert
f\right\vert _{\mu}^{\ast}(t/2))\chi_{(0,t/2)}(s)\left\vert g\right\vert
^{\ast}(s)ds\\
&  \leq\int_{0}^{t/2}\left(  \left\vert \left\vert f\right\vert -h\right\vert
_{\mu}^{\ast}(s)+\left\vert h\right\vert _{\mu}^{\ast}(s)\right)  \left\vert
g\right\vert ^{\ast}(s)ds-\left(  \left\vert h\right\vert _{\mu}^{\ast
}(t)-\left\vert \left\vert f\right\vert -h\right\vert _{\mu}^{\ast
}(t/2)\right)  \int_{0}^{t/2}\left\vert g\right\vert ^{\ast}(s)ds\\
&  =\int_{0}^{t/2}\left\vert \left\vert f\right\vert -h\right\vert _{\mu
}^{\ast}(s)\left\vert g\right\vert ^{\ast}(s)ds+\int_{0}^{t/2}\left\vert
h\right\vert _{\mu}^{\ast}(s)\left\vert g\right\vert ^{\ast}(s)ds+(\left\vert
\left\vert f\right\vert -h\right\vert _{\mu}^{\ast}(t/2)-\left\vert
h\right\vert _{\mu}^{\ast}(t))\int_{0}^{t/2}\left\vert g\right\vert ^{\ast
}(s)ds\\
&  =\int_{0}^{t/2}\left(  \left\vert \left\vert f\right\vert -h\right\vert
_{\mu}^{\ast}(s)+\left\vert \left\vert f\right\vert -h\right\vert _{\mu}%
^{\ast}(t/2)\right)  \left\vert g\right\vert ^{\ast}(s)ds+\int_{0}%
^{t/2}\left(  \left\vert h\right\vert _{\mu}^{\ast}(s)-\left\vert h\right\vert
_{\mu}^{\ast}(t)\right)  \left\vert g\right\vert ^{\ast}(s)ds\\
&  \leq2\int_{0}^{\mu(\Omega)}\left\vert \left\vert f\right\vert -h\right\vert
_{\mu}^{\ast}(s)\chi_{(0,t/2)}(s)\left\vert g\right\vert ^{\ast}(s)ds+\int%
_{0}^{\mu(\Omega)}\left(  \left\vert h\right\vert _{\mu}^{\ast}(s)-\left\vert
h\right\vert _{\mu}^{\ast}(t)\right)  \chi_{(0,t/2)}(s)\left\vert g\right\vert
^{\ast}(s)ds\\
&  \leq2\left\Vert \left\vert f\right\vert -h\right\Vert _{\bar{X}}+\left\Vert
(\left\vert h\right\vert _{\mu}^{\ast}(\cdot)-\left\vert h\right\vert _{\mu
}^{\ast}(t))\chi_{(0,t/2)}\right\Vert _{\bar{X}}\text{ (by (\ref{representada}%
), Chapter \ref{preliminar}).}%
\end{align*}

Now, let $\{h_{n}\}_{n\in N}$ be the sequence provided by [Lemma \ref{ll1},
Chapter \ref{preliminar}]. Then,
\begin{align*}
\left(  \left\vert h_{n}\right\vert _{\mu}^{\ast}(s)-\left\vert h_{n}%
\right\vert _{\mu}^{\ast}(t)\right)  \chi_{(0,t/2)}(s)  &  =\int_{s}%
^{t}\left(  -\left\vert h_{n}\right\vert _{\mu}^{\ast}\right)  ^{\prime
}(r)dr\chi_{(0,t/2)}(s)\\
&  \leq\int_{s}^{t}\left(  -\left\vert h_{n}\right\vert _{\mu}^{\ast}\right)
^{\prime}(r)I_{\Omega}(r)\frac{dr}{I_{\Omega}(r)}\\
&  \leq\frac{t}{I_{\Omega}(t)}\int_{s}^{t}\left(  -\left\vert h_{n}\right\vert
_{\mu}^{\ast}\right)  ^{\prime}(r)I_{\Omega}(r)\frac{dr}{r}\\
&  \leq\frac{t}{I_{\Omega}(t)}\int_{s}^{\mu(\Omega)}\left(  -\left\vert
h_{n}\right\vert _{\mu}^{\ast}\right)  ^{\prime}(r)I_{\Omega}(r)\frac{dr}{r}\\
&  =\frac{t}{I_{\Omega}(t)}Q(\left(  -\left\vert h_{n}\right\vert _{\mu}%
^{\ast}\right)  ^{\prime}I_{\Omega})(s).
\end{align*}
Applying $\left\Vert \cdot\right\Vert _{\bar{X}}$ (in the variable $s$) and
using the fact that $\underline{\alpha}_{X}>0$, we see that
\begin{align*}
\left\Vert \left(  \left\vert h_{n}\right\vert _{\mu}^{\ast}(\cdot)-\left\vert
h_{n}\right\vert _{\mu}^{\ast}(t)\right)  \chi_{(0,t/2)}(\cdot)\right\Vert
_{\bar{X}}  &  \leq c\frac{t}{I_{\Omega}(t)}\left\Vert \left(  -\left\vert
h_{n}\right\vert _{\mu}^{\ast}\right)  ^{\prime}(\cdot)I_{\Omega}%
(\cdot)\right\Vert _{\bar{X}}\\
&  \leq c\frac{t}{I_{\Omega}(t)}\left\Vert \left\vert \nabla h_{n}\right\vert
\right\Vert _{X}\text{ \ \ (by [(\ref{aa}), Chapter \ref{preliminar}])}\\
&  \leq c\frac{t}{I_{\Omega}(t)}(1+\frac{1}{n})\left\Vert \left\vert \nabla
h_{n}\right\vert \right\Vert _{X}\text{ \ (by [(\ref{cota01}), Chapter
\ref{preliminar}]).}%
\end{align*}
On the other hand, by [(\ref{converge}) Chapter \ref{preliminar}],
$h_{n}\underset{n\rightarrow\infty}{\rightarrow}h$ $\ $in $L^{1},$ and by
Lemma \ref{garlem} we get $\left\vert h_{n}\right\vert _{\mu}^{\ast
}(s)\rightarrow\left\vert h\right\vert _{\mu}^{\ast}(s),$ thus applying
Fatou's Lemma in the $X$ norm (cf. Section \ref{secc:ri} in Chapter
\ref{preliminar}), we find that
\begin{align}
\left\Vert \left(  \left\vert h\right\vert _{\mu}^{\ast}(\cdot)-\left\vert
h\right\vert _{\mu}^{\ast}(t)\right)  \chi_{(0,t)}(\cdot)\right\Vert _{\bar
{X}}  &  =\lim_{n\rightarrow\infty}\inf\left\Vert \left(  \left\vert
h_{n}\right\vert _{\mu}^{\ast}(\cdot)-\left\vert h_{n}\right\vert _{\mu}%
^{\ast}(t)\right)  \chi_{(0,t)}(s)\right\Vert _{\bar{X}}\nonumber\\
&  \leq c\frac{t}{I_{\Omega}(t)}\left\Vert \left\vert \nabla h\right\vert
\right\Vert _{X}. \label{este}%
\end{align}
Combining (\ref{este0}) and (\ref{este}) \ we get%
\begin{equation}
\int_{0}^{\mu(\Omega)}(\left\vert f\right\vert _{\mu}^{\ast}(s)-\left\vert
f\right\vert _{\mu}^{\ast}(t/2))\chi_{(0,t/2)}(s)\left\vert g\right\vert
^{\ast}(s)ds\leq c(\left\Vert \left\vert f\right\vert -h\right\Vert _{X}%
+\frac{t}{I_{\Omega}(t)}\left\Vert \left\vert \nabla h\right\vert \right\Vert
_{X}). \label{este2S}%
\end{equation}
Since $\left(  \left\vert f\right\vert _{\mu}^{\ast}(s)-\left\vert
f\right\vert _{\mu}^{\ast}(t/2)\right)  \chi_{(0,t/2)}(s)$ is a decreasing
function of $s,$ we can write
\[
\left(  \left\vert f\right\vert _{\mu}^{\ast}(s)-\left\vert f\right\vert
_{\mu}^{\ast}(t/2)\right)  \chi_{(0,t/2)}(s)=\left(  \left(  \left\vert
f\right\vert _{\mu}^{\ast}(\cdot)-\left\vert f\right\vert _{\mu}^{\ast
}(t/2)\right)  \chi_{(0,t/2)}(\cdot)\right)  ^{\ast}(s).
\]
Combining successively duality, the last formula and (\ref{este2S}), we get
\begin{align*}
&  \left\Vert \left(  \left\vert f\right\vert _{\mu}^{\ast}(\cdot)-\left\vert
f\right\vert _{\mu}^{\ast}(t/2)\right)  \chi_{(0,t/2)}(s)\right\Vert _{\bar
{X}}\\
&  =\sup_{\left\Vert g\right\Vert _{\bar{X}^{\prime}}\leq1}\int_{0}%
^{\mu(\Omega)}[\left(  \left\vert f\right\vert _{\mu}^{\ast}(\cdot)-\left\vert
f\right\vert _{\mu}^{\ast}(t/2)\right)  \chi_{(0,t/2)}(\cdot)]^{\ast
}(s)\left\vert g\right\vert ^{\ast}(s)ds\\
&  =\sup_{\left\Vert g\right\Vert _{\bar{X}^{\prime}}\leq1}\int_{0}%
^{\mu(\Omega)}\left(  \left\vert f\right\vert _{\mu}^{\ast}(s)-\left\vert
f\right\vert _{\mu}^{\ast}(t/2)\right)  \chi_{(0,t/2)}(s)\left\vert
g\right\vert ^{\ast}(s)ds\\
&  \leq2c\left(  \left\Vert \left\vert f\right\vert -h\right\Vert _{X}%
+\frac{t}{I_{\Omega}(t)}\left\Vert \left\vert \nabla h\right\vert \right\Vert
_{X}\right)  ,
\end{align*}
where $c$ is an absolute constant that depends only on $X$. Consequently, if
$f$ is bounded there exists an absolute constant $c>0$ such that
\begin{align*}
\left\Vert \left(  \left\vert f\right\vert _{\mu}^{\ast}(\cdot)-\left\vert
f\right\vert _{\mu}^{\ast}(t/2)\right)  \chi_{(0,t/2)}(\cdot)\right\Vert
_{\bar{X}}  &  \leq c\inf_{0\leq h\in S_{X}}\left\{  \left\Vert \left\vert
f\right\vert -h\right\Vert _{X}+\frac{t}{I_{\Omega}(t)}\left\Vert \left\vert
\nabla h\right\vert \right\Vert _{X}\right\} \\
&  \leq cK\left(  \frac{t}{I_{\Omega}(t)},f;X,S_{X}\right)  \text{ (by
(\ref{porfin})).}%
\end{align*}
Suppose now that $f$ is not bounded. Let $f_{n}=\min(\left\vert f\right\vert
,n)\nearrow\left\vert f\right\vert $. By the first part of the proof we have,
\[
\left\vert f_{n}\right\vert _{\mu}^{\ast\ast}(t)-\left\vert f_{n}\right\vert
_{\mu}^{\ast}(t)\leq cK\left(  \frac{t}{I_{\Omega}(t)},f_{n};X,S_{X}\right)
.
\]
Fix $n.$ For any $0\leq h\in S_{X},$ let $\tilde{h}_{n}=\min\left(
h,n\right)  .$ Then,%
\begin{align*}
K\left(  \frac{t}{I(t)},f_{n};X,S_{X}\right)   &  \leq\left\Vert \left\vert
f_{n}\right\vert -\tilde{h}_{n}\right\Vert _{X}+\frac{t}{I_{\Omega}%
(t)}\left\Vert \left\vert \nabla\tilde{h}_{n}\right\vert \right\Vert _{X}\\
&  \leq\left\Vert \left\vert f\right\vert -h\right\Vert _{X}+\frac
{t}{I_{\Omega}(t)}\left\Vert \left\vert \nabla h\right\vert \right\Vert _{X}.
\end{align*}
Taking infimum it follows that, for all $n\in\mathbb{N},$%
\[
K\left(  \frac{t}{I_{\Omega}(t)},f_{n};X,S_{X}\right)  \leq K\left(  \frac
{t}{I_{\Omega}(t)},f;X,S_{X}\right)  .
\]
Since $f_{n}=\min(\left\vert f\right\vert ,n)\nearrow\left\vert f\right\vert
$, then by the Fatou property of the norm we have%
\begin{align*}
\left\Vert \left(  \left\vert f\right\vert _{\mu}^{\ast}(\cdot)-\left\vert
f\right\vert _{\mu}^{\ast}(t/2)\right)  \chi_{(0,t/2)}(\cdot)\right\Vert
_{\bar{X}}  &  =\lim_{n}\left\Vert \left(  \left\vert f_{n}\right\vert _{\mu
}^{\ast}(\cdot)-\left\vert f_{n}\right\vert _{\mu}^{\ast}(t/2)\right)
\chi_{(0,t/2)}(\cdot)\right\Vert _{X}\\
&  \leq c\lim_{n}K\left(  \frac{t}{I_{\Omega}(t)},f_{n};X,S_{X}\right) \\
&  \leq cK\left(  \frac{t}{I_{\Omega}(t)},f;X,S_{X}\right)  ,
\end{align*}
as we wished to show.
\end{proof}

\begin{remark}
For perspective we now show that, under the extra assumption that the r.i.
space satisfies $\underline{\alpha}_{X}>0,$ (\ref{fertilizada}) can be used to
give a direct proof of (\ref{polaka1}).
\end{remark}

\begin{proof}
We have,
\begin{align*}
\left\vert f\right\vert _{\mu}^{\ast\ast}(t/2)-\left\vert f\right\vert _{\mu
}^{\ast}(t/2)  &  =\frac{2}{t}\int_{0}^{\mu(\Omega)}\left(  \left\vert
f\right\vert _{\mu}^{\ast}(s)-\left\vert f\right\vert _{\mu}^{\ast
}(t/2)\right)  \chi_{(0,t/2)}(s)ds\\
&  \leq\left\Vert \left(  \left\vert f\right\vert _{\mu}^{\ast}(\cdot
)-\left\vert f\right\vert _{\mu}^{\ast}(t/2)\right)  \chi_{(0,t/2)}%
(\cdot)\right\Vert _{\bar{X}}\frac{2\phi_{X^{\prime}}(t/2)}{t}\text{ \ \ (by
H\"{o}lder's inequality) }\\
&  \leq2cK\left(  \frac{t}{I_{\Omega}(t)},f;X,S_{X}\right)  \frac
{\phi_{X^{\prime}}(t)}{t}\text{ (by (\ref{fertilizada}))}\\
&  =2c\frac{K\left(  \frac{t}{I_{\Omega}(t)},f;X,S_{X}\right)  }{\phi_{X}(t)}.
\end{align*}

\end{proof}

\begin{example}
For familiar spaces (\ref{fertilizada}) takes a more concrete form. For
example, if $X=L^{p},$ $1\leq p<\infty,$ then $\underline{\alpha}_{L^{p}}>0,$
and (\ref{fertilizada}) becomes
\begin{equation}
\int_{0}^{t/2}\left(  \left\vert f\right\vert _{\mu}^{\ast}(s)-\left\vert
f\right\vert _{\mu}^{\ast}(t/2)\right)  ^{p}ds\leq c_{p}\left(  K\left(
\frac{t}{I_{\Omega}(t)},f;L^{p},S_{L^{p}}\right)  \right)  ^{p}.
\label{laizquierda}%
\end{equation}
In particular, when $p=1,$ the left hand side of (\ref{laizquierda}) becomes
\[
\frac{t}{2}(|f|_{\mu}^{\ast\ast}(t/2)-|f|_{\mu}^{\ast}(t/2))=\int_{0}%
^{t/2}\left(  f_{\mu}^{\ast}(s)-f_{\mu}^{\ast}(t/2)\right)  ds.
\]
As a consequence, when $X=L^{1}\ $and $0<t<\mu(\Omega)/4,$ (\ref{fertilizada})
and (\ref{polaka1}) represent the same inequality, modulo constants.
\end{example}

The next easy variant of Theorem \ref{main1} gives more flexibility for some applications.

\begin{theorem}
Let $X$ and $Y$ be a r.i. spaces on $\Omega.$ Then, for each $f\in X+S_{Y}$ we
have%
\begin{equation}
\left\vert f\right\vert _{\mu}^{\ast\ast}(t/2)-\left\vert f\right\vert _{\mu
}^{\ast}(t/2)\leq c\frac{K\left(  \frac{t}{I_{\Omega}(t)}\frac{\phi_{X}%
(t)}{\phi_{Y}(t)},f;X,S_{Y}\right)  }{\phi_{X}(t)},\,0<t\leq\mu(\Omega).
\label{polaca}%
\end{equation}

\end{theorem}

\begin{proof}
Since the proof is almost the same as the proof of Theorem \ref{main1} we
shall only briefly indicate the necessary changes. Let $f=f_{0}+f_{1}$ be a
decomposition of $f$, then using estimate (\ref{det2}), with $Y$ instead of
$X,$ we get
\[
\left\vert \nabla f_{1}\right\vert _{\mu}^{\ast\ast}(t)\leq\frac{\left\Vert
\left\vert \nabla f_{1}\right\vert \right\Vert _{Y}}{\phi_{Y}(t)}.
\]
Therefore,
\begin{equation}
\frac{t}{I_{\Omega}(t)}\left\vert \nabla f_{1}\right\vert _{\mu}^{\ast\ast
}(t)\leq\frac{\phi_{X}(t)}{\phi_{X}(t)}\frac{t}{I_{\Omega}(t)}\frac{\left\Vert
\left\vert \nabla f_{1}\right\vert \right\Vert _{Y}}{\phi_{Y}(t)}.
\label{det41}%
\end{equation}
Inserting (\ref{det1}) and (\ref{det41}) back in (\ref{det5}) we find that,
for $0<t<\mu(\Omega),$%
\begin{align*}
\left\vert f\right\vert _{\mu}^{\ast\ast}(t/2)-\left\vert f\right\vert _{\mu
}^{\ast}(t/2)  &  \leq\frac{\left\Vert f_{0}\right\Vert _{X}}{\phi_{X}%
(t)}+\frac{\phi_{X}(t)}{\phi_{X}(t)}\frac{t}{I_{\Omega}(t)}\frac{\left\Vert
\left\vert \nabla f_{1}\right\vert \right\Vert _{Y}}{\phi_{Y}(t)}\\
&  \leq\frac{1}{\phi_{X}(t)}\left(  \left\Vert f_{0}\right\Vert _{X}+\phi
_{X}(t)\frac{t}{I_{\Omega}(t)}\frac{\left\Vert \left\vert \nabla
f_{1}\right\vert \right\Vert _{Y}}{\phi_{Y}(t)}\right)  .
\end{align*}
The desired result follows taking infimum over all decompositions of $f.$
\end{proof}

\begin{remark}
Obviously if there exists a constant $c>0$ such that for all $t$
\begin{equation}
\phi_{X}(t)\leq c\phi_{Y}(t), \label{mass}%
\end{equation}
then for each $f\in X+S_{Y},$%
\[
\left\vert f\right\vert _{\mu}^{\ast\ast}(t/2)-\left\vert f\right\vert _{\mu
}^{\ast}(t/2)\leq2\frac{K\left(  \frac{ct}{I_{\Omega}(t)},f;X,S_{Y}\right)
}{\phi_{X}(t)},\text{ }0<t\leq\mu(\Omega).
\]
Note that $Y\subset X$ implies that (\ref{mass}) holds.
\end{remark}

\section{A variant of Theorem \ref{main1}\label{sec:variante}}

This section is devoted to the proof of an improvement of Theorem \ref{main1}
that will play an important role in Chapter \ref{contchap}. In this variant we
shall replace the variable inside the $K-$functional, namely $\frac
{t}{I_{\Omega}(t)},$ by a smaller function that depends on the space $X.$ The
relevant functions are defined next

\begin{definition}
\label{ladefdelapsi}Let $X$ be a r.i. space on $\Omega.$ For $t\in
(0,\mu(\Omega))$ we let
\begin{equation}
\psi_{X,\Omega}(t)=\frac{\phi_{X}(t)}{t}\left\Vert \frac{s}{I_{\Omega}(s)}%
\chi_{(0,t)}(s)\right\Vert _{\bar{X}^{^{\prime}}}, \label{polaka3}%
\end{equation}%
\begin{equation}
\Psi_{X,\Omega}(t)=\phi_{X}(t)\left\Vert \frac{1}{I_{\Omega}(s)}\chi
_{(0,t)}(s)\right\Vert _{\bar{X}^{^{\prime}}}. \label{polaka3sig}%
\end{equation}

\end{definition}

If either $X$ and/or $\Omega$ are clear from the context we shall drop the
corresponding sub-index. In the next Lemma we collect a few elementary remarks
connected with these functions.

\begin{lemma}
\label{ellemadelapsi}(i) The function $\psi_{X}(t)$ is always finite, and in
fact%
\[
\psi_{X}(t)\leq\frac{t}{I_{\Omega}(t)},\text{ }t\in(0,\mu(\Omega)).
\]
(ii) In the isoperimetric case there is no improvement: When $X=L^{1},$%
\[
\psi_{L^{1}}(t)=\frac{t}{I_{\Omega}(t)}.
\]
(iii) We always have%
\[
\psi_{X}\leq\Psi_{X}.
\]
(iv) The function $\Psi_{X}$ is increasing

(v) A necessary and sufficient condition for $\Psi_{X}(t)$ to take only finite
values is%
\begin{equation}
\left\Vert \frac{1}{I_{\Omega}(s)}\right\Vert _{\bar{X}^{^{\prime}}}<\infty.
\label{conditione}%
\end{equation}

\end{lemma}

\begin{proof}
(i) Since $\frac{s}{I_{\Omega}(s)}$ is increasing,
\begin{align*}
\frac{\phi_{X}(t)}{t}\left\Vert \frac{s}{I_{\Omega}(s)}\chi_{(0,t)}%
(s)\right\Vert _{\bar{X}^{^{\prime}}}  &  \leq\frac{\phi_{X}(t)}{t}\frac
{t}{I_{\Omega}(t)}\left\Vert \chi_{(0,t)}(s)\right\Vert _{\bar{X}^{^{\prime}}%
}\\
&  =\frac{\phi_{X}(t)}{t}\frac{t}{I_{\Omega}(t)}\phi_{X^{\prime}}(t)\\
&  =\frac{t}{I_{\Omega}(t)}.
\end{align*}

(ii) For $X=L^{1},$%
\begin{align*}
\psi_{L^{1}}(t)  &  =\frac{\phi_{L^{1}}(t)}{t}\left\|  \frac{s}{I_{\Omega}%
(s)}\chi_{(0,t)}(s)\right\|  _{L^{\infty}}\\
&  =\frac{t}{t}\sup_{s<t}\left\{  \frac{s}{I_{\Omega}(s)}\right\} \\
&  =\frac{t}{I_{\Omega}(t)}.
\end{align*}
(iii) Since $s\uparrow$%
\begin{align*}
\psi_{X}(t)  &  =\frac{\phi_{X}(t)}{t}\left\|  \frac{s}{I_{\Omega}(s)}%
\chi_{(0,t)}(s)\right\|  _{\bar{X}^{^{\prime}}}\\
&  \leq\phi_{X}(t)\left\|  \frac{1}{I_{\Omega}(s)}\chi_{(0,t)}(s)\right\|
_{\bar{X}^{^{\prime}}}\\
&  =\Psi_{X}(t).
\end{align*}

(iv) By inspection $\left\Vert \frac{1}{I_{\Omega}(s)}\chi_{(0,t)}%
(s)\right\Vert _{\bar{X}^{^{\prime}}}$ increases in the variable $t.$

(v) By (iv) for $t\in(0,\mu(\Omega)),$
\[
\Psi_{X}(t)\leq\Psi_{X}(\mu(\Omega))=\phi_{X}(\mu(\Omega))\left\Vert \frac
{1}{I_{\Omega}(s)}\right\Vert _{\bar{X}^{^{\prime}}}.
\]
On the other hand, by the triangle inequality%
\[
\Psi_{X}(\mu(\Omega)/2)\geq\phi_{X}(\mu(\Omega)/2)\left\Vert \frac
{1}{I_{\Omega}(s)}\right\Vert _{\bar{X}^{^{\prime}}}-\phi_{X}(\mu
(\Omega)/2)\left\Vert \frac{\chi_{(\mu(\Omega)/2,\mu(\Omega))}(s)}{I_{\Omega
}(s)}\right\Vert _{\bar{X}^{^{\prime}}}%
\]
It follows that%
\[
\phi_{X}(\mu(\Omega)/2)\left\Vert \frac{1}{I_{\Omega}(s)}\right\Vert _{\bar
{X}^{^{\prime}}}\leq2\Psi_{X}(\mu(\Omega)/2).
\]

\end{proof}

\begin{remark}
\label{rekatao}Unless mention to the contrary, we shall always assume that
(\ref{conditione}) holds when dealing with the functions introduced in
Definition\ \ref{ladefdelapsi}.
\end{remark}

\begin{theorem}
\label{main2}Let $X$ be a r.i. space on $\Omega.$ Then, for all $f\in
X+S_{X},$%
\begin{equation}
\left\vert f\right\vert _{\mu}^{\ast\ast}(t)-\left\vert f\right\vert _{\mu
}^{\ast}(t)\leq8\frac{K\left(  \psi_{X}(2t),f;X,S_{X}\right)  }{\phi_{X}%
(t)},\text{ }0<t\leq\mu(\Omega)/2. \label{polaka2}%
\end{equation}

\end{theorem}

\begin{remark}
Lemma \ref{ellemadelapsi} (i) implies that (\ref{polaka2}) is stronger than
(\ref{polaka1}).
\end{remark}

\begin{proof}
\textbf{(of Theorem \ref{main2})}

Let $f\in X+S_{X},$ be bounded. Proceeding as in the proof of Theorem
\ref{main1} we can show that for any $h\in S_{X},$ with $0\leq h\leq\left\|
f\right\|  _{L^{\infty}},$%
\[
\left|  f\right|  _{\mu}^{\ast\ast}(t)-\left|  f\right|  _{\mu}^{\ast}%
(t)\leq4\left|  \left|  f\right|  -h\right|  _{\mu}^{\ast\ast}(t)+2(\left|
h\right|  _{\mu}^{\ast\ast}(2t)-\left|  h\right|  _{\mu}^{\ast}(2t)),\text{
\ \ }0<t<\mu(\Omega)/2.
\]
The first term on the right hand side was estimated in (\ref{det1})
\[
4\left|  \left|  f\right|  -h\right|  _{\mu}^{\ast\ast}(t)\leq4\left\|
\left(  \left|  f\right|  -h\right)  \right\|  _{X}\frac{1}{\phi_{X}(t)}.
\]
To estimate $\left|  h\right|  _{\mu}^{\ast\ast}(2t)-\left|  h\right|  _{\mu
}^{\ast}(2t),$ consider $\left\{  h_{n}\right\}  _{n\in N}$ , the sequence of
Lip functions associated to $h$ that is provided by Lemma \ref{ll1}. Then
\begin{align*}
\left|  h_{n}\right|  _{\mu}^{\ast\ast}(2t)-\left|  h_{n}\right|  _{\mu}%
^{\ast}(2t)  &  =\frac{1}{2t}\int_{0}^{2t}s\left(  -\left|  h_{n}\right|
_{\mu}^{\ast}\right)  ^{\prime}(s)ds\text{ (integration by parts)}\\
&  =\frac{1}{2t}\int_{0}^{2t}s\left(  -\left|  h_{n}\right|  _{\mu}^{\ast
}\right)  ^{\prime}(s)I_{\Omega}(s)\frac{ds}{I(s)}\\
&  \leq\frac{1}{2t}\left\|  \frac{s}{I_{\Omega}(s)}\chi_{(0,2t)}(s)\right\|
_{\bar{X}^{^{\prime}}}\left\|  \left(  -\left|  h_{n}\right|  _{\mu}^{\ast
}\right)  ^{\prime}(s)I_{\Omega}(s)\right\|  _{\bar{X}}\text{ \ (by
H\"{o}lder's inequality)}\\
&  \leq\frac{1}{2t}\left\|  \frac{s}{I_{\Omega}(s)}\chi_{(0,2t)}(s)\right\|
_{\bar{X}^{^{\prime}}}\left\|  \left|  \nabla h_{n}\right|  \right\|
_{X}\text{ \ \ (by (\ref{aa}))}\\
&  \leq\frac{1}{2t}\left\|  \frac{s}{I_{\Omega}(s)}\chi_{(0,2t)}(z)\right\|
_{\bar{X}^{^{\prime}}}(1+\frac{1}{n})\left\|  \left|  \nabla h\right|
\right\|  _{X}\text{ \ (by [(\ref{cota01}), Chapter \ref{preliminar}]).}%
\end{align*}
On the other hand, from [(\ref{converge}) Chapter \ref{preliminar}] and Lemma
\ref{garlem}, we have $\left|  h_{n}\right|  _{\mu}^{\ast}(s)\rightarrow
\left|  h\right|  _{\mu}^{\ast}(s),$ and $\left|  h_{n}\right|  _{\mu}%
^{\ast\ast}(s)\rightarrow\left|  h\right|  _{\mu}^{\ast\ast}(s).$
Consequently,
\[
\left|  h\right|  _{\mu}^{\ast\ast}(2t)-\left|  h\right|  _{\mu}^{\ast
}(2t)\leq\frac{1}{2t}\left\|  \frac{s}{I_{\Omega}(s)}\chi_{(0,2t)}(s)\right\|
_{\bar{X}^{^{\prime}}}\left\|  \left|  \nabla h\right|  \right\|  _{X}.
\]
Summarizing,
\begin{align*}
\left|  f\right|  _{\mu}^{\ast\ast}(t)-\left|  f\right|  _{\mu}^{\ast}(t)  &
\leq4\frac{\left\|  \left|  f\right|  -h\right\|  _{X}}{\phi_{X}(t)}+\frac
{1}{2t}\left\|  \frac{s}{I_{\Omega}(s)}\chi_{(0,2t)}(s)\right\|  _{\bar
{X}^{^{\prime}}}\left\|  \left|  \nabla h\right|  \right\|  _{X}\\
&  =\frac{4}{\phi_{X}(t)}\left(  \left\|  \left|  f\right|  -h\right\|
_{X}+\frac{\phi_{X}(t)}{2t}\left\|  \frac{s}{I_{\Omega}(s)}\chi_{(0,2t)}%
(s)\right\|  _{\bar{X}^{^{\prime}}}\left\|  \left|  \nabla h\right|  \right\|
\right)  _{X}\\
&  \leq\frac{4}{\phi_{X}(t)}\left(  \left\|  \left|  f\right|  -h\right\|
_{X}+\psi_{X}(2t)\left\|  \left|  \nabla h\right|  \right\|  \right)  _{X}.
\end{align*}
Thus,%
\begin{align*}
\left|  f\right|  _{\mu}^{\ast\ast}(t)-\left|  f\right|  _{\mu}^{\ast}(t)  &
\leq\frac{4}{\phi_{X}(t)}\inf_{0\leq h\leq\left\|  f\right\|  _{\infty},h\in
S_{X}}\left\{  \left\|  \left|  f\right|  -h\right\|  _{X}+\psi(2t)\left\|
\left|  \nabla h\right|  \right\|  \right\} \\
&  \leq\frac{8}{\phi_{X}(t)}\frac{K\left(  \psi_{X}(2t),f;X,S_{X}\right)
}{\phi_{X}(t)}\text{ (by (\ref{porfin1})).}%
\end{align*}
When $f$ is not bounded we consider the sequence $f_{n}=\min(\left|  f\right|
,n)\nearrow\left|  f\right|  $, and we proceed as in the proof of Theorem
\ref{teo:elotro}.
\end{proof}

\section{Isoperimetry\label{sec:iso}}

In this section we show the connection of (\ref{polaka1}) and (\ref{polaka2})
with isoperimetry. Observe that (\ref{polaka1}) holds for all r.i. spaces. In
particular, it holds for $X=L^{1}.$

\begin{theorem}
\label{ref teo 2}Let $G$ be a continuous function on $(0,\mu(\Omega))$ which
is zero at zero and symmetric around $\mu(\Omega)/2.$ Then the following are equivalent

i) Isoperimetric inequality: There exists an absolute constant $c>0,$ such
that
\begin{equation}
G(t)\leq cI_{\Omega}(t),\text{ }0<t\leq\mu(\Omega). \label{vertigo}%
\end{equation}

ii) There exists an absolute constant $c>0$ such that for each $f\in L^{1},$%
\begin{equation}
\left|  f\right|  _{\mu}^{\ast\ast}(t)-\left|  f\right|  _{\mu}^{\ast}(t)\leq
c\frac{K\left(  \frac{t}{G(t)},f;L^{1},S_{L^{1}}\right)  }{t},\text{ }%
0<t\leq\mu(\Omega)/2. \label{polaka4}%
\end{equation}

\end{theorem}

\begin{proof}
(i) $\Rightarrow$(ii) Since
\[
\left\vert f\right\vert _{\mu}^{\ast\ast}(t)-\left\vert f\right\vert _{\mu
}^{\ast}(t)\leq c\frac{K\left(  \frac{t}{I_{\Omega}(t)},f;L^{1},S_{L^{1}%
}\right)  }{t},\text{ }0<t\leq\mu(\Omega)/2.
\]
it follows that (\ref{vertigo}) implies (\ref{polaka4}).

(ii) $\Rightarrow$ (i). Suppose that $A$ is a Borel set with $0<\mu
(A)<\mu(\Omega)/2.$ We may assume, without loss, that $P(A;\Omega)<\infty.$ By
\cite[Lemma 3.7]{bobk} we can select a sequence $\{f_{n}\}_{n\in N}$ of Lip
functions such that $f_{n}\underset{L^{1}}{\rightarrow}\chi_{A}$, and
\[
P(A;\Omega)\geq\lim\sup_{n\rightarrow\infty}\left\Vert \left\vert \nabla
f_{n}\right\vert \right\Vert _{L^{1}}.
\]
Going through a subsequence, if necessary, we can actually assume that for all
$n$ we have
\[
P(A;\Omega)\geq\left\Vert \left\vert \nabla f_{n}\right\vert \right\Vert
_{L^{1}}\text{ }.
\]
From (\ref{polaka4}) we know that there exists a constant $c>0$ such that for
all $0<t\leq\mu(\Omega)/2,$%
\[
\left\vert f_{n}\right\vert _{\mu}^{\ast\ast}(t)-\left\vert f_{n}\right\vert
_{\mu}^{\ast}(t)\leq c\frac{K\left(  \frac{t}{G(t)},f_{n};L^{1},S_{L^{1}%
}\right)  }{t}.
\]
We take limits when $n\rightarrow\infty$ on both sides of this inequality. To
compute on the left hand side we observe that, since, $f_{n}\underset{L^{1}%
}{\rightarrow}\chi_{A}$, Lemma \ref{garlem} implies that:
\begin{align*}
\left\vert f_{n}\right\vert _{\mu}^{\ast\ast}(t)\rightarrow\left\vert \chi
_{A}\right\vert _{\mu}^{\ast\ast}(t)  &  ,\text{ uniformly for }t\in
\lbrack0,1]\text{, and }\\
\left\vert f_{n}\right\vert _{\mu}^{\ast}(t)\rightarrow\left\vert \chi
_{A}\right\vert _{\mu}^{\ast}(t)\text{ }  &  \text{for }t\in\left(
0,1\right)  .
\end{align*}
Fix $1/2>r>\mu(A).$ We have
\begin{align*}
\lim_{n\rightarrow\infty}\left(  \left\vert f_{n}\right\vert _{\mu}^{\ast\ast
}(r)-\left\vert f_{n}\right\vert _{\mu}^{\ast}(r)\right)   &  =\left(
\chi_{A}\right)  _{\mu}^{\ast\ast}(r)-\left(  \chi_{A}\right)  _{\mu}^{\ast
}(r)\\
&  =\left(  \chi_{A}\right)  _{\mu}^{\ast\ast}(r)\text{ (since }\left(
\chi_{A}\right)  _{\mu}^{\ast}(r)=\chi_{(0,\mu(A))}(r)=0\text{)}\\
&  =\frac{\mu(A)}{r}.
\end{align*}
Now, to estimate the right hand side we observe that, for each $n,$ $f_{n}\in
L^{1}\cap S_{L^{1}}.$ Consequently, by the definition of $K$-functional, we
have
\begin{align*}
\frac{K(\frac{t}{G(t)},f_{n};L^{1},S_{L^{1}})}{t}  &  \leq\min\left\{
\frac{\left\Vert f_{n}\right\Vert _{L^{1}}}{t},\frac{1}{G(t)}\left\Vert
\left\vert \nabla f_{n}\right\vert \right\Vert _{L^{1}}\right\} \\
&  \leq\min\left\{  \frac{\left\Vert f_{n}\right\Vert _{L^{1}}}{t},\frac
{1}{G(t)}P(A;\Omega)\right\}  .
\end{align*}
Thus,
\[
\lim_{n\rightarrow\infty}\frac{K(\frac{t}{G(t)},f_{n};L^{1},S_{L^{1}})}{t}%
\leq\min\left\{  \frac{\mu(A)}{t},\frac{1}{G(t)}P(A;\Omega)\right\}  .
\]
Combining these estimates we find that for all $1/2>r>\mu(A),$%
\[
\frac{\mu(A)}{r}\leq c\frac{1}{G(r)}P(A;\Omega).
\]
Let $r\rightarrow\mu(A)$ then, by the continuity of $G,$ we find
\[
1\leq c\frac{1}{G(\mu(A))}P(A;\Omega),
\]
or
\[
G(\mu(A))\leq cP(A;\Omega).
\]
Thus,%
\begin{align*}
G(\mu(A))  &  \leq c\inf\{P(B;\Omega):\mu(B)=\mu(A)\}\\
&  =cI_{\Omega}(\mu(A)).
\end{align*}
Suppose now that $t\in(\mu(\Omega)/2,\mu(\Omega)).$ Then $1-t\in(0,\mu
(\Omega))$ and by symmetry,%
\[
G(t)=G(1-t)\leq cI(1-t)=cI(t),
\]
and we are done.
\end{proof}

\chapter{Embedding into continuous functions\label{contchap}}

\section{Introduction and Summary}

In this chapter we obtain a general version of the Morrey-Sobolev theorem on
metric measure spaces $\left(  \Omega,d,\mu\right)  $ satisfying the usual assumptions.

\subsection{Inequalities for signed rearrangements\label{sigre}}

Let $\left(  \Omega,d,\mu\right)  $ be a metric measure space satisfying the
usual assumptions. We collect a few more facts about signed rearrangements.
First let us note that for $c\in\mathbb{R},$%
\begin{equation}
\left(  f+c\right)  _{\mu}^{\ast}(t)=f_{\mu}^{\ast}(t)+c. \label{rela}%
\end{equation}
Moreover, if $X(\Omega)$ is a r.i. space$,$ we have
\[
\Vert\left|  f\right|  _{\mu}^{\ast}\Vert_{\bar{X}(0,\mu(\Omega))}=\left\|
\left|  f\right|  \right\|  _{X({\Omega})}=\Vert f\Vert_{X({\Omega})}=\Vert
f_{\mu}^{\ast}\Vert_{\bar{X}(0,\mu(\Omega))},
\]
where $\bar{X}(0,\mu(\Omega))$ is the representation space of $X\left(
\Omega\right)  .$

The results of the previous chapter can be easily formulated in terms of
signed rearrangements. In particular, we shall now discuss in detail the
following extension (variant) of Theorem \ref{main2}.

\begin{theorem}
\label{main2sig}Let $X$ be a r.i. space on $\Omega.$ Then, for all $f\in
X+S_{X},$ we have,
\begin{equation}
f_{\mu}^{\ast\ast}(t)-f_{\mu}^{\ast}(t)\leq8\frac{K\left(  \psi(2t),f;X,S_{X}%
\right)  }{\phi_{X}(t)},\text{ }0<t\leq\mu(\Omega)/2, \label{polaka2sig}%
\end{equation}
where $\psi(t):=\psi_{X,\Omega}(t)=\frac{\phi_{X}(t)}{s}\left\|  \frac
{s}{I_{\Omega}(s)}\chi_{(0,t)}(s)\right\|  _{\bar{X}^{^{\prime}}}$ is the
function introduced in [(\ref{polaka3}), Chapter \ref{main}]).
\end{theorem}

\begin{proof}
Let us first further assume that $f$ is bounded below, and let $c=\inf
_{\Omega}f.$ We can then apply Theorem \ref{main2} to the positive function
$f-c$, and we obtain
\begin{equation}
\left(  f-c\right)  _{\mu}^{\ast\ast}(t)-\left(  f-c\right)  _{\mu}^{\ast
}(t)\leq8\frac{K\left(  \psi(2t),f-c;X,S_{X}\right)  }{\phi_{X}(t)}.
\label{caru}%
\end{equation}
We can simplify the left hand side of (\ref{caru}) using (\ref{rela})
\[
\left(  f-c\right)  _{\mu}^{\ast\ast}(t)-\left(  f-c\right)  _{\mu}^{\ast
}(t)=f_{\mu}^{\ast\ast}(t)-f_{\mu}^{\ast}(t).
\]
On the other hand, the sub-additivity of the $K-$functional, and the fact that
it is zero on constant functions, allows us to estimate the right hand side of
(\ref{caru}) as follows
\begin{align*}
K\left(  \psi(t),f-c;X,S_{X}\right)   &  \leq K\left(  \psi(t),f;X,S_{X}%
\right)  +K\left(  \psi(t),c;X,S_{X}\right) \\
&  =K\left(  \psi(t),f;X,S_{X}\right)  .
\end{align*}
Combining these observations we see that,
\[
f_{\mu}^{\ast\ast}(t)-f_{\mu}^{\ast}(t)\leq8\frac{K\left(  \psi(2t),f;X,S_{X}%
\right)  }{\phi_{X}(t)}.
\]
If $f$ is not bounded from below, we use an approximation argument. Let
\[
f_{n}=\max(f,-n),\text{ \ }n=1,2,\ldots
\]
Then by the previous discussion we have
\[
(f_{n})_{\mu}^{\ast\ast}(t)-(f_{n})_{\mu}^{\ast}(t)\leq8\frac{K\left(
\psi(2t),f_{n};X,S_{X}\right)  }{\phi_{X}(t)}.
\]
Now $f_{n}(x)\rightarrow f(x)$ $\ \mu-$a.e., with $\left|  f_{n}\right|
\leq\left|  f\right|  ,$ therefore, $L^{1}$ convergence follows by dominated
convergence, and we can then apply Lemma \ref{garlem} to conclude that%
\[
f_{\mu}^{\ast\ast}(t)-f_{\mu}^{\ast}(t)\leq8\frac{K\left(  \psi(2t),f_{n}%
;X,S_{X}\right)  }{\phi_{X}(t)}.
\]
We estimate the right hand side as follows. Given $\varepsilon>0,$ select
$h^{\varepsilon}\in S_{X}$ such that
\begin{equation}
\left\|  f-h^{\varepsilon}\right\|  _{X}+\psi(t)\left\|  \left|  \nabla
h^{\varepsilon}\right|  \right\|  _{X}\leq K\left(  \psi(t),f;X,S_{X}\right)
+\varepsilon. \label{ss0}%
\end{equation}
For each $n\in\mathbb{N}$ let us define $h_{n}^{\varepsilon}=\max
(h^{\varepsilon},-n).$ Then
\begin{equation}
h_{n}^{\varepsilon}\in S_{X}\text{ \ with \ }\left|  \nabla h_{n}%
^{\varepsilon}\right|  \leq\left|  \nabla h^{\varepsilon}\right|  .
\label{ss1}%
\end{equation}
By a straightforward analysis of all possible cases we see that
\begin{equation}
\left\|  f_{n}-h_{n}^{\varepsilon}\right\|  _{X}\leq\left\|  f-h^{\varepsilon
}\right\|  _{X}. \label{ss2}%
\end{equation}
Therefore, combining (\ref{ss0}), (\ref{ss1}) and (\ref{ss2}), we obtain
\[
K\left(  \psi(t),f_{n};X,S_{X}\right)  \leq K\left(  \psi(t),f;X,S_{X}\right)
,
\]
thus
\[
f_{\mu}^{\ast\ast}(t)-f_{\mu}^{\ast}(t)\leq8\frac{K\left(  \psi(2t),f;X,S_{X}%
\right)  }{\phi_{X}(t)},\text{ }0<t\leq\mu(\Omega)/2.
\]
and (\ref{polaka2sig}) follows.
\end{proof}

\section{Continuity via rearrangement inequalities}

In this section we consider the following problems: Characterize, in terms of
$K-$functional conditions, the functions in $f\in X(\Omega)+S_{X}(\Omega)$
that are bounded, or essentially continuous. One can rephrase these questions
as suitable embedding theorems for Besov type spaces.

We consider boundedness first.

\begin{lemma}
\label{bound}Let $\left(  \Omega,d,\mu\right)  $ be a metric measure space$,$
and let $X(\Omega)$ be a r.i. space$.$ Then if $f\in X+S_{X}$ is such that
\[
\int_{0}^{\mu(\Omega)}\frac{K\left(  \Psi(t),f;X,S_{X}\right)  }{\phi
_{X}(t/2)}\frac{dt}{t}<\infty,
\]
then $f$ is essentially bounded, where $\Psi(t):=\Psi_{X,\Omega}(t)=\phi
_{X}(t)\left\Vert \frac{1}{I_{\Omega}(s)}\chi_{(0,t)}(s)\right\Vert _{\bar
{X}^{^{\prime}}}$ is the function introduced in [(\ref{polaka3sig}), Chapter
\ref{main}].
\end{lemma}

\begin{proof}
To simplify the notation we shall let $K(t,f):=K\left(  t,f;X,S_{X}\right)  .$
By Lemma \ref{ellemadelapsi} (i) and Theorem \ref{main2}, we have
\begin{equation}
\left|  f\right|  _{\mu}^{\ast\ast}(t/2)-\left|  f\right|  _{\mu}^{\ast
}(t/2)\leq8\frac{K\left(  \Psi(t),f\right)  }{\phi_{X}(t/2)},\text{ }0\text{
}<t\leq\mu(\Omega). \label{integra}%
\end{equation}
Fix $0<r<\frac{\mu(\Omega)}{2};$ then integrating both sides of (\ref{integra}%
) from $r$ to $\mu(\Omega),$ we find%
\[
\int_{r}^{\mu(\Omega)}\left(  \left|  f\right|  _{\mu}^{\ast\ast}(t/2)-\left|
f\right|  _{\mu}^{\ast}(t/2)\right)  \frac{dt}{t}\leq8\int_{r}^{\mu(\Omega
)}\frac{K\left(  \Psi(t),f\right)  }{\phi_{X}(t/2)}\frac{dt}{t}.
\]
We can compute the left hand side using the fundamental theorem of calculus%
\[
\int_{r}^{\frac{\mu(\Omega)}{2}}\left(  \left|  f\right|  _{\mu}^{\ast\ast
}(u)-\left|  f\right|  _{\mu}^{\ast}(u)\right)  \frac{du}{u}=f^{\ast\ast
}(r)-f^{\ast\ast}\left(  \frac{\mu(\Omega)}{2}\right)  ,
\]
thus,%
\[
f^{\ast\ast}(r)-f^{\ast\ast}\left(  \frac{\mu(\Omega)}{2}\right)  \leq
8\int_{r}^{\mu(\Omega)}\frac{K\left(  \Psi(t),f\right)  }{\phi_{X}(t/2)}%
\frac{dt}{t}.
\]

Therefore,%
\[
f^{\ast\ast}(r)\leq\frac{2}{\mu(\Omega)}\int_{0}^{\frac{\mu(\Omega)}{2}%
}\left|  f\right|  _{\mu}^{\ast}(s)ds+8\int_{r}^{\mu(\Omega)}\frac{K\left(
\Psi(t),f;X,S_{X}\right)  }{\phi_{X}(t/2)}\frac{dt}{t}.
\]
Thus, letting $r\rightarrow0,$%
\[
\left\|  f\right\|  _{L^{\infty}}\leq\frac{2}{\mu(\Omega)}\int_{0}^{\frac
{\mu(\Omega)}{2}}\left|  f\right|  _{\mu}^{\ast}(s)ds+8\int_{0}^{\mu(\Omega
)}\frac{K\left(  \Psi(t),f;X,S_{X}\right)  }{\phi_{X}(t/2)}\frac{dt}{t},
\]
and the result follows.
\end{proof}

To study essential continuity it will be useful to introduce some notation.
Let $G$ be an open subset of $\Omega.$ Recall (cf. Section \ref{secc:ri},
Chapter \ref{preliminar}) that $X_{r}(G)=X(G,d_{\mid G},\mu_{\mid G}).$ When
the open set $G$ is understood from the context, we shall simply
write$\ X_{r}$ and $S_{X_{r}}.$ We shall denote by $\bar{X}_{r}=\bar{X}%
_{r}(0,\mu(G))$ the representation space of $X_{r},$ and we let $X_{r}%
^{\prime}$ denote the corresponding associated space of $X_{r}$ (For more
information see Chapter \ref{preliminar} Section \ref{secc:ri} -
(\ref{dualres})).

If $f\in X(\Omega)+S_{X}(\Omega),$ then we obviously have that $f\chi_{G}\in
X_{r}(G)+S_{X_{r}}(G).$ However, we can not apply our fundamental inequalities
[(\ref{polaka1}), Chapter \ref{main}] or (\ref{polaka3sig})] since we are now
working in the metric space $(G,d_{\mid G},\mu_{\mid G})$ and therefore the
isoperimetric profile has changed.

Given $G\subset\Omega$ an open subset, and let $A\subset G.$ The
\textbf{perimeter} of $A$ \textbf{relative }to $G$ is defined by
\[
P(A;G)=\lim\inf_{h\rightarrow0}\frac{\mu\left(  A_{h}\right)  -\mu\left(
A\right)  }{h},
\]
where $A_{h}=\left\{  x\in G:d(x,A)<h\right\}  .$ Obviously
\[
P(A;G)\leq P(A;\Omega).
\]

The \textbf{relative isoperimetric profile }of\textbf{ }$G\subset\Omega$ is
defined by (see for example \cite{Amb} and the references quoted therein)
\[
I_{G}(s)=I_{(G,d,\mu)}(s)=\inf\left\{  P(A;G):\text{ }A\subset G,\text{ }%
\mu(A)=s\right\}  ,\text{ \ \ }0<s<\mu(G).
\]
We say that an \textbf{isoperimetric inequality relative }to $G$ holds true if
there exists a positive constant $C_{G}$ such that
\[
I_{G}(s)\geq C_{G}\min(I_{\Omega}(s),I_{\Omega}(\mu(G)-s))=J_{G}(t),\text{
\ \ }0<s<\mu(G),
\]
where $I_{\Omega}$ is the isoperimetric profile of $(\Omega,d,\mu)$. Notice
that, if $\mu(G)\leq\mu(\Omega)/2,$ then $J_{G}:[0,\mu(G)]\rightarrow\left[
0,\infty\right)  $ is increasing on $(0,\mu(G)/2),$ symmetric around the point
$\mu(G)/2,$ and\ such that
\[
I_{G}\geq J_{G},
\]
i.e. $J_{G}$ is an isoperimetric estimator for the metric space $(G,d_{\mid
G},\mu_{\mid G}).$

\begin{definition}
We will say that a metric measure space $\left(  \Omega,d,\mu\right)  $ has
the \textbf{relative uniform isoperimetric property }if\ there is a constant
$C$ such that for any ball $B\ $in $\Omega,$ its \textbf{relative
isoperimetric profile }$I_{B}$ satisfies:
\[
I_{B}(s)\geq C\min(I_{\Omega}(s),I_{\Omega}(\mu(B)-s)),\text{ \ \ }%
0<s<\mu(B).
\]

\end{definition}

The following proposition will be useful in what follows

\begin{proposition}
\label{esti}Let $J$ be an isoperimetric estimator of $(\Omega,d,\mu).$ Let
$G\subset\Omega$ be an open set with $\mu(G)\leq\mu(\Omega)/2$, and let
$Z=Z([0,\mu(G)])$ be a r.i. space on $[0,\mu(G)].$ Then
\[
R(t):=\left\Vert \frac{1}{\min(J(s),J(\mu(G)-s))}\chi_{(0,t)}(s)\right\Vert
_{Z}\leq2\left\Vert \frac{1}{J(s)}\chi_{(0,t)}(s)\right\Vert _{Z},\text{
\ \ \ }0<t\leq\mu(G).
\]

\end{proposition}

\begin{proof}
Since $\min(I_{\Omega}(s),I_{\Omega}(\mu(G)-s))$ is an isoperimetric estimator
for $(G,d_{\mid G},\mu_{\mid G}),$ we have
\begin{align*}
R(t)  &  \leq\left\Vert \frac{1}{\min(I_{\Omega}(s),I_{\Omega}(\mu(G)-s))}%
\chi_{(0,\mu(G)/2)}(s)\right\Vert _{Z}+\left\Vert \frac{1}{\min(I_{\Omega
}(s),I_{\Omega}(\mu(G)-s))}\chi_{(\mu(G)/2,t)}(s)\right\Vert _{Z}\\
&  =\left\Vert \frac{1}{I_{\Omega}(s)}\chi_{(0,\mu(G)/2)}(s)\right\Vert
_{Z}+\left\Vert \frac{1}{I_{\Omega}(\mu(G)-s)}\chi_{(\mu(G)/2,t)}%
(s)\right\Vert _{Z}\\
&  \leq\left\Vert \frac{1}{I_{\Omega}(s)}\chi_{(0,t)}(s)\right\Vert
_{Z}+\left\Vert \frac{1}{I_{\Omega}(\mu(G)-s)}\chi_{(\mu(G)/2,t)}%
(s)\right\Vert _{Z}.
\end{align*}
To estimate the second term on the right hand side, notice that, by the
properties of $I_{\Omega}(s),$ the functions
\[
\frac{1}{I_{\Omega}(\mu(G)-s)}\chi_{(\mu(G)/2,t)}(s)\text{ and }\frac
{1}{I_{\Omega}(s)}\chi_{(\mu(G)-t,\mu(G)/2)}(s)\text{ }%
\]
are equimeasurable (with respect to the Lebesgue measure), consequently
\begin{align*}
\left\Vert \frac{1}{I_{\Omega}(\mu(G)-s)}\chi_{(\mu(G)/2,t)}(s)\right\Vert
_{Z}  &  =\left\Vert \frac{1}{I_{\Omega}(s)}\chi_{(\mu(G)-t,\mu(G)/2)}%
(s)\right\Vert _{Z}\\
&  =\left\Vert \left(  \frac{1}{I_{\Omega}(s)}\chi_{(\mu(G)-t,\mu
(G)/2)}\right)  ^{\ast}(s)\right\Vert _{Z}\text{.}%
\end{align*}
Since $1/I_{\Omega}(s)$ is decreasing on $(0,\mu(G)/2),$
\[
\left\Vert \left(  \frac{1}{I_{\Omega}(s)}\chi_{(\mu(G)-t,\mu(G)/2)}\right)
^{\ast}(s)\right\Vert _{Z}\leq\left\Vert \frac{1}{I_{\Omega}(s)}\chi
_{(0,t-\mu(G)/2)}(s)\right\Vert _{Z}\leq\left\Vert \frac{1}{I_{\Omega}(s)}%
\chi_{(0,t)}(s)\right\Vert _{Z}.
\]
Consequently,
\[
R(t)\leq2\left\Vert \frac{1}{I_{\Omega}(s)}\chi_{(0,t)}(s)\right\Vert _{Z}.
\]

\end{proof}

\begin{theorem}
\label{continuo}Let $\left(  \Omega,d,\mu\right)  $ be a metric measure space
with the relative uniform isoperimetric property\textbf{ }and let $X$ be a
r.i. space on $\Omega.$ Then all the functions in $f\in X+S_{X}$ that satisfy
the condition
\begin{equation}
\int_{0}^{\mu(\Omega)}\frac{K\left(  \Psi_{X,\Omega}(t),f;X,S_{X}\right)
}{\phi_{X}(t)}\frac{dt}{t}<\infty, \label{conticondi}%
\end{equation}
are essentially continuos.
\end{theorem}

\begin{proof}
We will show that there exists a universal constant $c>0$ such that for any
function that satisfies (\ref{conticondi}) we have: For all balls $B$
$\mu(B)<\mu(\Omega)/2,$
\[
\left|  f(x)-f(y)\right|  \leq c\int_{0}^{2\mu(B)}\frac{K\left(
\Psi_{X,\Omega}(t),f;X,S_{X}\right)  }{\phi_{X}(t)}\frac{dt}{t},
\]
for $\mu-$almost every $x,y\in B.$

We shall use the following notation: $X=X(\Omega),$ $S_{X}=$ $S_{X}(\Omega),$
$X_{r}=X_{r}(B),$ $S_{X_{r}}=S_{X_{r}}(B)$, $\psi_{B}(t)=\psi_{B,X_{r}}(t),$
and $\psi_{B}=\psi_{X_{r},B},\Psi_{B}(t)=\Psi_{X_{r},B},\Psi_{X}%
=\Psi_{X,\Omega}$

Given $f\in X+S_{X},$ then $f\chi_{B}\in X_{r}+S_{X_{r}}.$ By Theorem
\ref{main2sig}
\begin{align}
\left(  f\chi_{B}\right)  _{\mu}^{\ast\ast}(t)-\left(  f\chi_{B}\right)
_{\mu}^{\ast}(t)  &  =\left(  f\chi_{B}\right)  _{\mu_{\mid B}}^{\ast\ast
}(t)-\left(  f\chi_{B}\right)  _{\mu_{\mid B}}^{\ast}(t)\label{qqqq}\\
&  \leq8\frac{K\left(  \psi_{B}(2t),f\chi_{B};X_{r},S_{X_{r}}\right)  }%
{\phi_{X_{r}}(t)},\text{ }0\leq t\leq\mu(B)/2.\nonumber
\end{align}
By [(\ref{galio}), Chapter \ref{preliminar}],%
\[
\frac{K\left(  \psi_{B}(2t),f\chi_{B};X_{r},S_{X_{r}}\right)  }{\phi_{X_{r}%
}(t)}\leq\frac{K\left(  \psi_{B}(2t),f;X,S_{X}\right)  }{\phi_{X}(t)},\text{
}0\leq t\leq\mu(B)/2.
\]
On the other hand, for $0\leq t\leq\mu(B)/2,$
\begin{align}
\psi_{B}(2t)  &  =\frac{\phi_{X_{r}}(2t)}{2t}\left\|  \frac{s}{I_{B}(s)}%
\chi_{(0,2t)}(s)\right\|  _{\bar{X}_{r}^{\prime}}\label{mas}\\
&  \leq\phi_{X_{r}}(2t)\left\|  \frac{1}{I_{B}(s)}\chi_{(02,t)}(s)\right\|
_{\bar{X}_{r}^{\prime}}\nonumber\\
&  =\phi_{X}(2t)\left\|  \frac{1}{I_{B}(s)}\chi_{(0,2t)}(s)\right\|  _{\bar
{X}^{\prime}}\nonumber\\
&  \leq c\phi_{X}(2t)\left\|  \frac{1}{\min(I_{\Omega}(B),I_{\Omega}%
(\mu(B)-s))}\chi_{(0,2t)}(s)\right\|  _{\bar{X}^{\prime}}\nonumber\\
&  \leq c\phi_{X}(2t)\left\|  \frac{1}{I_{\Omega}(s)}\chi_{(0,2t)}(s)\right\|
_{\bar{X}^{\prime}}\text{ \ (by Proposition \ref{esti})}\nonumber\\
&  =\text{ }c\Psi_{X}(2t).\nonumber
\end{align}

Consequently for\ $0<t\leq\mu(B)/2,$ we have%
\begin{align}
\left(  f\chi_{B}\right)  _{\mu}^{\ast\ast}(t)-\left(  f\chi_{B}\right)
_{\mu}^{\ast}(t)  &  \leq8\frac{K\left(  c\Psi_{X}(2t),f;X,S_{X}\right)
}{\phi_{X}(t)}\nonumber\\
&  \leq8C\frac{K\left(  \Psi_{X}(2t),f;X,S_{X}\right)  }{\phi_{X}%
(t)}\nonumber\\
&  =A(2t). \label{obam}%
\end{align}
Change variables: Let $t=zr,$ with $r=\frac{\mu(B)}{2\mu(\Omega)}$,
$0<z\leq\mu(\Omega).$ Then,
\[
\left(  f\chi_{B}\right)  _{\mu}^{\ast\ast}\left(  zr\right)  -\left(
f\chi_{B}\right)  _{\mu}^{\ast}\left(  zr\right)  \leq A(2zr),\text{
\ \ \ }0<z\leq\mu(\Omega).
\]
Integrating the previous inequality we obtain%
\begin{align*}
\int_{0}^{\mu(\Omega)}\left[  \left(  f\chi_{B}\right)  _{\mu}^{\ast\ast
}\left(  zr\right)  -\left(  f\chi_{B}\right)  _{\mu}^{\ast}\left(  zr\right)
\right]  \frac{dz}{z}  &  \leq\int_{0}^{\mu(\Omega)}A(2zr)\frac{dz}{z}\\
&  =\int_{0}^{\mu(B)}A(u)\frac{du}{u}.
\end{align*}
Using the formula
\[
\frac{d}{dz}\left\{  -\left(  f\chi_{B}\right)  _{\mu}^{\ast\ast}(zr)\right\}
=\frac{\left(  f\chi_{B}\right)  _{\mu}^{\ast\ast}(zr)-\left(  f\chi
_{B}\right)  _{\mu}^{\ast}(zr)}{z},
\]
we get
\begin{align*}
ess\sup(f\chi_{B})-\left(  f\chi_{B}\right)  _{\mu}^{\ast\ast}(\mu(B)/2)  &
=\left(  f\chi_{B}\right)  _{\mu}^{\ast\ast}(0)-\left(  f\chi_{B}\right)
_{\mu}^{\ast\ast}(\mu(B)/2)\\
&  =\int_{0}^{\mu(\Omega)}\left[  \left(  f\chi_{B}\right)  _{\mu}^{\ast\ast
}(zr)-\left(  f\chi_{B}\right)  _{\mu}^{\ast}(zr)\right]  \frac{dz}{z}\\
&  \leq\int_{0}^{\mu(B)}A(z)\frac{dz}{z}.
\end{align*}

Similarly, considering $-f\chi_{B},$ instead of $f\chi_{B},$ we obtain
\[
ess\sup(-f\chi_{B})-\left(  -f\chi_{B}\right)  _{\mu}^{\ast\ast}(\mu
(B)/2)\leq\int_{0}^{\mu(B)}A(z)\frac{dz}{z}.
\]
Adding both inequalities%
\begin{equation}
ess\sup(f\chi_{B})-ess\inf(f\chi_{B})\leq2\int_{0}^{\mu(B)}A(z)\frac{dz}%
{z}+\left[  \left(  f\chi_{B}\right)  _{\mu}^{\ast\ast}(\mu(B)/2)+\left(
-f\chi_{B}\right)  _{\mu}^{\ast\ast}(\mu(B)/2)\right]  \label{insertada}%
\end{equation}
To estimate the last term on the right hand side we let $t=\mu(B)/2$ in
(\ref{qqqq}),%
\[
\left[  \left(  f\chi_{B}\right)  _{\mu}^{\ast\ast}(\mu(B)/2)+\left(
-f\chi_{B}\right)  _{\mu}^{\ast\ast}(\mu(B)/2)\right]  \leq8\frac{K\left(
\psi_{B}(\mu(B)),f\chi_{B};X_{r},S_{X_{r}}\right)  }{\phi_{X_{r}}(\mu(B)/2)},
\]
and we do the same thing for the corresponding estimate for $-f,$
\[
\left[  \left(  -f\chi_{B}\right)  _{\mu}^{\ast\ast}(\mu(B)/2)+\left(
-f\chi_{B}\right)  _{\mu}^{\ast\ast}(\mu(B)/2)\right]  \leq8\frac{K\left(
\psi_{B}(\mu(B)),f\chi_{B};X_{r},S_{X_{r}}\right)  }{\phi_{X_{r}}(\mu(B)/2)}.
\]
Adding these inequalities, and recalling that, since $(-f\chi_{B})_{\mu}%
^{\ast}(t)=(-f\chi_{B})_{\mu_{\mid B}}^{\ast}(t)=-(f\chi_{B})_{\mu_{\mid B}%
}^{\ast}(\mu(B)-t)=-(f\chi_{B})_{\mu}^{\ast}(\mu(B)-t),$ we have
\[
(-f\chi_{B})_{\mu}^{\ast}(\mu(B)/2)=-(f\chi_{B})_{\mu}^{\ast}(\mu(B)/2),
\]
we see that
\[
\left(  f\chi_{B}\right)  _{\mu}^{\ast\ast}(\mu(B)/2)+\left(  -f\chi
_{B}\right)  _{\mu}^{\ast\ast}(\mu(B)/2)\leq16\frac{K\left(  \psi_{B}%
(\mu(B)),f\chi_{B};X_{r},S_{X_{r}}\right)  }{\phi_{X_{r}}(\mu(B)/2)}.
\]
Inserting this information back in (\ref{insertada}) we obtain%
\begin{align*}
ess\sup(f\chi_{B})-ess\inf(f\chi_{B})  &  \leq2\int_{0}^{\mu(B)}A(z)\frac
{dz}{z}+16\frac{K\left(  \psi_{B}(\mu(B)),f\chi_{B};X_{r},S_{X_{r}}\right)
}{\phi_{X_{r}}(\mu(B)/2)}\\
&  \leq C\left(  \int_{0}^{\mu(B)}\frac{K\left(  \Psi_{X}(t),f;X,S_{X}\right)
}{\phi_{X}(t/2)}\frac{dz}{z}+16\frac{K\left(  \Psi_{X}(\mu(B)),f;X,S_{X}%
\right)  }{\phi_{X}(\mu(B)/2)}\right) \\
&  \leq C\left(  \int_{0}^{\mu(B)}\frac{K\left(  \Psi_{X}(t),f;X,S_{X}\right)
}{\phi_{X}(t)}\frac{dz}{z}+\frac{K\left(  \Psi_{X}(\mu(B)),f;X,S_{X}\right)
}{\phi_{X}(\mu(B))}\right)  .
\end{align*}
Elementary considerations show that the second term on the right hand side can
be controlled by the first term. Indeed, we use that $K(t,f)$ increases and
$\Psi_{X}(t)$ increases (cf. Lemma \ref{ellemadelapsi} (iv) above) to derive
that $K\left(  \Psi_{X}(t),f;X,S_{X}\right)  $ increases, which we combine
with the fact that $\phi_{X^{\prime}}(t)$ increases, and $\phi_{X^{\prime}%
}(t)\phi_{X}(t)=t,$ then we obtain%
\begin{align*}
\int_{0}^{2\mu(B)}\frac{K\left(  \Psi_{X}(t),f;X,S_{X}\right)  }{\phi_{X}%
(t)}\frac{dt}{t}  &  =\int_{0}^{2\mu(B)}\frac{\phi_{X^{\prime}}(t)K\left(
\Psi_{X}(t),f;X,S_{X}\right)  }{t}\frac{dt}{t}\\
&  \geq\int_{\mu(B)}^{2\mu(B)}\frac{\phi_{X^{\prime}}(t)K\left(  \Psi
_{X}(t),f;X,S_{X}\right)  }{t}\frac{dt}{t}\\
&  \geq\phi_{X^{\prime}}(\mu(B))K\left(  \Psi_{X}(\mu(B)),f;X,S_{X}\right)
\int_{\mu(B)}^{2\mu(B)}\frac{1}{t}\frac{dt}{t}\\
&  =K\left(  \Psi_{X}(\mu(B)),f;X,S_{X}\right)  \frac{\phi_{X^{\prime}}%
(\mu(B))}{2\mu(B)}\\
&  =\frac{1}{2}\frac{K\left(  \Psi_{X}(\mu(B)),f;X,S_{X}\right)  }{\phi
_{X}(\mu(B))}%
\end{align*}
Thus,%
\[
ess\sup(f\chi_{B})-ess\inf(f\chi_{B})\leq c\int_{0}^{2\mu(B)}\frac{K\left(
\Psi_{X}(t),f;X,S_{X}\right)  }{\phi_{X}(t)}\frac{dt}{t}.
\]
It follows that for $\mu-$almost every $x,y\in B,$
\[
\left|  f(x)-f(y)\right|  \leq c\int_{0}^{2\mu(B)}\frac{K\left(  \Psi
_{X}(t),f;X,S_{X}\right)  }{\phi_{X}(t)}\frac{dt}{t},
\]
and the essential continuity of $f$ follows.
\end{proof}

\chapter{Examples and Applications\label{capitap}}

\section{Summary}

We verify the relative uniform isoperimetric property for a number of concrete
examples. As a consequence we shall show in detail how our methods provide a
unified treatment of embeddings of Sobolev and Besov spaces into spaces of
continuous functions in different contexts.

\section{Euclidean domains\label{class}}

Let $\Omega\subset\mathbb{R}^{n}$ be a bounded domain (i.e. a bounded, open
and connected set). For a measurable function $u:\Omega\rightarrow\mathbb{R},$
let
\[
u^{+}=\max(u,0)\text{ and }u^{-}=\min(u,0).
\]
Let $X=X(\Omega)$ be a r.i. space on $\Omega.$ The Sobolev space $W_{X}%
^{1}(\Omega):=W_{X}^{1}$ is the space of real-valued weakly differentiable
functions on $\Omega$ that, together with their first order derivatives,
belong to $X.$

In this setting the basic rearrangement inequality holds for all $f\in
W_{L^{1}}^{1}$%
\begin{equation}
\left|  f\right|  ^{\ast\ast}(t)-\left|  f\right|  ^{\ast}(t)\leq\frac
{t}{I_{\Omega}(t)}\frac{1}{t}\int_{0}^{t}\left|  \nabla f\right|  ^{\ast
}(s)ds,\text{ }0<t<\left|  \Omega\right|  , \label{sobu}%
\end{equation}
(rearrangements are taken with respect to the Lebesgue measure). We indicate
briefly the proof using the method of \cite{mamiadv}.

It is well known (see for example \cite{bakr}, \cite[Theorem 2.1.4]{Zie}) that
if $u\in W_{L^{1}}^{1}$ ($=W_{1}^{1})$ then $u^{+},$ $u^{-}\in W_{1}^{1}$ and
\[
\nabla u^{+}=\nabla u\chi_{\{u>0\}}\text{ and }\nabla u^{-}=\nabla
u\chi_{\{u<0\}}.
\]
For given a measurable function $g$ and $0<t_{1}<t_{2},$ the truncation
$g_{t_{1}}^{t_{2}}$ of\ $g$ is defined by
\[
g_{t_{1}}^{t_{2}}=\min\{\max\{0,g-t_{1}\},t_{2}-t_{2}\}\}.
\]
It follows that if $g\in W_{1}^{1}$, then $g_{t_{1}}^{t_{2}}\in W_{1}^{1}$
and, in fact,
\[
\nabla g_{t_{1}}^{t_{2}}=\nabla g\chi_{\{t_{1}<g<t_{2}\}}.
\]
In other words, $W_{1}^{1}$ is invariant by truncation. On the other hand,
given $g\in W_{1}^{1},$ the Federer-Fleming-Rishel co-area formula (cf.
\cite{flem}) states that
\[
\int_{\Omega}\left|  \nabla g(x)\right|  dx=\int_{-\infty}^{\infty}P_{\Omega
}(g>s)ds.
\]
Applying this result to $\left|  g\right|  _{t_{1}}^{t_{2}},$ we get
\begin{align}
\int_{\{t_{1}<\left|  g\right|  <t_{2}\}}\left|  \nabla\left|  g\right|
(x)\right|  dx  &  =\int_{0}^{\infty}P_{\Omega}(\left|  g\right|  _{t_{1}%
}^{t_{2}}>s)ds\label{A1}\\
&  \geq\int_{0}^{\infty}I_{\Omega}(\mu_{\left|  g\right|  _{t_{1}}^{t_{2}}%
}(s))ds\text{ (isoperimetric inequality)}\nonumber\\
&  =\int_{0}^{t_{2}-t_{1}}I_{\Omega}(\mu_{\left|  g\right|  _{t_{1}}^{t_{2}}%
}(s))ds.\nonumber
\end{align}

Observe that, for $0<s<t_{2}-t_{1}$,
\[
\left\vert \left\{  \left\vert f\right\vert \geq t_{2}\right\}  \right\vert
\leq\mu_{\left\vert f\right\vert _{t_{1}}^{t_{2}}}(s)\leq\left\vert \left\{
\left\vert f\right\vert >t_{1}\right\}  \right\vert .
\]
Consequently, by the properties of $I_{\Omega}$, we have
\[
\int_{0}^{t_{2}-t_{1}}I(\mu_{\left\vert g\right\vert _{t_{1}}^{t_{2}}%
}(s))ds\geq(t_{2}-t_{1})\min\left(  I_{\Omega}(\left\vert \left\{  \left\vert
g\right\vert \geq t_{1}\right\}  \right\vert ),I_{\Omega}(\left\vert \left\{
\left\vert g\right\vert \geq t_{2}\right\}  \right\vert \right)  .
\]
For $s>0$ and $h>0,$ pick $t_{1}=\left\vert g\right\vert ^{\ast}(s+h),$
$t_{2}=\left\vert g\right\vert ^{\ast}(s),$ then
\begin{equation}
s\leq\left\vert \left\{  \left\vert g\right\vert \geq\left\vert g\right\vert
^{\ast}(s)\right\}  \right\vert \leq\mu_{\left\vert g\right\vert _{t_{1}%
}^{t_{2}}}(s)\leq\left\vert \left\{  \left\vert g(x)\right\vert >\left\vert
g\right\vert ^{\ast}(s+h)\right\}  \right\vert \leq s+h. \label{A2}%
\end{equation}
Combining (\ref{A1}) and (\ref{A2}) we have,
\[
\left(  \left\vert g\right\vert ^{\ast}(s)-\left\vert g\right\vert ^{\ast
}(s+h)\right)  \min(I_{\Omega}(s+h),I_{\Omega}(s))\leq\int_{\left\{
\left\vert g\right\vert ^{\ast}(s+h)<\left\vert g\right\vert <\left\vert
g\right\vert ^{\ast}(s)\right\}  }\left\vert \nabla\left\vert g\right\vert
(x)\right\vert dx.
\]
At this stage we can continue as in \cite{mamiadv}, and we obtain that if
$f\in W_{1}^{1},$ then (\ref{sobu}) holds. Moreover, $\left\vert f\right\vert
^{\ast}$ is locally absolutely continuous, and
\begin{equation}
\int_{0}^{t}\left\vert (-\left\vert f\right\vert ^{\ast})^{\prime}%
(\cdot)I_{\Omega}(\cdot)\right\vert ^{\ast}(s)\leq\int_{0}^{t}\left\vert
\nabla f\right\vert ^{\ast}(s)ds. \label{sobre}%
\end{equation}
From here, using the same approximation method provided in the proof of
Theorem \ref{main2sig} Chapter \ref{contchap}, we find that, if $f\in
W_{1}^{1},$ then for $0<t<\left\vert \Omega\right\vert ,$
\begin{equation}
f^{\ast\ast}(t)-f^{\ast}(t)\leq\frac{t}{I_{\Omega}(t)}\frac{1}{t}\int_{0}%
^{t}\left\vert \nabla f\right\vert ^{\ast}(s)ds. \label{sobs}%
\end{equation}
Indeed, first assume that $f$ is bounded from below, and let $c=\inf_{\Omega
}f,$ then, since $f-c\geq0,$ we can apply (\ref{sobu}) to $f-c,$ and we
obtain
\[
\left\vert f-c\right\vert ^{\ast\ast}(t)-\left\vert f-c\right\vert ^{\ast
}(t)\leq\frac{t}{I_{\Omega}(t)}\frac{1}{t}\int_{0}^{t}\left\vert \nabla\left(
f-c\right)  \right\vert ^{\ast}(s)ds.
\]
Since $\left\vert f-c\right\vert ^{\ast}(t)=\left(  f-c\right)  ^{\ast}(t)$
and $f^{\ast\ast}(t)-f^{\ast}(t)=\left(  f-c\right)  ^{\ast\ast}(t)-\left(
f-c\right)  ^{\ast}(t),$ we get
\[
f^{\ast\ast}(t)-f^{\ast}(t)\leq\frac{t}{I_{\Omega}(t)}\frac{1}{t}\int_{0}%
^{t}\left\vert \nabla f\right\vert ^{\ast}(s)ds.
\]
If $f$ is not bounded from below, let $f_{n}=\max(f,-n),$ \ $n=1,2,\ldots$ The
previous discussion gives
\begin{align*}
(f_{n})^{\ast\ast}(t)-(f_{n})^{\ast}(t)  &  \leq\frac{t}{I_{\Omega}(t)}%
\frac{1}{t}\int_{0}^{t}\left\vert \nabla f_{n}\right\vert ^{\ast}(s)ds\\
&  \leq\frac{t}{I_{\Omega}(t)}\frac{1}{t}\int_{0}^{t}\left\vert \nabla
f\right\vert ^{\ast}(s)ds.
\end{align*}
We now take limits. To compute the left hand side we observe that
$f_{n}(x)\rightarrow f(x)$ $\ \mu-$a.e., and $\left\vert f_{n}\right\vert
\leq\left\vert f\right\vert ,$ then by dominated convergence $f_{n}%
\rightarrow_{L^{1}}f.$ Consequently, by Lemma \ref{garlem}, we have the
pointwise convergence $(f_{n})^{\ast\ast}(t)-(f_{n})^{\ast}%
(t)\underset{n\rightarrow\infty}{\rightarrow}f^{\ast\ast}(t)-f^{\ast}(t),$
($0<t<\mu(\Omega)),$ concluding the proof.

Let $X=X(\Omega)$ be a r.i. space on $\Omega.$ The homogeneous Sobolev space
$\dot{W}_{X}^{1}$ is defined by means of the seminorm
\[
\left\Vert u\right\Vert _{\dot{W}_{X}^{1}}:=\left\Vert \left\vert \nabla
u\right\vert \right\Vert _{X}.
\]
We consider the corresponding $K-$functional
\[
K(t,f;X,\dot{W}_{X}^{1})=\inf\{\left\Vert f-g\right\Vert _{X}+t\left\Vert
g\right\Vert _{\dot{W}_{X}^{1}}\}.
\]
The previous discussion shows that all the results of Chapters \ref{main} and
\ref{contchap} remain valid for functions in $\dot{W}_{X}^{1}$ or $X+\dot
{W}_{X}^{1}$ .

\subsection{\label{eudom}Sobolev spaces defined on Lipschitz domains of
$\mathbb{R}^{n}$}

We now discuss assumptions on the domain that translate into good estimates
for the corresponding isoperimetric profiles.

In this section we consider Sobolev spaces defined on Lipschitz domains of
$\mathbb{R}^{n}.$

Let $\Omega\subset\mathbb{R}^{n}$ be a bounded Lipschitz domain$.$ Then the
isoperimetric profile, satisfies (see for example \cite{ma})
\begin{equation}
I_{\Omega}(t)=c(n)\left(  \min(t,\left\vert \Omega\right\vert -t)\right)
^{\frac{n-1}{n}}. \label{perlip}%
\end{equation}
For any open ball that $B_{\alpha}\subset\Omega$ with $\left\vert B_{\alpha
}\right\vert =\alpha,$ we know that (see for example \cite{ma} or \cite{Zie})
\[
I_{B_{\alpha}}(t)\geq q(n)\left(  \min(t,\alpha-t)\right)  ^{\frac{n-1}{n}%
},\text{ \ \ \ }0<t<\alpha,
\]
where $q(n)\ $is a constant that only depends on $n$. Moreover, since
\[
c(n)\left(  \min(t,\alpha-t)\right)  ^{\frac{n-1}{n}}=\min(I_{\Omega
}(t),I_{\Omega}(\alpha-t))\text{ \ \ \ }0<t<\alpha,
\]
we see that there is a constant $C=C(n)$ such that
\[
I_{B_{\alpha}}(t)\geq C\min(I_{\Omega}(t),I_{\Omega}(\alpha-t))\text{
\ \ \ }0<t<\alpha.
\]
In particular the metric space $\left(  \Omega,\left\vert \cdot\right\vert
,dm\right)  $ has the relative uniform isoperimetric property.

\begin{theorem}
\label{sobcon}Let $X=X(\Omega)$ be a r.i. space on $\Omega,$ then
\[
\dot{W}_{X}^{1}(\Omega)\subset L^{\infty}\Leftrightarrow\left\Vert
t^{1/n-1}\chi_{(0,\left\vert \Omega\right\vert /2)}\right\Vert _{\bar
{X}^{\prime}}<\infty\Leftrightarrow\dot{W}_{X}^{1}(\Omega)\subset C_{b}%
(\Omega).
\]
(Here $C_{b}(\Omega)$ denotes the space of real valued continuous bounded
functions defined on $\Omega$).
\end{theorem}

\begin{proof}
Let us first observe that the condition of the Theorem can be reformulated in
terms of the isoperimetric profile of $\ \Omega$ as follows,%
\begin{equation}
\left\|  t^{1/n-1}\chi_{(0,\left|  \Omega\right|  /2)}\right\|  _{\bar
{X}^{\prime}}\simeq\left\|  \frac{1}{I_{\Omega}(t)}\right\|  _{\bar
{X}^{^{\prime}}}<\infty. \label{deantes}%
\end{equation}
Indeed, since $I_{\Omega}$ is symmetric around the point $\left|
\Omega\right|  /2,$ it follows that
\begin{align*}
\left\|  \frac{1}{I_{\Omega}(s)}\right\|  _{\bar{X}^{\prime}}  &  \leq\left\|
\frac{1}{I_{\Omega}(s)}\chi_{(0,\left|  \Omega\right|  /2)}\right\|  _{\bar
{X}^{\prime}}+\left\|  \frac{1}{I_{\Omega}(s)}\chi_{(\left|  \Omega\right|
/2,\left|  \Omega\right|  )}\right\|  _{\bar{X}^{\prime}}\\
&  =2\left\|  \frac{1}{I_{\Omega}(s)}\chi_{(0,\left|  \Omega\right|
/2)}\right\|  _{\bar{X}^{\prime}}\\
&  \leq2\left\|  \frac{1}{I_{\Omega}(s)}\right\|  _{\bar{X}^{\prime}}.
\end{align*}
Now (\ref{deantes}) follows since, in view of (\ref{perlip}), we have
\[
\left\|  \frac{1}{I_{\Omega}(s)}\chi_{(0,\left|  \Omega\right|  /2)}\right\|
_{\bar{X}^{\prime}}\approx\left\|  t^{1/n-1}\chi_{(0,\left|  \Omega\right|
/2)}\right\|  _{\bar{X}^{\prime}}.
\]
Suppose that $\left\|  \frac{1}{I_{\Omega}(t)}\right\|  _{\bar{X}^{^{\prime}}%
}\simeq\left\|  t^{1/n-1}\chi_{(0,\left|  \Omega\right|  /2)}\right\|
_{\bar{X}^{\prime}}<\infty.$ Let $f\in\dot{W}_{X}^{1}(\Omega).$ Since we have
shown in the previous section that $\left|  f\right|  ^{\ast}$ is locally
absolutely continuous (cf. \cite{leoni} and \cite{mamiadv}) we can write
\begin{align}
\left\|  f\right\|  _{L^{\infty}}-\left|  f\right|  ^{\ast}(\left|
\Omega\right|  )  &  =\left|  f\right|  ^{\ast}(0)-\left|  f\right|  ^{\ast
}(\left|  \Omega\right|  )=\int_{0}^{\left|  \Omega\right|  }(-\left|
f^{\ast}\right|  )^{\prime}(s)ds\label{iinf}\\
&  =\int_{0}^{\left|  \Omega\right|  }(-\left|  f^{\ast}\right|  )^{\prime
}(s)I_{\Omega}(s)\frac{ds}{I_{\Omega}(s)}\nonumber\\
&  \leq\left\|  (-\left|  f^{\ast}\right|  )^{\prime}(s)I_{\Omega}(s)\right\|
_{\bar{X}}\left\|  \frac{1}{I_{\Omega}(t)}\right\|  _{\bar{X}^{^{\prime}}%
}\text{ \ (by H\"{o}lder's inequality)}\nonumber\\
&  \leq\left\|  \left|  \nabla f\right|  \right\|  _{X}\left\|  \frac
{1}{I_{\Omega}(t)}\right\|  _{\bar{X}^{^{\prime}}}\text{ \ \ (by
(\ref{sobre})).}\nonumber
\end{align}
We have thus obtained%
\[
\left\|  f\right\|  _{L^{\infty}}\leq\left\|  f\right\|  _{L^{1}}+\left\|
\left|  \nabla f\right|  \right\|  _{X}\left\|  \frac{1}{I_{\Omega}%
(t)}\right\|  _{\bar{X}^{^{\prime}}},
\]
which applied to $f-\int_{\Omega}f$ yields%
\begin{align*}
\left\|  f-\int_{\Omega}f\right\|  _{L^{\infty}}  &  \leq\left\|
f-\int_{\Omega}f\right\|  _{L^{1}}+\left\|  \left|  \nabla f\right|  \right\|
_{X}\left\|  \frac{1}{I_{\Omega}(t)}\right\|  _{\bar{X}^{^{\prime}}}\\
&  \leq c(\left|  \Omega\right|  )\left\|  \left|  \nabla f\right|  \right\|
_{L^{1}}+\left\|  \left|  \nabla f\right|  \right\|  _{X}\left\|  \frac
{1}{I_{\Omega}(t)}\right\|  _{\bar{X}^{^{\prime}}}\text{ (by Poincar\'{e}'s
inequality)}\\
&  \leq c(\left|  \Omega\right|  )\phi_{X^{\prime}}(\left|  \Omega\right|
)\left\|  \left|  \nabla f\right|  \right\|  _{X}+\left\|  \left|  \nabla
f\right|  \right\|  _{X}\left\|  \frac{1}{I_{\Omega}(t)}\right\|  _{\bar
{X}^{^{\prime}}}\text{(by H\"{o}lder's inequality)}\\
&  =C(\left|  \Omega\right|  )\left\|  \frac{1}{I_{\Omega}(t)}\right\|
_{\bar{X}^{^{\prime}}}\left\|  \left|  \nabla f\right|  \right\|  _{X},
\end{align*}
where $C(\left|  \Omega\right|  )$ is a constant that depends on $X$ and the
measure of $\Omega.$

Conversely, suppose that $\dot{W}_{X}^{1}(\Omega)\subset L^{\infty},$ then
\[
\left\Vert f-\int_{\Omega}f\right\Vert _{L^{\infty}}\leq c\left\Vert
\left\vert \nabla f\right\vert \right\Vert _{X}.
\]
Since $\Omega$ has bounded Lipschitz boundary, this is equivalent (cf.
\cite{mamiadv} and \cite[Theorem 2]{mamipot}) to the existence of an absolute
constant $C>0$ such that for all $g\in\bar{X},$ $g\geq0$
\[
\left\Vert \int_{t}^{\left\vert \Omega\right\vert /2}\frac{g(s)}{I_{\Omega
}(s)}ds\right\Vert _{L^{\infty}}\leq C\left\Vert g\right\Vert _{\bar{X}%
}\text{.}%
\]
Thus,
\[
\sup_{\left\Vert g\right\Vert _{\bar{X}}\leq1}\int_{0}^{\left\vert
\Omega\right\vert /2}\frac{\left\vert g\right\vert (s)}{I_{\Omega}(s)}%
ds=\sup_{\left\Vert g\right\Vert _{\bar{X}}\leq1}\int_{0}^{\left\vert
\Omega\right\vert }\left\vert g\right\vert (s)\frac{\chi_{(0,\left\vert
\Omega\right\vert /2)}(s)}{I_{\Omega}(s)}ds<C,
\]
which, by duality, implies that
\[
\left\Vert \frac{1}{I_{\Omega}(s)}\chi_{(0,\left\vert \Omega\right\vert
/2)}\right\Vert _{\bar{X}^{\prime}}<\infty.
\]
To conclude the proof we show that (\ref{deantes}) and the relative uniform
isoperimetric property imply that
\[
\dot{W}_{X}^{1}(\Omega)\subset C_{b}(\Omega).
\]
Let $f\in\dot{W}_{X}^{1}(\Omega).$ Consider any open ball $B_{\alpha}$
contained in $\Omega$ with $\left\vert B_{\alpha}\right\vert =\alpha.$\ An
easy computation shows that
\[
\left\Vert \frac{1}{\min(t,\alpha-t)^{\frac{n-1}{n}}}\right\Vert _{\bar
{X}^{^{\prime}}}\simeq\left\Vert t^{1/n-1}\chi_{(0,\alpha/2)}\right\Vert
_{\bar{X}^{\prime}}.
\]
Applying the inequality (\ref{sobs}) to $f\chi_{B_{\alpha}}$ and integrating,
we get
\begin{align*}
ess\sup\left(  f\chi_{B_{\alpha}}\right)  -\frac{1}{\alpha}\int f\chi
_{B_{\alpha}}(x)dx  &  =\int_{0}^{\alpha}\left(  \left(  f\chi_{B_{\alpha}%
}\right)  ^{\ast\ast}(t)-\left(  f\chi_{B_{\alpha}}\right)  ^{\ast}(t)\right)
\frac{dt}{t}\\
&  \leq\int_{0}^{\alpha}\left(  \frac{t}{I_{B_{\alpha}}(t)}\frac{1}{t}\int%
_{0}^{t}\left\vert \nabla f\chi_{B_{\alpha}}\right\vert ^{\ast}(s)ds\right)
\frac{dt}{t}\\
&  \leq\left\Vert \frac{1}{t}\int_{0}^{t}\left\vert \nabla f\chi_{B_{\alpha}%
}\right\vert ^{\ast}(s)ds\right\Vert _{\bar{X}}\left\Vert \frac{1}%
{I_{B_{\alpha}}(t)}\right\Vert _{\bar{X}^{^{\prime}}}\\
&  \leq c(n,X)\left\Vert \left\vert \nabla f\chi_{B_{\alpha}}\right\vert
\right\Vert _{X}\left\Vert t^{1/n-1}\chi_{(0,\alpha/2)}\right\Vert _{\bar
{X}^{\prime}}.
\end{align*}
Similarly, considering $-f\chi_{B_{\alpha}},$ we get
\[
-ess\text{ }\inf f\left(  f\chi_{B_{\alpha}}\right)  +\frac{1}{\alpha}\int
f\chi_{B_{\alpha}}(x)dx\leq c(n,X)\left\Vert \left\vert \nabla f\chi
_{B_{\alpha}}\right\vert \right\Vert _{X}\left\Vert t^{1/n-1}\chi
_{(0,\alpha/2)}\right\Vert _{\bar{X}^{\prime}}.
\]
Adding these inequalities we see that
\[
ess\text{ }\inf\left(  f\chi_{B_{\alpha}}\right)  -ess\text{ }\inf\left(
f\chi_{B_{\alpha}}\right)  \leq c\left\Vert \left\vert \nabla f\chi
_{B_{\alpha}}\right\vert \right\Vert _{X}\left\Vert t^{1/n-1}\chi
_{(0,\alpha/2)}\right\Vert _{\bar{X}^{\prime}}.
\]
Thus, for almost every $x,y\in B_{\alpha},$%
\[
\left\vert f(x)-f(y)\right\vert \leq c(n)\left\Vert \left\vert \nabla
f\chi_{B_{\alpha}}\right\vert \right\Vert _{X}\left\Vert t^{1/n-1}%
\chi_{(0,\alpha/2)}\right\Vert _{\bar{X}^{\prime}}.
\]
The essential continuity of $f$ follows.
\end{proof}

\begin{remark}
\label{rema}Let us consider the case when $X=L^{p},$ with $p>n.$ An elementary
computation shows that
\[
\left\Vert t^{1/n-1}\chi_{(0,\alpha/2)}\right\Vert _{p^{\prime}}\leq
c_{(n,p)}\alpha^{\frac{1}{n}(1-\frac{n}{p})}.
\]
Let $B_{\alpha}$ be a ball with $\left\vert B_{\alpha}\right\vert =\alpha$,
then $\alpha^{1/n}$ is $c_{n}$ times the radius of the ball, thus, for almost
every $y,z\in$ $B_{\alpha}$ such that $\left\vert y-z\right\vert =c_{n}%
\alpha^{1/n},$ we get
\[
\left\vert f(y)-f(z)\right\vert \leq c(n,p)\left\vert y-z\right\vert
^{(1-\frac{n}{p})}\left\Vert \left\vert \nabla f\right\vert \right\Vert _{p}.
\]
The method of proof fails if $p=n.$ However, if we consider the smaller
Lorentz\footnote{See [(\ref{montana}), (\ref{vermont}), Chapter \ref{chapbmo}%
].} space $X=L^{n,1}\subset L^{n},$ then $X^{\prime}=L^{\frac{n}{n-1},\infty}%
$, and
\[
\left\Vert t^{1/n-1}\chi_{(0,\alpha/2)}\right\Vert _{L^{\frac{n}{n-1},\infty}%
}=\sup_{0<s<\alpha/2}s^{\frac{1}{n}-1}s^{1-\frac{1}{n}}=1.
\]
Thus, for almost every $y,z\in$ $B_{\alpha}$ such that $\left\vert
y-z\right\vert =c_{n}\alpha^{1/n},$ we have that
\[
\left\vert f(y)-f(z)\right\vert \leq c(n)\left\Vert \left\vert \nabla
f\chi_{B_{\alpha}}\right\vert \right\Vert _{L^{n,1}}=c(n)\int_{0}^{\left\vert
y-z\right\vert ^{n}}s^{1/n}\left\vert \nabla f\right\vert ^{\ast}(s)\frac
{ds}{s}.
\]
The essential continuity of $f$ follows. Thus we recover the classical result
independently due to Stein \cite{ST1} and C. P. Calder\'{o}n \cite{calx}.
\end{remark}

\begin{remark}
See \cite{Chi} for a related result, using a different method and involving
Orlicz norms.
\end{remark}

\subsection{Spaces defined in terms of the modulus of continuity on Lipschitz
domains of $\mathbb{R}^{n}$}

For Euclidean domains $\Omega$ with Lipschitz boundary it is known that (cf.
\cite[Theorem 1]{js}, \cite[Chapter 5, exercise 13, pag. 430]{bs}),
\begin{equation}
K(t,g;X\left(  {\Omega}\right)  ,\dot{W}_{X}^{1}(\Omega))\simeq\omega
_{X}(g,t),\;0<t<\left|  \Omega\right|  ,\nonumber
\end{equation}
where
\[
\omega_{X}(f,t)=\sup_{0<\left|  h\right|  \leq t}\left\|  (f(\cdot
+h)-f(\cdot))\chi_{\Omega(h)}\right\|  _{L^{p}(\Omega)},
\]
with $\Omega(h)=\left\{  x\in\Omega:x+\rho h\in\Omega,\text{ }0\leq\rho
\leq1\right\}  $ and $h\in\mathbb{R}^{n}$.

Moreover, as we have seen, $\left(  \Omega,\left\vert \cdot\right\vert
,dm\right)  $ has the relative uniform isoperimetric property. Consequently,
by Theorem \ref{continuo}, we have

\begin{theorem}
\label{pepe}Let $X$ be a r.i. space on $\Omega.$ If $f\in X+\dot{W}_{X}^{1}$
satisfies
\[
\int_{0}^{\mu(\Omega)}\frac{\omega_{X}\left(  f,\phi_{X}(t)\left\Vert \frac
{1}{\left(  \min(t,\left\vert \Omega\right\vert -t)\right)  ^{\frac{n-1}{n}}%
}\chi_{(0,t)}(s)\right\Vert _{\bar{X}^{^{\prime}}}\right)  }{\phi_{_{X}}%
(t)}\frac{dt}{t}<\infty,
\]
then, $f$ is essentially bounded and essentially continuous.
\end{theorem}

In particular, when $X=L^{p},$ we obtain

\begin{theorem}
\label{Besri}If$\ n/p<1,$ then there exists a constant $c>0,$ such that
\[
\left|  f(y)-f(z)\right|  \leq C\int_{0}^{\left|  y-z\right|  }\frac
{\omega_{p}(f,t)}{t^{n/p}}\frac{dt}{t}.
\]

\end{theorem}

\begin{proof}
Let $B_{\alpha}$ an open ball contained in $\Omega$ of measure $\alpha
\leq\left\vert \Omega\right\vert /2.$ Since $p>n$ an elementary computation
shows that
\[
\left\Vert \frac{1}{I_{B_{\alpha}}(s)}\chi_{(0,t)}(s)\right\Vert
_{L^{p^{\prime}}(B_{\alpha})}\leq c_{(n,p)}t^{\frac{1}{n}-\frac{1}{p}}.
\]
Thus, Theorem \ref{continuo} ensures for almost every $y,z\in B_{\alpha},$%
\begin{align*}
\left\vert f(y)-f(z)\right\vert  &  \leq c\int_{0}^{2\alpha}\frac{K\left(
t^{1/n},f;L^{p}(\Omega),\dot{W}_{p}^{1}(\Omega)\right)  }{t^{1/p}}\frac{dt}%
{t}\\
&  \simeq c\int_{0}^{2\alpha}\frac{\omega_{p}(f,t^{1/n})}{t^{1/p}}\frac{dt}%
{t}\\
&  =c\int_{0}^{\left(  2\alpha\right)  ^{1/n}}\frac{\omega_{p}(f,t)}{t^{n/p}%
}\frac{dt}{t}.
\end{align*}

Since $\left\vert B_{\alpha}\right\vert =\alpha$, $\alpha^{1/n}$ is a constant
times the radius of the ball, therefore, for almost every $y,z\in$ $B_{\alpha
}$ such that $\left\vert y-z\right\vert =c\alpha^{1/n},$
\[
\left\vert f(y)-f(z)\right\vert \preceq\int_{0}^{\left\vert y-z\right\vert
}\frac{\omega_{p}(f,t)}{t^{n/p}}\frac{dt}{t}.
\]
The essential continuity of $f$ follows.
\end{proof}

\begin{remark}
In Chapter \ref{Garetal} we shall discuss the connection with A. Garsia's work.
\end{remark}

\begin{theorem}
Let $X$ be a r.i. space such that $\overline{\alpha}_{\Lambda(X^{\prime}%
)}<\frac{1}{n}.$ If $f$ satisfies
\[
\int_{0}^{\left|  \Omega\right|  }\frac{\omega_{X}(f,t)}{\phi_{X}(t^{1/n}%
)}\frac{dt}{t}<\infty,
\]
then, $f$ is essentially continuous.
\end{theorem}

\begin{proof}
It is enough to prove
\[
R(t)=\left\Vert \frac{1}{\left(  \min(s,(\left\vert \Omega\right\vert
-s))\right)  ^{\frac{n-1}{n}}}\chi_{(0,t)}(s)\right\Vert _{\bar{X}^{\prime}%
}\leq c_{(n,X)}t^{\frac{1}{n}}\phi_{\bar{X}},(t),\text{ \ \ }0<t<\left\vert
\Omega\right\vert .
\]

Recall that if $\underline{\alpha}_{\Lambda(X^{\prime})}>0,$ the fundamental
function of $X^{^{\prime}}$ satisfies (see \cite[Theorem 2.4]{Sh})
\[
d\phi_{X^{^{\prime}}}(s)\simeq\frac{\phi_{X^{^{\prime}}}(s)}{s},
\]
and, moreover, for every $0<\gamma<\underline{\alpha}_{\Lambda(X^{\prime})}$
the function $\phi_{X^{^{\prime}}}(s)/s^{\gamma}$ is almost increasing (i.e.
$\exists c>0$ s.t. $\phi_{X^{^{\prime}}}(s)/s^{\gamma}\leq c\phi_{X^{^{\prime
}}}(t)/t^{\gamma}$ whenever $t\geq s).$ Pick $0<\beta\ $such that (notice that
$\overline{\alpha}_{\Lambda(X)}<\frac{1}{n}$ implies that $\underline{\alpha
}_{\Lambda(X^{\prime})}>1-\frac{1}{n})$
\[
1-\frac{1}{n}+\beta<\underline{\alpha}_{\Lambda(X^{\prime})}.
\]
Since $I_{\Omega}$ is symmetric around the point $\left\vert \Omega\right\vert
/2,$ we get
\[
R(t)\simeq\left\Vert \frac{1}{s^{\frac{n-1}{n}}}\chi_{(0,t)}(s)\right\Vert
_{\bar{X}^{\prime}},
\]
and
\begin{align*}
R(t)  &  \leq\int_{0}^{t}s^{\frac{1}{n}-1}d\phi_{X^{^{\prime}}}(s)\\
&  \simeq\int_{0}^{t}s^{\frac{1}{n}-1}\frac{\phi_{X^{^{\prime}}}(s)}{s}ds\\
&  =\int_{0}^{t}\frac{\phi_{X^{^{\prime}}}(s)}{s^{1-\frac{1}{n}+\beta}}%
\frac{ds}{s^{1-\beta}}\\
&  \preceq\frac{\phi_{X^{^{\prime}}}(t)}{t^{1-\frac{1}{n}+\beta}}\int_{0}%
^{t}\frac{ds}{s^{1-\beta}}\\
&  \simeq\frac{\phi_{X^{^{\prime}}}(t)}{t^{1-\frac{1}{n}}}\\
&  =t^{\frac{1}{n}}\phi_{\bar{X}}(t).
\end{align*}

\end{proof}

\section{Domains of Maz'ya's class $\mathcal{J}_{\alpha}$ ($1-1/n\leq\alpha
<1$)\label{Mazclass}}

\begin{definition}
(See \cite{maz}, \cite{ma}) A domain $\Omega\subset\mathbb{R}^{n}$ (with
finite measure) belongs to the class $\mathcal{J}_{\alpha}$ ($1-1/n\leq
\alpha<1$) if there exists a constant $M\in(0,\left\vert \Omega\right\vert )$
such that
\[
U_{\alpha}(M)=\sup\frac{\left\vert \mathcal{S}\right\vert ^{\alpha}}%
{P_{\Omega}(\mathcal{S})}<\infty,
\]
where the sup is taken over all $\mathcal{S}$ open bounded subsets of $\Omega$
such that $\Omega\cap\partial\mathcal{S}$ is a manifold of class $C^{\infty}$
and $\left\vert \mathcal{S}\right\vert \leq M,$ (in which case we will say
that $\mathcal{S}$ is an admissible subset) and where for a measurable set
$E\subset\Omega,$ $P_{\Omega}(E)$ is the De Giorgi perimeter of $E$ in
$\Omega$ defined by
\[
P_{\Omega}(E)=\sup\left\{  \int_{E}div\varphi\text{ }dx:\varphi\in\lbrack
C_{0}^{1}\left(  \Omega\right)  ]^{n},\text{ }\left\Vert \varphi\right\Vert
_{L^{\infty}(\Omega)}\leq1\right\}  .
\]
By an approximation process it follows that if $\Omega$ is a bounded domain in
$\mathcal{J}_{\alpha},$ then there exists a constant $c_{\Omega}>0$ such that,
for all measurable set $E\subset\Omega$ with $\left\vert E\right\vert
\leq\left\vert \Omega\right\vert /2,$ we have
\[
P_{\Omega}(E)\geq c_{\Omega}\left\vert E\right\vert ^{\alpha}.
\]
Since $E$ and $\Omega\backslash E$ have the same boundary measure, we obtain
the following isoperimetric inequality
\[
I_{\Omega}(t)\geq c_{\Omega}\left(  \min(t,\left\vert \Omega\right\vert
-t)\right)  ^{\alpha}:=J_{\Omega}(t),\text{ \ \ }0<t<\left\vert \Omega
\right\vert .
\]

\end{definition}

\begin{example}
If $\Omega$ is a bounded domain, star shaped with respect to a ball, then
$\Omega$ belongs to the class $\mathcal{J}_{1-1/n}$; if $\Omega$ is a bounded
domain with the cone property then $\Omega$ belongs to the class
$\mathcal{J}_{1-1/n}$; if $\Omega$ is a bounded Lipschitz domain, then
$\Omega$ belongs to the class $\mathcal{J}_{1-1/n}$; if $\Omega$ is a $s-$John
domain, then $\ \Omega\in\mathcal{J}_{(n-1)s/n}$; if $\Omega$ is a domain with
one $\beta-$cusp ($\beta\geq1),$ then it belongs to the Maz'ya class
$\mathcal{J}_{\frac{\beta\left(  n-1\right)  }{\beta(n-1)+1}}$.
\end{example}

\begin{theorem}
Let $\Omega$ be a domain in the Maz'ya class$\mathcal{J}_{\alpha}$, and let
$X$ a r.i. space on $\Omega.$ Suppose that
\begin{equation}
\left\Vert \frac{1}{J_{\Omega}(t)}\right\Vert _{X^{\prime}}<\infty.
\label{coo}%
\end{equation}
Then,

\begin{enumerate}
\item
\[
\dot{W}_{X}^{1}(\Omega)\subset C_{b}(\Omega).
\]

\item If $f\in X+\dot{W}_{X}^{1}$ satisfies
\[
\int_{0}^{\mu(\Omega)}\frac{K\left(  \phi_{X}(t)\left\Vert \frac{1}{J_{\Omega
}(t)}\chi_{(0,t)}(s)\right\Vert _{\bar{X}^{^{\prime}}},f;X,\dot{W}_{X}%
^{1}\right)  }{\phi_{_{X}}(t)}\frac{dt}{t}<\infty,
\]
then $f$ is essentially bounded and essentially continuous.
\end{enumerate}

\begin{proof}
Part 1. The inclusion $\dot{W}_{X}^{1}(\Omega)\subset L^{\infty}$ follows in
the same way as the corresponding part of Theorem \ref{sobcon} (cf. inequality
(\ref{iinf})). To prove the essential continuity we proceed as follows. Let
$B$ be any open ball contained in $\Omega$ with $\left\vert B\right\vert
\leq\min(1,\left\vert \Omega\right\vert /2).$ Notice that if $f\in\dot{W}%
_{X}^{1}(\Omega),$ then $f\chi_{B}\in$ $\dot{W}_{X}^{1}(B).$ Now, since $B$ is
a Lip domain, by Theorem \ref{sobcon} we just need to verify that
\[
\left\Vert t^{1/n-1}\chi_{(0,\left\vert B\right\vert /2)}\right\Vert _{\bar
{X}^{\prime}}<\infty.
\]
Since $1-1/n\leq\alpha<1,$ and $0<t<\left\vert B\right\vert /2<1,$ we have
\[
\sup_{0<t<\left\vert B\right\vert /2}t^{\alpha+1/n-1}=\left(  \frac{\left\vert
B\right\vert }{2}\right)  ^{\alpha+1/n-1}.
\]
Thus,%
\begin{align*}
\left\Vert t^{1/n-1}\chi_{(0,\left\vert B\right\vert /2)}\right\Vert _{\bar
{X}^{\prime}}  &  =\left\Vert \frac{t^{\alpha+1/n-1}}{t^{\alpha}}%
\chi_{(0,\left\vert B\right\vert /2)}\right\Vert _{\bar{X}^{\prime}}\\
&  \leq\left(  \frac{\left\vert B\right\vert }{2}\right)  ^{\alpha
+1/n-1}\left\Vert \frac{1}{t^{\alpha}}\chi_{(0,\left\vert B\right\vert
/2)}\right\Vert _{\bar{X}^{\prime}}<\infty\text{ \ \ (by (\ref{coo})).}%
\end{align*}

Part 2. Let $B$ be an open ball contained in $\Omega$ with $\left\vert
B\right\vert \leq\min(1,\left\vert \Omega\right\vert /2).$ Then
\[
K\left(  t,f\chi_{B};X(B),\dot{W}_{X}^{1}(B)\right)  \leq K\left(
t,f;X,\dot{W}_{X}^{1}\right)  .
\]
Using the same argument given in the first part of the proof, we obtain
\begin{align*}
\left\Vert \frac{\chi_{(0,t)}(s)}{\left(  \min(t,\left\vert B\right\vert
-t)\right)  ^{\frac{n-1}{n}}}\right\Vert _{\bar{X}^{^{\prime}}}  &
\leq2\left\Vert \frac{\chi_{(0,t)}(s)}{t^{\frac{n-1}{n}}}\right\Vert _{\bar
{X}^{^{\prime}}}\\
&  \preceq\left\Vert \frac{1}{J_{\Omega}(t)}\chi_{(0,t)}(s)\right\Vert
_{\bar{X}^{^{\prime}}},\text{ \ \ \ }0<t<\left\vert B\right\vert .
\end{align*}
Therefore,%
\[
\int_{0}^{\left\vert B\right\vert }\frac{K\left(  \phi_{X}(t)\left\Vert
\frac{\chi_{(0,t)}(s)}{\left(  \min(t,\left\vert B\right\vert -t)\right)
^{\frac{n-1}{n}}}\right\Vert _{\bar{X}^{^{\prime}}},f\chi_{B};X(B),\dot{W}%
_{X}^{1}(B)\right)  }{\phi_{_{X}}(t)}\frac{dt}{t}<\infty,
\]
and Theorem \ref{pepe} applies.
\end{proof}
\end{theorem}

\section{Ahlfors Regular Metric Measure Spaces\label{ahlfor}}

Let $(\Omega,d,\mu)$ be a complete connected metric Borel measure space and
let $k>1.$ We shall say that $(\Omega,d,\mu)$ is Ahlfors $k-$regular if there
exist absolute constants $c_{\Omega},C_{\Omega}$ such that
\begin{equation}
c_{\Omega}r^{k}\leq\mu\left(  B(x,r)\right)  \leq C_{\Omega}r^{k},\text{
\ \ }\forall x\in\Omega,\text{ \ }r\in(0,diam(\Omega)). \label{AR}%
\end{equation}

We will consider Ahlfors $k-$regular spaces $(\Omega,d,\mu)$ that support a
weak $(1,1)-$Poincar\'{e} inequality. In other words we shall assume the
existence of constants $C>0$ and $\lambda\geq1$ such that for all $u\in
Lip(\Omega),$
\begin{equation}
\int_{B(x,r)}\left\vert u(y)-u_{B_{x,r}}\right\vert d\mu(y)\leq Cr\int%
_{B(x,\lambda r)}\left\vert \nabla u(y)\right\vert d\mu(y), \label{onepoin}%
\end{equation}
where $u_{B_{x,r}}$ denotes the mean value of $u$ in $B$, i.e. $u_{B_{x,r}%
}=\frac{1}{\mu\left(  B(x,r)\right)  }\int_{B(x,r)}u(y)d\mu(y).$

Examples of spaces supporting a (weak) $(1,1)-$Poincar\'{e} inequality include
Riemannian manifolds with nonnegative Ricci curvature, Carnot-Carath\'{e}odory
groups, and more generally (in the case of doubling spaces)
Carnot-Carath\'{e}odory spaces associated to smooth (or locally Lipschitz)
vector fields satisfying H\"{o}rmander's condition (see for example
\cite{Amb}, \cite{HK} and the references quoted therein).

By a well known result of Hajlasz and Koskela (cf. \cite{HK}), (\ref{AR}) and
(\ref{onepoin}) imply%
\[
\left(  \int_{B(x,r)}\left\vert u(y)-u_{B_{x,r}}\right\vert ^{k/(k-1)}%
d\mu(y)\right)  ^{(k-1)/k}\leq D\int_{B(x,2\lambda r)}\left\vert \nabla
u(y)\right\vert d\mu(y),
\]
with $C=\left(  2C\right)  ^{(k-1)/k}.$

According to \cite{Mira} (see also the references quoted therein), given a
Borel set $E\subset\Omega$, and $A\subset\Omega$ open, the relative perimeter
of $E$ in $A$, denoted by $P(E,A),$ is defined by%
\[
P(E,A)=\inf\left\{  \lim\inf_{h\rightarrow\infty}\int_{A}\left\vert \nabla
u_{h}\right\vert d\mu:u_{h}\in Lip_{loc}(A),\text{ }u_{h}\rightarrow\chi
_{E}\text{ in }L_{loc}^{1}(A)\right\}  .
\]

\begin{lemma}
\label{lempo1}The following relative isoperimetric inequality holds:%
\begin{equation}
\min\left(  \mu(E\cap B(x,r)),\mu(E^{c}\cap B(x,r))\right)  \leq D\left(
P(E,B(x,r))\right)  ^{k/(k-1)}. \label{porfinsale}%
\end{equation}

\end{lemma}

\begin{proof}
The proof of this result is contained in the proof of \cite[Theorem 4.3]{Amb}
\end{proof}

\begin{theorem}
Let $k>1$ be the exponent satisfying (\ref{AR}), and let $B:=B(x,r)\ $\ be a
ball. The following statements are equivalent.

\begin{enumerate}
\item For every set of finite perimeter $E$ in $\Omega,$
\[
c(k,C)\left(  \min(\mu(E\cap B),\mu(E^{c}\cap B))\right)  ^{\frac{k}{k-1}}\leq
P(E,B),
\]
where the constant $c(k,C)$ does not depend on $B.$

\item $\forall u\in Lip(\Omega),$ the function $\left\vert u\chi
_{B}\right\vert _{\mu}^{\ast}$ is locally absolutely continuous and, for
$0<t<\mu(B),$
\[
c(k,C)\int_{0}^{t}\left\vert \left(  \left(  -\left\vert u\chi_{B}\right\vert
\right)  _{\mu}^{\ast}\right)  ^{\prime}(\cdot)\left(  \min(\cdot,\mu
(B)-\cdot)\right)  ^{\frac{k-1}{k}}\right\vert ^{\ast}(s)ds\leq\int_{0}%
^{t}\left\vert \nabla u\chi_{B}\right\vert _{\mu}^{\ast}(s)ds\text{.}%
\]

\item Oscillation inequality: $\forall u\in Lip(\Omega)$ and, for
$0<t<\mu(B),$%
\[
(\left\vert u\chi_{B}\right\vert _{\mu}^{\ast\ast}(t)-\left\vert u\chi
_{B}\right\vert _{\mu}^{\ast}(t))\leq\frac{t}{c(k,C)\left(  \min
(t,\mu(B)-t)\right)  ^{\frac{k-1}{k}}}\left\vert \nabla u\chi_{B}\right\vert
_{\mu}^{\ast\ast}(t).
\]

\end{enumerate}
\end{theorem}

\begin{proof}
Consider the metric space $\left(  B,d_{\mid B},\mu_{\mid B}\right)  ,$ then
the Theorem is a particular case of Theorem 1 of \cite{mamiadv}.
\end{proof}

The local version of Theorem \ref{main2sig} is

\begin{theorem}
Let $X$ be a r.i. space on $\Omega$ and let $B\subset\Omega$ be an open ball.
Then, for each $f\in X+S_{X},$%
\[
\left(  f\chi_{B}\right)  _{\mu}^{\ast\ast}(t/2)-\left(  f\chi_{B}\right)
_{\mu}^{\ast}(t/2)\leq4\frac{K\left(  \psi(t),f\chi_{B};X,S_{X}\right)  }%
{\phi_{X}(t)},\text{ }0<t<\mu(B),
\]
where
\[
\psi(t)=\frac{\phi_{X}(t)}{t}\left\|  \frac{s}{c(k,C)\left(  \min
(s,\mu(B)-s)\right)  ^{\frac{k-1}{k}}}\chi_{(0,t)}(s)\right\|  _{\bar
{X}^{^{\prime}}}.
\]

\end{theorem}

\begin{proof}
Let $f\in X+S_{X},$ then $f\chi_{B}\in X_{r}(B)+S_{X_{r}}(B),$ where $B$ is
the metric space $\left(  B,d_{\mid B},\mu_{\mid B}\right)  .$ By Lemma
\ref{lempo1} we know that
\[
c(k,C)\left(  \min(\mu(E\cap B),\mu(E^{c}\cap B))\right)  ^{\frac{k}{k-1}}\leq
P(E,B).
\]
Thus, for any Borel set $E\subseteq B,$%
\[
c(k,C)\left(  \min(\mu(E),\mu(B)-\mu(E))\right)  ^{\frac{k}{k-1}}\leq
P_{B}(E).
\]
Consequently, $J_{B}(t)=c(k,C)\left(  \min(t,\mu(B)-t)\right)  ^{\frac{k}%
{k-1}}$ ($0<t<\mu(B))$ is an isoperimetric estimator of $\left(  B,d_{\mid
B},\mu_{\mid B}\right)  ,$ and now we finish the proof in the same way as in
Theorem \ref{main2sig}.
\end{proof}

\begin{theorem}
Let $f\in X+S_{X}$ and let $B$ be an open ball, if
\[
\int_{0}^{\mu(B)}\frac{K\left(  \phi_{X}(t)\left\|  \left(  \min
(s,\mu(B)-s)\right)  ^{1-1/k}\chi_{(0,t)}(s)\right\|  _{\bar{X}^{^{\prime}}%
},f\chi_{B};X,S_{X}\right)  }{\phi_{_{X}}(t)}\frac{dt}{t}<\infty
\]
then, $f\chi_{B}$ is essentially bounded and essentially continuous.
\end{theorem}

\begin{proof}
By the proof of the previous Theorem we know that
\[
J_{B}(t)=c(k,C)\left(  \min\left(  t,\mu(B)-t\right)  \right)  ^{\frac{k}%
{k-1}}%
\]
is an isoperimetric estimator of $\left(  B,d_{\mid B},\mu_{\mid B}\right)  .$
For any open ball $B(x,r)\subset B,$ it follows from Lemma \ref{lempo1} that,
for $0<s<\mu(Q_{B(x,r)}),$
\begin{align*}
c(k,C)\left(  \min\left(  t,\mu(B(x,r))-t\right)  \right)  ^{\frac{k}{k-1}}
&  \preceq c(k,C)\min(J_{B}(t),J_{B}(\mu(Q_{B(x,r)})-t))\\
&  \leq P_{B(x,r)}(s).
\end{align*}
Therefore $\left(  B,d_{\mid B},\mu_{\mid B}\right)  $ has the relative
isoperimetric property and Theorem \ref{continuo} applies.
\end{proof}

\begin{remark}
In the particular case $X=L^{p},$ we can thus use the same argument given in
Theorem \ref{Besri} to obtain that for $k/p<1,$ there exists an absolute
constant such that
\[
\left\vert f(y)-f(z)\right\vert \preceq\int_{0}^{\left\vert y-z\right\vert
}\frac{K\left(  t,f\chi_{B};X,S_{X}\right)  }{t^{k/p}}\frac{dt}{t},\text{
\ \ }y,z\in B.
\]

\end{remark}

\chapter{Fractional Sobolev inequalities in Gaussian measures
\label{chapgauss}}

\section{Introduction and Summary}

As another application of our theory, in this chapter we consider in detail
fractional logarithmic Sobolev inequalities. We will deal not only with
Gaussian measures but also with measures that interpolate between Gaussian and exponential.

In the context of classical Gaussian measures a typical result in this chapter
is given by the following fractional logarithmic Sobolev inequality. Let
$d\gamma_{n}$ be the Gaussian measure on $\mathbb{R}^{n}$, let $1\leq
q<\infty,$ $\theta\in(0,1)$; then, there exists an absolute constant $c>0,$
independent of the dimension, such that (cf. Theorem \ref{fustado} below)%
\begin{equation}
\left\{  \int_{0}^{1/2}\left\vert f\right\vert _{\gamma_{n}}^{\ast}%
(t)^{q}\left(  \log\frac{1}{t}\right)  ^{\frac{q\theta}{2}}dt\right\}
^{1/q}\leq c\left\Vert f\right\Vert _{B_{L^{q}}^{\theta,q}(\gamma_{n})},
\label{fusta}%
\end{equation}
where $B_{L^{q}}^{\theta,q}(\gamma_{n})\ $is the Gaussian Besov space, see
(\ref{besgaus}) below. Note that if $q=2,$ (\ref{fusta}) interpolates between
the embedding that follows from the classical logarithmic Sobolev inequality
(which corresponds to the case $\theta=1$) and the trivial embedding
$L^{2}\subset L^{2}$ (the case $\theta=0$). For related inequalities using
semigroups see \cite{bakrmey} and also \cite{fei}.

More generally, we will also prove fractional Sobolev inequalities for tensor
products of measures that, on the real line are defined as follows. Let
$\alpha\geq0,$ \ $r\in\left[  1,2\right]  $ and $\gamma=\exp(2\alpha/(2-r)),$
($\alpha=0$ if $r=2$) and let
\[
d\mu_{r,\alpha}(x)=Z_{r,\alpha}^{-1}\exp\left(  -\left\vert x\right\vert
^{r}(\log(\gamma+\left\vert x\right\vert )^{\alpha}\right)  dx,
\]%
\[
\mu_{r,\alpha,n}=\mu_{p,\alpha}^{\otimes n},
\]
where $Z_{r,\alpha}^{-1}$ is chosen to ensure that $\mu_{r,\alpha
}(\mathbb{R)=}1.$ The corresponding results are apparently new and give
fractional Sobolev inequalities, that just like the logarithmic Sobolev
inequalities of \cite{mamiadv}, exhibit logarithmic gains of integrability
that are directly related to the corresponding isoperimetric profiles. For
example, if $\alpha=0,$ then the corresponding fractional Sobolev inequalities
take the following form. Let $1\leq q<\infty,$ $\theta\in(0,1)$, then there
exists an absolute constant $c>0,$ independent of the dimension, such that
(cf. Theorem \ref{fustado} below)%
\[
\left(  \int_{0}^{1/4}\left\vert f\right\vert _{\mu_{r,0,n}}^{\ast}%
(t)^{q}\left(  \log\frac{1}{t}\right)  ^{q\theta\left(  1-1/r\right)
}dt\right)  ^{1/q}\leq c\left\Vert f\right\Vert _{B_{L^{q}}^{\theta,q}%
(\mu_{r,0,n})}.
\]

Likewise, for $q=\infty$ (cf. (\ref{vera1}) below)%
\[
\sup_{t\in(0,\frac{1}{4})}\left(  \left|  f\right|  _{\mu_{r,0,n}}^{\ast\ast
}(t)-\left|  f\right|  _{\mu_{r,0,n}}^{\ast}(t)\right)  \left(  \log\frac
{1}{t}\right)  ^{(1-\frac{1}{r})\theta}\leq c\left\|  f\right\|  _{\dot
{B}_{L^{\infty}}^{\theta,\infty}(\mu_{r,0,n})}.
\]

We also explore the scaling of fractional inequalities for Gaussian Besov
spaces based on exponential Orlicz spaces. We show that in this context the
gain of integrability can be measured directly in the power of the exponential.

We start by considering the corresponding embeddings of Gaussian-Sobolev
spaces into $L^{\infty}$.

\section{Boundedness of functions in Gaussian-Sobolev spaces}

Let $\alpha\geq0,$ \ $r\in\left[  1,2\right]  $ and $\gamma=\exp
(2\alpha/(2-r))$ ($\alpha=0$ if $r=2$), and let $\mu_{r,\alpha}$ be the
probability measure on $\mathbb{R}$ defined by%
\[
d\mu_{r,\alpha}(x)=Z_{r,\alpha}^{-1}\exp\left(  -\left|  x\right|  ^{r}%
(\log(\gamma+\left|  x\right|  )^{\alpha}\right)  dx=\varphi_{r,\alpha
}(x)dx,\text{ }x\in\mathbb{R}\text{,}%
\]
where $Z_{r,\alpha}^{-1}$ is chosen to ensure that $\mu_{r,\alpha
}(\mathbb{R)=}1.$ Then we let%
\[
\varphi_{\alpha,r}^{n}(x)=\varphi_{r,\alpha}(x_{1})\cdots\varphi_{r,\alpha
}(x_{n}),\text{ }x\in\mathbb{R}^{n},
\]
and $\mu_{r,\alpha,n}=\mu_{r,\alpha}^{\otimes n}.$ In other words%
\[
d\mu_{r,\alpha,n}(x)=\varphi_{r,\alpha}^{n}(x)dx.
\]
In particular, $\mu_{2,0,n}=\gamma_{n}$ (Gaussian measure).

It is known that the isoperimetric problem for $\mu_{r,\alpha}$ is solved by
half-lines (cf. \cite{Bor} and \cite{Bob}) and the isoperimetric profile is
given by
\[
I_{\mu_{r,\alpha}}(t)=\varphi\left(  H^{-1}(\min(t,1-t)\right)  =\varphi
\left(  H^{-1}(t\right)  ),\text{ \ \ \ }t\in\lbrack0,1],
\]
where $H:\mathbb{R}\rightarrow(0,1)$ is the increasing function given by
\[
H(r)=\int_{-\infty}^{r}\varphi(x)dx.
\]
Moreover (cf. \cite{bar} and \cite{BCR1}), there exist constants $c_{1},c_{2}$
such that, for all $t\in\lbrack0,1]$,
\begin{equation}
c_{1}L_{\mu_{r,\alpha}}(t)\leq I_{\mu_{r,\alpha}}(t)\leq c_{2}L_{\mu
_{r,\alpha}}(t), \label{qq}%
\end{equation}
where
\[
L_{\mu_{r,\alpha}}(t)=\min(t,1-t)\left(  \log\frac{1}{\min(t,1-t)}\right)
^{1-\frac{1}{r}}\left(  \log\log\left(  e+\frac{1}{\min(t,1-t)}\right)
\right)  ^{\frac{\alpha}{r}}.
\]

Moreover, we have%
\begin{equation}
I_{\mu_{r,\alpha,n}}(t)\simeq t\left(  \log\frac{1}{t}\right)  ^{1-\frac{1}%
{r}}\left(  \log\log\left(  e+\frac{1}{t}\right)  \right)  ^{\alpha/r},\text{
for }t\in\left(  0,\frac{1}{2}\right)  . \label{debase1}%
\end{equation}

For the rest of the section we shall let $\mu$ denote the measure
$\mu_{r,\alpha,n}$ on $\mathbb{R}^{n}.$ For a given r.i. space $X:=X\left(
\mathbb{R}^{n},\mu\right)  ,$ let $W_{X}^{1}\left(  \mathbb{R}^{n},\mu\right)
$ be the classical Sobolev space endowed with the norm $\left\|  u\right\|
_{W_{X}^{1}(\mathbb{R}^{n},\mu)}=\left\|  u\right\|  _{X}+\left\|  \left|
\nabla u\right|  \right\|  _{X}.$ The homogeneous Sobolev space $\dot{W}%
_{X}^{1}(\mathbb{R}^{n},\mu)$ is defined by means of the quasi norm $\left\|
u\right\|  _{\dot{W}_{X}^{1}}:=\left\|  \left|  \nabla u\right|  \right\|
_{X}.$

The discussion of Chapter \ref{capitap} - Section \ref{class} applies and
therefore we see that $W_{L_{1}}^{1}(\mathbb{R}^{n},\mu)$ is invariant under
truncation. Moreover, if $u\in W_{L_{1}}^{1}(\mathbb{R}^{n},\mu)$ then the
following co-area formula holds:
\[
\int_{\mathbb{R}^{n}}\left|  \nabla u(x)\right|  d\mu(x)=\int_{\mathbb{R}^{n}%
}\left|  \nabla u(x)\right|  \varphi_{\alpha,p}^{n}(x)dx=\int_{-\infty
}^{\infty}P_{\mu}(u>s)ds.
\]
From here we see that inequalities [Chapter \ref{capitap}, (\ref{sobu}),
(\ref{sobre}) and (\ref{sobs})] hold for all $W_{L_{1}}^{1}(\mathbb{R}^{n}%
,\mu)$ functions (of course, the rearrangements are now with respect to the
measure $\mu).$ Finally, if we consider
\[
K(t,f;X,\dot{W}_{X}^{1})=\inf\{\left\|  f-g\right\|  _{X}+t\left\|  g\right\|
_{\dot{W}_{X}^{1}(\mathbb{R}^{n},\mu)}\},
\]
all the results that we have obtained in Chapter \ref{main}, remain true.

\begin{theorem}
If $\mu=\mu_{r,\alpha,n}$ then
\[
\dot{W}_{X}^{1}(\mathbb{R}^{n},\mu)\nsubseteq L^{\infty}.
\]

\end{theorem}

\begin{proof}
By \cite[Theorem 6]{mamiadv} the embedding $\dot{W}_{X}^{1}(\mathbb{R}^{n}%
,\mu)\subset L^{\infty}$ is equivalent to the existence of a positive constant
$c>0,$ such that for all $f\in\bar{X},$ supported on $(0,\frac{1}{2})$ we
have
\[
\sup_{t\geq0}\int_{t}^{1/2}\frac{\left\vert f(s)\right\vert }{L_{\mu
_{r,\alpha,n}}(s)}ds\leq c\left\Vert f\right\Vert _{\bar{X}}.
\]
In particular this implies that
\[
\int_{0}^{1/2}\frac{ds}{L_{\mu_{r,\alpha,n}}(s)}\leq c.
\]
But this is not possible since $1/L_{\mu_{r,\alpha,n}}(s)\notin L^{1}.$
\end{proof}

It follows that the results of Chapter \ref{capitap} cannot be applied
directly to obtain the continuity of functions in the space $\dot{W}_{X}^{1}.$

\begin{remark}
Let us remark that since continuity is a local property, a weak version of the
Morrey-Sobolev theorem (that depends on the dimension) is available. Let
$\mu=\mu_{r,\alpha,n},$ and let $X=X\left(  \mathbb{R}^{n},\mu\right)  $ be a
r.i space on $\left(  \mathbb{R}^{n},\mu\right)  $ such that
\[
\left\Vert \frac{1}{\min(1,1-t)^{1-1/n}}\right\Vert _{\bar{X}^{\prime}}%
<\infty.
\]
Then every function in $\dot{W}_{X}^{1}(\mathbb{R}^{n},\mu)$ is essentially continuous.
\end{remark}

\begin{proof}
Let $f\in\dot{W}_{X}^{1}(\mathbb{R}^{n},\mu)$ and let $B\subset\mathbb{R}^{n}$
be an arbitrary ball with Lebesgue measure equal to $1.$ To prove that $f$ is
continuous on $B$ let us note that $f\chi_{B}\in\dot{W}_{X}^{1}(B,\mu),$ i.e.
\[
\left\Vert \left\vert \nabla f\chi_{B}\right\vert \right\Vert _{X}<\infty.
\]
Let $m$ be the Lebesgue measure on $\mathbb{R}^{n},$ it is plain that for all
$t>0,$
\[
c_{B}m\left\{  x\in B:\left\vert \nabla f\right\vert >t\right\}  \leq
\mu\left\{  x\in B:\left\vert \nabla f\right\vert >t\right\}  \leq
C_{B}m\left\{  x\in B:\left\vert \nabla f\right\vert >t\right\}  ,
\]
where $c_{B}=\inf_{x\in B}\varphi_{\alpha,p}^{n}(x)$ and $C_{B}=\max_{x\in
B}\varphi_{\alpha,p}^{n}(x).$ Therefore,
\[
c_{Q}\left\Vert \left\vert \nabla f\chi_{B}\right\vert \right\Vert
_{X(B,m)}\leq\left\Vert \left\vert \nabla f\chi_{B}\right\vert \right\Vert
_{X(\mathbb{R}^{n},\mu)}\leq C_{Q}\left\Vert \left\vert \nabla f\chi
_{B}\right\vert \right\Vert _{X(B,m)}.
\]
Consequently, $f\chi_{B}\in\dot{W}_{X}^{1}(B,m),$ and by Theorem \ref{sobcon},
$f\chi_{B}\in C(B).$
\end{proof}

\section{Embeddings of Gaussian Besov spaces}

In what follows unless it is necessary to be more specific we shall let
$\mu:=\mu_{r,\alpha,n}.$ We consider the Besov spaces $\dot{B}_{X}^{\theta
,q}(\mu),$ $B_{X}^{\theta,q}(\mu)$ can be defined using real interpolation
(cf. \cite{bl}, \cite{tr}). In other words for $1\leq q\leq\infty,$ $\theta
\in(0,1),$ and let%
\[
\dot{B}_{X}^{\theta,q}(\mu)=\{f:\left\Vert f\right\Vert _{\dot{B}_{X}%
^{\theta,q}(\mu)}<\infty\},
\]%
\begin{equation}
B_{X}^{\theta,q}(\mu)=\{f:\left\Vert f\right\Vert _{B_{X}^{\theta,q}(\mu
)}=\left\Vert f\right\Vert _{\dot{B}_{X}^{\theta,q}(\mu)}+\left\Vert
f\right\Vert _{X}<\infty\}, \label{besgaus}%
\end{equation}
where%
\[
\left\Vert f\right\Vert _{\dot{B}_{X}^{\theta,q}(\mu)}=\left\{
\begin{array}
[c]{cc}%
\left(  \int_{0}^{1}\left(  K\left(  s,f;X(\mu),\dot{W}_{X}^{1}(\mu)\right)
s^{-\theta}\right)  ^{q}\frac{ds}{s}\right)  ^{1/q} & \text{if }q<\infty\\
\sup_{s}\left(  K\left(  s,f;X(\mu),\dot{W}_{X}^{1}(\mu)\right)  s^{-\theta
}\right)  & \text{if }q=\infty.
\end{array}
\right.
\]

The embeddings we prove in this section will follow from%
\begin{equation}
\left|  f\right|  _{\mu}^{\ast\ast}(t)-\left|  f\right|  _{\mu}^{\ast}(t)\leq
c\frac{K\left(  \frac{t}{I_{\mu}(t)},f;X(\mu),\dot{W}_{X}^{1}(\mu)\right)
}{\phi_{X}(t)},\text{ }0<t\leq1/2,\text{ \ }\left(  \text{ }f\in X+\dot{W}%
_{X}^{1}\right)  . \label{debase}%
\end{equation}

To simplify the presentation we shall state and prove our results only for the
Gaussian measures $\mu_{r,0,n},$ $r\in(1,2],$ which include the most important
examples: Gaussian measures and the so called interpolation measures between
exponential and Gaussian.

\begin{theorem}
\label{fustado}Let $1\leq q<\infty,$ $\theta\in(0,1),$ $r\in(1,2]$. Then there
exists an absolute constant $c=c(q,\theta,r)>0$ such that,%
\begin{equation}
\left\{  \int_{0}^{1/2}\left\vert f\right\vert _{\mu_{r,0,n}}^{\ast}%
(t)^{q}\left(  \log\frac{1}{t}\right)  ^{q\theta\left(  1-1/r\right)
}dt\right\}  ^{1/q}\leq c\left\Vert f\right\Vert _{B_{L^{q}}^{\theta,q}%
(\mu_{r,0,n})}. \label{vera}%
\end{equation}
Let $q=\infty,$ then there exists an absolute constant $c=c(\theta,r)>0$ such
that%
\begin{equation}
\sup_{t\in(0,\frac{1}{2})}\left(  \left\vert f\right\vert _{\mu_{r,0,n}}%
^{\ast\ast}(t)-\left\vert f\right\vert _{\mu_{r,0,n}}^{\ast}(t)\right)
\left(  \log\frac{1}{t}\right)  ^{(1-\frac{1}{r})\theta}\leq c\left\Vert
f\right\Vert _{\dot{B}_{L^{\infty}}^{\theta,\infty}(\mu_{r,0,n})}.
\label{vera1}%
\end{equation}

\end{theorem}

\begin{proof}
We shall let $\mu=:\mu_{r,0,n},$ $K\left(  s,f\right)  :=$ $K\left(
s,f;L^{q}(\mu),\dot{W}_{L^{q}}^{1}(\mu)\right)  .$ Suppose that $1\leq
q<\infty.$ We start by rewriting the term we want to estimate%
\begin{align*}
&  \left\{  \int_{0}^{1/2}\left|  f\right|  _{\mu}^{\ast}(t)^{q}\left(
\log\frac{1}{t}\right)  ^{q\theta\left(  1-1/r\right)  }dt\right\}  ^{1/q}\\
&  \preceq\left\{  \int_{0}^{1/2}\left|  f\right|  _{\mu}^{\ast}(t)^{q}\left(
\int_{t}^{1/2}\left(  \log\frac{1}{s}\right)  ^{q\theta\left(  1-1/r\right)
-1}\frac{ds}{s}+\left(  \log2\right)  ^{q\theta\left(  1-1/r\right)  }\right)
dt\right\}  ^{1/q}\\
&  \preceq\left\{  \int_{0}^{1/2}\left(  \log\frac{1}{s}\right)
^{q\theta\left(  1-1/r\right)  -1}\frac{1}{s}\int_{0}^{s}\left|  f\right|
_{\mu}^{\ast}(t)^{q}dtds\right\}  ^{1/q}+\left\{  \int_{0}^{1/2}\left|
f\right|  _{\mu}^{\ast}(t)^{q}dt\right\}  ^{1/q}\\
&  =(I)+(II)
\end{align*}
The term $(II)$ is under control since%
\[
(II)\leq\left\|  f\right\|  _{L^{q}}\leq\left\|  f\right\|  _{B_{L^{q}%
}^{\theta,q}(\mu)}.
\]
To estimate $(I)$ we first note that the elementary inequality\footnote{Which
follows readily by Jensen's inequality.}: $\left|  x\right|  ^{q}\leq
2^{q-1}(\left|  x-y\right|  ^{q}+\left|  y\right|  ^{q}),$ yields%
\[
\frac{1}{s}\int_{0}^{s}\left|  f\right|  _{\mu}^{\ast}(t)^{q}dt\preceq\frac
{1}{s}\int_{0}^{s}\left(  f_{\mu}^{\ast}(t)-f_{\mu}^{\ast}(s)\right)
^{q}dt+f_{\mu}^{\ast}(s)^{q}.
\]
Consequently,
\begin{align*}
(I)  &  \preceq\left\{  \int_{0}^{1/2}\left(  \log\frac{1}{s}\right)
^{q\theta\left(  1-1/r\right)  -1}\left(  \frac{1}{s}\int_{0}^{s}\left(
\left|  f\right|  _{\mu}^{\ast}(t)-\left|  f\right|  _{\mu}^{\ast}(s)\right)
^{q}dt\right)  ds\right\}  ^{1/q}\\
&  +\left\{  \int_{0}^{1/2}\left(  \log\frac{1}{s}\right)  ^{q\theta\left(
1-1/r\right)  -1}\left|  f\right|  _{\mu}^{\ast}(s)^{q}ds\right\}  ^{1/q}\\
&  =(I_{1})+(I_{2}),\text{ say.}%
\end{align*}
To control $(I_{1})$ we first use Example \ref{laizquierda} in Chapter
\ref{main} and (\ref{debase1}) to estimate the inner integral as follows%
\[
\frac{1}{s}\int_{0}^{s}\left(  \left|  f\right|  _{\mu}^{\ast}(t)-\left|
f\right|  _{\mu}^{\ast}(s)\right)  ^{q}dt\preceq\frac{1}{s}\left(  K\left(
\left(  \log\frac{1}{s}\right)  ^{1/r-1},f)\right)  \right)  ^{q},\text{
\ }0<s\leq1/2.
\]
Thus,%
\[
(I_{1})\preceq\left\{  \int_{0}^{1/2}\left(  \left(  \log\frac{1}{s}\right)
^{\theta\left(  1-1/r\right)  }\left(  K\left(  \left(  \log\frac{1}%
{s}\right)  ^{1/r-1},f)\right)  \right)  \right)  ^{q}\frac{ds}{s\log\frac
{1}{s}}\right\}  ^{1/q}%
\]
The change of variables $u=(\log\frac{1}{s})^{1/r-1}$ then yields%
\[
(I_{1})\preceq\left\|  f\right\|  _{B_{L^{q}}^{\theta,q}(\mu)}.
\]
It remains to estimate $(I_{2}).$ We write%
\[
(I_{2})\leq\left\{  \int_{0}^{1/2}\left(  \log\frac{1}{s}\right)
^{q\theta\left(  1-1/r\right)  -1}\left|  f\right|  _{\mu}^{\ast\ast}%
(s)^{q}ds\right\}  ^{1/q},
\]
then, using the fundamental theorem of calculus, we have%
\begin{align*}
(I_{2})  &  \leq\left\{  \int_{0}^{1/2}\left(  \log\frac{1}{s}\right)
^{q\theta\left(  1-1/r\right)  -1}\left(  \int_{s}^{1/2}\left(  \left|
f\right|  _{\mu}^{\ast\ast}(z)-\left|  f\right|  _{\mu}^{\ast}(z)\right)
\frac{dz}{z}+\left|  f\right|  _{\mu}^{\ast\ast}(1/2)\right)  ^{q}ds\right\}
^{1/q}\\
&  \leq\left\{  \int_{0}^{1/2}\left(  \left(  \log\frac{1}{s}\right)
^{\theta\left(  1-1/r\right)  -1/q}\int_{s}^{1/2}\left(  \left|  f\right|
_{\mu}^{\ast\ast}(z)-\left|  f\right|  _{\mu}^{\ast}(z)\right)  \frac{dz}%
{z}\right)  ^{q}ds\right\}  ^{1/q}\\
&  +\left|  f\right|  _{\mu}^{\ast\ast}(1/2)\left\{  \int_{0}^{1/2}\left(
\log\frac{1}{t}\right)  ^{q\theta\left(  1-1/r\right)  }dt\right\}  ^{1/q}\\
&  =(A)+(B),\text{ say.}%
\end{align*}

To use the Hardy logarithmic inequality of \cite[(6.7)]{br} we first write%
\[
(A)=\left\{  \int_{0}^{1/2}\left(  \left(  \log\frac{1}{s}\right)
^{\theta\left(  1-1/r\right)  -1/q}s^{1/q}\int_{s}^{1/2}\left(  \left|
f\right|  _{\mu}^{\ast\ast}(z)-\left|  f\right|  _{\mu}^{\ast}(z)\right)
\frac{dz}{z}\right)  ^{q}\frac{ds}{s}\right\}  ^{1/q}%
\]
and then find that%
\[
(A)\preceq\left\{  \int_{0}^{1/2}\left(  \left(  \left|  f\right|  _{\mu
}^{\ast\ast}(s)-\left|  f\right|  _{\mu}^{\ast}(s)\right)  s^{1/q}\left(
\log\frac{1}{s}\right)  ^{\theta\left(  1-1/r\right)  -1/q}\right)  ^{q}%
\frac{ds}{s}\right\}  ^{1/q}.
\]
Now we use the fact that in the region of integration $s^{1/q}\leq1,$ combined
with (\ref{debase}) and (\ref{debase1}), to conclude that%
\begin{align*}
&  \left\{  \int_{0}^{1/2}\left(  \left(  \left|  f\right|  _{\mu}^{\ast\ast
}(s)-\left|  f\right|  _{\mu}^{\ast}(s)\right)  s^{1/q}\left(  \log\frac{1}%
{s}\right)  ^{\theta\left(  1-1/r\right)  -1/q}\right)  ^{q}\frac{ds}%
{s}\right\}  ^{1/q}\\
&  \preceq\left\{  \int_{0}^{1/2}\left(  K\left(  \left(  \log\frac{1}%
{s}\right)  ^{\frac{1}{r}-1},f\right)  \right)  ^{q}\left(  \log\frac{1}%
{s}\right)  ^{q\theta\left(  1-1/r\right)  }\frac{ds}{s\left(  \log\frac{1}%
{s}\right)  }\right\}  ^{1/q}\\
&  \simeq\left\{  \int_{0}^{\left(  \log2\right)  ^{\frac{1}{r}-1}}\left(
K(u,f)\right)  ^{q}u^{-\theta q}\frac{du}{u}\right\}  ^{1/q}\text{ }\left(
\text{change of variables }u=\left(  \log\frac{1}{s}\right)  ^{\frac{1}{r}%
-1}\right) \\
&  \leq\left\|  f\right\|  _{B_{L^{q}}^{\theta,q}(\mu)}.
\end{align*}
Finally it remains to estimate $(B):$%
\begin{align*}
(B)  &  =\frac{1}{2}\left(  2\left|  f\right|  _{\mu}^{\ast\ast}(1/2)\right)
\left\{  \int_{0}^{1/2}t^{1/2}\left(  \log\frac{1}{t}\right)  ^{q\theta\left(
1-1/r\right)  -1}\frac{dt}{t^{1/2}}\right\}  ^{1/q}\\
&  \leq4\left\|  f\right\|  _{L^{1}}\left(  \sup_{t\in(0,1/2]}t^{1/2}\left(
\log\frac{1}{t}\right)  ^{q\theta\left(  1-1/r\right)  -1}\right)  \left\{
\int_{0}^{1/2}\frac{dt}{t^{1/2}}\right\} \\
&  \preceq\left\|  f\right\|  _{L^{1}}\leq\left\|  f\right\|  _{L^{q}}\\
&  \leq\left\|  f\right\|  _{B_{L^{q}}^{\theta,q}(\mu)}.
\end{align*}

We consider now the case $q=\infty.$ We apply (\ref{debase}), observing that
for $X=L^{\infty},$ we have $\phi_{L^{\infty}}(t)=1,$ and obtain that for
$t\in\left(  0,\frac{1}{2}\right]  ,$%
\begin{align*}
\left|  f\right|  _{\mu}^{\ast\ast}(t)-\left|  f\right|  _{\mu}^{\ast}(t)  &
\preceq K\left(  \left(  \log\frac{1}{t}\right)  ^{\frac{1}{r}-1},f\right) \\
&  =K\left(  \left(  \log\frac{1}{t}\right)  ^{\frac{1}{r}-1},f\right)
\left(  \log\frac{1}{t}\right)  ^{-(\frac{1}{r}-1)\theta}\left(  \log\frac
{1}{t}\right)  ^{(\frac{1}{r}-1)\theta}\\
&  \leq\left(  \log\frac{1}{t}\right)  ^{(\frac{1}{r}-1)\theta}\left(
\sup_{u}\left(  K(u,f)u^{-\theta}\right)  \right) \\
&  =\left(  \log\frac{1}{t}\right)  ^{(\frac{1}{r}-1)\theta}\left\|
f\right\|  _{\dot{B}_{L^{\infty}}^{\theta,\infty}(\mu_{r})},
\end{align*}
as we wished to show.
\end{proof}

\begin{remark}
Gaussian measure corresponds to $r=2,$ in this case, for $q=2,$ (\ref{vera})
yields the logarithmic Sobolev inequality%
\[
\left\{  \int_{0}^{1/2}\left\vert f\right\vert _{\gamma_{n}}^{\ast}%
(t)^{2}\left(  \log\frac{1}{t}\right)  ^{\theta}dt\right\}  ^{1/2}\leq
c\left\Vert f\right\Vert _{B_{L^{2}}^{\theta,2}(\gamma_{n})},\text{ }\theta
\in(0,1).
\]
Formally the case $\theta=1$ corresponds to an $L^{2}$ Logarithmic Sobolev
inequality, while the case $\theta=0,$ corresponds to the trivial
$L^{2}\subset L^{2}$ embedding. One could formally approach such inequalities
by complex interpolation (cf. \cite{bakrmey} as well as the calculations
provided in \cite{mey})%
\[
\lbrack L^{2},\dot{W}_{L^{2}}^{1}]_{\theta}\subset\lbrack L^{2},L^{2}%
LogL]_{\theta}=L^{2}(LogL)^{\theta}.
\]
The case $r=2,$ $q=1,$ corresponds to a fractional version of Ledoux's
inequality (cf. \cite{le}). Besides providing a unifying approach our method
can be applied to deal with more general domains and measures.
\end{remark}

\begin{remark}
When $q=\infty$ the inequality (\ref{vera1}) reflects a refined estimate of
the exponential integrability of $f.$ In particular, note that the case
$\theta=1,$ formally gives the following inequality (cf. \cite{bogo} and the
references therein)%
\[
\left\|  f\right\|  _{L^{[\infty,\infty]}}\simeq\sup_{t\in(0,\frac{1}{2}%
]}\left(  \left|  f\right|  _{\gamma_{n}}^{\ast\ast}(t)-\left|  f\right|
_{\gamma_{n}}^{\ast}(t)\right)  \leq c\left\|  f\right\|  _{\dot{W}_{e^{L^{2}%
}}^{1}(\gamma_{n})}%
\]
(cf. [(\ref{vermont}, Chapter \ref{chapbmo}] below for the definition of the
$L^{[p,q]}$ spaces ). The previous inequality can be proved readily using%
\[
\left|  f\right|  _{\gamma_{n}}^{\ast\ast}(t)-\left|  f\right|  _{\gamma_{n}%
}^{\ast}(t)\leq c\frac{1}{(\log\frac{1}{t})^{1/2}}\left|  \nabla f\right|
_{\gamma_{n}}^{\ast\ast}(t),\text{ }t\in(0,1/2].
\]

\end{remark}

\begin{remark}
Using the transference principle of \cite{mamiadv} the Gaussian results can be
applied to derive results related to the dimensionless Sobolev inequalities on
Euclidean cubes studied by Krbec-Schmeisser (cf. \cite{krb0}, \cite{krb}) and
Triebel \cite{tr1}.
\end{remark}

\section{Exponential Classes}

There is a natural connection between Gaussian measure and the exponential
class $e^{L^{2}}$. Likewise, this is also true with more general exponential
measures and other exponential spaces. Although there are many nice
inequalities associated with this topic that follow from our theory, we will
not develop the matter in great detail here. Instead, we shall only give a
flavor of possible results by considering Besov embeddings connected with the
Sobolev space $\dot{W}_{e^{L^{2}}}^{1}:=\dot{W}_{e^{L^{2}}}^{1}(\mathbb{R}%
^{n},\gamma_{n}).$

In this setting (\ref{debase}) takes the form%
\[
\left|  f\right|  _{\gamma_{n}}^{\ast\ast}(t)-\left|  f\right|  _{\gamma_{n}%
}^{\ast}(t)\leq c\frac{K\left(  \left(  \log\frac{1}{t}\right)  ^{-\frac{1}%
{2}},f;e^{L^{2}},\dot{W}_{e^{L^{2}}}^{1}\right)  }{\phi_{e^{L^{2}}}(t)}%
,\,t\in\left(  0,\frac{1}{2}\right]  .
\]
Now, since $\phi_{e^{L^{2}}}(t)=\left(  \log\frac{1}{t}\right)  ^{-\frac{1}%
{2}},$ $t\in(0,\frac{1}{2})$ we formally have%
\begin{align*}
\left(  \left|  f\right|  _{\gamma_{n}}^{\ast\ast}(t)-\left|  f\right|
_{\gamma_{n}}^{\ast}(t)\right)   &  \leq cK\left(  \left(  \log\frac{1}%
{t}\right)  ^{-\frac{1}{2}},f;e^{L^{2}},\dot{W}_{e^{L^{2}}}^{1}\right)
\left(  \log\frac{1}{t}\right)  ^{\frac{1}{2}}\\
&  \leq c\left\|  f\right\|  _{\dot{B}_{e^{L^{2}},\infty}^{1}(\gamma_{n})},
\end{align*}
or%
\begin{equation}
\left\|  f\right\|  _{L^{[\infty,\infty]}(\gamma_{n})}\leq c\left\|
f\right\|  _{\dot{B}_{e^{L^{2}},\infty}^{1}(\gamma_{n})}. \label{dependita}%
\end{equation}
More generally,%
\begin{align*}
\left(  \left|  f\right|  _{\gamma_{n}}^{\ast\ast}(t)-\left|  f\right|
_{\gamma_{n}}^{\ast}(t)\right)  \left(  \log\frac{1}{t}\right)  ^{-\frac{1}%
{2}+\frac{\theta}{2}}  &  \leq cK\left(  \left(  \log\frac{1}{t}\right)
^{-\frac{1}{2}},f;e^{L^{2}},\dot{W}_{e^{L^{2}}}^{1}\right)  \left(  \log
\frac{1}{t}\right)  ^{\frac{\theta}{2}}\\
&  \leq c\left\|  f\right\|  _{\dot{B}_{e^{L^{2}},\infty}^{\theta}},
\end{align*}
which shows directly the improvement on the exponential integrability in the
$\dot{B}_{e^{L^{2}},\infty}^{\theta}$ scale.

\chapter{On limiting Sobolev embeddings and $BMO$\label{chapbmo}}

\section{Introduction and Summary}

The discussion in this chapter is connected with the role of $BMO$ in some
limiting Sobolev inequalities. We start by reviewing some definitions, and
then proceed to describe Sobolev inequalities which follow readily from our
symmetrization inequalities, and will be relevant for our discussion.

Let $(\Omega,d,\mu)$ be a metric measure space satisfying the usual
assumptions, including the relative uniform isoperimetric property. The space
$BMO(\Omega)=BMO,$ introduced by John-Nirenberg, is the space of integrable
functions $f:\Omega\rightarrow\mathbb{R}$, such that%
\[
\left\Vert f\right\Vert _{BMO}=\sup_{B}\left\{  \inf_{c}\left(  \frac{1}%
{\mu(B)}\int_{B}\left\vert f-c\right\vert d\mu\right)  :B\text{ ball in
}\Omega\right\}  <\infty.
\]
In fact, it is enough to consider averages $f_{B}=\frac{1}{\mu(B)}\int%
_{B}fd\mu,$ or a median $m(f)$ of $f$ (cf. [Definition \ref{mediandef},
Chapter \ref{main}]),
\[
\left\Vert f\right\Vert _{BMO}\simeq\sup_{B}\left\{  \frac{1}{\mu(B)}\int%
_{B}\left\vert f-f_{B}\right\vert d\mu:B\text{ ball in }\Omega\right\}
<\infty.
\]
To obtain a norm we may set%
\[
\left\Vert f\right\Vert _{BMO_{\ast}}=\left\Vert f\right\Vert _{BMO}%
+\left\Vert f\right\Vert _{L^{1}}.
\]

\begin{remark}
One can also control $\left\Vert f\right\Vert _{\ast}$ through the use of
maximal operators (cf. \cite{fs}, \cite{coifweiss}, \cite{aalto}). Let%
\[
f^{\#}(x)=\sup_{B\backepsilon x}\frac{1}{\mu(B)}\int_{B}\left\vert
f-f_{B}\right\vert d\mu,
\]
where the sup is taken over all open balls containing $x.$ Then we have%
\[
\left\Vert f\right\Vert _{BMO}\simeq\left\Vert f^{\#}\right\Vert _{\infty}.
\]

\end{remark}

Let $\theta\in(0,1),1\leq p\leq\infty,1\leq q\leq\infty.$ Consider the Besov
spaces $\dot{b}_{p}^{\theta,q}(\Omega)$ (resp. $b_{p}^{\theta,q}(\Omega)),$
defined by%
\begin{align}
\left\Vert f\right\Vert _{\dot{b}_{p}^{\theta,q}(\Omega)}  &  =\left(
\int_{0}^{\mu(\Omega)}\left(  t^{-\theta}K(t,f;L^{p}(\Omega),S_{L^{p}}%
(\Omega)\right)  ^{q}\frac{dt}{t}\right)  ^{1/q}\label{besovmetrico}\\
\left\Vert f\right\Vert _{b_{p}^{\theta,q}(\Omega)}  &  =\left\Vert
f\right\Vert _{\dot{b}_{p,q}^{\theta}(\Omega)}+\left\Vert f\right\Vert
_{L^{p}}.\nonumber
\end{align}

For ready comparison with classical embedding theorems, from now on in this
section, unless explicitly stated to the contrary, we shall consider metric
measure spaces $(\Omega,d,\mu)$ such that the corresponding isoperimetric
profiles satisfy%
\begin{equation}
t^{1-1/n}\preceq I_{\Omega}(t),\text{ }t\in(0,\mu(\Omega)/2).
\label{lapropiedad}%
\end{equation}
We now recall the definition of the $L^{p,q}$ spaces. Moreover, in order to
incorporate in a meaningful way the limiting cases that correspond to the
index $p=\infty,$ we also recall the definition of the modified $L^{[p,q]}$
spaces\footnote{The $L^{p,q}$ and $L^{[p,q]}$ spaces are equivalent for
$p<\infty.$}. Let $1\leq p<\infty,1\leq q\leq\infty$ (cf. \cite{bds},
\cite{bmr})$,$ and let\footnote{with the usual modifications when $q=\infty.$}%
\begin{equation}
L^{p,q}(\Omega)=\left\{  f:\left\Vert f\right\Vert _{L^{p,q}}=\left(  \int%
_{0}^{\mu(\Omega)}\left(  \left\vert f\right\vert _{\mu}^{\ast}(s)s^{1/p}%
\right)  ^{q}\frac{ds}{s}\right)  ^{1/q}<\infty\right\}  . \label{montana}%
\end{equation}
For $1\leq p\leq\infty,1\leq q\leq\infty,$ we let
\begin{equation}
L^{[p,q]}(\Omega)=\left\{  f:\left\Vert f\right\Vert _{L^{[p,q]}}=\left(
\int_{0}^{\mu(\Omega)}\left(  (\left\vert f\right\vert _{\mu}^{\ast\ast
}(s)-\left\vert f\right\vert _{\mu}^{\ast}(s))s^{1/p}\right)  ^{q}\frac{ds}%
{s}\right)  ^{1/q}<\infty\right\}  . \label{vermont}%
\end{equation}
It is known that (cf. \cite{mamicon} and the references therein)
\[
L^{p,q}(\Omega)=L^{[p,q]}(\Omega),\text{ for }1\leq p<\infty,1\leq q\leq
\infty.
\]

Then, under our current assumptions on the isoperimetric profile of $\Omega,$
Theorem \ref{main1} states that%
\begin{equation}
\left\vert f\right\vert _{\mu}^{\ast\ast}(t)-\left\vert f\right\vert _{\mu
}^{\ast}(t)\leq c\frac{K(t^{1/n},f;L^{p}(\Omega),S_{L^{p}}(\Omega))}{t^{1/p}%
},\text{ }t\in(0,\mu(\Omega)/2). \label{basicaa}%
\end{equation}
The following basic version of the Sobolev embedding follows readily

\begin{proposition}%
\begin{equation}
b_{p}^{\theta,q}(\Omega)\subset L^{\bar{p},q}(\Omega),\text{ where }\frac
{1}{\bar{p}}=\frac{1}{p}-\frac{\theta}{n},\text{ }\theta\in(0,1),\text{ }1\leq
q\leq\infty,\text{ }\theta p\leq n. \label{edita}%
\end{equation}

\end{proposition}

\begin{proof}
Indeed, from the relationship between the indices and (\ref{basicaa}), we can
write
\[
\left(  \left\vert f\right\vert _{\mu}^{\ast\ast}(t)-\left\vert f\right\vert
_{\mu}^{\ast}(t)\right)  t^{1/\bar{p}}\preceq t^{-\frac{\theta}{n}}%
K(t^{1/n},f;L^{p}(\Omega),S_{L^{p}}(\Omega)),\text{ }t\in(0,\mu(\Omega)/2).
\]
If $q=\infty$, (\ref{edita}) follows taking supremum on both sides of the
inequality above. Likewise, if $q<\infty,$ then the desired result follows
raising both sides to the power $q$ and integrating from $0$ to $\mu
(\Omega)/2.$ In reference to the role of the $L^{[\infty,q]}$ spaces here let
us remark that, in the limiting case $\theta p=n,$ we have $\bar{p}=\infty.$
\end{proof}

We consider the limiting case, $\theta=\frac{n}{p},$ $p>n,$ in more detail. In
this case (\ref{edita}) reads (cf. \cite{mamiproc})%
\[
b_{p}^{n/p,q}(\Omega)\subset L^{[\infty,q]}(\Omega),\text{ }p>n,\text{ }1\leq
q\leq\infty.
\]
Note that when $q=1,$ $L^{[\infty,1]}(\Omega)=L^{\infty}(\Omega),$ and we
recover the well known result (for Euclidean domains),%
\begin{equation}
b_{p}^{n/p,1}(\Omega)\subset L^{\infty}(\Omega). \label{auto1}%
\end{equation}
On the other hand, when $q=\infty,$ from (\ref{edita}) we only get%
\begin{equation}
\dot{b}_{p}^{n/p,\infty}(\Omega)\subset L^{[\infty,\infty]}(\Omega).
\label{auto2}%
\end{equation}

In the Euclidean world better results are known. Recall that given a domain
$\Omega\subset\mathbb{R}^{n}$ the Besov spaces $\dot{B}_{p}^{\theta,q}%
(\Omega)$ (resp. $B_{p}^{\theta,q}(\Omega)),$ are defined by%
\begin{align}
\left\Vert f\right\Vert _{\dot{B}_{p}^{\theta,q}(\Omega)}  &  =\left(
\int_{0}^{\left\vert \Omega\right\vert }\left(  t^{-\theta}K(t,f;L^{p}%
(\Omega),\dot{W}_{L^{p}}^{1}(\Omega)\right)  ^{q}\frac{dt}{t}\right)
^{1/q}\label{besoveuclideo}\\
\left\Vert f\right\Vert _{B_{p}^{\theta,q}(\Omega)}  &  =\left\Vert
f\right\Vert _{\dot{B}_{p}^{\theta,q}(\Omega}+\left\Vert f\right\Vert _{L^{p}%
}.\nonumber
\end{align}
Indeed, for smooth domains, we have a better result than (\ref{auto1}),
namely
\begin{equation}
B_{p}^{n/p,1}(\Omega)\subset C(\Omega), \label{auto3}%
\end{equation}
and, moreover, it is well known that (cf. \cite{bourdon})%
\begin{equation}
\dot{B}_{p}^{n/p,\infty}([\Omega)\subset BMO(\Omega). \label{auto4}%
\end{equation}
We note that since we have\footnote{This is an easy consequence of
(\ref{bendevsha}) below.} $BMO([0,1]^{n})\subset L^{\left[  \infty
,\infty\right]  }:$ i.e.
\[
\sup_{t}\left(  \left\vert f\right\vert ^{\ast\ast}(t)-\left\vert f\right\vert
^{\ast}(t)\right)  \leq C\left\Vert f\right\Vert _{BMO},
\]
then (\ref{auto4}) is stronger than (\ref{auto2}).

In Chapter \ref{contchap} we have shown that for Sobolev and Besov spaces that
are based on metric probability spaces with the relative uniform isoperimetric
property, the rearrangement inequality (\ref{auto1}) self-improves to
(\ref{auto3}). Let $X$ be a r.i. space on $\Omega,$ we will show that the
$K-$Poincar\'{e} inequality (cf. [Theorem \ref{teobmo}, Chapter \ref{main}])
\begin{equation}
\frac{1}{\mu(\Omega)}\int_{\Omega}\left|  f-f_{\Omega}\right|  d\mu\leq
c\frac{K\left(  \frac{\mu(\Omega)/2}{I_{\Omega}(\mu(\Omega)/2)},f;X,S_{X}%
\right)  }{\phi_{X}(\mu(\Omega))},f\in X+S_{X}, \label{delotro}%
\end{equation}
combined with the relative uniform isoperimetric property self improves to
(\ref{auto4}). In fact, the self improve result reads%
\begin{equation}
\left\|  f\right\|  _{BMO(\Omega)}\leq C\sup_{0<t<\mu(\Omega)}\frac{K\left(
\frac{\mu(\Omega)/2}{I_{\Omega}(\mu(\Omega)/2)},f;X,S_{X}\right)  }{\phi
_{X}(\mu(\Omega))}, \label{muchomas}%
\end{equation}
and is valid for our the general class of isoperimetric profiles considered in
this paper. Indeed, the result exhibits a new connection between the geometry
of the ambient space and the embedding of Besov and BMO spaces. For example,
for an Ahlfors $k-$regular space $\left(  \Omega,d,\mu\right)  $ (cf. [Section
\ref{ahlfor}, Chapter \ref{capitap}]) given a ball $B,$ consider the metric
space $\left(  B,d_{\mid B},\mu_{\mid B}\right)  $ then%
\begin{equation}
\left\|  f\right\|  _{BMO(B)}\leq c\left\|  f\right\|  _{\dot{b}%
_{p}^{k/p,\infty}(B)},\text{ }p>k. \label{dependitalpha}%
\end{equation}
We shall also discuss a connection between our development in this paper and a
characterization of $BMO$ provided by John \cite{jhn} and Stromberg
\cite{stro}.

Finally, re-interpreting $BMO$ as a limiting Lip space we were lead to an
analog of [Chapter \ref{main}, (\ref{polaka1})] which we now describe. We
argue that in $\mathbb{R}^{n}$ the natural replacement of [(\ref{cuatro}),
Chapter \ref{intro}] involving the space $BMO$ is given by the
Bennett-DeVore-Sharpley inequality (cf. \cite{bds}, \cite{bs}, \cite{aalto},
\cite{aalto1})%
\begin{equation}
\left\vert f\right\vert ^{\ast\ast}(t)-\left\vert f\right\vert ^{\ast}(t)\leq
c(f^{\#})^{\ast}(t),\text{ }0<t<\frac{\left\vert B\right\vert }{6},\text{
where }B\text{ is a ball on }\mathbb{R}^{n}. \label{bendevsha}%
\end{equation}
Variants of this inequality are known to hold in more general contexts. For
our purposes here the following inequality will suffice%
\begin{equation}
\left\vert f\right\vert _{\mu}^{\ast\ast}(t)-\left\vert f\right\vert _{\mu
}^{\ast}(t)\leq C\left\Vert f\right\Vert _{BMO},\text{ }0<t<\mu(\Omega).
\label{weakinfty}%
\end{equation}
We shall therefore assume for this particular discussion that our metric
measure space $(\Omega,d,\mu)$ also satisfies the following condition: There
exists a constant $C>0$ such that (\ref{weakinfty}) holds for all $f\in BMO.$
For example, in \cite[see (3.8)]{saghschv} it is shown that (\ref{weakinfty})
holds for doubling measures on\ Euclidean domains. More general results can be
found in \cite{aalto}.

Assuming the validity of (\ref{weakinfty}), and using the method of the proof
of Theorem \ref{main1}, we will show below (cf. Theorem \ref{teoremarkao})
that if $X(\Omega)$ is a r.i. space, then we have\footnote{On $\mathbb{R}^{n}$
(\ref{bendevsha1}) is known and can be obtained by combining (\ref{bendevsha})
with \cite[theorem 8.8]{bs}.}%
\begin{equation}
\left\vert f\right\vert ^{\ast\ast}(t)-\left\vert f\right\vert ^{\ast}(t)\leq
c\frac{K(\phi_{X}(t),f;X(\Omega),BMO(\Omega))}{\phi_{X}(t)},\text{ }%
0<t<\mu(\Omega)/2. \label{bendevsha1}%
\end{equation}
This result should be compared with Theorem \ref{main1} above. For
perspective, we now show a different road to a special case of
(\ref{bendevsha1}). Recall that for Euclidean domains it is shown in
\cite[(8.11)]{bs} that%
\[
\frac{K(t,f;L^{1},BMO)}{t}\simeq(f^{\#})^{\ast}(t).
\]
Combining this inequality with (\ref{bendevsha}), we obtain a different
approach to (\ref{bendevsha1}) in the special case $X=L^{1},$ at least when
$t$ is close to zero.

\subsection{Self Improving inequalities and $BMO$}

We show that (\ref{delotro}) combined with the relative uniform isoperimetric
property yields the following embedding

\begin{theorem}
\label{bmomarkao}Let $\left(  \Omega,d,\mu\right)  $ be a metric space
satisfying the standard assumptions and with the relative uniform
isoperimetric property\textbf{. }Let $X$ be a r.i. space on $\Omega$, then,
there exists an absolute constant $C>0$ such that,
\[
\left\Vert f\right\Vert _{BMO(\Omega)}\leq C\sup_{0<t<\mu(\Omega)}%
\frac{K\left(  \frac{t/2}{I_{\Omega}(t/2)},f;X,S_{X}\right)  }{\phi_{X}(t)}.
\]

\end{theorem}

\begin{proof}
Given an integrable function $f$ and a ball $B\ $in $\Omega,$ consider
$f\chi_{B}.$ By Theorem \ref{teobmo}, applied to the metric space $(B,d_{\mid
B},\mu_{\mid B})$, we have%
\begin{align*}
\frac{1}{\mu(B)}\int_{B}\left\vert f-f_{B}\right\vert d\mu &  \leq
c\frac{K\left(  \frac{\mu(B)/2}{I_{B}(\mu(B)/2)},f\chi_{B};X_{r}(B),S_{X_{r}%
}(B)\right)  }{\phi_{X_{r}}(\mu(B))}.\\
&  .
\end{align*}
Since $\left(  \Omega,d,\mu\right)  $ has the relative uniform isoperimetric
property, we have
\begin{align*}
\frac{K\left(  \frac{\mu(B)/2}{I_{B}(\mu(B)/2)},f\chi_{B};X_{r}(B),S_{X_{r}%
}(B)\right)  }{\phi_{X_{r}}(\mu(B))}  &  \leq C\frac{K\left(  \frac{\mu
(B)/2}{I_{\Omega}(\mu(B)/2)},f\chi_{B};X,S_{X}\right)  }{\phi_{X}(\mu(B))}\\
&  \leq C\sup_{0<t<\mu(\Omega)}\frac{K\left(  \frac{t/2}{I_{\Omega}%
(t/2)},f;X,S_{X}\right)  }{\phi_{X}(t)}.
\end{align*}
Consequently,%
\[
\sup_{B}\frac{1}{\mu(B)}\int_{B}\left\vert f-f_{B}\right\vert d\mu\leq
C\sup_{0<t<\mu(\Omega)}\frac{K\left(  \frac{t/2}{I_{\Omega}(t/2)}%
,f;X,S_{X}\right)  }{\phi_{X}(t)}.
\]

\end{proof}

We now give a concrete application of the previous result.

\begin{corollary}
\label{corolariomarkao}Let $\Omega\subset\mathbb{R}^{n}$ be a bounded domain
that belongs to the Maz'ya's class $\mathcal{J}_{1-1/n}$ (cf. [Section
\ref{Mazclass}, Chapter \ref{capitap}])$.$ Suppose that $p>n,$ then
\[
\dot{B}_{p}^{n/p,\infty}(\Omega)\subset BMO(\Omega).
\]

\end{corollary}

\begin{proof}
Since $\Omega$ belongs to the Maz'ya's class $\mathcal{J}_{1-1/n}$ the
following isoperimetric estimate holds.
\[
t^{1-1/n}\preceq I_{\Omega}(t),\text{ }t\in(0,\left\vert \Omega\right\vert
/2).
\]
On the other hand, for any ball $B$ in $\Omega,$ we have
\[
I_{B}(t)\geq c(n)\min(s,(\left\vert B\right\vert -s))^{\frac{n-1}{n}},\text{
\ \ \ }0<s<\left\vert B\right\vert .
\]
Since $\Omega\subset\mathbb{R}^{n}$, using the same argument we provided in
Section \ref{class}, it follows readily that the inequality (\ref{poinbesov})
remains valid for all functions in $f\in X+\dot{W}_{X}^{1}$, i.e.%
\begin{equation}
\frac{1}{\mu(\Omega)}\int_{\Omega}\left\vert f-f_{\Omega}\right\vert d\mu\leq
c\frac{K\left(  \frac{\mu(\Omega)/2}{I_{\Omega}(\mu(\Omega)/2)},f;X,\dot
{W}_{X}^{1}\right)  }{\phi_{X}(\mu(\Omega))}. \label{besovbmo}%
\end{equation}
Thus, by the argument given in the previous Theorem, we see that%
\begin{align*}
\frac{1}{\mu(B)}\int_{B}\left\vert f-f_{B}\right\vert d\mu &  \leq
C(n)\sup_{0<t<\left(  \left\vert \Omega\right\vert /2\right)  ^{1/n}}%
\frac{K\left(  \left(  t/2\right)  ^{1/n},f;L^{p},\dot{W}_{L^{p}}^{1}\right)
}{t^{1/p}}\\
&  \leq C(n)\sup_{0<t<\left\vert \Omega\right\vert }t^{-n/p}K\left(
t,f;L^{p},\dot{W}_{L^{p}}^{1}\right) \\
&  =C(n)\left\Vert f\right\Vert _{\dot{B}_{p}^{n/p,\infty}(\Omega)}.
\end{align*}

\end{proof}

\section{On the John-Stromberg characterization of $BMO$ \label{john}}

Our discussion in this chapter is closely connected with a characterization of
$BMO([0,1]^{n})$ using rearrangements due to John \cite{jhn} and Stromberg
\cite{stro}. Let $\lambda\in(0,\frac{1}{2}],$ then
\[
\left\|  f\right\|  _{BMO_{\ast}}\simeq\sup_{Q\subset\lbrack0,1]^{n}}%
\inf_{c\in\mathbb{R}}\left(  (f-c)\chi_{Q}\right)  ^{\ast}(\lambda\left|
Q\right|  ).
\]
See also Jawerth-Torchinsky \cite{jat}, Lerner \cite{ler}, \cite{css}, and the
references therein.

\begin{theorem}
\label{median}Let $(\Omega,d,\mu)$ be a measure metric space satisfying our
standard assumptions, let $f:\Omega\rightarrow\mathbb{R}$ be an integrable
function. For a measurable set $Q\subset\Omega,$ $\left(  f\chi_{Q}\right)
_{\mu}^{\ast}(\mu(Q)/2)\ $is a median of $f$ on $Q.$
\end{theorem}

\begin{proof}
It is easy to convince oneself that Definition \ref{mediandef} of median in is
equivalent to%
\[
\mu\{f>m(f)\}\leq\mu(Q);\text{ and }\mu\{f<m(f)\}\leq\mu(Q).
\]
Now%
\[
\mu\{f\chi_{Q}<(f\chi_{Q})_{\mu}^{\ast}(\mu(Q)/2)\}=\mu\{-f\chi_{Q}%
>-(f\chi_{Q})_{\mu}^{\ast}(\mu(Q)/2)\}
\]
But since%
\[
\left(  -f\chi_{Q}\right)  _{\mu}^{\ast}(t)=-(f\chi_{Q})_{\mu}^{\ast}%
(\mu(Q)-t)
\]
it follows that%
\[
\left(  -f\chi_{Q}\right)  _{\mu}^{\ast}(\mu(Q)/2)=-(f\chi_{Q})_{\mu}^{\ast
}(\mu(Q)/2).
\]
Consequently%
\begin{align*}
\mu\{f\chi_{Q}  &  <(f\chi_{Q})_{\mu}^{\ast}(\mu(Q)/2)\}=\mu\{-f\chi
_{Q}>-(f\chi_{Q})_{\mu}^{\ast}(\mu(Q)/2)\}\\
&  =\mu\{-f\chi_{Q}>\left(  -f\chi_{Q}\right)  _{\mu}^{\ast}(\mu(Q)/2)\}\\
&  \leq\mu(Q)/2\text{ (by definition).}%
\end{align*}
Therefore $(f\chi_{Q})_{\mu}^{\ast}(\mu(Q)/2)$ is a median as we wished to show.
\end{proof}

As a consequence we have the following John-Stromberg inequality: for any ball
$B,$
\begin{equation}
\left(  (f\chi_{B})_{\mu}^{\ast\ast}(\mu(B)/2)-(f\chi_{B})_{\mu}^{\ast}%
(\mu(B)/2)\right)  \leq\frac{1}{2}\left\|  f\right\|  _{BMO}. \label{mediana}%
\end{equation}

\begin{theorem}
Let $\left(  \Omega,d,\mu\right)  $ be a metric space satisfying the standard
assumptions. Then there exists a constant $C>0$ such that for all $f,$ it
holds
\[
\left\Vert f\right\Vert _{BMO(\Omega)}\leq C\sup_{B\text{ ball in }%
\Omega\text{\ }}\left\{  (f\chi_{B})_{\mu}^{\ast\ast}(\mu(B))-(f\chi_{B}%
)_{\mu}^{\ast}(\mu(B))\right\}  .
\]

\end{theorem}

\begin{proof}
For $t\in(0,\mu(\Omega))$ let us write%
\begin{align*}
t(f_{\mu}^{\ast\ast}(t)-f_{\mu}^{\ast}(t))  &  =\int_{0}^{t}(f_{\mu}^{\ast
}(x)-f_{\mu}^{\ast}(t))d\mu\\
&  =\int_{0}^{\mu(\Omega)}\max\left(  0,f_{\mu}^{\ast}(x)-f_{\mu}^{\ast
}(t)\right)  d\mu\\
&  =\int_{\{s:f(s)>f_{\mu}^{\ast}(t)\}}\max\left(  0,f(x)-f_{\mu}^{\ast
}(t)\right)  d\mu.
\end{align*}
Fix a ball $B$ and apply the preceding equality to $f\chi_{B}$ and $t=\mu(B):$%
\begin{align*}
&  \mu(B)\left(  (f\chi_{B})_{\mu}^{\ast\ast}(\mu(B))-(f\chi_{B})_{\mu}^{\ast
}(\mu(B))\right) \\
&  =\int_{\{s\in B:f(s)>(f\chi_{B})_{\mu}^{\ast}(\mu(B))\}}\max(0,f\chi
_{B}(s)-(f\chi_{B})_{\mu}^{\ast}(\mu(B)))d\mu.
\end{align*}
To estimate the right hand side from below we observe that%
\[
f_{B}:=\frac{1}{\mu(B)}\int_{B}f(x)d\mu=\frac{1}{\mu(B)}\int_{0}^{\mu
(B)}(f\chi_{B})_{\mu}^{\ast}(s)ds\geq(f\chi_{B})_{\mu}^{\ast}(\mu(B)),
\]
therefore
\[
f(s)-(f\chi_{Q})_{\mu}^{\ast}(\mu(B))\geq f(s)-f_{B}.
\]
Consequently,%
\begin{align*}
&  \int_{\{s\in B:f(s)>(f\chi_{B})_{\mu}^{\ast}(\mu(B))\}}\max(0,f\chi
_{B}(s)-(f\chi_{B})_{\mu}^{\ast}(\mu(B)))d\mu.\\
&  \geq\int_{\left\{  s\in B:\text{ }f(s)>f_{B}\right\}  }\max(0,f\chi
_{B}(s)-(f\chi_{B})_{\mu}^{\ast}(\mu(B)))d\mu\\
&  \geq\int_{\left\{  x\in B:\text{ }f(s)>f_{B}\right\}  }\left[
f(s)-f_{B}\right]  d\mu.
\end{align*}
We will verify in a moment that%
\begin{align}
&  \frac{1}{\mu(B)}\int_{\left\{  x\in B:\text{ }f(s)>f_{B}\right\}  }\left[
f(s)-f_{B}\right]  d\mu\label{lamitad}\\
&  =\frac{1}{2}\frac{1}{\mu(B)}\int_{B}\left|  f(s)-f_{B}\right|
d\mu\nonumber
\end{align}
Combining (\ref{lamitad}) with the previous estimates we see that%
\[
\left(  (f\chi_{B})_{\mu}^{\ast\ast}(\mu(B))-(f\chi_{B})_{\mu}^{\ast}%
(\mu(B))\right)  \geq\frac{1}{2}\frac{1}{\mu(B)}\int_{B}\left|  f(s)-f_{B}%
\right|  d\mu.
\]
Hence%
\begin{align*}
\left\|  f\right\|  _{BMO}  &  =\sup_{B}\frac{1}{\mu(B)}\int_{B}\left|
f(s)-f_{B}\right|  d\mu\\
&  \leq2\sup_{B}\left(  (f\chi_{B})_{\mu}^{\ast\ast}(\mu(B))-(f\chi_{B})_{\mu
}^{\ast}(\mu(B))\right)  ,
\end{align*}
as we wished to show.

It remains to see (\ref{lamitad}). Since%
\[
\int_{\left\{  x\in B:\text{ }f(s)>f_{B}\right\}  }\left[  f(s)-f_{B}\right]
d\mu+\int_{\left\{  x\in B:\text{ }f(s)<f_{B}\right\}  }\left[  f(s)-f_{B}%
\right]  d\mu=0,
\]
we have that%
\[
\int_{\left\{  x\in B:\text{ }f(s)>f_{B}\right\}  }\left[  f(s)-f_{B}\right]
d\mu=\int_{\left\{  x\in B:\text{ }f(s)<f_{B}\right\}  }\left[  f_{B}%
-f(s)\right]  d\mu.
\]
Consequently,%
\begin{align*}
\int_{B}\left\vert f(s)-f_{B}\right\vert d\mu &  =\int_{\left\{  x\in B:\text{
}f(s)>f_{B}\right\}  }\left[  f(s)-f_{B}\right]  d\mu+\int_{\left\{  x\in
B:\text{ }f(s)<f_{B}\right\}  }\left[  f_{B}-f(s)\right]  d\mu\\
&  =2\int_{\left\{  x\in B:\text{ }f(s)>f_{B}\right\}  }\left[  f(s)-f_{B}%
\right]  d\mu.
\end{align*}

\end{proof}

\section{Oscillation, $BMO$ and $K-$functionals}

As is well known in the Euclidean world or even for fairly general metric
spaces (cf. \cite{coifweiss}) one can realize $BMO$ as a limiting $Lip$ space.
The easiest way to see this is through the equivalence%
\[
\left\|  f\right\|  _{Lip_{\alpha}}\simeq\sup_{Q}\frac{1}{\left|  Q\right|
^{1-\alpha/n}}\int_{Q}\left|  f-f_{Q}\right|  dx<\infty.
\]
From this point of view $BMO$ corresponds to a Lip space of order $\alpha=0.$

This observation leads naturally to consider the analogs of the results of
Chapter \ref{contchap} in the context of $BMO.$

\begin{theorem}
\label{teoremarkao}Suppose that $(\Omega,d,\mu)$ is a metric space with finite
measure and such that there exists an absolute constant $C>0$ such that for
all $f\in L^{1}(\Omega)$ we have%
\begin{equation}
\left\vert f\right\vert _{\mu}^{\ast\ast}(t)-\left\vert f\right\vert _{\mu
}^{\ast}(t)\leq C\left\Vert f\right\Vert _{BMO},\text{ }0<t<\mu(\Omega).
\label{sagsh}%
\end{equation}
Then, for every r.i. space $X(\Omega)$ there exists a constant $c>0$ such that%
\begin{equation}
\left\vert f\right\vert _{\mu}^{\ast\ast}(t/2)-\left\vert f\right\vert _{\mu
}^{\ast}(t/2)\leq c\frac{K(\phi_{X}(t),f;X(\Omega),BMO(\Omega))}{\phi_{X}%
(t)},\text{ }0<t<\mu(\Omega), \label{bmoada}%
\end{equation}
where%
\[
K(t,f;X(\Omega),BMO(\Omega))=\inf_{h\in BMO}\{\left\Vert f-h\right\Vert
_{X}+t\left\Vert h\right\Vert _{BMO}\}.
\]

\end{theorem}

\begin{proof}
The proof follows exactly the same lines as the proof of Theorem \ref{main1},
so we shall be brief. We start by noting three important properties that
functional $\left\Vert f\right\Vert _{BMO}$ shares with $\left\Vert \left\vert
\nabla f\right\vert \right\Vert :$ $(i)$ For any constant $c,$ $\left\Vert
f+c\right\Vert _{BMO}=\left\Vert f\right\Vert _{BMO}$, $(ii)$ $\left\Vert
\left\vert f\right\vert \right\Vert _{BMO}\leq\left\Vert f\right\Vert _{BMO},$
and more generally $(iii)$ for any Lip $1$ function $\Psi,$ $\left\Vert
\Psi(f)\right\Vert _{BMO}\leq\left\Vert f\right\Vert _{BMO}.$ Let $t>0,$ then
using the corresponding arguments in Theorem \ref{main1} shows that we have
\begin{equation}
\inf_{0\leq h\in BMO}\left\{  \left\Vert \left\vert f\right\vert -h\right\Vert
_{X}+\phi_{X}(t)\left\Vert h\right\Vert _{BMO}\right\}  \leq K(\phi
_{X}(t),f;X(\Omega),BMO(\Omega)), \label{primbmo}%
\end{equation}

To prove (\ref{bmoada}) we proceed as in the proof of Theorem \ref{main1}
until we arrive to%
\[
\left\vert f\right\vert _{\mu}^{\ast\ast}(t/2)-\left\vert f\right\vert _{\mu
}^{\ast}(t/2)\leq\left\vert \left\vert f\right\vert -h\right\vert _{\mu}%
^{\ast\ast}(t)+\left\vert \left\vert f\right\vert -h\right\vert _{\mu}^{\ast
}(t)+\left\vert h\right\vert _{\mu}^{\ast\ast}(t)-\left\vert h\right\vert
_{\mu}^{\ast}(2t)
\]
which we now estimate as%
\[
\left\vert f\right\vert _{\mu}^{\ast\ast}(t/2)-\left\vert f\right\vert _{\mu
}^{\ast}(t/2)\leq2\left\vert \left\vert f\right\vert -h\right\vert _{\mu
}^{\ast\ast}(t)+\left(  \left\vert h\right\vert _{\mu}^{\ast\ast
}(t)-\left\vert h\right\vert _{\mu}^{\ast}(t)\right)  +\left(  \left\vert
h\right\vert _{\mu}^{\ast}(t)-\left\vert h\right\vert _{\mu}^{\ast
}(2t)\right)
\]
Note that%
\begin{align*}
\left(  \left\vert f\right\vert -h\right)  _{\mu}^{\ast\ast}(t/2)  &
=\frac{2}{t}\int_{0}^{t}\left(  \left\vert f\right\vert -h\right)  _{\mu
}^{\ast}(s)ds\\
&  \leq2\frac{\left\Vert \left\vert f\right\vert -h\right\Vert _{X}%
\phi_{X^{\prime}}(t)}{t}\text{ (H\"{o}lder's inequality)}\\
&  =2\frac{\left\Vert \left\vert f\right\vert -h\right\Vert _{X}}{\phi_{X}%
(t)}\text{ (since }\phi_{X^{\prime}}(t)\phi_{X}(t)=t).
\end{align*}
On the other hand, by (\ref{sagsh})%
\[
\left\vert h\right\vert _{\mu}^{\ast\ast}(t)-\left\vert h\right\vert _{\mu
}^{\ast}(t)\leq C\left\Vert h\right\Vert _{BMO}.
\]
While by [(\ref{davobonee}), Chapter \ref{intro}]%
\begin{align*}
\left\vert h\right\vert _{\mu}^{\ast}(t)-\left\vert h\right\vert _{\mu}^{\ast
}(2t)  &  \leq2\left(  \left\vert h\right\vert _{\mu}^{\ast\ast}%
(2t)-\left\vert h\right\vert _{\mu}^{\ast}(2t)\right) \\
&  \leq C\left\Vert h\right\Vert _{BMO}.\text{ (again by (\ref{sagsh}))}%
\end{align*}
Therefore, combining our findings we see that%
\begin{align*}
\left\vert f\right\vert _{\mu}^{\ast\ast}(t/2)-\left\vert f\right\vert _{\mu
}^{\ast}(t/2)  &  \leq C\inf_{0\leq h\in BMO}\left\{  \frac{\left\Vert
\left\vert f\right\vert -h\right\Vert _{X}}{\phi_{X}(t)}+\left\Vert
h\right\Vert _{BMO}\right\} \\
&  =\frac{C}{\phi_{X}(t)}\inf_{0\leq h\in BMO}\{\left\Vert \left\vert
f\right\vert -h\right\Vert _{X}+\phi_{X}(t)\left\Vert h\right\Vert _{BMO}\}\\
&  \leq\frac{C}{\phi_{X}(t)}K(\phi_{X}(t),f;X(\Omega),BMO(\Omega))\text{ (by
(\ref{primbmo}).}%
\end{align*}

\end{proof}

\chapter{Estimation of growth \textquotedblleft envelopes\textquotedblright%
\label{triebel}}

\section{Summary}

Triebel and his school, in particular we should mention here the extensive
work of Haroske, have studied the concept of *envelopes* (cf. \cite{haroske},
a book mainly devoted to the computation of growth and continuity envelopes of
function spaces defined on $\mathbb{R}^{n}$). On the other hand, as far as we
are aware, the problem of estimating growth envelopes for Sobolev or Besov
spaces based on general measure spaces has not been treated systematically in
the literature. For a function space $Z(\Omega),$ which we should think as
measuring smoothness, one attempts to find precise estimates of (*the growth
envelope*)
\[
E^{Z}(t)=\sup_{\left\|  f\right\|  _{Z(\Omega)}\leq1}\left|  f\right|  ^{\ast
}(t).
\]
A related problem is the estimation of *continuity envelopes* (cf.
\cite{haroske}). For example, suppose that $Z:=Z(\mathbb{R}^{n})\subset
C(\mathbb{R}^{n}),$ then we let%
\[
E_{C}^{Z}(t)=\sup_{\left\|  f\right\|  _{Z}\leq1}\frac{\omega_{L^{\infty}%
}(t,f)}{t},
\]
and the problem at hand is to obtain precise estimates of $E_{C}^{Z}(t).$

In this chapter we estimate growth envelopes of function spaces based on
metric probability spaces using our symmetrization inequalities. Most of the
results we shall obtain, including those for Gaussian function spaces, are
apparently new. Our method, moreover, gives a unified approach.

In a somewhat unrelated earlier work \cite{mamiconver}, we proposed some
abstract ideas on how to study certain convergence and compactness properties
in the context of interpolation scales. In Section \ref{sec:markao} we shall
briefly show a connection with the estimation of envelopes.

\section{Spaces defined on measure spaces with Euclidean type profile}

To fix ideas, and for easier comparison, in this section we consider metric
probability\footnote{Note that, when we are dealing with domains $\Omega$ with
finite measure, we can usually assume without loss that we are dealing with
functions such that $|f|^{\ast\ast}(\infty)=0.$ Indeed, we have%
\[
\left\vert f\right\vert ^{\ast\ast}(t)\leq\frac{1}{t}\left\Vert f\right\Vert
_{L^{1}(\Omega)}.
\]
For the usual function spaces on $\mathbb{R}^{n},$ we can usually work with
functions in $C_{0}(\mathbb{R}^{n}),$ which again obviously satisfy
$|f|^{\ast\ast}(\infty)=0.$} spaces $(\Omega,d,\mu)$ satisfying the standard
assumptions and such that the corresponding profiles satisfy%
\begin{equation}
t^{1-1/n}\preceq I_{\Omega}(t),\text{ }t\in(0,1/2). \label{cota}%
\end{equation}
Particular examples are the $\mathcal{J}_{1-\frac{1}{n}}-$Maz'ya domains on
$\mathbb{R}^{n}.$ By the L\'{e}vy-Gromov isoperimetric inequality, Riemannian
manifolds with positive Ricci curvature also satisfy (\ref{cota}).

In this context the basic rearrangement inequalities (cf. (\ref{cuatro}),
(\ref{desiK})) take the following form, if $f\in Lip(\Omega),$ then
\begin{equation}
\left|  f\right|  _{\mu}^{\ast\ast}(t)-\left|  f\right|  _{\mu}^{\ast
}(t)\preceq t^{1/n}\left|  \nabla f\right|  _{\mu}^{\ast\ast}(t),\text{ }%
t\in(0,1/2),\text{ } \label{lareemplazada}%
\end{equation}
and, if $f\in X+S_{X},$ then
\begin{equation}
\left|  f\right|  _{\mu}^{\ast\ast}(t)-\left|  f\right|  _{\mu}^{\ast
}(t)\preceq\frac{K(t^{1/n},f,X,S_{X})}{\phi_{X}(t)},\text{ }t\in(0,1/2).\text{
} \label{lanueva}%
\end{equation}

\begin{theorem}
Let $X=X(\Omega)$ be a r.i. space on $\Omega,$ and let $\bar{S}_{X}(\Omega)$
be defined by
\[
\bar{S}_{X}(\Omega)=\left\{  f\in Lip(\Omega):\left\|  f\right\|  _{\bar
{S}_{X}(\Omega)}=\left\|  \left|  \nabla f\right|  \right\|  _{X}+\left\|
f\right\|  _{X}<\infty\right\}  .
\]

Then,%
\begin{equation}
E^{\bar{S}_{X}(\Omega)}(t)\preceq\int_{t}^{1}s^{1/n-1}\phi_{\bar{X}^{\prime}%
}(s)\frac{ds}{s},\text{ }t\in(0,1/2). \label{eresfea}%
\end{equation}
(ii) In particular, if $X=L^{p},1\leq p<n,$ then (compare with \cite{haroske}
and see also Remark \ref{trece} below)%
\begin{equation}
E^{\bar{S}_{L^{p}}(\Omega)}(t)\preceq t^{1/n-1/p},\text{ }t\in(0,1/2).
\label{eresmuyfea}%
\end{equation}

\end{theorem}

\begin{proof}
Let $f$ $\ $be such that $\left\|  f\right\|  _{\bar{S}_{X}(\Omega)}\leq1.$
Using the fundamental theorem of Calculus we can write
\begin{equation}
\left|  f\right|  _{\mu}^{\ast\ast}(t)-\left|  f\right|  _{\mu}^{\ast\ast
}(1/2)=\int_{t}^{1/2}\left(  \left|  f\right|  _{\mu}^{\ast\ast}(s)-\left|
f\right|  _{\mu}^{\ast}(s)\right)  \frac{ds}{s}. \label{porfavor}%
\end{equation}
This representation combined with (\ref{lareemplazada}) and H\"{o}lder's
inequality\footnote{Write
\begin{align*}
s\left|  \nabla f\right|  ^{\ast\ast}(s)  &  =\int_{0}^{s}\left|  \nabla
f\right|  ^{\ast}(u)du\\
&  \leq\left\|  \left|  \nabla f\right|  \chi_{(0,s)}\right\|  _{\bar{X}%
}\left\|  \chi_{(0,s)}\right\|  _{\bar{X}^{\prime}}\\
&  =\left\|  \left|  \nabla f\right|  \chi_{(0,s)}\right\|  _{\bar{X}}%
\phi_{\bar{X}^{\prime}}(s).
\end{align*}
}, yields%
\begin{align*}
\left|  f\right|  _{\mu}^{\ast\ast}(t)  &  \leq c_{n}\int_{t}^{1/2}%
s^{1/n}\left|  \nabla f\right|  _{\mu}^{\ast\ast}(s)\frac{ds}{s}+\left|
f\right|  _{\mu}^{\ast\ast}(1/2)\\
&  \leq c_{n}\int_{t}^{1}s^{1/n-1}\left\|  \left|  \nabla f\right|
\chi_{(0,s)}\right\|  _{\bar{X}}\phi_{\bar{X}^{\prime}}(s)\frac{ds}%
{s}+2\left\|  f\right\|  _{L^{1}}\\
&  \leq\left\|  f\right\|  _{\bar{S}_{X}}c_{n}\int_{t}^{1}s^{1/n-1}\phi
_{\bar{X}^{\prime}}(s)\frac{ds}{s}+2\left\|  f\right\|  _{L^{1}}\\
&  \leq c_{n}\int_{t}^{1}s^{1/n-1}\phi_{\bar{X}^{\prime}}(s)\frac{ds}{s}+2c,
\end{align*}
where in the last step we used the fact that $\left\|  f\right\|  _{\bar
{S}_{X}(\Omega)}\leq1,$ and $\left\|  f\right\|  _{L^{1}}\leq c\left\|
f\right\|  _{X}.$ Therefore,
\begin{align*}
E^{\bar{S}_{X}(\Omega)}(t)  &  =\sup_{\left\|  f\right\|  _{\bar{S}_{X}%
(\Omega)}\leq1}\left|  f\right|  _{\mu}^{\ast}(t)\\
&  \leq\sup_{\left\|  f\right\|  _{\bar{S}_{X}(\Omega)}\leq1}\left|  f\right|
_{\mu}^{\ast\ast}(t)\\
&  \leq c\left(  \int_{t}^{1}s^{1/n-1}\phi_{\bar{X}^{\prime}}(s)\frac{ds}%
{s}+1\right)  ,\text{ \ \ }t\in(0,1/2).
\end{align*}
The second part of the result follows readily by computation since, if
$X=L^{p},$ $1\leq p<n,$ then
\begin{align*}
\int_{t}^{1}s^{1/n-1}\phi_{\bar{X}^{\prime}}(s)\frac{ds}{s}  &  =\int_{t}%
^{1}s^{1/n-1}s^{1-1/p}\frac{ds}{s}\\
&  \leq\int_{t}^{\infty}s^{1/n-1}s^{1-1/p}\frac{ds}{s}\\
&  \simeq t^{1/n-1/p},
\end{align*}
and%
\[
1\leq ct^{1/n-1/p},\text{ for }t\in(0,1/2).
\]

\end{proof}

\begin{remark}
\label{trece} We can also deal in the same fashion with infinite measures. For
comparison with \cite{haroske} let us consider the case of $\mathbb{R}^{n}$
provided with Lebesgue measure. In this case $I_{\mathbb{R}^{n}}%
(t)=c_{n}t^{1-1/n},$ for $t>0,$ and (\ref{lareemplazada}) is known to hold for
all $t>0,$ and for all functions in $C_{0}(\mathbb{R}^{n})$ (cf.
\cite{mmlog}). For functions in $C_{0}(\mathbb{R}^{n})$ we can replace
(\ref{porfavor}) by
\[
\left|  f\right|  ^{\ast\ast}(t)=\int_{t}^{\infty}\left(  \left|  f\right|
^{\ast\ast}(s)-\left|  f\right|  ^{\ast}(s)\right)  \frac{ds}{s}.
\]
Suppose further that $X$ is such that%
\[
\int_{t}^{\infty}s^{1/n-1}\phi_{\bar{X}^{\prime}}(s)\frac{ds}{s}<\infty.
\]
Then, proceeding with the argument given in the proof above, we see that there
is no need to restrict the range of $t^{\prime}s$ for the validity of
(\ref{eresfea}), (\ref{eresmuyfea}), etc. Therefore, for $1\leq p<n,$ we have
(compare with \cite[Proposition 3.25]{haroske})%
\begin{equation}
E^{S_{X}^{1}(\mathbb{R}^{n})}(t)\leq c\left(  \int_{t}^{\infty}s^{1/n-1}%
\phi_{\bar{X}^{\prime}}(s)\frac{ds}{s}+1\right)  ,\text{ }t>0.
\label{greenspan}%
\end{equation}

\end{remark}

The use of H\"{o}lder's inequality as effected in the previous theorem does
not give the sharp result at the end point $p=n.$ Indeed, following the
previous method for $p=n,$ we only obtain%
\begin{align*}
E^{W_{L^{n}}^{1}(\Omega)}(t)  &  \preceq\int_{t}^{1}s^{1/n-1}s^{1-1/n}%
\frac{ds}{s}\\
&  \preceq\ln\frac{1}{t},\text{ }t\in(0,1/2).
\end{align*}
Our next result shows that using (\ref{lareemplazada}) in a slightly different
form (applying H\"{o}lder's inequality on the left hand side) we can obtain
the sharp estimate in the limiting cases (compare with \cite[Proposition
3.27]{haroske}).

\begin{theorem}
\label{precisado}$E^{\bar{S}_{L^{n}}(\Omega)}(t)\preceq\left(  \ln\frac{1}%
{t}\right)  ^{1/n^{\prime}},$ for $t\in(0,1/2).$
\end{theorem}

\begin{proof}
Suppose that $\left\|  f\right\|  _{\bar{S}_{L^{n}}(\Omega)}\leq1.$ First we
rewrite (\ref{lareemplazada}) as%
\[
\left(  \frac{\left|  f\right|  _{\mu}^{\ast\ast}(s)-\left|  f\right|  _{\mu
}^{\ast}(s)}{s^{1/n}}\right)  ^{n}\leq c_{n}\left(  \left|  \nabla f\right|
_{\mu}^{\ast\ast}(s)\right)  ^{n},s\in(0,1/2).
\]
Integrating, we thus find,%
\begin{align*}
\int_{t}^{1/2}\left(  \left|  f\right|  _{\mu}^{\ast\ast}(s)-\left|  f\right|
_{\mu}^{\ast}(s)\right)  ^{n}\frac{ds}{s}  &  =\int_{t}^{1/2}\left(
\frac{\left|  f\right|  _{\mu}^{\ast\ast}(s)-\left|  f\right|  _{\mu}^{\ast
}(s)}{s^{1/n}}\right)  ^{n}ds\\
&  \leq c_{n}\int_{t}^{1}\left(  \left|  \nabla f\right|  _{\mu}^{\ast\ast
}(s)\right)  ^{n}ds\\
&  \leq C_{n}\left\|  \left|  \nabla f\right|  \right\|  _{L^{n}}^{n}\text{
(by Hardy's inequality)}\\
&  \leq C_{n}.
\end{align*}
Now, for $t\in(0,1/2),$%
\begin{align*}
\left|  f\right|  _{\mu}^{\ast\ast}(t)-\left|  f\right|  _{\mu}^{\ast\ast
}(1/2)  &  =\int_{t}^{1/2}\left(  \left|  f\right|  _{\mu}^{\ast\ast
}(s)-\left|  f\right|  _{\mu}^{\ast}(s)\right)  \frac{ds}{s}\\
&  \leq\left(  \int_{t}^{1/2}\left(  \left|  f\right|  _{\mu}^{\ast\ast
}(s)-\left|  f\right|  _{\mu}^{\ast}(s)\right)  ^{n}\frac{ds}{s}\right)
^{1/n}\left(  \int_{t}^{1}\frac{ds}{s}\right)  ^{1/n^{\prime}}\text{
(H\"{o}lder's inequality)}\\
&  \leq C_{n}^{1/n}\left(  \ln\frac{1}{t}\right)  ^{1/n^{\prime}}.
\end{align*}
Therefore,
\begin{align*}
\left|  f\right|  _{\mu}^{\ast\ast}(t)  &  \leq C_{n}^{1/n}\left(  \ln\frac
{1}{t}\right)  ^{1/n^{\prime}}+\left\|  f\right\|  _{L^{1}}\\
&  \leq C_{n}^{1/n}\left(  \ln\frac{1}{t}\right)  ^{1/n^{\prime}}+\left\|
f\right\|  _{L^{n}}\\
&  \leq C_{n}^{1/n}\left(  \ln\frac{1}{t}\right)  ^{1/n^{\prime}},\text{ for
}t\in(0,1/2).
\end{align*}
Consequently,
\begin{equation}
E^{\bar{S}_{L^{n}}(\Omega)}(t)\preceq\left(  \ln\frac{1}{t}\right)
^{1/n^{\prime}},\text{ for }t\in(0,1/2). \label{bernanke}%
\end{equation}

\end{proof}

The case $p>n,$ is somewhat less interesting for the computation of growth
envelopes since we should have $E^{\bar{S}_{L^{p}}(\Omega)}(t)\leq c.$ We now
give a direct proof of this fact just to show that our method unifies all the cases.

\begin{proposition}
Let $p>n,$ then there exists a constant $c=c(n,p)$ such that%
\[
E^{\bar{S}_{L^{p}}(\Omega)}(t)\leq c,\text{ }t\in(0,1/2).
\]

\end{proposition}

\begin{proof}
Suppose that $\left\|  f\right\|  _{\bar{S}_{L^{p}}(\Omega)}\leq1,$ and let
$s\in(0,1/2).$ We estimate as follows%
\begin{align*}
\left|  f\right|  _{\mu}^{\ast\ast}(t)-\left|  f\right|  _{\mu}^{\ast}(t)  &
\leq ct^{1/n-1}\int_{0}^{t}\left|  \nabla f\right|  _{\mu}^{\ast}(s)ds\\
&  \leq ct^{1/n-1}\left(  \int_{0}^{1}\left|  \nabla f\right|  _{\mu}^{\ast
}(s)^{p}ds\right)  ^{1/p}t^{1/p^{\prime}}\\
&  \leq ct^{1/n-1/p}.
\end{align*}
Thus, using a familiar argument, we see that for $t\in(0,1/2),$%
\begin{align*}
\left|  f\right|  _{\mu}^{\ast\ast}(t)-\left|  f\right|  _{\mu}^{\ast\ast
}(1/2)  &  \leq c\int_{t}^{1}s^{1/n-1/p-1}ds\\
&  \leq c\frac{1-t^{1/n-1/p}}{(1/n-1/p)}\\
&  \leq\frac{c}{(1/n-1/p)}.
\end{align*}
It follows that%
\begin{align*}
\left|  f\right|  _{\mu}^{\ast}(t)  &  \leq\left|  f\right|  _{\mu}^{\ast\ast
}(t)\leq\frac{c}{(1/n-1/p)}+\left|  f\right|  _{\mu}^{\ast\ast}(1/2)\\
&  \leq\frac{c}{(1/n-1/p)},
\end{align*}
and we obtain%
\[
E^{\bar{S}_{L^{p}}(\Omega)}(t)\leq c,\text{ }t\in(0,1/2).
\]

\end{proof}

The methods discussed above apply to Sobolev spaces based on general r.i.
spaces. As another illustration we now consider in detail the case of the
Sobolev spaces based on the Lorentz spaces $L^{n,q}\left(  \Omega\right)  ,$
where $\Omega\subset\mathbb{R}^{n}$ is a bounded Lip domain of measure $1$.
The interest here lies in the fact that in the critical case $p=n,$ the second
index $q$ plays an important role$.$ Indeed, for $q=1,$ as is well known, we
have $W_{L^{n,1}}^{1}\subset L^{\infty},$ while this is no longer true for
$W_{L^{n,q}}^{1}$ if $q>1.$ In particular, for the space $W_{L^{n,n}}%
^{1}=W_{L^{n}}^{1}.$ The next result thus extends Theorem \ref{precisado} and
provides an explanation of the situation we have just described through the
use of growth envelopes.

\begin{theorem}
\label{comoantes}Let $1\leq q\leq\infty,$ then $E^{W_{L^{n,q}}^{1}(\Omega
)}(t)\preceq\left(  \ln\frac{1}{t}\right)  ^{1/q^{\prime}},$ for
$t\in(0,1/2).$
\end{theorem}

\begin{proof}
Consider first the case $1\leq q<\infty.$ Suppose that $\left\|  f\right\|
_{W_{L^{n,q}}^{1}(\Omega)}\leq1.$ From (\ref{lareemplazada}) we get%
\[
\left(  \left|  f\right|  ^{\ast\ast}(s)-\left|  f\right|  ^{\ast}(s)\right)
^{q}\leq c_{n}\left(  \left|  \nabla f\right|  ^{\ast\ast}(s)s^{1/n}\right)
^{q},\text{ }s\in(0,1/2).
\]
Then,
\begin{align*}
\left|  f\right|  _{\mu}^{\ast\ast}(t)-\left|  f\right|  _{\mu}^{\ast\ast
}(1/2)  &  =\int_{t}^{1/2}\left(  \left|  f\right|  _{\mu}^{\ast\ast
}(s)-\left|  f\right|  _{\mu}^{\ast}(s)\right)  \frac{ds}{s}\\
&  \leq\left(  \int_{t}^{1/2}\left(  \left|  f\right|  _{\mu}^{\ast\ast
}(s)-\left|  f\right|  _{\mu}^{\ast}(s)\right)  ^{q}\frac{ds}{s}\right)
^{1/q}\left(  \int_{t}^{1/2}\frac{ds}{s}\right)  ^{1/q^{\prime}}\\
&  \leq c\left(  \int_{t}^{1/2}\left(  \left|  \nabla f\right|  _{\mu}%
^{\ast\ast}(s)s^{1/n}\right)  ^{q}\frac{ds}{s}\right)  ^{1/q}\left(  \int%
_{t}^{1/2}\frac{ds}{s}\right)  ^{1/q^{\prime}}\\
&  \leq c\left\|  \left|  \nabla f\right|  \right\|  _{L^{n,q}}\left(
\ln\frac{1}{t}\right)  ^{1/q^{\prime}}.
\end{align*}
Therefore, as before%
\begin{align*}
\left|  f\right|  _{\mu}^{\ast\ast}(t)  &  \leq c\left\|  \left|  \nabla
f\right|  \right\|  _{L^{n,q}}\left(  \ln\frac{1}{t}\right)  ^{1/q^{\prime}%
}+\left|  f\right|  _{\mu}^{\ast\ast}(1/2)\\
&  \leq c\left(  \ln\frac{1}{t}\right)  ^{1/q^{\prime}},\text{ }t\in(0,1/2),
\end{align*}
and the desired estimate for $E^{W_{L^{n,q}}^{1}(\Omega)}(t)$ follows.

When $q=\infty,$ and $\left\|  f\right\|  _{W_{L^{n,\infty}}^{1}(\Omega)}%
\leq1$, we estimate%
\begin{align*}
\left|  f\right|  _{\mu}^{\ast\ast}(s)-\left|  f\right|  _{\mu}^{\ast}(s)  &
\leq c_{n}\left|  \nabla f\right|  ^{\ast\ast}(s)s^{1/n}\\
&  \leq c_{n}\left\|  \left|  \nabla f\right|  \right\|  _{L^{n,\infty}}\\
&  \leq c_{n}.
\end{align*}
Consequently,%
\[
\left|  f\right|  _{\mu}^{\ast\ast}(t)-\left|  f\right|  _{\mu}^{\ast\ast
}(1/2)\leq c_{n}\int_{t}^{1}\frac{ds}{s}%
\]
and we readily get%
\[
E^{W_{L^{n,\infty}}^{1}(\Omega)}(t)\leq c\left(  \ln\frac{1}{t}\right)
,\text{ }t\in(0,1/2).
\]

\end{proof}

\begin{remark}
Note that, in particular,
\[
E^{W_{L^{n,1}}^{1}(\Omega)}(t)\leq c,
\]
which again reflects the fact that $W_{L^{n,1}}^{1}(\Omega)\subset L^{\infty
}(\Omega).$
\end{remark}

\begin{remark}
As before, all the previous results hold for the $W_{L^{p,q}}^{1}%
(\mathbb{R}^{n})$ spaces.
\end{remark}

We now show that a similar method, replacing the use of (\ref{lareemplazada})
by (\ref{lanueva}), allow us to obtain sharp estimates for growth envelopes of
Besov spaces (see (\ref{besoveuclideo})).

\begin{theorem}
Let $p>n>1,$ $1\leq q\leq\infty.$ Then%
\[
E^{B_{p}^{n/p,q}([0,1]^{n})}(t)\preceq\left(  \log\frac{1}{t}\right)
^{1/q^{\prime}},\text{ }t\in(0,1/2).
\]

\end{theorem}

\begin{proof}
The starting point is%
\[
\left\vert f\right\vert _{\mu}^{\ast\ast}(t)-\left\vert f\right\vert _{\mu
}^{\ast}(t)\leq c\frac{\omega_{L^{p}}(t^{1/n},f)}{t^{1/p}},\text{ }%
t\in(0,1/2).
\]
Then,%
\begin{align*}
\left\vert f\right\vert _{\mu}^{\ast\ast}(t)-\left\vert f\right\vert _{\mu
}^{\ast\ast}(1/2)  &  =\int_{t}^{1/2}\left(  \left\vert f\right\vert _{\mu
}^{\ast\ast}(s)-\left\vert f\right\vert _{\mu}^{\ast}(s)\right)  \frac{ds}%
{s}\\
&  \leq c\int_{t}^{1/2}\frac{\omega_{L^{p}}(s^{1/n},f)}{s^{1/p}}\frac{ds}{s}\\
&  \leq c\left(  \int_{t}^{1}\left(  \frac{\omega_{L^{p}}(s^{1/n},f)}{s^{1/p}%
}\right)  ^{q}\frac{ds}{s}\right)  ^{1/q}\left(  \int_{t}^{1}\frac{ds}%
{s}\right)  ^{1/q^{\prime}}(\text{H\"{o}lder's inequality)}\\
&  \leq c\left\Vert f\right\Vert _{B_{p}^{n/p,q}([0,1]^{n})}\left(  \log
\frac{1}{t}\right)  ^{1/q^{\prime}}.
\end{align*}
Thus, a familiar argument now gives (compare with \cite[(1.9)]{haroske})%
\[
E^{B_{p}^{n/p,q}([0,1]^{n})}(t)\preceq\left(  \log\frac{1}{t}\right)
^{1/q^{\prime}},t\in(0,1/2).
\]

\end{proof}

\section{Continuity Envelopes}

Suppose that $Z:=Z(\mathbb{R}^{n})\subset C(\mathbb{R}^{n}),$ then (cf.
\cite{haroske} and the references therein) one defines the continuity envelope
by%
\[
E_{C}^{Z}(t)=\sup_{\left\|  f\right\|  _{Z(\mathbb{R}^{n})}\leq1}\frac
{\omega_{L^{\infty}}(t,f)}{t}.
\]

At this point it is instructive to recall some known interpolation
inequalities. Let $\left\Vert f\right\Vert _{\dot{W}_{L^{n,1}}^{1}}=\int%
_{0}^{\infty}\left\vert \nabla f\right\vert ^{\ast}(s)s^{1/n}\frac{ds}{s}$. We
interpolate the following known embeddings (cf. \cite{St}): for $f\in
C_{0}^{\infty}(\mathbb{R}^{n}),$ we have%
\[
\left\Vert f\right\Vert _{L^{\infty}(\mathbb{R}^{n})}\preceq\left\Vert
f\right\Vert _{\dot{W}_{L^{n,1}(\mathbb{R}^{n})}^{1}}.
\]
Consequently,%
\begin{equation}
K(t,f;L^{\infty}(\mathbb{R}^{n}),\dot{W}_{L^{\infty}}^{1}(\mathbb{R}%
^{n}))\preceq K(t,f,\dot{W}_{L^{n,1}}^{1}(\mathbb{R}^{n}),\dot{W}_{L^{\infty}%
}^{1}(\mathbb{R}^{n})). \label{orla}%
\end{equation}
It is well-known that for continuous functions we have (cf. \cite{bs})%
\[
K(t,f;L^{\infty}(\mathbb{R}^{n}),\dot{W}_{L^{\infty}}^{1}(\mathbb{R}%
^{n}))=\omega_{L^{\infty}}(t,f)\simeq\sup_{\left\vert h\right\vert \leq
t}\left\Vert f(\cdot+h)-f(\cdot)\right\Vert _{L^{\infty}}.
\]
On the other hand using \cite[Theorem 2]{mmletter} and Holmstedt's Lemma (see
\cite[Theorem 3.6.1]{bl}) we find\
\[
K(t,f,\dot{W}_{L^{n,1}}^{1}(\mathbb{R}^{n}),\dot{W}_{L^{\infty}}%
^{1}(\mathbb{R}^{n}))\simeq\int_{0}^{t^{n}}\left\vert \nabla f\right\vert
^{\ast}(s)s^{1/n}\frac{ds}{s}.
\]
Inserting this information back to (\ref{orla}) we find%
\[
\omega_{L^{\infty}}(t,f)\preceq\int_{0}^{t^{n}}\left\vert \nabla f\right\vert
^{\ast}(s)s^{1/n}\frac{ds}{s}.
\]
Therefore,%
\begin{align*}
\frac{\omega_{L^{\infty}}(t,f)}{t}  &  \preceq\frac{1}{t}\int_{0}^{t^{n}%
}\left\vert \nabla f\right\vert ^{\ast}(s)s^{1/n}\frac{ds}{s}\\
&  \preceq\frac{1}{t}\left\Vert \left\vert \nabla f\right\vert \right\Vert
_{L^{p,q}}\left(  \int_{0}^{t^{n}}s^{(1/n-1/p)q^{\prime}}\frac{ds}{s}\right)
^{1/q^{\prime}}\\
&  \preceq\frac{1}{t}\left\Vert \left\vert \nabla f\right\vert \right\Vert
_{L^{p,q}}t^{1-n/p}.
\end{align*}
Thus, we have (compare with \cite[(1.15)]{haroske}) that for $1\leq q<\infty,$%
\begin{equation}
E_{C}^{W_{L^{p,q}}^{1}(\mathbb{R}^{n})}(t)\preceq t^{-n/p}.
\label{orlacontinua}%
\end{equation}

\begin{remark}
It is actually fairly straightforward at this point to derive a general
relation between $E^{Z}$ and $E_{C}^{Z}.$ In \cite{mamiproc} we have shown for
$f\in C_{0}^{\infty}(\mathbb{R}^{n})$ we have%
\begin{equation}
\left|  f\right|  ^{\ast\ast}(t)-\left|  f\right|  ^{\ast}(t)\leq c_{n}%
\omega_{L^{\infty}}(t^{1/n},f),\text{ }t>0. \label{deducida}%
\end{equation}
We now proceed formally, although the details can be easily filled-in by the
interested reader. From (\ref{deducida}) we find%
\[
\frac{\left|  f\right|  ^{\ast\ast}(t)-\left|  f\right|  ^{\ast}(t)}{t}\leq
c_{n}\frac{\omega_{L^{\infty}}(t^{1/n},f)}{t},\text{ }t>0.
\]
Then%
\begin{align*}
\left|  f\right|  ^{\ast}(t)  &  \leq\left|  f\right|  ^{\ast\ast}(t)\\
&  =\int_{t}^{\infty}\frac{\left|  f\right|  ^{\ast\ast}(s)-\left|  f\right|
^{\ast}(s)}{s}ds\text{ (since }f^{\ast\ast}(\infty)=0)\\
&  \leq c_{n}\int_{t}^{\infty}\frac{\omega_{L^{\infty}}(s^{1/n},f)}{s}ds.
\end{align*}
Taking supremum over the unit ball of $Z(\mathbb{R}^{n})$ we obtain%
\[
E^{Z(\mathbb{R}^{n})}(t)\preceq\int_{t^{1/n}}^{\infty}E_{C}^{Z(\mathbb{R}%
^{n})}(s)ds.
\]
Thus, for example, from (\ref{orlacontinua}) we find that for $p<n$%
\begin{align*}
E^{W_{L^{p,q}}^{1}(\mathbb{R}^{n})}(t)  &  \preceq\int_{t^{1/n}}^{\infty
}s^{-n/p}ds\\
&  \simeq t^{1/n-1/p},
\end{align*}
which should be compared with (\ref{eresmuyfea}).
\end{remark}

\section{General isoperimetric profiles}

In the previous sections we have focussed mainly on function spaces on domains
with isoperimetric profiles of Euclidean type; but our inequalities also
provide a unified setting to study estimates for general profiles. For a
metric measure space $(\Omega,d,\mu)$ of finite measure we consider r.i.
spaces $X\left(  \Omega\right)  .$ Let $0<\theta<1$ and $1\leq q\leq\infty,$
the homogeneous Besov space $\dot{b}_{X,q}^{\theta}(\Omega)$ is defined by
\[
\dot{b}_{X}^{\theta,q}(\Omega)=\left\{  f\in X+S_{X}:\left\Vert f\right\Vert
_{\dot{b}_{X}^{\theta,q}(\Omega)}=\left(  \int_{0}^{\mu\left(  \Omega\right)
}\left(  K\left(  s,f;X,S_{X}\right)  s^{-\theta}\right)  ^{q}\frac{ds}%
{s}\right)  ^{1/q}<\infty\right\}  ,
\]
with the usual modifications when $q=\infty.$ The Besov space $b_{X,q}%
^{\theta}(\Omega,\mu)$ is defined by
\[
\left\Vert f\right\Vert _{b_{X}^{\theta,q}(\Omega,\mu)}=\left\Vert
f\right\Vert _{X}+\left\Vert f\right\Vert _{\dot{b}_{X}^{\theta,q}(\Omega)}.
\]
Notice that if $X=L^{p},$ then $\dot{b}_{L^{p}}^{\theta,q}(\Omega)=\dot{b}%
_{p}^{\theta,q}(\Omega)$ (resp. $b_{L^{p}}^{\theta,q}(\Omega)=b_{p}^{\theta
,q}(\Omega))$ (see (\ref{besovmetrico})).

\begin{theorem}
Let $X$ be a r.i. space on $\Omega.$ Let $g(s)=\frac{s}{I_{\Omega}(s)}$ where
$I_{\Omega}$ denotes the isoperimetric profile of $(\Omega,d,\mu).$ Let
$b_{X}^{\theta,q}(\Omega)$ be a Besov space ($0<\theta<1,$ $1<q<\infty)$, then
for $t\in(0,\mu(\Omega)/2)$ we have that%
\[
E^{b_{X}^{\theta,q}(\Omega)}(t)\leq c\left(  1+\left(  \int_{t}^{\mu
(\Omega)/2}\left(  g(s)^{\theta}\left(  \frac{g(s)}{g^{\prime}(s)}\right)
^{1/q}\right)  ^{q^{\prime}}\frac{ds}{\left(  s\phi_{X}(s)\right)
^{q^{\prime}}}\right)  \right)  ,\text{ }%
\]
where, as usual $1/q^{\prime}+1/q=1.$
\end{theorem}

\begin{proof}
Let $f\in X+S_{X},$ and let us write $K\left(  t,f;X,S_{X}\right)  :=K\left(
t,f\right)  .$ By Theorem \ref{main1} we know that
\[
\left\vert f\right\vert _{\mu}^{\ast\ast}(t)-\left\vert f\right\vert _{\mu
}^{\ast}(t)\leq c\frac{K\left(  g(t),f\right)  }{\phi_{X}(t)},\text{ }%
0<t<\mu(\Omega).
\]
Taking in account that, $\left(  -\left\vert f\right\vert _{\mu}^{\ast\ast
}\right)  ^{\prime}(t)=\left(  \left\vert f\right\vert _{\mu}^{\ast\ast
}(t)-\left\vert f\right\vert _{\mu}^{\ast}(t)\right)  /t,$ we get
\[
\left\vert f\right\vert _{\mu}^{\ast\ast}(t)-\left\vert f\right\vert _{\mu
}^{\ast\ast}(\mu(\Omega)/2)=\int_{t}^{\mu(\Omega)/2}\left(  -\left\vert
f\right\vert _{\mu}^{\ast\ast}\right)  ^{\prime}(s)ds\leq c\int_{t}%
^{\mu(\Omega)/2}\frac{K\left(  g(s),f\right)  }{\phi_{X}(s)}\frac{ds}{s}.
\]
Since $I_{\Omega}(s)$ is a concave continuous increasing function on
$(0,\mu(\Omega)/2),$ $g(s)$ is differentiable on $(0,\mu(\Omega)/2).$ Then, by
H\"{o}lder's inequality we have%
\begin{align*}
R(t)  &  =\int_{t}^{\mu(\Omega)/2}\frac{K\left(  g(s),f\right)  }{\phi_{X}%
(s)}\frac{ds}{s}\\
&  =\int_{t}^{\mu(\Omega)/2}K\left(  g(s),f\right)  \left(  \frac{g(s)}%
{g(s)}\right)  ^{\theta}\left(  \frac{g^{\prime}(s)}{g(s)}\right)
^{1/q}\left(  \frac{g(s)}{g^{\prime}(s)}\right)  ^{1/q}\frac{ds}{s\phi_{X}%
(s)}\\
&  \leq R_{1}(t)R_{2}(t),
\end{align*}
where%
\[
R_{1}(t)=\left(  \int_{t}^{\mu(\Omega)/2}\left(  K\left(  g(s),f\right)
g(s)^{-\theta}\right)  ^{q}\left(  \frac{g^{\prime}(s)}{g(s)}\right)
ds\right)  ^{1/q},
\]
and
\begin{equation}
R_{2}(t)=\left(  \int_{t}^{\mu(\Omega)/2}\left(  g(s)^{\theta}\left(
\frac{g(s)}{g^{\prime}(s)}\right)  ^{1/q}\right)  ^{q^{\prime}}\frac
{ds}{\left(  s\phi_{X}(s)\right)  ^{q^{\prime}}}\right)  ^{1/q^{\prime}}.
\label{estex2}%
\end{equation}
By a change of variables
\begin{equation}
R_{1}(t)=\left(  \int_{g^{-1}(t)}^{g^{-1}(\mu(\Omega)/2)}\left(  K\left(
z,f\right)  z^{-\theta}\right)  ^{q}\frac{dz}{z}\right)  ^{1/q}\leq\left\Vert
f\right\Vert _{\dot{b}_{X,q}^{\theta}(\Omega)} \label{estex1}%
\end{equation}
Combining (\ref{estex2}) and (\ref{estex1}) we obtain
\begin{align*}
\left\vert f\right\vert _{\mu}^{\ast\ast}(t)  &  \leq c\left\Vert f\right\Vert
_{\dot{b}_{X,q}^{\theta}(\Omega)}R_{2}(t)+2\left\Vert f\right\Vert _{X}\\
&  \leq2c\left(  1+R_{2}(t)\right)  \left\Vert f\right\Vert _{b_{X,q}^{\theta
}(\Omega)}.
\end{align*}
Therefore, taking sup over all $f$ such that $\left\Vert f\right\Vert
_{b_{X,q}^{\theta}(\Omega)}\leq1$ we see that%
\[
E^{b_{X}^{\theta,q}(\Omega)}(t)\leq2c\left(  1+R_{2}(t)\right)  ,\,t\in
(0,\mu(\Omega)/2).
\]

\end{proof}

\begin{example}
Consider the Gaussian measure $(\mathbb{R}^{n},\gamma_{n})$. Then (cf.
\cite{Bob}) we can take as isoperimetric estimator
\[
I_{\gamma_{n}}(t)=t\left(  \log\frac{1}{t}\right)  ^{1/2},\text{ \ \ \ }%
t\in(0,1/2).
\]
Thus,
\[
\text{\ }g(t)=\frac{1}{\left(  \log\frac{1}{t}\right)  ^{1/2}}\text{\ \ and
\ }g^{\prime}(s)=\frac{1}{2\left(  \log\frac{1}{s}\right)  ^{\frac{3}{2}}s},
\]
and
\begin{align*}
&  \left(  1+\left(  \int_{t}^{1/2}\left(  g(s)^{\theta}\left(  \frac
{g(s)}{g^{\prime}(s)}\right)  ^{1/q}\right)  ^{q^{\prime}}\frac{ds}{\left(
s\phi_{X}(s)\right)  ^{q^{\prime}}}\right)  ^{1/q^{\prime}}\right) \\
&  =\left(  \int_{t}^{1/2}\left(  \log\frac{1}{s}\right)  ^{q^{\prime}\left(
1-\frac{\theta}{2}\right)  -1}\frac{ds}{s\left(  \phi_{X}(s)\right)
^{q^{\prime}}}\right)  ^{1/q^{\prime}}\\
&  \leq\frac{1}{\phi_{X}(t)}\left(  \int_{t}^{1/2}\left(  \log\frac{1}%
{s}\right)  ^{q^{\prime}\left(  1-\frac{\theta}{2}\right)  -1}\frac{ds}%
{s}\right)  ^{1/q^{\prime}}\\
&  \preceq\frac{1}{\phi_{X}(t)}\left(  \log\frac{1}{t}\right)  ^{\left(
1-\frac{\theta}{2}\right)  }.
\end{align*}
Therefore we find that
\[
E^{B_{X}^{\theta,q}(\mathbb{R}^{n},\gamma_{n})}(t)\preceq\frac{1}{\phi_{X}%
(t)}\left(  \log\frac{1}{t}\right)  ^{\left(  1-\frac{\theta}{2}\right)
},\text{ \ \ \ }t\in(0,1/2).
\]

\end{example}

\section{Envelopes for higher order spaces}

In general it is not clear how to define higher order Sobolev and Besov spaces
in metric spaces. On the other hand for classical domains (Euclidean,
Riemannian manifolds, etc.) there is a well developed theory of embeddings
that one can use to estimate growth envelopes. The underlying general
principle is very simple. Suppose that the function space $Z=Z\left(
\Omega\right)  $ is continuously embedded in $Y=Y\left(  \Omega\right)  $ and
$Y$ is a rearrangement invariant space, then, since (cf. \ref{tango}), Chapter
\ref{preliminar}) $Y\subset M\left(  Y\right)  ,$ where $M(Y)$ is the
Marcinkiewicz space associated with $Y$ (cf. Section \ref{secc:ri} in Chapter
\ref{preliminar}), we have (cf. (\ref{fier})) for all $f\in Z,$%
\[
\sup_{t}\left|  f\right|  ^{\ast}(t)\phi_{Y}(t)\leq\left\|  f\right\|
_{Y}\leq c\left\|  f\right\|  _{Z},
\]
where $\phi_{Y}(t)$ is the fundamental function of $Y,$ and $c$ is the norm of
the embedding $Z\subset Y.$ Consequently, for all $f\in Z,$ with $\left\|
f\right\|  _{Z}\leq1,$ for all $t>0,$
\[
\left|  f\right|  ^{\ast}(t)\leq\frac{c}{\phi_{Y}(t)}.
\]
Therefore,%
\[
E^{Z}(t)\preceq\frac{1}{\phi_{Y}(t)}.
\]

For example, suppose that $p<\frac{n}{k},$ then from
\[
W_{p}^{k}(\mathbb{R}^{n})\subset L^{q,p},\text{ with }\frac{1}{q}=\frac{1}%
{p}-\frac{k}{n}%
\]
and%
\[
\phi_{L^{q,p}}(t)=t^{1/q}=t^{1/p-k/n}%
\]
we get (compare with (\ref{eresmuyfea}) above and \cite[(1.7)]{haroske})
\[
E^{W_{p}^{k}(\mathbb{R}^{n})}(t)\preceq t^{k/n-1/p}.
\]
In the limiting case we have (cf. \cite{bmr}, \cite{mp})%
\[
W_{\frac{n}{k}}^{k}(\mathbb{R}^{n})\subset L^{[\infty,\frac{n}{k}]}.
\]
For comparison consider $W_{\frac{n}{k}}^{k}(\Omega),$ where $\Omega$ is a
domain on $\mathbb{R}^{n},$ with $\left|  \Omega\right|  =1.$ One can readily
estimate the decay of functions in $L^{[\infty,\frac{n}{k}]}$ as follows:
\begin{align*}
\left|  f\right|  ^{\ast\ast}(t)-\left|  f\right|  ^{\ast\ast}(1)  &
=\int_{t}^{1}\left(  \left|  f\right|  ^{\ast\ast}(s)-\left|  f\right|
^{\ast}(s)\right)  \frac{ds}{s}\\
&  \leq\left(  \int_{t}^{1}\left(  \left|  f\right|  ^{\ast\ast}(s)-\left|
f\right|  ^{\ast}(s)\right)  ^{\frac{n}{k}}\frac{ds}{s}\right)  ^{1/(\frac
{n}{k})}\left(  \int_{t}^{1}\frac{ds}{s}\right)  ^{1/(\frac{n}{k})^{\prime}}\\
&  \leq\left\|  f\right\|  _{L^{[\infty,\frac{n}{k}]}}\left(  \log\frac{1}%
{t}\right)  ^{1-\frac{k}{n}}.
\end{align*}
Combining these observations we see that for functions in the unit ball of
$W_{\frac{n}{k}}^{k}(\Omega)$ we have
\[
\left|  f\right|  ^{\ast\ast}(t)\preceq c\left(  \log\frac{1}{t}\right)
^{1-\frac{k}{n}},\text{ for }t\in(0,1/2).
\]
Consequently
\[
E^{W_{\frac{n}{k}}^{k}(\Omega)}(t)\preceq\left(  \log\frac{1}{t}\right)
^{1-\frac{k}{n}}.
\]
In particular, when $k=1$ then $1-\frac{k}{n}=\frac{1}{n^{\prime}},$ and the
result coincides with Theorem \ref{precisado} above.

Likewise we can deal with the case of general isoperimetric profiles but we
shall leave the discussion for another occasion.

\section{$K$ and $E$ functionals for families\label{sec:markao}}

It is of interest to point out a connection between the different
\textquotedblleft envelopes\textquotedblright\ discussed above and a more
general concept introduced somewhat earlier in \cite{mamiconver}, but in a
different context. One of the tools introduced in \cite{mamiconver} was to
consider the $K$ and $E$ functionals for families, rather than single elements.

Given a compatible pair of spaces $(X,Y)$ (cf. \cite{bl}), and a family of
elements, $F\subset X+Y,$ we can define the $K-$functional and $E-$%
functional\footnote{Recall that (cf. \cite{bl}, \cite{mamiconver}),%
\[
K(t,f;X,Y)=\inf\{\left\|  f-g\right\|  _{X}+t\left\|  g\right\|  _{Y}:g\in
Y\}
\]%
\[
E(t,f;X,Y)=\inf\{\left\|  f-g\right\|  _{Y}:\left\|  g\right\|  _{X}\leq t\}.
\]
} of the family $F$ by (cf. \cite{mamiconver})%
\[
K(t,F;X,Y)=\sup_{f\in F}K(t,f;X,Y).
\]%
\[
E(t,F;X,Y)=\sup_{f\in F}E(t,f;X,Y).
\]
The connection with the Triebel-Haroske envelopes can be seen from the
following known computations. If we let $\left\|  f\right\|  _{L^{0}}=\mu
\{$supp$f\},$ then
\[
\left|  f\right|  ^{\ast}(t)=E(t,f;L^{0},L^{\infty}).
\]
Therefore,%
\[
E^{Z(\Omega)}(t)=E(t,\text{unit ball of }Z(\Omega);L^{0}(\Omega),L^{\infty
}(\Omega).
\]
Moreover, since on Euclidean space we have
\[
\omega_{L^{\infty}}(t,f)\simeq K(t,f;L^{\infty}(\mathbb{R}^{n}),\dot
{W}_{L^{\infty}}^{1}(\mathbb{R}^{n}))
\]
we therefore see that%
\[
E_{C}^{Z(\mathbb{R}^{n})}(t)=K(t,f;\text{unit ball of }Z(\mathbb{R}%
^{n});L^{\infty}(\mathbb{R}^{n}),\dot{W}_{L^{\infty}}^{1}(\mathbb{R}^{n})).
\]
This suggests the general definition for metric spaces
\[
E_{C}^{Z(\Omega)}(t)=K(t,f;\text{unit ball of }Z(\Omega);L^{\infty}%
(\Omega),S_{L^{\infty}}(\Omega)).
\]
This provides a method to expand the known results to the metric setting using
the methods discussed in this paper. Another interesting aspect of the
connection we have established here lies in the fact, established in
\cite{mamiconver}, that one can reformulate classical convergence and
compactness criteria for function spaces (e.g. Kolmogorov's compactness
criteria for sets contained in $L^{p}$) in terms of conditions on these (new)
functionals. For example, according to the Kolmogorov criteria, for a set of
functions $F$ to be compact on $L^{p}(\mathbb{R}^{n})$ one needs the uniform
continuity on $F$ at zero of $\omega_{L^{p}}(t,F).$ In our formulation we
replace this condition by demanding the continuity at zero of%
\[
K(t,F;L^{p},W_{L^{p}}^{1}).
\]
Again, to develop this material in detail is a long paper on its own, however,
let us note in passing that the failure of compactness of the embedding
$W_{L^{p}}^{1}\left(  \Omega\right)  \subset L^{\bar{p}}\left(  \Omega\right)
,$ for $p=n,$ is consistent with the blow up at zero predicted by the fact
that the converse of (\ref{bernanke}) also holds. One should compare this with
the estimate (\ref{greenspan}) which is consistent with the Relich compactness
criteria for Sobolev embeddings on bounded domains, when $p<n.$

\chapter{Lorentz spaces with negative indices\label{negativo}}

\section{Introduction and Summary}

As we have shown elsewhere (cf. \cite{bmr}, \cite{mamicwk}), the basic
Euclidean inequality
\[
f^{\ast\ast}(t)-f^{\ast}(t)\leq c_{n}t^{1/n}\left|  \nabla f\right|
^{\ast\ast}(t)
\]
leads to the optimal Sobolev inequality%
\begin{equation}
\left\|  f\right\|  _{L^{[\bar{p},p]}}=\left\{  \int_{0}^{\infty}\left(
\left(  \left|  f\right|  ^{\ast\ast}(t)-\left|  f\right|  ^{\ast}(t)\right)
t^{1/\bar{p}}\right)  ^{p}\frac{dt}{t}\right\}  ^{1/p}\leq c_{n}\left\|
\left|  \nabla f\right|  \right\|  _{L^{p}}, \label{negativa}%
\end{equation}
where $1<p\leq n,$ $\frac{1}{\bar{p}}=\frac{1}{p}-\frac{1}{n}.$ The use of the
$L^{[\bar{p},p]}$ conditions makes it possible to consider the limiting case
$p=n$ in a unified way. Now (\ref{negativa}) is also meaningful when $p>n,$
albeit the only reason for the restriction $p\leq n,$ is that, if we don't
impose it, then $\bar{p}<0,$ and thus the condition defined by $\left\|
f\right\|  _{L^{[\bar{p},p]}}<\infty$ is not well understood$.$ It is was
shown in \cite{mp} that these conditions are meaningful. In this chapter we
show a connection between the Lorentz $L^{[\bar{p},p]}$ spaces with negative
indices and Morrey's theorem.

\subsection{Lorentz conditions}

Let $(\Omega,d,\mu)$ be a metric measure space. Let $0<q\leq\infty,$
$s\in\mathbb{R}.$ We define%
\[
L^{[s,q]}=L^{[s,q]}\left(  0,\mu\left(  \Omega\right)  \right)  =\left\{  f\in
L^{1}\left(  \Omega\right)  :\left\{  \int_{0}^{\mu\left(  \Omega\right)
}\left(  \left(  \left|  f\right|  _{\mu}^{\ast\ast}(t)-\left|  f\right|
_{\mu}^{\ast}(t)\right)  t^{1/s}\right)  ^{q}\frac{dt}{t}\right\}
^{1/q}<\infty\right\}  .
\]
For $0<q\leq\infty,$ $s\in\lbrack1,\infty],$ these spaces were defined in
Chapter \ref{chapbmo}. They coincide with the usual $L^{s,q}$ spaces when
$0<q\leq\infty,$ $s\in\lbrack1,\infty)$ (cf. \cite{mamicon}).

Our first observation is that $L^{[s,q]}\neq\emptyset.$ Indeed, for $s<0,$ we
have%
\begin{align*}
0  &  <\left\|  \chi_{A}\right\|  _{L^{[s,q]}}=\frac{\mu(A)}{\left(
q-q/s\right)  ^{1/q}}[\mu(A)^{q/s-q}-1]^{1/q}\\
&  \leq\frac{\mu(A)^{1/s}}{\left(  q-q/s\right)  ^{1/q}}.
\end{align*}
It is important to remark that the cancellation at zero afforded by $\left|
f\right|  _{\mu}^{\ast\ast}(t)-\left|  f\right|  _{\mu}^{\ast}(t)$ is crucial
here. Indeed, if we attempt to extend the usual definition of Lorentz spaces
by letting $s<0,$ then we find that $\left\|  \chi_{A}\right\|  _{L^{(s,q)}%
}=\int_{0}^{\mu(A)}t^{q/s}\frac{dt}{t}<\infty$ iff $\mu(A)=0.$

\section{The role of the $L^{[\bar{p},p]}$ spaces in Morrey's theorem}

For definiteness we work on $\mathbb{R}^{n}$ with Lebesgue measure $m.$ We
show that many arguments we have discussed in this paper are available in the
context of Lorentz spaces with negative index.

Let $f\in L^{[\bar{p},p]}$ where $\frac{1}{\bar{p}}=\frac{1}{p}-\frac{1}%
{n}<0.$ Then, for $0<t_{1}<t_{2},$ we can write%
\begin{align*}
f^{\ast\ast}(t_{1})-f^{\ast\ast}(t_{2})  &  =\int_{t_{1}}^{t_{2}}\left(
f^{\ast\ast}(t)-f^{\ast}(t)\right)  t^{1/\bar{p}}t^{-1/\bar{p}}\frac{dt}{t}\\
&  \leq\left(  \int_{t_{1}}^{t_{2}}\left(  \left(  f^{\ast\ast}(t)-f^{\ast
}(t)\right)  t^{1/\bar{p}}\right)  ^{p}\frac{dt}{t}\right)  ^{1/p}\left(
\int_{t_{1}}^{t_{2}}t^{-p^{\prime}/\bar{p}}\frac{dt}{t}\right)  ^{1/p^{\prime
}}\\
&  =\left\|  f\right\|  _{L^{[\bar{p},p]}}\left(  \int_{t_{1}}^{t_{2}%
}t^{-p^{\prime}/\bar{p}}\frac{dt}{t}\right)  ^{1/p^{\prime}}.
\end{align*}
Note that since $\frac{-p^{\prime}}{\bar{p}}-1=\frac{p}{p-1}(\frac{n-p}%
{np})-1=\frac{1}{p-1}[\frac{n-p-n}{n}]<0,$ the function $t^{-p^{\prime}%
/\bar{p}-1}$ is decreasing and therefore,%
\[
\int_{t_{1}}^{t_{2}}t^{-p^{\prime}/\bar{p}}\frac{dt}{t}\leq\int_{0}%
^{t_{2}-t_{1}}t^{-p^{\prime}/\bar{p}}\frac{dt}{t}=\frac{-\bar{p}}{p^{\prime}%
}\left|  t_{2}-t_{1}\right|  ^{\frac{-p^{\prime}}{\bar{p}}}.
\]
Thus,%
\begin{equation}
f^{\ast\ast}(t_{1})-f^{\ast\ast}(t_{2})\leq\left(  \frac{-\bar{p}}{p^{\prime}%
}\right)  ^{1/p^{\prime}}\left\|  f\right\|  _{L^{[\bar{p},p]}}\left|
t_{2}-t_{1}\right|  ^{\frac{-1}{\bar{p}}}. \label{nega2}%
\end{equation}
The localization property in this context takes the following form. Suppose
that $f\in L^{[\bar{p},p]}$ is such that there exists a constant $C>0,$ such
that $\forall B$ open ball$,$ it follows that $f\chi_{B}\in L^{[\bar{p}%
,p]}(0,m(B)),$ with $\left\|  f\chi_{B}\right\|  _{L^{[\bar{p},p]}}\leq
C\left\|  f\right\|  _{L^{[\bar{p},p]}}.$ Then, from (\ref{nega2}) we get%
\[
\left(  f\chi_{B}\right)  ^{\ast\ast}(t_{1})-\left(  f\chi_{B}\right)
^{\ast\ast}(t_{2})\leq C\left(  \frac{-\bar{p}}{p^{\prime}}\right)
^{1/p^{\prime}}\left\|  f\right\|  _{L^{[\bar{p},p]}}\left|  t_{2}%
-t_{1}\right|  ^{\frac{\alpha}{n}},
\]
where $\alpha=1-\frac{n}{p}.$ Applying this inequality replacing $t_{i}$ by
$t_{i}m(B),i=1,2;$ we get%
\[
\left(  f\chi_{B}\right)  ^{\ast\ast}(t_{1}m(B))-\left(  f\chi_{B}\right)
^{\ast\ast}(t_{2}m(B))\leq C\left(  \frac{-\bar{p}}{p^{\prime}}\right)
^{1/p^{\prime}}\left\|  f\right\|  _{L^{[\bar{p},p]}}\left|  t_{2}%
-t_{1}\right|  ^{\frac{\alpha}{n}}m(B)^{\frac{\alpha}{n}}.
\]
Letting $t_{1}\rightarrow0,t_{2}\rightarrow1,$ we then find%
\[
ess\sup_{B}\left(  f\right)  -\frac{1}{m(B)}\int_{B}f\leq C\left(  \frac
{-\bar{p}}{p^{\prime}}\right)  ^{1/p^{\prime}}\left\|  f\right\|
_{L^{[\bar{p},p]}}m(B)^{a/n}.
\]
Applying this inequality to $-f$ and adding we arrive at%
\[
ess\sup_{B}\left(  f\right)  -ess\text{ }\inf_{B}f\leq2C\left(  \frac{-\bar
{p}}{p^{\prime}}\right)  ^{1/p^{\prime}}\left\|  f\right\|  _{L^{[\bar{p},p]}%
}m(B)^{a/n}.
\]
Let $x,y\in\mathbb{R}^{n},$ and consider $B=B(x,3\left|  x-y\right|  )$ (i.e.
the ball centered at $x,$ with radius $3\left|  x-y\right|  $), then%
\begin{align*}
\left|  f(x)-f(y)\right|   &  \leq ess\sup_{B}f-ess\inf_{B}f\\
&  \leq c_{n}2\left(  \frac{-\bar{p}}{p^{\prime}}\right)  ^{1/p^{\prime}%
}C\left\|  f\right\|  _{L^{[\bar{p},p]}}\left|  x-y\right|  ^{\alpha}.
\end{align*}
At this point we could appeal to (\ref{negativa}) to conclude that%
\[
\left|  f(x)-f(y)\right|  \leq c_{n}2\left(  \frac{-\bar{p}}{p^{\prime}%
}\right)  ^{1/p^{\prime}}C\left\|  \left|  \nabla f\right|  \right\|
_{p}\left|  x-y\right|  ^{\alpha}.
\]

Similar arguments apply when dealing with Besov spaces. In this case the point
of departure is the corresponding replacement for (\ref{negativa}) that is
provided by the Besov embedding%
\[
\int\left[  \left(  \left|  f\right|  _{\mu}^{\ast\ast}(t)-\left|  f\right|
_{\mu}^{\ast}(t)\right)  t^{\frac{1}{p}-\frac{\theta}{n}}\right]  ^{q}%
\frac{dt}{t}\leq c\int\left[  t^{-\frac{\theta}{n}}K(t^{1/n},f;L^{p},\dot
{W}_{L^{p}}^{1})\right]  ^{q}\frac{dt}{t},
\]
where $\frac{1}{\bar{p}}=\frac{1}{p}-\frac{\theta}{n}.$ $\theta\in(0,1),$
$1\leq q\leq\infty.$ Notice that we don't assume anymore that $\theta p\leq
n.$

\begin{remark}
In the usual argument the use of the Lorentz spaces with negative indices was
implicit. The idea being that we can estimate $\left(  \int_{t_{1}}^{t_{2}%
}\left(  \left(  f^{\ast\ast}(t)-f^{\ast}(t)\right)  t^{1/\bar{p}}\right)
^{p}\frac{dt}{t}\right)  ^{1/p}$ through the use of%
\[
f^{\ast\ast}(t)-f^{\ast}(t)\leq c_{n}t^{1/n}\left|  \nabla f\right|
^{\ast\ast}(t).
\]
Namely,%
\begin{align*}
f^{\ast\ast}(t_{1})-f^{\ast\ast}(t_{2})  &  =\int_{t_{1}}^{t_{2}}\left(
f^{\ast\ast}(t)-f^{\ast}(t)\right)  \frac{dt}{t}\\
&  \leq\int_{t_{1}}^{t_{2}}t^{1/n-1}\left|  \nabla f\right|  ^{\ast\ast
}(t)dt\text{ (basic inequality)}\\
&  \leq\left(  \int_{t_{1}}^{t_{2}}\left|  \nabla f\right|  ^{\ast\ast}%
(t)^{p}dt\right)  ^{1/p}\left(  \int_{t_{1}}^{t_{2}}t^{p^{\prime}%
(1/n-1)}dt\right)  ^{1/p^{\prime}}\text{ (H\"{o}lder's inequality)}\\
&  \leq c_{p}\left\|  \left|  \nabla f\right|  \right\|  _{p}\left(
\int_{t_{1}}^{t_{2}}t^{p^{\prime}(1/n-1)}dt\right)  ^{1/p^{\prime}}\text{
(Hardy's inequality})\\
&  \leq c_{p}\left\|  \left|  \nabla f\right|  \right\|  _{p}\left(  \int%
_{0}^{\left|  t_{2}-t_{1}\right|  }t^{p^{\prime}(1/n-1)}dt\right)
^{1/p^{\prime}}\text{ (since }t^{p^{\prime}(1/n-1)}\text{ decreases})\\
&  =c_{p,n}\left\|  \left|  \nabla f\right|  \right\|  _{p}\left|  t_{2}%
-t_{1}\right|  ^{1/n-1/p}\\
&  =c_{p,n}\left\|  \left|  \nabla f\right|  \right\|  _{p}\left|  t_{2}%
-t_{1}\right|  ^{\alpha/n}.
\end{align*}

\end{remark}

At this point it is not difficult to reformulate many of the results in this
paper using the notion of Lorentz spaces with negative index. As an example we
simply state the following result and safely leave the details to the reader.

\begin{theorem}
Let $(\Omega,d,\mu)$ be a probability metric space that satisfies the relative
isoperimetric property and such that
\[
t^{1-1/n}\preceq I_{\Omega}(t),\text{ }t\in(0,1/2).
\]
Then, if $p>n$%
\[
b_{p}^{n/p,1}\left(  \Omega\right)  \subset L^{\bar{p},1}%
\]
where $\frac{1}{\bar{p}}=\frac{1}{p}-\frac{1}{n}.$ Moreover, if $f\in
b_{p}^{n/p,1}\left(  \Omega\right)  ,$ then $\forall B\subset\Omega,$
$f\chi_{B}\in L^{\bar{p},1},$ and $\left\Vert f\chi_{B}\right\Vert
_{L^{\bar{p},1}}\preceq\left\Vert f\right\Vert _{b_{p}^{n/p,1}\left(
\Omega\right)  },$ with constants independent of $B.$ In particular, it
follows that $f\in C(\Omega).$
\end{theorem}

\section{An interpolation inequality}

In this section we formulate the basic argument of this chapter as in
interpolation inequality.

\begin{lemma}
Suppose that $(\Omega,d,\mu)$ is a probability measure. Let $s<0,1\leq
q\leq\infty,$ and suppose that $-q^{\prime}>s.$ Then for all $f\in
L^{1}\left(  \Omega\right)  $ we have,
\[
\left\|  f\right\|  _{L^{\infty}}\leq\left(  \frac{-s}{q^{\prime}}\right)
^{1/q^{\prime}}\left\|  f\right\|  _{L^{[s,q]}}+\left\|  f\right\|  _{L^{1}}.
\]

\end{lemma}

\begin{proof}
We use the argument of the previous section verbatim. Let $0<t_{1}<t_{2}<1.$
By the fundamental theorem of calculus, we have%
\begin{align*}
\left|  f_{\mu}^{\ast\ast}\right|  (t_{1})-\left|  f\right|  _{\mu}^{\ast\ast
}(t_{2})  &  =\int_{t_{1}}^{t_{2}}\left(  \left|  f\right|  _{\mu}^{\ast\ast
}(t)-\left|  f\right|  _{\mu}^{\ast}(t)\right)  \frac{dt}{t}\\
&  =\int_{t_{1}}^{t_{2}}\left(  \left|  f\right|  _{\mu}^{\ast\ast}(t)-\left|
f\right|  _{\mu}^{\ast}(t)\right)  t^{1/s}t^{-1/s}\frac{dt}{t}\\
&  \leq\left\{  \int_{t_{1}}^{t_{2}}\left\{  \left(  \left|  f\right|  _{\mu
}^{\ast\ast}(t)-\left|  f\right|  _{\mu}^{\ast}(t)\right)  t^{1/s}\right\}
^{q}\frac{dt}{t}\right\}  ^{1/q}\left\{  \int_{t_{1}}^{t_{2}}t^{-q^{\prime}%
/s}\frac{dt}{t}\right\}  ^{1/q^{\prime}}\\
&  \leq\left(  \frac{-s}{q^{\prime}}\right)  ^{1/q^{\prime}}\left\|
f\right\|  _{L^{[s,q]}}\left|  t_{1}-t_{2}\right|  ^{-q^{\prime}/s}.
\end{align*}
Therefore letting $t_{1}\rightarrow0^{+},$ $t_{2}\rightarrow1^{-},$ we find%
\[
\left\|  f\right\|  _{L^{\infty}}-\left\|  f\right\|  _{L^{1}}\leq\left(
\frac{-s}{q^{\prime}}\right)  ^{1/q^{\prime}}\left\|  f\right\|  _{L^{[s,q]}%
}.
\]

\end{proof}

\section{Further remarks}

Good portions of the preceding discussion can be extended to the context of
real interpolation spaces. In this framework one can consider spaces that are
defined in terms of conditions on $\frac{K(t,f;\bar{X})}{t}-K^{\prime
}(t,f;\bar{X}),$ where $\bar{X}$ is a compatible pair of Banach spaces. An
example of such construction are the modified Lions-Peetre spaces defined, for
example, in \cite{ho}, \cite{jm} and the references therein. The usual
conditions defining these spaces are of the form%
\[
\left\|  f\right\|  _{[X_{0},X_{1}]_{\theta,q}}=\left\{  \int_{0}^{\infty
}(t^{-\theta}(K(t,f;X_{0},X_{1})-tK^{\prime}(t,f;X_{0},X_{1}))^{q}\frac{dt}%
{t}\right\}  ^{1/q}<\infty,
\]
where $\theta\in(0,1),q\in(0,\infty].$ Adding the end points $\theta=0,1,$
produces conditions that still make sense and are useful in analysis (cf.
\cite{mamicwk} and the references therein). Observe that when $\bar{X}%
=(L^{1},L^{\infty}),$ we have
\[
\frac{K(t,f;\bar{X})}{t}-K^{\prime}(t,f;\bar{X})=\left|  f\right|  ^{\ast\ast
}(t)-\left|  f\right|  ^{\ast}(t),
\]
and therefore%
\[
\lbrack X_{0},X_{1}]_{\theta,q}=L^{[\frac{1}{1-\theta},q]}.
\]
Therefore the discussion in this chapter suggests that it is of interest to
consider, more generally, the spaces $[X_{0},X_{1}]_{\theta,q},$ for
$\theta\in\mathbb{R}.$ In particular this may allow, in some cases, to treat
$L^{p}$ and $Lip$ conditions in a unified manner. For example, in \cite{cjm}
and \cite{mamicon} results are given that imply that for certain operators
\thinspace$T,$ that include gradients, an inequality of the form
\[
\left\|  f\right\|  _{[Y_{0},Y_{1}]_{\theta_{0},q_{0}}}\leq c\left\|
Tf\right\|  _{[X_{0},X_{1}]_{\theta_{1},q_{1}}}%
\]
can be extrapolated to a family of inequalities that involve the $[Y_{0}%
,Y_{1}]_{\theta_{0},q_{0}}$ spaces defined here. In particular%
\[
\left\|  f\right\|  _{L^{n^{\prime}}}\leq c\left\|  \nabla f\right\|  _{L^{1}%
},f\in C_{0}^{1}(R^{n})
\]
implies%
\[
K(t,f;L^{1},L^{\infty})-tK^{\prime}(t,f;L^{1},L^{\infty})\leq ct^{1/n}%
K(t,\nabla f;L^{1},L^{\infty}).
\]
Thus from one inequality we can extrapolate ``all'' the classical Sobolev
inequalities through the use of the $[Y_{0},Y_{1}]_{\theta,q}$ spaces with
$\theta$ possibly negative. To pursue this point further would take us too far
away from our main concerns in this paper, so we must leave more details and
applications for another occasion.

\chapter{Connection with the work of Garsia and his
collaborators\label{Garetal}\bigskip}

\setcounter{section}{1} In this section we shall discuss the connection of our
results with the work of Garsia and his collaborators (cf. \cite{garr},
\cite{garsiagrenoble}, \cite{garsia0}, \cite{garsiaind}, \cite{garsia},
\cite{park}, \cite{deland}...). We argue that our results can be seen as an
extension the work by Garsia \cite{garsia}, \cite{garsiaind}, and
some\footnote{The results of \cite{garsiagrenoble}, while very similar, are
formulated in terms of moduli of continuity that in some cases cannot be
readily identified with the ones we consider in this paper.} of the work by
Garsia-Rodemich \cite{garsiagrenoble}, to the metric setting. Indeed,
\cite{garsiaind}, \cite{garsia} were one of the original motivations behind
\cite{mamiproc} and some of our earlier writings.

In \cite{garsiaind} it is shown that for functions on $[0,1],$ if\footnote{The
case $p=1$ is also trivially true since%
\[
\int_{0}^{1}Q_{1}(\delta,f)\frac{d\delta}{\delta^{2}}<\infty,
\]
readily implies that $f$ is constant.} $p\geq1,$%
\begin{equation}
\left.
\begin{array}
[c]{c}%
f^{\ast}(x)-f^{\ast}(1/2)\\
f^{\ast}(1/2)-f^{\ast}(1-x)
\end{array}
\right\}  \leq\frac{4^{1/p}}{\log\frac{3}{2}}\int_{x}^{1}Q_{p}(\delta
,f)\frac{d\delta}{\delta^{1+1/p}}, \label{gar1}%
\end{equation}
(see Section \ref{sigre} in Chapter \ref{contchap} above), and where%
\[
Q_{p}(\delta,f)=\left\{  \frac{1}{\delta}\int\int_{\left\vert x-y\right\vert
<\delta}\left\vert f(x)-f(y)\right\vert ^{p}dxdy\right\}  ^{1/p}.
\]
In particular, if $\int_{0}^{1}Q_{p}(\delta,f)\frac{d\delta}{\delta^{1+1/p}%
}<\infty,$ then $f$ is essentially continuous, and in fact, $a.e.$
$x,y\in\lbrack0,1]$%
\begin{equation}
\left\vert f(x)-f(y)\right\vert \leq2\frac{4^{1/p}}{\log\frac{3}{2}}\int%
_{0}^{\left\vert x-y\right\vert }Q_{p}(\delta,f)\frac{d\delta}{\delta^{1+1/p}%
}. \label{gar2}%
\end{equation}
Moreover, in \cite{garsiaind} more general moduli of continuity based on
Orlicz spaces are considered: for a Young's function $A,$ normalized so that
$A(1)=1,$ let%
\[
Q_{A}(\delta,f)=\inf\left\{  \lambda>0:\frac{1}{\delta}\int\int_{\left\vert
x-y\right\vert <\delta}A\left(  \frac{\left\vert f(x)-f(y)\right\vert
}{\lambda}\right)  dxdy\leq1\right\}  .
\]
In \cite{garsiaind} and Deland \cite[(1.1),(1.3)]{deland} the following
analogues of (\ref{gar1}) and (\ref{gar2}) are shown to hold:%
\begin{equation}
\left.
\begin{array}
[c]{c}%
f^{\ast}(x)-f^{\ast}(1/2)\\
f^{\ast}(1/2)-f^{\ast}(1-x)
\end{array}
\right\}  \leq\frac{2}{\log\frac{3}{2}}\int_{x}^{1}Q_{A}(\delta,f)A^{-1}%
(\frac{4}{\delta})\frac{d\delta}{\delta} \label{gar3}%
\end{equation}
and%
\begin{equation}
\left\vert f(x)-f(y)\right\vert \leq c\int_{0}^{\left\vert x-y\right\vert
}Q_{A}(\delta,f)A^{-1}(\frac{4}{\delta})\frac{d\delta}{\delta}. \label{gar4}%
\end{equation}

We will show in a moment that our inequalities readily give the following
version of (\ref{gar3}) for all r.i. spaces $X[0,1]:$%
\begin{equation}
\left.
\begin{array}
[c]{c}%
f^{\ast}(x)-f^{\ast}(1/2)\\
f^{\ast}(1/2)-f^{\ast}(1-x)
\end{array}
\right\}  \leq c\int_{x}^{1}\frac{K(\delta,f;X,\dot{W}_{X}^{1})}{\phi
_{X}(\delta)}\frac{d\delta}{\delta}. \label{garneta}%
\end{equation}

To relate this inequality to Garsia's results we compare the modulus of
continuity to $K-$functionals. Thus, we let%
\[
\omega_{A}(\delta,f)=\inf\left\{  \lambda>0:\sup_{h\leq\delta}\int%
_{0}^{1-\delta}A\left(  \frac{\left\vert f(x+h)-f(x)\right\vert }{\lambda
}\right)  dx\leq1\right\}  ,\text{ }\delta\in(0,1).
\]
Then, as is well known (cf. \cite{bs}, \cite{mamiproc}),%
\begin{equation}
K(\delta,f;L_{A},\dot{W}_{L_{A}}^{1})\simeq\omega_{A}(\delta,f),
\label{garne2}%
\end{equation}
and we have

\begin{lemma}
sup$_{0<\sigma<\delta}Q_{A}(\sigma,f)\preceq K(\delta,f;L_{A},\dot{W}_{L_{A}%
}^{1}).$
\end{lemma}

\begin{proof}
To see this note that, for all $\lambda>0,$ $\delta\in(0,1),$ we have
\begin{align*}
\frac{1}{\delta}\int\int_{\{(x,y)\in\lbrack0,1]^{2}:\left|  x-y\right|
<\delta\}}A\left(  \frac{\left|  f(x)-f(y)\right|  }{2\lambda}\right)  dxdy
&  \leq\frac{1}{\delta}\int_{0}^{\delta}\int_{0}^{1-\delta}A\left(
\frac{\left|  f(x+h)-f(x)\right|  }{\lambda}\right)  dxdh\\
&  \leq\sup_{h\leq\delta}\int_{0}^{1-\delta}A\left(  \frac{\left|
f(x+h)-f(x)\right|  }{\lambda}\right)  dx.
\end{align*}
Therefore, if we let $\lambda=\omega_{A}(\delta,f),$ by the definitions,
\[
\frac{1}{\delta}\int\int_{\{(x,y)\in\lbrack0,1]^{2}:\left|  x-y\right|
<\delta\}}A\left(  \frac{\left|  f(x)-f(y)\right|  }{2\lambda}\right)
dxdy\leq1,
\]
and consequently%
\begin{align*}
Q_{A}(\delta,f)  &  \leq2\omega_{A}(\delta,f)\\
&  \preceq K(\delta,f;L_{A},\dot{W}_{L_{A}}^{1}).
\end{align*}

\end{proof}

To complete the picture let us also note that%
\[
\phi_{L_{A}}(t)=\frac{1}{A^{-1}(\frac{1}{t})}.
\]

We now show in detail (\ref{garneta}). One technical problem we have to
overcome is that the results of this paper do not apply directly for functions
on $[0,1],$ since the isoperimetric profile of $[0,1]$ is $I(t)\equiv1,$ and
therefore $I$ does not satisfy the required hypotheses to apply our general
machinery (cf. Condition \ref{condition} in Chapter \ref{preliminar}, and
\cite{mamiadv}, \cite{mamicon}). Therefore while the inequalities
[(\ref{seis}), Chapter 1], and their corresponding signed rearrangement
variants are valid (cf. Chapter \ref{contchap}), our results cannot be applied
directly. However, we will now show that our methods can be readily adapted to
yield the one dimensional result as well.

To prove [(\ref{seis}), Chapter 1] for $n=1,$ we need to establish the
following inequality (compare with [(\ref{cuatro}), Chapter 1, letting
formally $I(t)=1)$%
\[
f^{\ast\ast}(t)-f^{\ast}(t)\leq t\left(  \left|  f^{\prime}\right|  \right)
^{\ast\ast}(t),t\in(0,1).
\]
While \cite{mamiadv} formally does not cover this case, it turns out that we
can easily prove this inequality directly using the method of ``truncation by
symmetrization'', which was apparently introduced in \cite{mmp}. Indeed, a
known elementary result of Duff \cite{duff} states that%
\[
\left\|  (f^{\ast})^{\prime}\right\|  _{L^{p}[0,1]}\leq\left\|  f^{\prime
}\right\|  _{L^{p}[0,1]}.
\]
The truncation method of \cite{mmp} (cf. also \cite[discussion before
Corallaire 2.4]{ehrhard}), as it is developed in detail in \cite{leoni}, when
applied to the case $p=1,$ yields the corresponding P\'{o}lya-Szeg\"{o}
inequality (as formulated in \cite{mmp})%
\[
t\left(  (f^{\ast})^{\prime}\right)  ^{\ast\ast}(t)\leq t\left(  \left|
f^{\prime}\right|  \right)  ^{\ast\ast}(t),\text{ }t\in(0,1).
\]
We can (and will) assume without loss that $f$ is bounded, then (cf.
\cite{leoni}),
\[
t\left(  (f^{\ast})^{\prime}\right)  ^{\ast\ast}(t)=\int_{0}^{t}\left|
(f^{\ast})^{\prime}\right|  ds=f^{\ast}(0)-f^{\ast}(t)<\infty.
\]
Now, since $f^{\ast\ast}(0)=f^{\ast}(0),$ and $f^{\ast\ast}$ is decreasing, we
have
\[
f^{\ast\ast}(t)-f^{\ast}(t)\leq f^{\ast\ast}(0)-f^{\ast}(t)=f^{\ast
}(0)-f^{\ast}(t).
\]
Therefore, combining these estimates we arrive at
\[
f^{\ast\ast}(t)-f^{\ast}(t)\leq t\left(  \left|  f^{\prime}\right|  \right)
^{\ast\ast}(t),\text{ }t\in(0,1),
\]
as required. At this point the proof of Theorem \ref{main1} applies without
changes to yield for $0<t\leq1/2,$%
\[
f^{\ast\ast}(t)-f^{\ast}(t)\leq c\frac{K(t,f;X,\dot{W}_{X}^{1})}{\phi_{X}%
(t)}.
\]
Moreover, using \cite[(4.1)]{bmr} we have%
\[
f^{\ast}\left(  t/2\right)  -f^{\ast}(t)\leq2\left(  f^{\ast\ast}(t)-f^{\ast
}(t)\right)  .
\]
Thus,%
\begin{align}
f^{\ast}\left(  t/2\right)  -f^{\ast}(t)  &  \leq c\frac{K(t,f;X,\dot{W}%
_{X}^{1})}{\phi_{X}(t)}\nonumber\\
&  \leq2c\frac{K(\frac{t}{2},f;X,\dot{W}_{X}^{1})}{\phi_{X}(t)}\text{ (since
}\frac{K(t,f;X,\dot{W}_{X}^{1})}{t}\text{ and }\frac{1}{\phi_{X}(t)}\text{
decrease)}\nonumber\\
&  \leq\frac{2c}{\ln2}\int_{\frac{t}{2}}^{t}\frac{K(s,f;X,\dot{W}_{X}^{1}%
)}{\phi_{X}(s)}\frac{ds}{s}. \label{fornaio}%
\end{align}
Given $t\in(0,1/2),$ let $N=N(t)$ be such that $\frac{t}{2}\leq2^{-(N+1)}%
<t<2^{-N}\leq\frac{1}{2},$ then%
\begin{align*}
f^{\ast}(t)-f^{\ast}(1/2)  &  \leq f^{\ast}(2^{-(N+1)})-f^{\ast}(1/2)\\
&  =%
{\displaystyle\sum\limits_{j=1}^{N}}
\left(  f^{\ast}(2^{-(j+1)})-f^{\ast}(2^{-j})\right) \\
&  \leq C%
{\displaystyle\sum\limits_{j=1}^{N}}
\int_{2^{-(j+1)}}^{2^{-j}}\frac{K(s,f;X,\dot{W}_{X}^{1})}{\phi_{X}(s)}%
\frac{ds}{s}\\
&  \leq C\int_{2^{-(N+1)}}^{1/2}\frac{K(s,f;X,\dot{W}_{X}^{1})}{\phi_{X}%
(s)}\frac{ds}{s}\\
&  =C\int_{2^{-(N+1)}}^{2^{-N}}\frac{K(s,f;X,\dot{W}_{X}^{1})}{\phi_{X}%
(s)}\frac{ds}{s}+C\int_{2^{-N}}^{1/2}\frac{K(s,f;X,\dot{W}_{X}^{1})}{\phi
_{X}(s)}\frac{ds}{s}.
\end{align*}
Now,%
\[
\int_{2^{-N}}^{1/2}\frac{K(s,f;X,\dot{W}_{X}^{1})}{\phi_{X}(s)}\frac{ds}%
{s}\leq\int_{t}^{1/2}\frac{K(s,f;X,\dot{W}_{X}^{1})}{\phi_{X}(s)}\frac{ds}%
{s}.
\]
Moreover, we will show in a moment that%
\begin{equation}
\int_{2^{-(N+1)}}^{2^{-N}}\frac{K(s,f;X,\dot{W}_{X}^{1})}{s}\frac{1}{\phi
_{X}(s)}ds\leq\frac{4}{\ln(1/2)}\int_{t}^{1/2}K(s,f;X,\dot{W}_{X}^{1})\frac
{1}{\phi_{X}(s)}\frac{ds}{s}. \label{pucha}%
\end{equation}
Collecting these results we see that there exists a universal constant $c>0$
such that,%
\[
f^{\ast}(t)-f^{\ast}(1/2)\leq c\int_{t}^{1/2}K(s,f;X,\dot{W}_{X}^{1})\frac
{1}{\phi_{X}(s)}\frac{ds}{s},t\in(0,1/2).
\]
The previous inequality applied to $-f$ yields the second half of Garsia's
inequality%
\[
f^{\ast}(1/2)-f^{\ast}(1-t)\leq c\int_{t}^{1/2}K(s,f;X,\dot{W}_{X}^{1}%
)\frac{1}{\phi_{X}(s)}\frac{ds}{s},t\in(0,1/2).
\]
We complete the details of the proof of (\ref{pucha}) using various
monotonicity properties of the functions involved and the position of $t$ in
the interval:
\begin{align*}
\int_{2^{-(N+1)}}^{2^{-N}}\frac{K(s,f;X,\dot{W}_{X}^{1})}{s}\frac{1}{\phi
_{X}(s)}ds  &  \leq\frac{K(2^{-(N+1)},f;X,\dot{W}_{X}^{1})}{2^{-(N+1)}}%
\frac{1}{\phi_{X}(2^{-(N+1)})}2^{-(N+1)}\text{ (}\frac{K(r}{r}\downarrow
,\frac{1}{\phi_{X}(r)}\downarrow)\\
&  =K(2^{-(N+1)},f;X,\dot{W}_{X}^{1})\frac{1}{\phi_{X}(2^{-(N+1)})}\frac
{1}{\ln(2^{-N}/t)}\int_{t}^{2^{-N}}\frac{ds}{s}\\
&  \leq\frac{4}{\ln(1/2)}\int_{t}^{2^{-N}}K(s,f;X,\dot{W}_{X}^{1})\frac
{1}{\phi_{X}(s)}\frac{ds}{s}\text{ (}K\uparrow,\frac{r}{\phi_{X}(r)}%
\uparrow)\\
&  \leq\frac{4}{\ln(1/2)}\int_{t}^{1/2}K(s,f;X,\dot{W}_{X}^{1})\frac{1}%
{\phi_{X}(s)}\frac{ds}{s}.
\end{align*}

In particular, our results thus give versions of (\ref{gar1}), (\ref{gar2}),
(\ref{gar3}), but replacing $Q_{A}(\delta,f)$ with the usual modulus of
continuity $K(\delta,f;L_{A},\dot{W}_{L_{A}}^{1}).$ We also note that Deland
\cite{deland} found the following improvement to (\ref{gar3})%
\[
\left.
\begin{array}
[c]{c}%
f^{\ast}(x)-f^{\ast}(1/2)\\
f^{\ast}(1/2)-f^{\ast}(1-x)
\end{array}
\right\}  \preceq\int_{x}^{1}Q_{A}(\delta,f)dA^{-1}(\frac{c}{\delta}),\text{
}0<x<1/2.
\]
This is of particular interest when dealing with the space $X=e^{L^{2}}.$
Indeed, in this case $A(t)=e^{t^{2}}-1,$ and therefore%
\[
\phi_{X}(t)=\frac{1}{\left(  \ln\frac{e}{t}\right)  ^{1/2}}.
\]
Consequently, from (\ref{gar4}) (or (\ref{garneta})) one finds that a
sufficient condition for continuity can be formulated as: there exists
$0<a<1,c>0,$ such that%
\begin{equation}
\int_{0}^{a}Q_{A}(\delta,f)\left(  \ln\frac{c}{\delta}\right)  ^{1/2}%
\frac{d\delta}{\delta}<\infty. \label{captu1}%
\end{equation}
On the other hand, Deland's improved condition for continuity replaces
(\ref{captu1}) by%
\begin{equation}
\int_{0}^{a}Q_{A}(\delta,f)\frac{d\delta}{\left(  \ln\frac{c}{\delta}\right)
^{1/2}\delta}<\infty. \label{captu2}%
\end{equation}

In our formulation (\ref{captu2}) corresponds to a condition of the form%
\[
\int_{0}^{a}K(\delta,f;L_{A},\dot{W}_{L_{A}}^{1})d(\frac{1}{\phi_{A}%
(t)})<\infty.
\]
While we don't have any new insight to add to Deland's improvement we should
point out here that Deland's improvement is automatic for spaces far away from
$L^{\infty},$ in the sense that $\underline{\alpha}_{\Lambda(X)}>0.$ Indeed,
we have

\begin{lemma}
Suppose that $X=X[0,1]$ is a r.i. space such that $\underline{\alpha}%
_{\Lambda(X)}>0.$ Then there exists a re-norming of $X,$ that we shall call
$\bar{X},$ such that
\begin{equation}
\int_{0}^{1}K(\delta,f;\bar{X},\dot{W}_{\bar{X}}^{1})d(\frac{1}{\phi_{\bar{X}%
}(\delta)})<\infty\Longleftrightarrow\int_{0}^{1}\frac{K(\delta,f;X,\dot
{W}_{X}^{1})}{\phi_{X}(\delta)}\frac{d\delta}{\delta}<\infty. \label{captu4}%
\end{equation}

\end{lemma}

\begin{proof}
Let $\bar{\phi}(t)=\int_{0}^{t}\phi_{X}(s)\frac{ds}{s},$ then, since
$\frac{\phi_{X}(s)}{s}$ decreases, we have $\bar{\phi}(t)\geq\phi_{X}(t),$ and%
\[
([-\bar{\phi}(t)]^{-1})^{\prime}=\frac{1}{\bar{\phi}(t)^{2}}\frac{\phi_{X}%
(t)}{t}\leq\frac{1}{\bar{\phi}(t)t}\leq\frac{1}{t\phi_{X}(t)}.
\]
Moreover, since $\underline{\alpha}_{\Lambda(X)}>0,$ we have (cf. \cite[Lemma
2.1]{Sh})%
\[
\bar{\phi}(t)\preceq\phi_{X}(t).
\]
Therefore there exists an equivalent re-norming of $X$, which we shall call
$\bar{X},$ such that%
\[
\phi_{X}(t)\simeq\phi_{\bar{X}}(t)=\bar{\phi}(t).
\]
Moreover, we clearly have%
\[
K(\delta,f;X,\dot{W}_{X}^{1})\simeq K(\delta,f;\bar{X},\dot{W}_{\bar{X}}%
^{1}).
\]
We can also see that,
\begin{align*}
\left(  \left[  \phi_{\bar{X}}(t)\right]  ^{-1}\right)  ^{\prime}  &
=([-\bar{\phi}(t)]^{-1})^{\prime}\\
&  =\bar{\phi}(t)^{-2}\frac{\phi_{X}(t)}{t}\\
&  \simeq\frac{1}{\left(  \phi_{X}(t)\right)  ^{2}}\frac{\phi_{X}(t)}{t}\\
&  \simeq\frac{1}{\phi_{X}(t)t}.
\end{align*}
Consequently (\ref{captu4}) holds when $\underline{\alpha}_{\Lambda(X)}>0$.
\end{proof}

On the other hand Deland's improvement does not follow from the previous
Lemma, since from the point of view of the theory of indices
$\underline{\alpha}_{\Lambda(e^{L^{2}})}=0.\,$For more details on how to
overcome this difficulty for spaces close to $L^{\infty}$ we must refer to
Deland's thesis \cite{deland}.

For applications to Fourier series, the appropriate moduli of continuity
defined for periodic functions on, say, $[0,2\pi],$ are defined by (cf.
\cite{garsiaind}, \cite[(1.1),(1.3)]{deland})
\[
W_{A}(h,f)=\inf\left\{  \lambda>0:\int_{0}^{2\pi}A\left(  \frac{\left|
f(x+h)-f(x)\right|  }{\lambda}\right)  )dx\leq1\right\}  .
\]
Then, we also have%
\begin{align*}
\sup_{\sigma<\delta}Q_{A}(\sigma,f)  &  \preceq K(\delta,f;L_{A}[0,2\pi
],\dot{W}_{L_{A}}^{1}[0,2\pi])\\
&  \simeq\sup_{h\leq\delta}W_{A}(h,f).
\end{align*}
It follows from our work that the results of \cite{garsiaind} can be now
extended to r.i. spaces. In this connection we note that (just like in
\cite{garsiaind} for $L^{p}$ spaces) one could also use the boundedness of the
Hilbert transform on r.i. spaces where one has control of the Boyd indices
(cf. \cite{boyd}, \cite{bs}). However, to continue with this topic will take
us too far away from our main concerns here so we must leave the discussion
for another occasion.

For further applications to: the path continuity of stochastic processes,
Fourier series, random Fourier series and embeddings we refer to
\cite{garsiagrenoble}, \cite{garsiaind}, \cite{garsia}, \cite{deland} and the
references therein. Moreover, under suitable assumptions on the connection
between the isoperimetric profile and the measure of balls (cf. \cite{tes})
one can also formulate the Besov conditions as entropy conditions as it is
customarily done in probability (cf. the discussion in Pisier \cite[cf.
Remarque, p 14.]{pis}).

\chapter{Appendix: Some remarks on the calculation of $K-$%
functionals\label{Kfuntionals}}

\section{Introduction}

It seemed to us useful to collect for our reader some known computations of
$K-$functionals of the form $K(t,f;X(\Omega),\dot{W}_{X}^{1}(\Omega)),$ where
$X$ is a r.i. space. We don't claim any originality, but we provide detailed
proofs when we could not find suitable references.

In the Euclidean case, for smooth (Lip) domains, these estimates are well
known for $L^{p}$ spaces (cf. \cite{bs}, \cite{js}, \cite{tr}), and can be
readily extended to r.i. spaces (cf. \cite{mamiproc}):%
\begin{align*}
K(t,f;X(\Omega),\dot{W}_{X}^{1}(\Omega)  &  \simeq\omega_{X}(t,f)\\
&  =\sup_{\left|  h\right|  \leq t}\left\|  \left(  f(.+h)-f(.)\right)
\chi_{\Omega(h)}\right\|  _{X}%
\end{align*}
where%
\[
\Omega(h)=\{x\in\Omega:x+th\in\Omega,0\leq t\leq1\}.
\]

Consider $(\mathbb{R}^{n},\left|  \cdot\right|  ,d\gamma_{n}),$ i.e.
$\mathbb{R}^{n}$ with Gaussian measure. The fact that this measure is not
translation invariant makes the computation of the $K$-functional somewhat
more complicated. We discuss the necessary modifications in some detail for
$n=1.$

We consider spaces on $(\mathbb{R},\left|  \cdot\right|  ,d\gamma_{1}).$ Let
$p\in\lbrack1,\infty],$ and let%
\[
K_{\gamma}(t,f,L^{p},\dot{W}_{p}^{1})=\inf\left\{  \left\|  f-g\right\|
_{L^{p}(\mathbb{R},d\gamma_{1})}+t\left\|  g^{\prime}\right\|  _{L^{p}%
(d\gamma_{1})}\right\}  .
\]
This functional was studied by the approximation theory community (cf.
Ditzian-Totik \cite{ditzian}, Ditzian-Lubinsky \cite{dil} and the references
therein). For example, from \cite[page 183]{ditzian}, we have%
\begin{align}
K_{\gamma}(t,f,L^{p},\dot{W}_{p}^{1})  &  \simeq\sup_{0<h\leq t}\left\|
f(.+h)-f(.)\right\|  _{L^{p}\left(  (-\frac{1}{2h},\frac{1}{2h}),d\gamma
_{1}\right)  }+\inf_{c}\left\|  f-c\right\|  _{L^{p}\left(  (\frac{1}%
{2t},\infty),d\gamma_{1}\right)  }\nonumber\\
&  +\inf_{c}\left\|  f-c\right\|  _{L^{p}\left(  (-\infty,\frac{-1}%
{2t}),d\gamma_{1}\right)  } \label{kaf0}%
\end{align}
The main part of the right hand side of (\ref{kaf0}) is the modulus
\[
\Omega_{\gamma}(t,f)=\sup_{0<h\leq t}\left\|  f(.+h)-f(.)\right\|
_{L^{p}\left(  (-\frac{1}{2h},\frac{1}{2h}),d\gamma_{1})\right)  }.
\]
Indeed, $\Omega_{\gamma}(t,f)$ controls the characterization of the
corresponding interpolation spaces. For example, it follows (cf. \cite[Theorem
11.2.5]{ditzian}) that for $\theta\in(0,1),$
\begin{equation}
K_{\gamma}(t,f,L^{p},\dot{W}_{p}^{1})=O(t^{\theta})\Longleftrightarrow
\Omega_{\gamma}(t,f)=O(t^{\theta}). \label{kaf1}%
\end{equation}
More generally, a similar result holds for $(\mathbb{R},d\gamma_{\lambda})$,
where for $\lambda>1,$ $d\gamma_{\lambda}(x)=e^{-x^{\lambda}}dx.$ Indeed, in
this case (\ref{kaf0}) holds replacing $\frac{1}{2h}$ throughout by $\frac
{1}{\lambda h^{1/\left(  1-\lambda\right)  }}:$%
\begin{align}
K_{\gamma_{\lambda}}(t,f,L^{p},\dot{W}_{p}^{1})  &  \simeq\sup_{0<h\leq
t}\left\|  f(.+h)-f(.)\right\|  _{L^{p}\left(  (-\frac{1}{\lambda h^{1/\left(
1-\lambda\right)  }},\frac{1}{\lambda h^{1/\left(  1-\lambda\right)  }%
}),d\gamma_{\lambda}\right)  }+\label{emancipada}\\
&  \inf_{c}\left\|  f-c\right\|  _{L^{p}\left(  (\frac{1}{\lambda t^{1/\left(
1-\lambda\right)  }},\infty),d\gamma_{\lambda}\right)  }+\inf_{c}\left\|
f-c\right\|  _{L^{p}\left(  (-\infty,\frac{-1}{\lambda t^{1/\left(
1-\lambda\right)  }}),d\gamma_{\lambda}\right)  }.\nonumber
\end{align}

Again the main part of the right hand side is the modulus of continuity
\[
\Omega_{\gamma_{\lambda}}(t,f)=\sup_{0<h\leq t}\left\|  f(.+h)-f(.)\right\|
_{L^{p}((-\frac{1}{\lambda h^{1/\left(  1-\lambda\right)  }},\frac{1}{\lambda
h^{1/\left(  1-\lambda\right)  }}),d\gamma_{\lambda})}.
\]
Likewise the analogue of (\ref{kaf1}) holds.

More generally, the estimates above have been extended to the class of the so
called ``Freud weights'' of the form $w(x)=e^{Q(x)}.$ Here we assume that $Q$
is a given function in $C^{1}(\mathbb{R})$ such that $Q$ is even,
$\lim_{x\rightarrow\infty}Q^{\prime}(x)=\infty,$ and such that there exists
$A>0,$ such that $Q^{\prime}(x+1)\leq AQ^{\prime}(x),$ for all $x>0.$ For
complete details we refer again to \cite{ditzian}.

Although one would expect that the $n-$dimensional extensions of the
$K-$func\-tio\-nal estimates above should not be very difficult, we have not
been able find references, even after consultation with many experts. On the
other hand, as is well known, one can avoid this difficulty through the use of
an alternate characterization of $K-$functionals for Gaussian measure using
appropriate semigroups. We provide some details in the next sections.

For the last section of this chapter, connecting semigroups and Gaussian Besov
spaces, we are grateful to Stefan Geiss and Alessandra Lunardi for precious
information, in particular, for pointing out the relevant literature. In this
last regard we also refer to the recent paper by Geiss-Toivola \cite{geto}. In
connection with this last section we should mention the recent formulation of
fractional Poincar\'{e} inequalities in \cite{moru}.

\section{Semigroups and Interpolation}

A family $\{G(t)\}_{t>0}$ of operators on a Banach space $A$ is called an
equibounded, strongly continuous semigroup if the following conditions are
satisfied:
\[
(i)\text{ \ \ \ \ \ }G(t+s)=G(t)G(s)
\]%
\[
(ii)\text{ \ \ \ There exists }M>0\text{ such that \ \ \ }\sup_{t>0}\left\|
G(t)\right\|  _{A\rightarrow A}\leq M
\]%
\[
(iii)\text{ \ \ \ \ }\lim_{t\rightarrow0}\left\|  G(t)a-a\right\|
_{A}=0,\text{ for }a\in A.
\]
The infinitesimal generator $\Lambda$ is defined on
\[
D(\Lambda)=\{a\in A:\lim_{t\rightarrow0}\frac{G(t)a-a}{t}\text{ exists}\}
\]
by
\[
\Lambda a=\lim_{t\rightarrow0}\frac{G(t)a-a}{t}.
\]
We consider
\[
K(t,a;A,D(\Lambda))=\inf\{\left\|  a_{0}\right\|  _{A}+t\left\|  \Lambda
a_{1}\right\|  _{A}:a=a_{0}+a_{1}\}.
\]
For equibounded strongly continuous semigroups we have the well known
estimate, apparently going back to Peetre \cite{peebra} (cf. \cite{bl},
\cite{ditiv}, \cite{pee})
\begin{equation}
K(t,a;A,D(\Lambda))\simeq\sup_{0<s\leq t}\left\|  (G(s)-I)a\right\|  _{A},
\label{Kf}%
\end{equation}
where $I=$identity operator on $A.$ The proof can be accomplished using the
decomposition
\[
a=\underset{a_{0}\epsilon A}{\underbrace{\left(  a-\frac{1}{t}\int_{0}%
^{t}G(s)ads\right)  }}+\underset{a_{1}\epsilon D(\Lambda)}{\underbrace{\frac
{1}{t}\int_{0}^{t}G(s)ads}}.
\]
Note that the right hand side of (\ref{Kf}) should be thought as a generalized
modulus of continuity which in the classical case corresponds to the semigroup
of translations $G(s)f=f(s+\cdot).$

In \cite[Corollary 7.2]{ditiv} the following alternate estimates were pointed
out%
\begin{align*}
K(t,a;A,D(\Lambda))  &  \simeq\frac{1}{t}\int_{0}^{t}\left\|
(G(s)-I)a\right\|  _{A}ds\\
&  \simeq\frac{1}{t}\left\|  \int_{0}^{t}(G(s)-I)ads\right\|  _{A}\\
&  \simeq\frac{1}{t}\left\|  \int_{t/2}^{t}(G(s)-I)ads\right\|  _{A}\\
&  \simeq\frac{1}{t}\int_{t/2}^{t}\left\|  (G(s)-I)ads\right\|  _{A}.
\end{align*}

The preceding estimates can be further improved under more restrictions on the
semigroups. Recall that a semigroup is said to be holomorphic if:
\[
(i)\text{ \ \ \ \ }G(t)a\in D(\Lambda)\text{ \ \ for all }a\in A,
\]
and
\[
(ii)\text{ \ There exists a constant }C>0\text{ such that }\left\|  \Lambda
G(t)a\right\|  _{A}\leq C\frac{\left\|  a\right\|  _{A}}{t},\forall a\in
A,t>0.
\]
In \cite{ditiv} it is shown that for holomorphic semigroups we have the
following improvement of (\ref{Kf})
\begin{equation}
K(t,a;A,D(\Lambda))\approx\left\|  (G(t)-I)a\right\|  _{A}. \label{holomo}%
\end{equation}
Peetre \cite[page 33]{pee} pointed out, without proof, that for holomorphic
semigroups we also have%
\[
K(t,a;A,D(\Lambda))\simeq\sup_{s\leq t}\left\|  \Lambda G(s)a\right\|  _{A}.
\]
However, we can only prove a somewhat weaker result here.

\begin{lemma}
\label{lemalimon}Suppose that $\{G(t)\}_{t>0}$ is an holomorphic semigroup on
a Banach space $A$. Let $c_{1}>1,$ be such that for all $t>0,$ and for all
$a\in A$ (cf. (\ref{holomo}) above),
\[
\frac{1}{c_{1}}\left\|  (G(t)-I)a\right\|  _{A}\leq K(t,a;A,D(\Lambda))\leq
c_{1}\left\|  (G(t)-I)a\right\|  _{A}.
\]
Then, there exist absolute constants $c_{2}(m),c_{3}(m)$ such that for all
$t>0,$ for all $a\in A,$ for all $m\geq2,$
\begin{equation}
K(t,a;A,D(\Lambda))-c_{1}^{2}K(\frac{t}{m},a;A,D(\Lambda))\leq c_{2}%
(m)\sup_{s\leq t}\left\|  \Lambda G(s)a\right\|  _{A}\leq c_{3}%
(m)K(t,a;A,D(\Lambda)). \label{beta}%
\end{equation}

\end{lemma}

\begin{proof}
It is easy to show that there exists an absolute constant $C>0$ such that%
\begin{equation}
\sup_{s\leq t}\left\|  \Lambda G(s)a\right\|  _{A}\leq CK(t,a;A,D(\Lambda)).
\label{sesigue}%
\end{equation}
Indeed, let $a=a_{0}+a_{1},$ be any decomposition with $a_{0}\in A,$ $a_{1}\in
D(\Lambda).$ Then, using the properties of holomorphic semigroups, we have%
\begin{align*}
\sup_{s\leq t}s\left\|  \Lambda G(s)a\right\|  _{A}  &  \leq\sup_{s\leq
t}s\left\|  G(s)\Lambda a_{0}\right\|  _{A}+\sup_{s\leq t}s\left\|
G(s)\Lambda a_{1}\right\|  _{A}\\
&  \leq C\left(  \left\|  a_{0}\right\|  _{A}+t\left\|  \Lambda a_{1}\right\|
_{A}\right)  .
\end{align*}
Consequently, (\ref{sesigue}) follows by taking infimum over all such decompositions.

We now prove the left hand side of (\ref{beta}). Observe that, for $t>0$ we
have $G(t)a\in D(\Lambda),$ therefore we can write $\frac{d}{dt}%
(G(t)a)=\Lambda G(t)a.$ Consequently, for all $m\geq2,$
\begin{align*}
K(t,a;A,D(\Lambda))  &  \leq c_{1}\left\|  (G(t)-I)a\right\|  _{A}\\
&  \leq c_{1}\left\|  \int_{0}^{t/m}\Lambda G(s)ads\right\|  _{A}%
+c_{1}\left\|  \int_{t/m}^{t}\Lambda G(s)ads\right\|  _{A}\\
&  =c_{1}\left\|  (G(t/m)-I)a\right\|  _{A}+c_{1}\left\|  \int_{t/m}^{t}%
\frac{s}{s}\Lambda G(s)ads\right\|  _{A}\\
&  \leq c_{1}\left\|  (G(t/m)-I)a\right\|  _{A}+c_{1}\int_{t/m}^{t}\frac{1}%
{s}\left\|  s\Lambda G(s)a\right\|  _{A}ds\\
&  \leq c_{1}\left\|  (G(t/m)-I)a\right\|  _{A}+c_{1}\frac{m}{t}\sup_{s\leq
t}\left\|  s\Lambda G(s)a\right\|  _{A}\frac{(m-1)}{m}t\\
&  \leq c_{1}^{2}K(\frac{t}{m},a;A,D(\Lambda))+c_{1}(m-1)\sup_{s\leq
t}\left\|  s\Lambda G(s)a\right\|  _{A},
\end{align*}
as we wished to show.
\end{proof}

Recall the definition of real interpolation spaces. Let $\theta\in(0,1),$
$q\in(0,\infty),$
\[
(A,D(\Lambda))_{\theta,q}=\left\{  a\in A:\left\|  a\right\|  _{(A,D(\Lambda
))_{\theta,q}}^{q}=\int_{0}^{\infty}\left(  t^{-\theta}K(s,a;A,D(\Lambda
))\right)  ^{q}\frac{dt}{t}<\infty\right\}  ,
\]
and%
\[
(A,D(\Lambda))_{\theta,\infty}=\left\{  a\in A:\left\|  a\right\|
_{(A,D(\Lambda))_{\theta,\infty}}=\sup_{t>0}\left\{  t^{-\theta}%
K(s,a;A,D(\Lambda))\right\}  <\infty\right\}  .
\]

From the previous Lemma we see that

\begin{proposition}
\label{carbao}Suppose that $\{G(t)\}_{t>0}$ is an holomorphic semigroup on a
Banach space $A.$ Then $(A,D(\Lambda))_{\theta,q}$ can be equivalently
described by%
\[
(A,D(\Lambda))_{\theta,q}=\left\{  a:\left\{  \int_{0}^{\infty}\left(
t^{-\theta}\sup_{s\leq t}\left\|  \Lambda G(s)a\right\|  _{A}\right)
^{q}\frac{dt}{t}\right\}  ^{1/q}<\infty\right\}  ,
\]
with the obvious modification if $q=\infty,$ and where the constants of the
underlying norm equivalences depend only on $\theta.$
\end{proposition}

\begin{proof}
One part follows readily from (\ref{sesigue}). For the less trivial inclusion
we proceed as follows. Given $\theta\in(0,1),$ select $m$ such that
$m^{-\theta}c_{1}^{2}<1.$ Then from Lemma \ref{lemalimon}, there exists an
absolute $c_{2}(m)>0$ such that%
\[
K(t,a;A,D(\Lambda))\leq c_{2}(m)\sup_{s\leq t}\left\|  \Lambda G(s)a\right\|
_{A}+c_{1}^{2}K(\frac{t}{m},a;A,D(\Lambda))
\]
Thus%
\begin{align*}
\left(  \int_{0}^{\infty}\left(  t^{-\theta}K(t,a;A,D(\Lambda))\right)
^{q}\frac{dt}{t}\right)  ^{1/q}  &  \leq c_{2}(m)\left(  \int_{0}^{\infty
}\left(  t^{-\theta}\sup_{s\leq t}\left\|  \Lambda G(s)a\right\|  _{A}\right)
^{q}\frac{dt}{t}\right)  ^{1/q}\\
&  +c_{1}^{2}\left(  \int_{0}^{\infty}(t^{-\theta}K(\frac{t}{m},a;A,D(\Lambda
)))^{q}\frac{dt}{t}\right)  ^{1/q}\\
&  =c_{2}(m)\left(  \int_{0}^{\infty}\left(  t^{-\theta}\sup_{s\leq t}\left\|
\Lambda G(s)a\right\|  _{A}\right)  ^{q}\frac{dt}{t}\right)  ^{1/q}\\
&  +c_{1}^{2}m^{-\theta}\left(  \int_{0}^{\infty}\left(  t^{-\theta
}K(t,a;A,D(\Lambda))\right)  ^{q}\frac{dt}{t}\right)  ^{1/q}.
\end{align*}

Hence,%
\[
\left(  \int_{0}^{\infty}\left(  t^{-\theta}K(t,a;A,D(\Lambda))\right)
^{q}\frac{dt}{t}\right)  ^{1/q}\leq(1-c_{1}^{2}m^{-\theta})^{-1}%
c_{2}(m)\left(  \int_{0}^{\infty}\left(  t^{-\theta}\sup_{s\leq t}\left\|
\Lambda G(s)a\right\|  _{A}\right)  ^{q}\frac{dt}{t}\right)  ^{1/q}.
\]

\end{proof}

We have the following well known result (cf. \cite{bube}, \cite{bl})

\begin{theorem}
\label{markante}Let $\{G(t)\}_{t>0}$ be an equibounded, strongly continuous
semigroup on the Banach space $A.$ Let $\theta\in(0,1),q\in(0,\infty];$ then
(with the usual modifications when $q=\infty)$

(i)%
\[
(A,D(\Lambda))_{\theta,q}=\left\{  a:\left\{  \int_{0}^{\infty}\left(
t^{-\theta}\sup_{0<s\leq t}\left\|  G(s)a-a\right\|  _{A}\right)  ^{q}%
\frac{dt}{t}\right\}  ^{1/q}<\infty\right\}  .
\]
(ii) Moreover, if the semigroup is analytic then we also have the following
characterizations (with the usual modifications when $q=\infty)$

(ii$_{1})$
\[
(A,D(\Lambda))_{\theta,q}=\left\{  a:\left\{  \int_{0}^{\infty}\left(
t^{-\theta}\left\|  G(t)a-a\right\|  _{A}\right)  ^{q}\frac{dt}{t}\right\}
^{1/q}<\infty\right\}  ,
\]
(ii$_{2})$%
\[
(A,D(\Lambda))_{\theta,q}=\left\{  a\in A:\int_{0}^{\infty}\left(  t^{-\theta
}\sup_{s\leq t}s\left\|  \Lambda G(s)a\right\|  _{A}\right)  ^{q}\frac{dt}%
{t}<\infty\right\}  ,
\]

(ii$_{3})$%
\[
(A,D(\Lambda))_{\theta,q}=\left\{  a\in A:\int_{0}^{\infty}\left(
t^{1-\theta}\left\|  \Lambda G(t)a\right\|  _{A}\right)  ^{q}\frac{dt}%
{t}<\infty\right\}  .
\]

\end{theorem}

\begin{proof}
The characterizations $(i)-(ii_{1})-(ii_{2})$ follow (respectively) from
(\ref{Kf}), (\ref{holomo}) and Proposition \ref{carbao}. To prove $(ii_{3})$
we remark that, on the one hand,%
\[
\int_{0}^{\infty}\left(  t^{1-\theta}\left\|  \Lambda G(t)a\right\|
_{A}\right)  ^{q}\frac{dt}{t}\leq\int_{0}^{\infty}\left(  t^{-\theta}%
\sup_{s\leq t}s\left\|  \Lambda G(s)a\right\|  _{A}\right)  ^{q}\frac{dt}{t}.
\]
On the other hand, since $\frac{d}{dt}(G(t)a)=\Lambda G(t)a,$
\begin{align*}
\int_{0}^{\infty}\left(  t^{-\theta}\left\|  G(t)a-a\right\|  _{A}\right)
^{q}\frac{dt}{t}  &  =\int_{0}^{\infty}\left(  t^{-\theta}\left\|  \int%
_{0}^{t}\Lambda G(s)ads\right\|  _{A}\right)  ^{q}\frac{dt}{t}\\
&  \leq\int_{0}^{\infty}\left(  t^{-\theta}\int_{0}^{t}\left\|  \Lambda
G(s)ads\right\|  _{A}\right)  ^{q}\frac{dt}{t}\\
&  \leq c_{\theta,q}\int_{0}^{\infty}\left(  t^{1-\theta}\left\|  \Lambda
G(t)a\right\|  _{A}\right)  ^{q}\frac{dt}{t},
\end{align*}
where the last step follows from Hardy's inequality.
\end{proof}

\begin{remark}
Related interpolation spaces (obtained by the ``complex method'') can be
characterized, under suitable conditions, using functional calculus. By the
known relations between these different interpolation methods one can obtain
further characterizations and embedding theorems for the real method (cf.
\cite{xiao}). In this setting fractional powers of the infinitesimal generator
$\Lambda,$ play the role of fractional derivatives. We must refer to \cite{tr}
and \cite{pee} for a complete treatment.
\end{remark}

\section{Specific Semigroups}

Two basic examples of semigroups on $L^{p}((\mathbb{R}^{n}),d\gamma_{n}),$
which are relevant for this paper are given by

1. Ornstein-Uhlenbeck semigroup, defined by%
\[
G(t)f(x)=(1-e^{-2t})^{-n/2}\int e^{-\frac{e^{-2t}(\left|  x\right|
^{2}+\left|  y\right|  ^{2}-2<x,y>)}{1-e^{-2t}}}f(y)d\gamma_{n}(y),
\]
with generator%
\[
\Lambda=\frac{1}{2}\Delta_{x}-<x,\nabla_{x}>.
\]

2. Poisson-Hermite semigroup%
\[
P_{t}f(x)=\frac{1}{\sqrt{\pi}}\int_{0}^{\infty}\frac{e^{-s}}{\sqrt{s}}%
G(\frac{t^{2}}{4s})f(x)ds,
\]
with generator%
\[
\Lambda_{1/2}=-(-\Lambda)^{1/2}.
\]
For example, $P_{t}$ on $L^{\infty}(\mathbb{R}^{n})$ is analytic although not
strongly continuous. Restricting $P_{t}$ to $\widetilde{L^{\infty}%
(\mathbb{R}^{n})}$ , the subspace of elements of $L^{\infty}(\mathbb{R}^{n})$
such that $\lim\left\Vert P_{t}f-f\right\Vert _{\infty}=0,$ remedies this
deficiency and we have (cf. \cite{tr})%
\[
(L^{\infty}(\mathbb{R}^{n}),d\gamma_{n}),D(\Lambda_{1/2}))_{\theta,\infty
}=(\widetilde{L^{\infty}(\mathbb{R}^{n})},d\gamma_{n}),D(\Lambda
_{1/2}))_{\theta,\infty}=Lip_{\theta}(\mathbb{R}^{n}).
\]
In particular, it follows from Theorem \ref{markante} that $f\in Lip_{\theta
}(\mathbb{R}^{n}),$ iff%
\[
\left\Vert P_{t}f-f\right\Vert _{\infty}=0(t^{\theta}).
\]
For other characterizations of Besov spaces we must refer to \cite{pee},
\cite{tr} and the references therein. For a treatment of fractional
derivatives in Gaussian Lipschitz spaces using semigroups and classical
analysis we refer to \cite{gaturb}.

\backmatter

\end{document}